\documentclass[reqno]{amsart}
\usepackage{mathrsfs}
\usepackage{amsfonts}
\usepackage[centertags]{amsmath}

\usepackage{amssymb}
\usepackage{amsthm}
\usepackage{graphicx}
\usepackage{caption}
\usepackage{dynkin-diagrams}
\usepackage[colorlinks]{hyperref}
\hypersetup{
bookmarksnumbered,
pdfstartview={FitH},
breaklinks=true,
linkcolor=blue,
urlcolor=blue,
citecolor=blue,
bookmarksdepth=2
}
\usepackage[sorted,compressed-cites,sorted-cites,initials]{amsrefs}
\usepackage{enumitem}
\usepackage[margin=1in]{geometry}
\usepackage{tikz-cd}
\usepackage{rotating}
\usetikzlibrary{matrix,arrows,decorations.pathmorphing}
\usepackage{tikz-cd}
\usetikzlibrary{calc}
\usetikzlibrary{decorations.pathmorphing}

\tikzset{curve/.style={settings={#1},to path={(\tikztostart)
    .. controls ($(\tikztostart)!\pv{pos}!(\tikztotarget)!\pv{height}!270:(\tikztotarget)$)
    and ($(\tikztostart)!1-\pv{pos}!(\tikztotarget)!\pv{height}!270:(\tikztotarget)$)
    .. (\tikztotarget)\tikztonodes}},
    settings/.code={\tikzset{quiver/.cd,#1}
        \def\pv##1{\pgfkeysvalueof{/tikz/quiver/##1}}},
    quiver/.cd,pos/.initial=0.35,height/.initial=0}

\tikzset{tail reversed/.code={\pgfsetarrowsstart{tikzcd to}}}
\tikzset{2tail/.code={\pgfsetarrowsstart{Implies[reversed]}}}
\tikzset{2tail reversed/.code={\pgfsetarrowsstart{Implies}}}
\tikzset{no body/.style={/tikz/dash pattern=on 0 off 1mm}}
\newcommand{\ceil}[1]{\lceil #1\rceil}
\newcommand{\floor}[1]{\lfloor #1\rfloor}

\newcommand{\Bfloor}[1]{\Big\lfloor #1\Big\rfloor}

\usepackage{listing}
\usepackage{color}
\usepackage{float}
\usepackage[scr=boondox]{mathalpha}

\newcommand{\He}{\mathscr{H\!\!e\!c}}
\newcommand{\Co}{\mathscr{C\!o\!x}}
\DeclareMathOperator{\supp}{supp}
\newcommand{\brd}[2]{\underbrace{#1\cdots}_{#2}}

\newcommand{\OEIS}[1]{\href{https:oeis.org/#1}{#1}}
\newcommand{\mP}{\mathsf P}
\newcommand{\la}{\langle}
\newcommand{\ra}{\rangle}
\newcommand{\wh}[1]{\widehat{#1}}

\newcommand{\ascprod}{\mathop{\overrightarrow{\prod}}}
\newcommand{\dscprod}{\mathop{\overleftarrow{\prod}}}
\newcommand{\ascprodst}{\mathop{\overrightarrow{\prod}^\times}}
\newcommand{\dscprodst}{\mathop{\overleftarrow{\prod}^\times}}
\newcommand{\prodst}{\mathop{{\prod}^\times}}

\usepackage{xstring}
\makeatletter
\newcommand{\partref}[1]{\StrCut{#1}{.}\@pref@\@suff@\ref{\@pref@}\StrLen{\@suff@}[\@slen@]\ifnum\@slen@>0\ref{#1}\fi}
\makeatother

\allowdisplaybreaks[4]
\usetikzlibrary{matrix,arrows,decorations.pathmorphing}

\newtheorem{theorem}{Theorem}[section]
\newtheorem{corollary}[theorem]{Corollary}
\newtheorem{lemma}[theorem]{Lemma}
\newtheorem{proposition}[theorem]{Proposition}
\theoremstyle{definition}
\newtheorem{definition}[theorem]{Definition}
\newtheorem{remark}[theorem]{Remark}
\newtheorem{example}[theorem]{Example}

\newtheorem{problem}{Problem}

\newtheorem{conjecture}[theorem]{Conjecture}

\newcommand{\cx}[2]{c_{#1\to #2}}

\newcommand{\cxr}[2]{c_{#1\leftarrow #2}}

\newenvironment{enmalph}{\begin{enumerate}[label={\rm(\alph*)},leftmargin=*]}{\end{enumerate}}

\SetLabelAlign{center}{\hfil#1\hfil}
\newenvironment{enmroman}{\begin{enumerate}[label={\rm(\roman*)},align=center,leftmargin=*]}{\end{enumerate}}

\numberwithin{equation}{section}
\newcommand{\Cox}[1]{\mathbf{Cox}(#1)}

\newcommand{\lie}[1]{{\mathfrak{#1}}}
\newcommand{\ZZ}{{\mathbb Z}}

\DeclareMathOperator{\id}{id}

\DeclareMathOperator{\ad}{ad}

\DeclareMathOperator{\Br}{Br}

\DeclareMathOperator{\Hom}{Hom}

\DeclareMathOperator{\SQF}{SQF}
\makeatletter
\renewcommand*{\Hdynkin}{\Adynkin\dynkinEdgeLabel{\numexpr\dynkin@rank-1}{\dynkin@rank}{5}}
\newcommand{\plink}[1]{\hypertarget{#1}{}\label{page:#1}}
\makeatother
\pgfkeys{/Dynkin diagram,
edge length=1.0cm,text style/.style={scale=0.8},Coxeter,root radius=0.07}

\begin{document}

\title{Hecke monoids, their homomorphisms
and parabolicity}
\author{Arkady Berenstein, Jacob Greenstein and Jian-Rong Li}
\address{Arkady Berenstein, Department of Mathematics, University of Oregon, Eugene, OR 97403, USA}
\email{arkadiy@math.uoregon.edu}
\address{Jacob Greenstein, Department of Mathematics, University of California, Riverside, CA 92521, USA}
\email{jacobg@ucr.edu}
\address{Jian-Rong Li, Faculty of Mathematics, University of Vienna, Oskar-Morgenstern-Platz 1, 1090 Vienna, Austria}
\email{lijr07@gmail.com}
\date{}

\thanks{This work was partially supported by the Simons Foundation Collaboration Grant no.~636972 (A.~Berenstein), the Simons
foundation collaboration grant no.~245735 (J.~Greenstein), Austrian Science Fund (FWF): P 34602, Grant DOI: 10.55776/P34602, and PAT 9039323, Grant-DOI 10.55776/PAT9039323 (J.-R. Li).
}

\begin{abstract}
One of the motivations of this work was the rather mysterious commutative submonoid of parabolic elements in any Hecke
monoid of finite type, discovered by the first author and D.~Kazhdan. To understand it, we introduce the class of (weakly) parabolic
homomorphisms, which map parabolic elements to parabolic elements. The most prominent ones are those we call parabolic projections.
We prove that products of parabolic elements can be described explicitly in terms of these remarkable homomorphisms. We also use them
to find all pairs of commuting idempotents in Hecke monoids and prove that the number of such pairs in finite types miraculously depends
only on the rank and satisfies a Fibonacci-like recurrence. We also introduce the class of locally injective homomorphisms. As a major
step towards the classification of injective homomorphisms, we find all locally injective connected homomorphisms between Hecke
monoids of types $A$, $B$ and~$D$ and identify, albeit conjecturally, the injective ones. As a surprising byproduct of our study of
parabolic and injective homomorphisms, we describe {\em all} homomorphisms between Hecke monoids. This gives hope
that the classification of injective ones is achievable.
\end{abstract}

\maketitle

\tableofcontents

\section{Introduction and main results}
The aim of the present work is to systematically study
homomorphisms of Hecke monoids and thus to develop their comprehensive theory. Homomorphisms, especially injective ones, are
instrumental in representation theory of these monoids and Coxeter groups as well as corresponding unipotent algebraic groups.

We start with {\em parabolic} homomorphisms of Hecke monoids, also referred to as 0-Hecke monoids, Demazure monoids and even Coxeter monoids in the literature. Our interest in them was motivated by the following surprising
observation.
Let $M$ be a Coxeter matrix over a finite set~$I$ and let $W(M)=\la s_i\,:\, i\in I\ra$
(respectively, $(W(M),\star)$) be the corresponding Coxeter group
(respectively, the Hecke
monoid generated by simple idempotents also denoted by $s_i$, $i\in I$, see~\S\S\ref{subs:Br(M)W(M)}, \ref{subs:Hecke} for details).
Given~$J\subset I$, let $W_J(M)=\la s_j\,:\, j\in J\ra$ be the {\em parabolic} subgroup of~$W(M)$
corresponding to~$J$ (\S\ref{subs:parab}).
It is well-known that the assignments
$s_j\mapsto s_j$, $j\in J$ and $s_j\mapsto 1$, $j\in I\setminus J$, define a homomorphism of monoids $p_J:(W(M),\star)\to
(W_J(M),\star)$, which we refer to as a {\em parabolic projection}. It should be noted
that such homomorphisms of Coxeter groups are very rare (see Remark~\ref{rem:parab proj Cox}).

We say that~$K\subset I$ is of {\em finite type} (in the literature they are also called {\em spherical}) if~$W_K(M)$ is finite.
For any~$K\subset I$ of finite type,
there is a distinguished
family of elements $w_{J;K}\in W_K(M)$, $J\subset K$ called {\em parabolic} (cf.~\S\ref{subs:parab elts}).
In particular, $w_\circ^K:=w_{\emptyset;K}$
is the longest element in~$W_K(M)$.
We say that a homomorphism  of Hecke monoids is weakly parabolic if it maps
all parabolic elements to parabolic elements, and is parabolic if
it satisfies an additional set-theoretic
condition (see Definition~\ref{defn:parabolic hom} for the precise statement).
It should be noted that, for instance,
homomorphisms from commutative Hecke monoids are always weakly parabolic but not necessarily parabolic (see the discussion below and~\S\ref{subs:parab hom}).
Our first main result is the following
 \begin{theorem}\label{thm:mthm I}
Each parabolic projection is a parabolic homomorphism. Moreover, all {\em light} homomorphisms of Hecke monoids, i.e. those homomorphisms sending generators to generators or 1, are also parabolic.
 \end{theorem}
 This theorem is highly non-trivial and is proven in~\S\S\ref{subs:pJ parabolic}
 and~\ref{subs:light->parab}.
 The argument relies very heavily on
 the structure of the rather mysterious monoid $\mathsf P_K(M)=\{ w_{J;K}\,:\, J\subset K\}$.
These monoids for Weyl groups appeared
for the first time in the joint work of the first author with David Kazhdan (\cite{BK07}). In particular, it was shown, by a geometric argument,
that $\mathsf P_K(M)$
is an abelian submonoid of the Hecke monoid~$(W_K(M),\star)$.
For non-crystallographic finite Coxeter groups, we provide a
different proof here for the reader's convenience (Proposition~\ref{prop:submonoid *}).

This result was a byproduct of geometric considerations, namely of our study of homomorphisms of nilpotent Lie algebras and the corresponding (unipotent) Lie groups. We show (Theorem~\ref{thm:Lie light}) that any light homomorphism of
Hecke monoids of Weyl type $W(M)\to W(M')$, finite or infinite, gives rise to a homomorphism of nilpotent Lie algebras $\lie n(A)\to \lie n(A')$ and their (unipotent) Lie groups $U(A)\to U(A')$, where
$A$ and~$A'$ are (generalized) Cartan matrices associated with~$M$ and~$M'$, respectively (see~\S\ref{subs:Lie} for the details).

Parabolic projections are also instrumental in establishing structural results about Hecke monoids, e.g., in
our classification of commuting idempotents (see~\S\ref{subs:prb projs}
and in particular Proposition~\ref{prop:comm idemp}).
The enumeration of pairs of commuting
idempotents is quite remarkable since,
for finite irreducible types, it depends only on~$|I|$ (Proposition~\ref{prop:enumerate comm pairs}). The
corresponding sequence~$c_n$, $n\ge 0$ satisfies the recursion
$c_{n+1}=4c_n-2c_{n-1}$, $n\ge 1$, with~$c_0=1$ and~$c_1=4$
(\cite{OEIS}*{\OEIS{A007070}}).
This looks like a miracle by the contrast with Coxeter groups (cf. Remark~\ref{rem:graph property}).

Motivated by Theorem~\ref{thm:mthm I}, we pose the following
\begin{problem}\label{prob:parabolic}
Classify all parabolic homomorphisms of Hecke monoids.
\end{problem}
This problem is interesting and, hopefully, manageable since, in particular, the class of parabolic homomorphisms is closed under compositions, and so
it is enough to classify ``prime'' ones.
For instance, we found a series of
such homomorphisms $(W(B_2),\star)\to
(W(A_{2n-1}),\star)$ (see Remark~\ref{rem:B2 A2n-1 parab noninj}) which, like most light homomorphisms,
are non-injective and we expect that, up to decorations
in style of Lemma~\ref{lem:decoration}, they
exhaust non-injective solutions of Problem~\ref{prob:parabolic}
for~$M=B_2$ and~$M'=A_n$.

The following provides additional classes of solutions for Problem~\ref{prob:parabolic} which are very different from light ones.
\begin{theorem}[Propositions~\ref{prop:CH-prop} and~\ref{prop:CH inj}]\label{thm:Cox type Hecke hom}
Let~$\phi:(W(M),\star)\to (W(M'),\star)$ be a
homomorphism of Hecke monoids. Suppose that~$\phi$
is also a homomorphism of Coxeter groups.
Then~$\phi$ is parabolic and, moreover, is injective if and only if it does not map any generator to~$1$.
\end{theorem}
We show that homomorphisms satisfying the first assumption
of this theorem are the same as
{\em homogeneous homomorphisms} (see~\S\ref{subs:homogeneous}). For finite~$W(M)$
they are classified in Theorem~\ref{thm:adm finite class} when both the domain and the codomain are irreducible
and turn out to be, essentially, the unfoldings. The general
case is settled in Lemma~\ref{lem:diagonal} and Proposition~\ref{prop:coprod homogeneous}.

Another class of injective solutions of Problem~\ref{prob:parabolic} is obtained as follows.
We prove
that the assignments $s_i\mapsto w_\circ^{[i,n-2+i]}$,
$i\in\{1,2\}$ define injective parabolic homomorphisms, respectively,
$(W(A_2),\star)\to (W(A_n),\star)$
(Proposition~\ref{prop:A2->An}),
$(W(B_2),\star)\to (W(B_n),\star)$ (Proposition~\ref{prop:B2 Bn all})
and~$(W(B_2),\star)\to (W(D_n),\star)$ for~$n$ even
(Proposition~\ref{prop:B2 Dn+1}).
\begin{conjecture}\label{conj:parab homs}
Homogeneous homomorphisms and the homomorphisms described above
generate under composition, up to diagram automorphisms and decorations (see Lemma~\ref{lem:decoration}), all injective parabolic homomorphisms
between Hecke monoids of irreducible crystallographic types
(except for~$G_2$).
\end{conjecture}
We verified this conjecture by computer for all fully supported homomorphisms with domains of types~$A_k$, $B_k$, $k\le 4$, and~$D_4$
and codomains of types~$A_n$, $B_n$, $D_n$, $n\le 6$ ($n\le 5$ for domains of rank~4); in this range, decorations are not needed.

It should be noted
that homogeneous homomorphisms in the crystallographic case are closely related to homomorphisms of the corresponding reductive groups:
by Theorem~\ref{thm:groups}, a morphism of reductive groups with finite kernel induces injective homomorphisms of their Weyl groups, and
the unfoldings~\eqref{eq:unfold Bn A2n-1}, \eqref{eq:unfold Bn A2n}, \eqref{eq:unfold Bn Dn+1} and~\eqref{eq:unfold F4 E6} arise in
this way from natural embeddings of simple algebraic groups (such as~$SO(2n+1)\hookrightarrow SL(2n+1)$). These were our second motivation.

Since all left-total (see Definition~\ref{defn:total}) homomorphisms from Theorem~\ref{thm:Cox type Hecke hom}
and all those listed before Conjecture~\ref{conj:parab homs} are injective, it is natural to pose the following
\begin{problem}\label{prob:injective}
Classify injective homomorphisms of Hecke monoids.
\end{problem}
Like parabolic ones, this class of homomorphisms is also closed under compositions. This problem is certainly
harder than Problem~\ref{prob:parabolic}, nevertheless we hope that it is still manageable. As the first step, we pose the following (see Definitions~\ref{def:types heck hom} and~\ref{defn:locally inj} for the terms in italic)
\begin{problem}\label{prob:locally injective}
Classify all {\em locally injective connected}
homomorphisms of Hecke monoids.
\end{problem}
The class of connected homomorphisms of Hecke monoids is closed under compositions, like those of parabolic and injective homomorphisms.
We expect that all (locally) injective
homomorphisms are obtained from connected
ones either
by using ``decorations'' (see Lemma~\ref{lem:decoration}) or by taking compositions
with homogeneous homomorphisms (cf. Conjecture~\ref{conj:families of inj}).
We hope that this can be
used in representation theory of Hecke monoids
and instrumental for solving Problem~\ref{prob:injective} for infinite Hecke monoids.

We completely solved Problem~\ref{prob:locally injective} when the domain and the codomain
are irreducible of classical
type.
\begin{theorem}[Theorems~\ref{thm:Ak to An}, \ref{thm:Bk Bn}, \ref{thm:Br D_n+1}; Propositions~\ref{prop:A2->An}, \ref{prop:B2 Bn all},
\ref{prop:simply laced to B, B to A}, \ref{prop:B2 Dn+1}, \ref{prop:negative}; Lemmata~\ref{lem:sigma equiv}, \ref{lem:Dr Dn} and Corollary~\ref{cor:Dr Bn}]\label{thm:list}
A fully supported connected homomorphism between irreducible Hecke monoids of
classical type is locally injective
if and only if, up to diagram automorphisms and compositions with homogeneous homomorphisms,
it is one of the homomorphisms~$\psi_{\boldsymbol\lambda}$ (type~$A$ to type~$A$), $\theta_{\boldsymbol\mu}$ (type~$B$ to type~$B$),
$\eta_{\boldsymbol\lambda}$ (type~$B$ to type~$D$), or one of the homomorphisms with domain of rank~$2$ from
Propositions~\ref{prop:A2->An}, \ref{prop:B2 Bn all} and~\ref{prop:B2 Dn+1}. In particular, there are no such homomorphisms
between types~$A$ and~$B$, between types~$A$ and~$D$, or from type~$D$ to type~$B$, and the only such homomorphisms
between Hecke monoids of type~$D$ are isomorphisms.
\end{theorem}
For instance, all such homomorphisms
$(W(A_k),\star)\to (W(A_n),\star)$, $n\ge k\ge 2$ are parametrized by
restricted integer partitions~$\mathcal A_{k+1}(n+1)$
(see~\S\ref{subs:A->A}).
Having completely solved Problem~\ref{prob:locally injective} for classical series, we checked the resulting homomorphisms for injectivity by computer (up to rank 7 for the domain and 30 for the codomain). This led to a precise injectivity criterion (Conjecture~\partref{conj:families of inj.a}) which, once proven, settles Problem~\ref{prob:injective} for connected homomorphisms and for classical series.
The remaining cases of Problem~\ref{prob:injective}, some of which are very complicated, will be addressed in a subsequent work. The first in line are disconnected ones which can be cooked up from both locally injective and
non-locally injective ones using Corollary~\ref{cor:non-loc inj->loc inj}.

Quite unexpectedly, as a byproduct of our approach to Problems~\ref{prob:parabolic} and~\ref{prob:injective}
(which are still open despite all solutions already found) and the
classification of light homomorphisms, it turned out to be possible to describe, to some extent, {\em all} homomorphisms of Hecke monoids. To begin with,
since the generators $s_i$, $i\in I$ of~$(W(M),\star)$
are idempotents, to find
all homomorphisms of Hecke monoids it is necessary to find all idempotents in Hecke monoids.
One can show (see~\S\ref{subs:Hecke} and e.g.~\cite{K14}) that an element of~$(W(M),\star)$
is an idempotent if and only if it is equal to~$w_\circ^J$
for some~$J\subset I$ such that~$W_J(M)$ is finite
(see~\S\ref{subs:w0J}).
\begin{theorem}[Theorem~\ref{thm:Hom Heck Mon}]\label{thm:1.3}
Let $M=(m_{ij})_{i,j\in I}$, $M'=(m'_{ij})_{i,j\in I'}$ be
Coxeter matrices and let
$(W(M),\star)$, $(W(M'),\star)=
\la s'_i: i\in I'\ra$ be the respective Hecke monoids.
\begin{enmalph}
    \item Given a family $\boldsymbol K=\{K_i\,:\,i\in I'\}$
    of subsets of~$I$ such that each~$W_{K_i}(M)$ is finite,
    the assignments $s'_i\mapsto w_\circ^{K_i}$, $i\in I'$,
    define a homomorphism $\phi_{\boldsymbol K}:(W(M'),\star)\to (W(M),\star)$
    if and only if  for all~$i\not=j\in I'$ either~$m'_{ij}=
    \infty$ or
    $W_{K_i\cup K_j}(M)$ is finite
    and
            $\brd{w_\circ^{K_i}\star w_\circ^{K_j}\star}{m'_{ij}}=w_\circ^{K_i\cup K_j}$.
    \item All homomorphisms of Hecke monoids
    are of this form.
\end{enmalph}
\end{theorem}
This provides the classification of homomorphisms of
Hecke monoids (for example, $p_J=\phi_{\boldsymbol K(J)}$
where $K(J)_i=\{i\}$ if~$i\in J$ and~$K(J)_i=\emptyset$
otherwise)
and will enhance representation theory of Hecke monoids developed in~\cite{Cart}.
In particular, if the domain is commutative (that is, the $n$th direct power of~$(W(A_1),\star)$),
$\mathbf K$ is a collection
of subsets~$K_i\subset I$, $1\le i\le n$, such that,
for each~$1\le i\not=j\le n$,
$W_{K_i\cup K_j}(M)$ is finite and any connected components~$L_i$ of~$K_i$ and~$L_j$ of~$K_j$
are either orthogonal or nested. All such homomorphisms are
weakly parabolic but not necessarily parabolic (see Remark~\ref{rem:weakly not parab}).
This follows
from our classification of pairwise
commuting idempotents in Hecke monoids
(Proposition~\ref{prop:comm idemp}) which
depends on a remarkable formula for
the product of any two idempotents in any Hecke monoid
$$
w_\circ^J\star w_\circ^K=w_\circ^J w_\circ^{J\cap K}w_\circ^K
=w_{J\cap K;J}{}^{-1}w_\circ^{J\cap K}
w_{J\cap K;K}
$$
(see Lemma~\ref{lem:w0J*w0K}).
However, this classification does not immediately solve Problem~\ref{prob:parabolic}, let alone Problem~\ref{prob:injective}.

Another application of Theorem~\ref{thm:1.3} is the classification of isomorphisms of Hecke monoids (see Proposition~\ref{prop:Hecke iso}). Even though it looks obvious, an analogous statement fails for Coxeter groups (see Remark~\ref{rem:Hecke iso Cox}).

Motivated by Theorem~\ref{thm:1.3},
we introduce a remarkable family~$\{u_{J,K}\}\subset W(M)$ of ``indicators of local injectivity'' via
\begin{equation}\label{eq:u JK}
u_{J,K}:=
w_\circ^{J\cup K}(\brd{w_\circ^{J}\star
w_\circ^{K}\star}{\mu_M(J,K)-1}),
\end{equation}

for all~$J,K\subset I$
such that~$W_{J\cup K}(M)$ is finite. Here the outer product is taken {\em in the Coxeter group}~$W(M)$,
and $\mu_M(J,K)$ is the minimal positive integer $m$ such that $\brd{w_\circ^J\star w_\circ^K\star}{m}$ is equal to $w_\circ^{J\cup K}$ (see~\S\ref{subs:Hecke homs} and Lemma~\ref{lem:G JK finite} for the details).
They
play a central role in our
theory of locally injective homomorphisms
(see Section~\ref{sec:loc inj Homs}
and especially
Remarks~\ref{rem:u JK 1} and~\ref{rem: uJK 2}). In all cases
we studied so far such an element
turns out to be a maximal parabolic element in
the smallest parabolic submonoid of~$(W(M),\star)$ containing it. We plan to investigate their parabolicity elsewhere.

\subsection*{Acknowledgments}
The main part of this work was carried out while the authors were visiting Erwin Schr\"odinger
International Institute for Theoretical Physics (ESI), Vienna, Austria,
in the framework of the ``Research in teams'' program. It is our pleasure to thank the ESI for its hospitality. During the work on
this project the first author was visiting Max Planck Institute for Mathematics in the Sciences (MIS), Leipzig, Germany and the second author was visiting Institut des Hautes \'Etudes Scientifiques (IHES), Bures-sur-Yvette, France. The hospitality of both institutions is gratefully acknowledged. We are very grateful to the referee whose insightful suggestions helped us to dramatically simplify several proofs and greatly improved the exposition.

\section{Preliminaries}\label{sec:Prelim}

\subsection{General notation}\label{subs:gen notation}
We extend the natural order on~$\mathbb Z$ to~$\mathbb Z\cup\{\infty\}$
via $\infty>n$ for all~$n\in\mathbb Z$
and use the convention that $n \infty=n+\infty=\infty$ for all~$n\in\mathbb Z_{>0}\cup\{\infty\}$.
In particular, $\infty$ is assumed
to be divisible by all elements of~$\mathbb Z_{>0}\cup\{\infty\}$.
Given~\plink{bar s}$s\in\mathbb Z$, let $\bar s\in\{0,1\}$
be the remainder of~$s$ when divided by~$2$.
For any $a,b\in\mathbb Z$ we write $[a,b]=\{ i\in\mathbb Z\,:\,
a\le i\le b\}$ and \plink{[a,b]2}$[a,b]_2=\{ k\in [a,b]\,:\, \overline{b-k}=0\}$. Given $a,b\in\ZZ$ and~$J\subset \ZZ$, set \plink{a+bJ}$a+b J:=\{a +b j\,:\, j\in J\}$.
The power set of a set~$S$ is denoted by~\plink{PS}$\mathscr P(S)$. Given a category~$\mathscr{C}$, we denote by $\Hom_{\mathscr C}(X,Y)$ the set of morphisms from $X\in\mathscr C$ to~$Y\in\mathscr C$.

\subsection{Monoids}\label{subs:monoids}
Throughout this paper, a homomorphism of monoids is assumed to map the identity element of the domain
to the identity element of the codomain.

Let~$\mathsf M$ be a multiplicative monoid.
Given any finite~$J=\{j_1,\dots,j_r\}\subset \ZZ$, where~$j_1<\cdots<j_r$, and
a family $x_j$, $j\in J$ of elements of~$\mathsf M$ we set
$$\plink{ascp}
\ascprod_{j\in J} x_i=x_{i_1}\cdots x_{i_r},\qquad
\dscprod_{j\in J} x_i=x_{i_r}\cdots x_{i_1}.
$$
This notation
will also be used for infinite families $x_j$, $j\in \mathbb Z$, with all but finitely many of the~$x_j$
equal to~$1$.

Given a family~$S$ of generators of~$\mathsf M$,
the length function $\ell_S:\mathsf M\to \mathbb Z_{\ge0}$ is defined by setting $\ell_S(x)$,
$x\in \mathsf M$
to be
the minimal length of a word in~$S$ which is equal
to~$x$. Clearly, $\ell_S(xy)\le \ell_S(x)+\ell_S(y)$
for all~$x,y\in\mathsf M$.

For any $x,y\in \mathsf M$ and~$m\in\ZZ_{\ge0}$  define
$$\plink{brd}
\brd{xy}{m}:=(xy)^{\floor{ \frac12 m}} x^{\bar m}.
$$
In other words, $\brd{xy}0=1$, $\brd{xy}{m+1}=
\brd{xy}m\,x$ if~$m$ is even, $\brd{xy}{m+1}=
\brd{xy}m\,y$ if~$m$ is odd, while $\brd{xy}{m+1}=
x\brd{yx}{m}$ for all~$m\in\mathbb Z_{\ge 0}$.

\subsection{Coxeter matrices, Artin monoids and Coxeter groups}\label{subs:Br(M)W(M)}
Let~$I$ be a finite set and let~$M=(m_{ij})_{i,j\in I}$ be
a symmetric matrix with $m_{ii}=1$ and $m_{ij}\in\ZZ_{>1}\cup\{\infty\}$, $i\not=j$.
Such a matrix is called a {\em Coxeter matrix} (over~$I$), and we denote
by~\plink{Cox I}$\Cox I$ the set of
all Coxeter matrices over~$I$.
The Coxeter
graph~\plink{Gamma(M)}$\Gamma(M)$ associated with~$M$ is the undirected graph with vertex set~$I$
and with a unique edge connecting $i,j\in I$ if and only if~$m_{ij}>2$.
The edge is labeled with~$m_{ij}$ if~$m_{ij}>3$.
We say that~$J,K\subset I$ are {\em orthogonal}
if $m_{jk}=2$ for all~$j\in J$, $k\in K$.
In particular, if~$J$ and~$K$ are orthogonal then~$J\cap K=\emptyset$.
A Coxeter
matrix~$M$ over~$I$ is said to be {\em irreducible} if~$I$ cannot be written as a
disjoint union of two non-empty orthogonal subsets or, equivalently,
if~$\Gamma(M)$ is connected. We denote~$\Gamma_J(M)$ the full
weighted subgraph
of~$\Gamma(M)$ with vertex set~$J$. Clearly, $\Gamma_J(M)=\Gamma(M_J)$. We say that~$J\subset I$
is {\em connected} if~$\Gamma_J(M)$ is connected as a graph or,
equivalently, if $J$ is not the disjoint union of two non-empty
orthogonal subsets. By abuse of terminology, we say that~$J\subset I$
is a {\em connected component} of~$I$ if~$\Gamma_J(M)$
is a connected component of~$\Gamma(M)$ or, equivalently, if~$J$ is a maximal connected subset of~$I$.
We say that~$J\subset I$
is {\em totally disconnected} if~$m_{ij}\le 2$ for all~$i,j\in J$,
that is, if the only connected subsets of~$J$ are singletons.

The {\em Artin monoid} $\Br^+(M)$\plink{Br+(M)} associated with~$M$ (see for example~\cites{BrSa,Del,Tits}) is generated by the~$T_i$, $i\in I$ subject to relations
$$
\brd{T_iT_j}{m_{ij}}=\brd{T_jT_i}{m_{ij}},\qquad i\not=j\in I,\, m_{ij}<\infty.
$$
Since defining relations of~$\Br^+(M)$ are homogeneous in the number of generators, the
length function with respect to~$\{T_i\}_{i\in I}$
is a homomorphism of monoids~\plink{ell}$\ell:\Br^+(M)\to (\ZZ_{\ge 0},+)$. If~$|I|=1$ this homomorphism
is actually an isomorphism.
Furthermore, the monoid~$\Br^+(M)$ admits a unique \plink{op}anti-involu\-tion~${}^{op}$ fixing each
generator, since the set of its defining relations is invariant under reversing the order of factors in words.

The {\em Coxeter group}\plink{W(M)}
$W=W(M)$ associated with~$M$ is generated by the~$s_i$, $i\in I$ subject
to relations
$$
(s_i s_j)^{m_{ij}}=1,\qquad i,j\in I,\, m_{ij}\not=\infty.
$$
Clearly, $W(M)$ is isomorphic to the quotient monoid of~$\Br^+(M)$ by the minimal congruence relation containing the~$(T_i^2,1)$, $i\in I$.
Let~\plink{piM}$\pi_M:\Br^+(M)\to W(M)$,
$T_i\mapsto s_i$, $i\in I$,
be the canonical
surjective homomorphism of monoids~$\Br^+(M)\to W(M)$.
We denote by~$\ell$ the length function for~$W(M)$
with respect to~$\{s_i\}_{i\in I}$.
An expression~$w=s_{i_1}\cdots s_{i_k}$,
$i_1,\dots,i_k\in I$ is called {\em reduced} if~$k=\ell(w)$. Clearly,
$\ell(\pi_M(T))\le \ell(T)$ for all~$T\in\Br^+(M)$ and we
set\plink{SQF}
$$
\SQF^+(M)=\{ T\in\Br^+(M)\,:\,\ell(\pi_M(T))=\ell(T)\}.
$$
Elements of~$\SQF^+(M)$ are called {\em square free}. The following is well-known.
\begin{theorem}[\cite{Tits}*{Theorem~3}] \label{thm:Tits}
\begin{enmalph}
\item\label{thm:Tits.a} $\pi_M$ restricts to a bijection~$\SQF^+(M)\to W(M)$.
\item\label{thm:Tits.b} Given $w\in W(M)$, denote by $T_w$ the unique
element of~$\SQF^+(M)\cap \pi^{-1}_M(\{w\})$.
Then $T_w T_{w'}=T_{ww'}$ if and only if
$\ell(ww')=\ell(w)+\ell(w')$. In particular,
for any~$w\in W(M)$,
an expression $w=s_{i_1}\cdots s_{i_k}$, $i_1,\dots,i_k\in I$ is reduced
if and only if~$T_w=T_{i_1}\cdots T_{i_k}$.
\end{enmalph}
\end{theorem}

The anti-involution~${}^{op}$ factors through to an anti-involution of~$W(M)$ which coincides with
the anti-involution $w\mapsto w^{-1}$, $w\in W(M)$.

\subsection{Parabolic submonoids and subgroups}\label{subs:parab}
Given~$J\subset I$, let $M_J=(m_{ij})_{i,j\in J}\in \Cox J$ be the corresponding
submatrix of~$M$. Then the submonoid \plink{Br+J(M)}$\Br^+_J(M):=\la T_j\,:\, j\in J\ra$ of~$\Br^+(M)$ is isomorphic to~$\Br^+(M_J)$.
The subgroups~$W_J(M)$ of~$W(M)$ are defined similarly and are
isomorphic to respective objects corresponding to~$M_J$.
Those subobjects are called {\em parabolic} submonoids (subgroups).
We will usually identify $W_J(M)$ with~$W(M_J)$ and so on and denote by~\plink{iotaJ}$\iota_{J}$
the natural inclusion of~$W_J(M)$ (respectively, $\Br^+_J(M)$)
into~$W(M)$ (respectively, $\Br^+(M)$).

We say that~$J\subset I$ is of {\em finite type} if~$W(M_J)$ is finite.
The corresponding subgroups and submonoids are often referred to as being of {\em spherical type} in the literature. We let~\plink{F(M)}$\mathscr F(M)$ denote the
set of all subsets of~$I$ of finite type. Clearly,
$\mathscr F(M)=\mathscr P(I)$ if and only if~$I\in\mathscr F(M)$, in which case we also say that~$M$ is of finite type. Note that~$\emptyset\in\mathscr F(M)$, the corresponding
parabolic subgroups and submonoids being trivial.

Define \plink{supp}$\supp:\Br^+(M)\to \mathscr P(I)$
by $$
\supp T=\bigcap_{J\subset I\,:\, T\in\Br^+_J(M)} J,
\qquad T\in\Br^+(M).
$$
The map~$\supp:W(M)\to \mathscr P(I)$ is defined
similarly.
Clearly, $\supp \pi_M(T)\subset \supp T$ for all~$T\in\Br^+(M)$.
Given a subset~$S$ of~$\Br^+(M)$ or~$W(M)$, define $\supp S:=\bigcup_{x\in S} \supp x$. Observe that~$\supp TT'=\supp T\,\cup\,\supp T'$ for $T,T'\in\Br^+(M)$ while $\supp ww'\subset \supp w\,\cup\,\supp w'$ for $w,w'\in W(M)$. In particular, given
any expression $T=T_{i_1}\cdots T_{i_k}$
(respectively, a {\em reduced} expression
$w=s_{i_1}\cdots s_{i_k}$) where~$i_1,\dots,i_k\in I$
we have~$\supp T=\{i_1,\dots,i_k\}$ (respectively,
$\supp w=\{i_1,\dots,i_k\}$). It follows that the map
$\supp$ is surjective.
The following is well-known
(cf.~\cite{Tits}*{Theorem~3}, \cite{Bou}*{Ch. IV, \S1.5}).
\begin{lemma}\label{lem:extend supp}
Let~$w,w'\in W(M)$ with $\supp w\cap \supp w'=\emptyset$. Then
$\ell(ww')=\ell(w)+\ell(w')$ and~$\supp ww'=\supp w\cup\supp w'$.
\end{lemma}

It is well-known (see, e.g.~\cite{Bou}*{Ch. VI, \S4, Thm.~1}) that the Coxeter group~$W(M)$ with irreducible~$M$ is finite if and only if
$\Gamma(M)$ is isomorphic to one of the following weighted graphs
\begin{alignat}{3}
A_n:&\dynkin[text style/.style={scale=0.8},Coxeter,root radius=0.07,expand labels={1,2,n-1,n},make indefinite edge={2-3},edge length=1.0cm]A4,&&n\ge 1,
\nonumber\\
B_n:&\dynkin[text style/.style={scale=0.8},Coxeter,root radius=0.07,expand labels={1,2,n-1,n},make indefinite edge={2-3},edge length=1.0cm]B4,&&n\ge 2,\nonumber\\
D_{n+1}: &\dynkin[text style/.style={scale=0.8},Coxeter,root radius=0.07,expand labels={1,2,n-1,n,n+1},label directions={,,right,,},make indefinite edge={2-3},edge length=1.0cm]D5,&&n\ge 3,\nonumber\\
E_n:&\dynkin[text style/.style={scale=0.8},Coxeter,ordering=Kac,root radius=0.07,expand labels={1,2,3,4,n-1,n},make indefinite edge={4-5},edge length=1.0cm]E6,&\qquad&
n\in\{6,7,8\},\nonumber\\
F_4:&\dynkin[text style/.style={scale=0.8},Coxeter,ordering=Kac,root radius=0.07,expand labels={1,2,3,4},edge length=1.0cm]F4,\nonumber\\
I_2(m):&
\dynkin[text style/.style={scale=0.8},Coxeter,ordering=Kac,root radius=0.07,expand labels={1,2},edge length=1.0cm,gonality=m]I2,&&m\ge 4,\nonumber\\
H_n:&\dynkin[text style/.style={scale=0.8},Coxeter,ordering=Kac,root radius=0.07,ordering=Adams,expand labels={1,2,n-1,n},make indefinite edge={2-3},edge length=1.0cm]H4
,&&n\in\{3,4\}.\label{eq:Coxeter graphs}
\end{alignat}
The labeling of vertices shown in~\eqref{eq:Coxeter graphs} will be used throughout the rest of
the paper unless specified otherwise.
Clearly, $I_2(3)$ (respectively,~$I_2(4)$) coincides with $A_2$ (respectively, $B_2$); the graph of type~$I_2(6)$
is traditionally denoted as~$G_2$. We will use the same symbol
to denote the Coxeter graph and the corresponding Coxeter matrix.

An automorphism~$\sigma$ of the weighted graph~$\Gamma(M)$, or, equivalently
a permutation~$\sigma$ of~$I$ such that $m_{\sigma(i)\sigma(j)}=m_{ij}$
for all~$i,j\in I$, induces an automorphism of~$\Br^+(M)$ (respectively, of
$W(M)$), called a {\em diagram automorphism} and also denoted
by~$\sigma$, via $\sigma(T_i)=T_{\sigma(i)}$ (respectively,
$\sigma(s_i)=s_{\sigma(i)}$), $i\in I$. If~$W(M)$ is finite and~$\Gamma(M)$ is connected, diagram automorphisms of order~$2$
exist only if $\Gamma(M)$ is of type $A_n$, $n\ge 1$, $D_{n+1}$, $n\ge 3$, $F_4$, $E_6$ or~$I_2(m)$, the corresponding permutation of~$I$ being
\begin{equation}\label{eq:diag aut}
\sigma=\begin{cases}
\prod_{1\le i\le \frac12n} (i,n+1-i),& M=A_n,\, n\ge 2,\\
(n,n+1),&M=D_{n+1},\, n\ge 3,\\
(1,4)(2,3),&M=F_4,\\
(1,5)(2,4),&M=E_6.
\end{cases}
\end{equation}
In type $D_4$, there is also a diagram
automorphism of order~$3$ given by the permutation $(1,3,4)$ of~$[1,4]$ and
so the group of all diagram automorphisms of~$D_4$ is isomorphic to~$S_3$.

If~$J\in\mathscr F(M)$, then $W_J(M)$ contains the unique element~\plink{w0J}
$w_\circ^J$ of maximal length (see, e.g.~\cites{Bou,Tits}), which is obviously an involution. It is well-known (see e.g.~\cite{Bou}*{Ch. IV, Ex. 22} or~\cite{BjBr}*{Proposition~2.3.2}) that
\begin{equation}\label{eq:ell w w0}
\ell(ww_\circ^J)=\ell(w_\circ^J w)=\ell(w_\circ^J)-\ell(w),\qquad
w\in W_J(M).
\end{equation}
It is also well-known (cf. e.g.~\cite{BjBr}*{Exercise~4.10} and~\cite{BrSa}) that for~$M$ irreducible and
of finite type~$w_\circ^I$ is central unless~$M=A_n$,
$M=D_{n+1}$ with~$n$ even, $M=E_6$ or~$M=I_2(2m+1)$. If~$w_\circ^I$
is not central then
$ww_\circ^I=w_\circ^I\sigma(w)$ for all~$w\in W(M)$
where~$\sigma$ is the diagram automorphism.

\subsection{Hecke monoids}\label{subs:Hecke}
The {\em Hecke monoid} associated with $M\in\Cox I$ is the quotient
of~$\Br^+(M)$ by the minimal congruence relation
containing $(T_i^2,T_i)$ for all~$i\in I$.
We denote~\plink{pi*M}$\pi^\star_M$ the canonical homomorphism from $\Br^+(M)$
to the corresponding Hecke monoid.
Thus, the Hecke monoid is generated
by the $s_i:=\pi^\star_M(T_i)$, $i\in I$
subject to relations $s_i\star s_i=s_i$, $i\in I$ and
$$
\brd{s_i\star s_j\star}{m_{ij}}=\brd{s_j\star s_i\star}{m_{ij}},\qquad
i\not=j\in I,\, m_{ij}\not=\infty.
$$
Note that~${}^{op}$ and diagram automorphisms factor through to
the Hecke monoid.
\begin{remark}
In the literature, Hecke monoids are also referred to as Coxeter monoids
(see e.g.~\cite{K14}), $0$-Hecke monoids or Demazure monoids. The latter term is due to the fact that idempotent Demazure operators provide a representation of Hecke monoids.
\end{remark}
\begin{proposition}\label{prop:prod *}
For all~$i\in I$, $w\in W(M)$
\begin{equation}\label{eq:prod *}
s_i\star w=\begin{cases}
s_i w,&\ell(s_i w)>\ell(w),\\
w,&\ell(s_i w)<\ell(w),
\end{cases}\qquad
w\star s_i=\begin{cases}
ws_i,&\ell(ws_i)>\ell(w),\\
w,&\ell(ws_i)<\ell(w),
\end{cases}
\end{equation}
where we abbreviate $w=\pi^\star_M(T_w)$.
In particular, $\pi^\star_M(\Br^+(M))$ identifies with~$W(M)$ as a set,  the restriction
of~$\pi^\star_M$ to~$\SQF^+(M)$ is a bijection onto~$W(M)$ and
$\pi^\star_M|_{\SQF^+(M)}=\pi_M|_{\SQF^+(M)}$.
\end{proposition}
\begin{proof}
By~\cite{Bou}*{Ch.~IV, \S1.5}, if~$\ell(s_i w)>\ell(w)$ then $\ell(s_i w)=\ell(w)+1$ and
so $T_{s_iw}=T_i T_w$ and it remains to apply~$\pi^\star_M$.
Also by~\cite{Bou}*{Ch.~IV, \S1.5}, if $\ell(s_i w)<\ell(w)$ then
$w=s_i u$ for some~$u\in W(M)$ with~$\ell(w)=\ell(u)+1$. Then
$T_w=T_i T_u$ and so applying~$\pi^\star_M$ yields
$s_i\star w=s_i\star (s_i\star u)=s_i\star u=w$. The second identity
follows by using~${}^{op}$.

We now prove by induction on~$\ell(T)$ that for any~$T\in\Br^+(M)$, $\pi^\star_M(T)=w$ for some~$w\in W(M)$. The induction base $\ell(T)=0$ is obvious.
For the inductive step, if~$\ell(T)>0$ then $T=T_i T'$ for some~$i\in I$, $T'\in\Br^+(M)$ with $\ell(T')=\ell(T)-1$.
 Therefore, $\pi^\star_M(T)=s_i\star \pi^\star_M(T')=s_i\star w'$ for some~$w'\in W(M)$ by the induction hypothesis. Then $\pi^\star_M(T)\in \{w',s_i w'\}$ by~\eqref{eq:prod *}, which proves the inductive step. The last assertion
follows by an obvious induction on~$\ell(T)$, $T\in \SQF^+(M)$
since $T\in \SQF^+(M)$ implies
that $T=T_{i_1}\cdots T_{i_k}$, $k=\ell(T)\ge 0$ with $\ell(s_{i_1}
\cdots s_{i_r})=r$ for all~$1\le r\le k$.
\end{proof}
In particular, $
\SQF^+(M)=\{ T\in\Br^+(M)\,:\,\ell(\pi^\star_M(T))=\ell(T)\}$.
From now on, we identify the Hecke monoid associated with
the Coxeter matrix~$M$ with the Coxeter group~$W(M)$ {\em as a set} and denote it
\plink{HeMon}$(W(M),\star)$. Note that $\supp (w\star w')=\supp w\cup\supp w'$ for all
$w,w'\in W(M)$.

We write $u\vdash v$ if $\ell(uv)=\ell(u)+\ell(v)$, $u,v\in W(M)$. Following~\cite{K14}, we write
\plink{times}$w=u\times v$ if~$w=uv=u\star v$.
In particular, $w=s_{i_1}\cdots s_{i_r}$, $i_t\in I$, $1\le i\le r$, is a reduced
expression if and only if $w=s_{i_1}\times\cdots\times s_{i_r}$.
Given~$w\in W(M)$, define\plink{DL(w)}
$$
D_L(w)=\{ i\in I\,:\, \ell(s_iw)<\ell(w)\},\quad
D_R(w)=\{ i\in I\,:\, \ell(ws_i)<\ell(w)\}.
$$
Clearly, $D_R(w)=D_L(w^{-1})$.
We will need the following well-known result.
\begin{proposition}[{see e.g.~\cite{BjBr}*{Proposition~2.4.4}}]\label{prop:parab factorization}
Let~$M\in\Cox I$ and~$K\subset I$. Every~$w\in W(M)$ can be uniquely written as~$w=u\times {}^K[w]$ and as~$w=[w]^K\times u'$,
where~$u,u'\in W_K(M)$ and~\plink{parab fact}${}^K[w]$ (respectively, $[w]^K$) is the unique element of minimal length in the coset~$W_K(M)w$
(respectively, $wW_K(M)$). Moreover, $v\vdash{}^K[w]$ and~$[w]^K\vdash v$ for all~$v\in W_K(M)$, and~${}^K[w]=w$ (respectively, $[w]^K=w$)
if and only if~$D_L(w)\cap K=\emptyset$ (respectively, $D_R(w)\cap K=\emptyset$).
\end{proposition}
\begin{corollary}\label{cor:descent w0J}
Let~$M\in\Cox I$, $J\in\mathscr F(M)$ and~$w\in W(M)$. Then~$w\vdash w_\circ^J$ (that is, $w\star w_\circ^J=w\times w_\circ^J$) if and only
if~$D_R(w)\cap J=\emptyset$, and~$w_\circ^J\vdash w$ (that is, $w_\circ^J\star w=w_\circ^J\times w$) if and only if~$D_L(w)\cap J=\emptyset$.
\end{corollary}
\begin{proof}
If~$D_R(w)\cap J=\emptyset$ then~$w=[w]^J$ by Proposition~\ref{prop:parab factorization} and so~$w\vdash w_\circ^J$. Conversely,
if~$w\vdash w_\circ^J$ and~$j\in J$ then, since~$w_\circ^J=s_j\times(s_jw_\circ^J)$, we have~$w\vdash s_j$, that is, $j\notin D_R(w)$.
The second assertion follows by applying~${}^{op}$.
\end{proof}
\begin{remark}
Proposition~\ref{prop:prod *} can be regarded as a presentation of the Hecke monoid. Namely,
we can define it as $W(M)$, as a set, equipped with the unique associative operation~$\star$ satisfying the
first property in~\eqref{eq:prod *}.
\end{remark}

The following Lemmata are immediate.
\begin{lemma}\label{lem:free product Artin}\label{lem:free product}
Let~$M\in\Cox I$, $M'\in\Cox{I'}$.
\begin{enmalph}
\item\label{lem:free product.a}
Define~$M\times M'\in\Cox{I\sqcup I'}$ by
$$(M\times M')_{ij}=(M\times M')_{ji}=\begin{cases}
                                     m_{ij},&i,j\in I,\\
                                     m'_{ij},&i,j\in I',\\
                                     2,&i\in I,\,j\in I'.
                                \end{cases}$$
Then $W(M)\times W(M')\cong W(M\times M')$
and $(W(M),\star)\times (W(M'),\star)\cong
(W(M\times M'),\star)$;
\item \label{lem:free product.b}
Define~$M\coprod M'\in\Cox{I\sqcup I'}$ by
$$(M\textstyle\coprod M')_{ij}=(M\textstyle\coprod M')_{ji}=\begin{cases}
m_{ij},& i,j\in I,\\
m'_{ij},& i,j\in I',\\
\infty,& i\in I,\,j\in I'.
\end{cases}
$$
Then the free product of~$W(M)$ and~$W(M')$
(respectively, of $(W(M),\star)$ and~$(W(M'),\star)$
is isomorphic to~$W(M\coprod M')$ (respectively,
$(W(M\coprod M'),\star)$).
\end{enmalph}
\end{lemma}
\begin{lemma}\label{lem:orth factors}
Let~$M\in\Cox I$ and let $J,K\subset I$ be orthogonal.
Then
\begin{enmalph}
    \item\label{lem:orth factors.b}
    $W_{J\cup K}(M)\cong W_J(M)\times
    W_K(M)$;
    \item\label{lem:orth factors.c}
    $(W_{J\cup K}(M),\star)\cong (W_J(M),\star)\times
    (W_K(M),\star)$.
\end{enmalph}
In particular, submonoids $W_J(M)$, $W_K(M)$ (respectively, $(W_J(M),\star)$, $(W_K(M),\star)$) commute element-wise in $W(M)$ (respectively, in $(W(M),\star)$).
\end{lemma}

The multiplication in~$(W(M),\star)$ has an additional
characterization which will be important later.
Define a relation $\longrightarrow$ on~$W(M)$
by $u\longrightarrow w$, $u,w\in W(M)$ if and only if
$\ell(u)<\ell(w)$ and $u^{-1}w$ is conjugate to~$s_i$
for some~$i\in I$.
The {\em strong Bruhat order} on~$W(M)$, which we will denote by $\le$, is the transitive
closure of this relation and is easily seen to be
a partial order (see e.g.~\cite{BjBr}*{\S2.1}).
We will need the following properties of the strong
Bruhat order.
\begin{proposition}[see e.g.~\cite{BjBr}*{Theorems~2.2.2 and~2.2.6, Proposition~2.3.4}]\label{prop:Bruhat order}
\begin{enmalph}
    \item\label{prop:Bruhat order.a}
    $u\le w\in W(M)$ if and only if any reduced expression for~$w$ contains a reduced expression for~$u$ as a subexpression. More precisely,
    let $w=s_{i_1}\cdots s_{i_k}\in W(M)$,
    $i_t\in I$, $1\le t\le k=\ell(w)$.
    Then $u\le w$ if and only if there exists $J\subset [1,k]$
    such that $|J|=\ell(u)$ and $u=\ascprod_{j\in J} s_{i_j}$.
    In particular, the restriction of the strong Bruhat order on~$W(M)$
    to~$W_K(M)$ coincides with the strong Bruhat order on~$W_K(M)$ for any~$K\subset I$.
     \item\label{prop:Bruhat order.b} If~$u\le w\in W(M)$
     then there exists a chain $x_0=u<x_1<\cdots<x_r=w$ in~$W(M)$
     such that $\ell(x_i)=\ell(u)+i$, $0\le i\le r$.
     \item\label{prop:Bruhat order.c} If~$J\in\mathscr F(M)$
     then for any $w,w'\in W_J(M)$, $w<w'$ if and only
     if $w_\circ^J w'<w_\circ^J w$ and if and only if $w'w_\circ^J<
     ww_\circ^J$.
\end{enmalph}
\end{proposition}
Given $u\in W(M)$, denote \plink{dar}$\downarrow u=\{ w\in W(M)\,:\, w\le u\}$ and $\uparrow u=\{ w\in W(M)\,:\, u\le w\}$. The following
result will be frequently used in the rest of the paper.
\begin{proposition}[\cite{He09}*{Lemma~1 and Corollary~1}, \cite{K14}*{Lemma~2 and Proposition~8}]\label{prop:Bruhat order *}
Let $w,w'\in W(M)$.
\begin{enmalph}
    \item\label{prop:Bruhat order *.a}
    $u\le w$, $u'\le w'$
    implies that $u\star u'\le w\star w'$;
    \item\label{prop:Bruhat order *.b}
    $(\downarrow w)(\downarrow w'):=
    \{ uu'\,:\, u\in \downarrow w,\,u'\in\downarrow w'\}=
    \downarrow w\star w'$. In particular,  $w\star w'$ is
    the unique maximal element of
    $(\downarrow w)(\downarrow w')$
    with respect to the strong Bruhat order. Moreover,
    $w\star w'=u\times w'=w\times u'$
    where
    $u=(w\star w')w'{}^{-1}\le w$, $u'=w^{-1}(w\star w')\le w'$.

\item\label{prop:Bruhat order *.c} $w\le w\star w'$ and~$w'\le w\star w'$
for any $w,w'\in W(M)$. In particular, $\ell(w\star w')\ge \max(\ell(w),\ell(w'))$.
\item\label{prop:Bruhat order *.d} $ww'\le w\star w'$ for all~$w,w'\in W(M)$,
with the equality if and only if~$\ell(ww')=\ell(w)+\ell(w')$.
\end{enmalph}
\end{proposition}

In fact, the Bruhat order
is the unique partial order preserved by the multiplication in the Hecke monoid for
which the identity is the smallest element
(see e.g.~\cite{K14}*{Theorem~25}).

\subsection{Idempotents in Hecke monoids}\label{subs:idemp}
First, note the following characterization of the~$w_\circ^J$, $J\in\mathscr F(M)$, in Hecke monoids.
\begin{lemma}\label{lem:char w_0 monoid}
Suppose that~$J\in\mathscr F(M)$. The following
are equivalent for $w\in W_J(M)$:
\begin{enmroman}
 \item\label{lem:char w_0 monoid.i} $w=w_\circ^J$;
 \item\label{lem:char w_0 monoid.ii} $s_i\star w=w$ for all $i\in J$;
 \item\label{lem:char w_0 monoid.iv} $x\star w=w$ for all $x\in W_J(M)$;
 \item\label{lem:char w_0 monoid.iii} $w\star s_i=w$ for all $i\in J$;
 \item\label{lem:char w_0 monoid.v} $w\star x=w$ for all $x\in W_J(M)$.
\end{enmroman}
In particular, $w_\circ^J$ is an idempotent in~$(W(M),\star)$.
\end{lemma}
\begin{proof}
We may assume, without loss of generality, that~$J=I$. Clearly,
\ref{lem:char w_0 monoid.ii} (respectively, \ref{lem:char w_0 monoid.iii})
is equivalent to~\ref{lem:char w_0 monoid.iv} (respectively, \ref{lem:char w_0 monoid.v}).
Since
$\ell(s_iw_\circ^I),\ell(w_\circ^Is_i)<\ell(w_\circ^I)$ for all~$i\in I$, it follows from
Proposition~\ref{prop:prod *} that~$s_i\star w_\circ^I=w_\circ^I=
w_\circ^I\star s_i$
for all~$i\in I$ and so~\ref{lem:char w_0 monoid.i} implies~\ref{lem:char w_0 monoid.ii} and~\ref{lem:char w_0 monoid.iii}. Finally, if $x\star w=w$ for all~$x\in W(M)$
then $w_\circ^I\star w=w$. Yet since $w_\circ^I\star x=w_\circ^I$
for all~$x\in W(M)$, it follows that $w_\circ^I\star w=w_\circ^I$
hence \ref{lem:char w_0 monoid.iv} implies~\ref{lem:char w_0 monoid.i}.
Similarly, \ref{lem:char w_0 monoid.v} implies~\ref{lem:char w_0 monoid.i}.
\end{proof}
We will now prove that
$$
\{ w_\circ^J\,:\, J\in\mathscr F(M)\}=
\{ x\in (W(M),\star)\,:\, x\star x=x\}.
$$

\begin{proposition}[Absorption property]\label{prop:absorb prop}
Let~$u,v\in W(M)$. Then~$\ell(u\star v)=\ell(u)$ (respectively, $\ell(u\star v)=\ell(v)$) if and only if~$u\star v=u$
(respectively, $u\star v=v$), and in that case~$\supp v\subset \supp u$ (respectively, $\supp u\subset \supp v$).
In particular, if $\ell(u\star v)=\ell(u)=\ell(v)$ then
$\supp u=\supp v\in\mathscr F(M)$ and
$u=v=w_\circ^{\supp u}$.
\end{proposition}
\begin{proof}
By Proposition~\partref{prop:Bruhat order *.b}, $u\star v=u\times v'$ for some~$v'\le v$.
Since $\ell(u)=\ell(u\star v)=\ell(u\times v')=\ell(u)+\ell(v')$, it follows that~$v'=1$ and so~$u\star v=u$.
In particular, $v\le u\star v=u$ by Proposition~\partref{prop:Bruhat order *.c} and so $\supp v\subset \supp u$. The converse is trivial.

The second assertion is proved verbatim.

Finally, if $\ell(u\star v)=\ell(u)=\ell(v)$ then
$u=v$ and $\supp u=\supp v=:J$. Since $1\le s_j\le v$ for any~$j\in J$,
$u=u\star 1\le u\star s_j\le u\star v=u$ by Proposition~\partref{prop:Bruhat order *.a}, whence $u\star s_j=u$.
Therefore, $u\star w=u$ for all~$w\in W_J(M)$ and so
$\ell(u)=\ell(u\star w)\ge \ell(w)$ by Proposition~\partref{prop:Bruhat order *.c}. Thus, $J\in\mathscr F(M)$ and it remains to apply Lemma~\ref{lem:char w_0 monoid}.
\end{proof}

\begin{corollary}\label{cor:max elts}
\begin{enmalph}
\item\label{cor:max elts.a} Let~$G$ be
a finite subsemigroup of~$(W(M),\star)$. Then~$\supp G\in\mathscr F(M)$, $w_\circ^{\supp G}\in G$ and is the unique element of~$G$ of
maximal length, and $G\subset \downarrow w_\circ^{\supp G}=W_{\supp G}(M)$.
\item\label{cor:max elts.b} $w\in (W(M),\star)$ is an idempotent if
and only if $\supp w\in\mathscr F(M)$ and $w=w_\circ^{\supp w}$.
\item\label{cor:max elts.c} Let~$J\subset I$. Then~$J\in\mathscr F(M)$
if and only if $(W(M),\star)$ contains an idempotent~$w$ with~$\supp w=J$.
\end{enmalph}
\end{corollary}
\begin{proof}
Since~$G$ is finite, it contains an element~$u$ of maximal length.
Suppose that $v\in G$  with~$\ell(v)=\ell(u)$.
By Proposition~\partref{prop:Bruhat order *.c}, $\ell(u\star v)=\ell(u)=\ell(v)$ and so by Proposition~\ref{prop:absorb prop}, $\supp u\in\mathscr F(M)$ and~$u=v=w_\circ^{\supp u}$. If $w\in G$
then, as $\ell(w\star u)\ge \ell(u)$ by Proposition~\partref{prop:Bruhat order *.c}, $\ell(w\star u)=\ell(u)$ and hence $\supp w\subset \supp u$ by Proposition~\ref{prop:absorb prop}. It follows that~$\supp G=\supp u$. The last
assertion in part~\ref{cor:max elts.a}
is immediate.

Part~\ref{cor:max elts.b}
follows from~\ref{cor:max elts.a} since
the subset consisting of an idempotent
element is a subsemigroup of~$W(M)$.

The forward direction in part~\ref{cor:max elts.c} is established in Lemma~\ref{lem:char w_0 monoid} while the converse is established in Proposition~\ref{prop:absorb prop}.
\end{proof}
The following Lemma will allow us to relate products in Coxeter groups and the corresponding Hecke monoids.

\begin{lemma}\label{lem: u * u^(-1) w_0}
Let~$M\in\Cox I$, $J\in\mathscr F(M)$, $u,v\in W_J(M)$ and~$w\in W(M)$.
\begin{enmalph}
\item\label{lem: u * u^(-1) w_0.a'}
$w_\circ^J\star (uw)=w_\circ^J\star w=w_\circ^J\times {}^J[w]$ and~$(wu)\star w_\circ^J=w\star w_\circ^J=[w]^J\times w_\circ^J$.
\item\label{lem: u * u^(-1) w_0.c}
Let~$K\subset J$ and~$u\in W_K(M)$. Then~$w_\circ^K\star (u w_\circ^J)=(w_\circ^J u)\star w_\circ^K=w_\circ^J$.
\item\label{lem: u * u^(-1) w_0.b}
Suppose that $\ell(uv)=\ell(v)-\ell(u)$. Then
$u\star (vw_\circ^J)=uvw_\circ^J$,
$(uv)\star (v^{-1}w_\circ^J)=
u w_\circ^J$,
$w_\circ^Jv^{-1}u^{-1}=(w_\circ^J v^{-1})\star u^{-1}$ and~$(w_\circ^J v)\star (v^{-1}u^{-1})=w_\circ^Ju^{-1}$.
\item\label{lem: u * u^(-1) w_0.a}
Suppose that~$u\vdash v$. Then~$u^{-1}\star (uv w_\circ^J)=vw_\circ^J$ and $(w_\circ^J v^{-1}u^{-1})\star u=
w_\circ^Jv^{-1}$.
\end{enmalph}
\end{lemma}
\begin{proof}
To prove part~\ref{lem: u * u^(-1) w_0.a'}, write~$w=u'\times{}^J[w]$ with~$u'\in W_J(M)$ as in Proposition~\ref{prop:parab factorization}.
Then~$uw=(uu')\,{}^J[w]$ with~$uu'\in W_J(M)$, so~${}^J[uw]={}^J[w]$. Since~$w_\circ^J\star u'=w_\circ^J$ by Lemma~\ref{lem:char w_0 monoid}
and~$w_\circ^J\vdash{}^J[w]$ by Proposition~\ref{prop:parab factorization},
$$
w_\circ^J\star(uw)=w_\circ^J\star{}^J[w]=w_\circ^J\times{}^J[w]=w_\circ^J\star w.
$$
The second identity follows by applying~${}^{op}$.

Part~\ref{lem: u * u^(-1) w_0.c}
follows from part~\ref{lem: u * u^(-1) w_0.a'} (with~$J$ replaced by~$K$ and~$w=w_\circ^J$) and Lemma~\ref{lem:char w_0 monoid}.

Since~$\ell(uvw_\circ^J)=\ell(w_\circ^J)-\ell(uv)=\ell(w_\circ^J)-\ell(v)+\ell(u)
=\ell(vw_\circ^J)+\ell(u)$, we get
$u\star (vw_\circ^J)=uvw_\circ^J$.
Since~$uv(v^{-1}w_\circ^J)=uw_\circ^J$,
the second identity is equivalent to
~$\ell(uw_\circ^J)=\ell(uv)+\ell(v^{-1}w_\circ^J)$. Indeed, $\ell(uv)+\ell(v^{-1}w_\circ^J)=\ell(v)-\ell(u)+\ell(w_\circ^J)-\ell(v^{-1})=\ell(w_\circ^J)-\ell(u)=
\ell(uw_\circ^J)$.
As before, the remaining identities follow by applying~${}^{op}$.

Finally, since~$u^{-1}(uvw_\circ^J)=vw_\circ^J$,
the first identity in part~\ref{lem: u * u^(-1) w_0.a} is equivalent to
$\ell(u^{-1})+\ell(uvw_\circ^J)=\ell(vw_\circ^J)$.
By~\eqref{eq:ell w w0}
\begin{align*}
\ell(u^{-1})+\ell(uvw_\circ^J)=
\ell(u)+\ell(w_\circ^J)-\ell(uv)=
\ell(w_\circ^J)-\ell(v)=\ell(vw_\circ^J).
\end{align*}
The second identity in part~\ref{lem: u * u^(-1) w_0.a}
follows from the first by applying~${}^{op}$.
\end{proof}

The following is apparently well-known (see, e.g.~\cite{BjBr}*{Proof of Lemma~3.2.3}). We
provide a proof here since the argument is quite elegant in the setting of Hecke monoids.
\begin{lemma}\label{lem:left desc}
For any~$w\in W(M)$, $\{
x\in W(M)\,:\, x\star w=w\}$ is a finite
submonoid of~$(W(M),\star)$ and is
equal to $(W_{D_L(w)}(M),\star)$. In particular, $D_L(w)\in\mathscr F(M)$
and for any~$u\in W_{D_L(w)}(M)$,
$w=u\times w'$ for some~$w'\in\downarrow w$.
Similarly, for any~$w\in W(M)$, $\{
x\in W(M)\,:\, w\star x=w\}$ is a finite
submonoid of~$(W(M),\star)$ and is
equal to $(W_{D_R(w)}(M),\star)$. In particular, $D_R(w)\in\mathscr F(M)$
and for any~$u\in W_{D_R(w)}(M)$,
$w=w'\times u$ for some~$w'\in\downarrow w$.
\end{lemma}
\begin{proof}
Let $G(w)=\{x\in W(M)\,:\, x\star w=w\}$
which is manifestly a submonoid of $(W(M),\star)$.
Furthermore, if~$x\in G(w)$
then~$(\downarrow x)\subset (\downarrow x)(\downarrow w)=\downarrow w$. In particular,
$G(w)\subset\downarrow w$ and therefore is finite. Furthermore, if~$x'\le x$ for some~$x\in G(w)$ then~$x'\in G(w)$.
Let~$J(w)=\supp G(w)$. By Corollary~\partref{cor:max elts.a}, $J(w)\in\mathscr F(M)$,  and so
$G(w)=
\downarrow w_\circ^{J(w)}=W_{J(w)}(M)$.
The equality $D_L(w)=J(w)$ is now immediate.
The last assertion follows from Proposition~\partref{prop:Bruhat order *.b}.
The statements concerning~$D_R(w)$ are proven similarly.
\end{proof}
\begin{remark}
In view of Proposition~\ref{prop:absorb prop},
$\{x\in W(M)\,:\, x\star w=w\}=
\{x \in W(M)\,:\, \ell(x\star w)=\ell(w)\}$.
\end{remark}
\begin{lemma}\label{lem:DL uw0}
Let~$M\in\Cox I$, $J\in\mathscr F(M)$ and~$u\in W_J(M)$.
Then~$D_L(u w_\circ^J)=J\setminus D_L(u)$
and~$D_R(w_\circ^J u)=J\setminus D_R(u)$.
\end{lemma}
\begin{proof}
Since~$D_L(u w_\circ^J)\subset J$, we may assume,
without loss of generality, that~$J=I$.
Since~$i\in D_L(u)$ if and only if $s_i u< u$,
which is equivalent to
$uw_\circ<s_i uw_\circ$ by Proposition~\partref{prop:Bruhat order.c},
the first assertion follows.
The second assertion
follows from the first by using~${}^{op}$.
\end{proof}

\subsection{Coxeter elements}\label{subs:w0J}

Given~$J\subset I$,
we say that~$c\in W_J(M)$ is a
{\em Coxeter element}
if~$\supp c=J$ and~$\ell(c)=|J|$.
In the sequel, we will often consider
special Coxeter elements corresponding to an interval~$J=[a,b]\subset I\subset \ZZ$,
namely \plink{cab}$\cx ab=\ascprod_{a\le i\le b} s_i$, $\cxr ab=
(\cx ab)^{op}=\dscprod_{a\le i\le b}s_i$.
We will use the convention that~$\cx ij=\cxr ij=1$ if~$i>j$.

It is well-known (see e.g.~\cite{Bou}*{Ch. V, \S6})
that if~$J\in\mathscr F(M)$ then all Coxeter elements~$c\in W_J(M)$
are conjugate in the Coxeter group~$W_J(M)$ and
in particular are of the same order~\plink{h(M)}$h(M_J)$, called the {\em Coxeter number}
of~$W_J(M)$. The Coxeter number is even for all irreducible finite types
except~$I_2(2m+1)$, $m>0$ and $A_{2m}$.
Note also that if~$J\subset I$ is totally disconnected then $w_\circ^J$ is the unique Coxeter element of~$W_J(M)$.
Some important properties of Coxeter elements are
summarized in the following
\begin{proposition}[cf.~\cite{BrSa}*{\S5.8}]\label{prop:Coxeter splitting}
Let~$M\in\Cox I$ and let~$J\in\mathscr F(M)$ be connected.
\begin{enmalph}
\item\label{prop:Coxeter splitting.b}
If $w_\circ^J$ is central in~$W_J(M)$ then~$h(M_J)$ is even and $w_\circ^J=
c^{\times \frac12h(M_J)}$ for any Coxeter
element~$c\in W_J(M)$;
\item\label{prop:Coxeter splitting.c}
For any partition~$J=J'\cup J''$
of~$J$ into disjoint non-empty totally disconnected subsets
$$w_\circ^J=\brd{w_\circ^{J'}w_\circ^{J''}}{h(M_J)}
=\brd{w_\circ^{J''}w_\circ^{J'}}{h(M_J)}.
$$
\end{enmalph}
\end{proposition}
In~$(W(A_n),\star)$, we have
the following expressions for~$w_\circ^{[a,b]}$,
$1\le a\le b\le n$ in terms of Coxeter elements,
which we will often use in the sequel
\begin{equation}
w_\circ^{[a,b]}=\ascprodst_{a\le k\le b} \cxr ak=
\dscprodst_{a\le k\le b} \cx ak=
\ascprodst_{a\le k\le b} \cxr kb=
\dscprodst_{a\le k\le b} \cx kb.\label{eq:w0 type A}
\end{equation}
It is easy to see that
\begin{equation}\label{eq:comm cox}
\cx ab\times s_i=s_{i+1}\times\cx ab,\qquad s_i\times\cxr ab=\cxr ab\times s_{i+1},\qquad i\in[a,b-1].
\end{equation}
Indeed,
$\cx ab s_i=\cx a{(i-1)}s_is_{i+1}\cx{(i+2)}b\, s_i=\cx a{(i-1)}
s_i s_{i+1}s_i \cx{(i+2)}b
=\cx a{(i-1)}s_{i+1}s_i s_{i+1}\cx{(i+2)}b=
s_{i+1}\cx ab$, and both sides are reduced since~$D_R(\cx ab)=\{b\}$
and~$D_L(\cx ab)=\{a\}$. The second identity follows by applying~${}^{op}$.

\subsection{Parabolic elements}\label{subs:parab elts}
Given~$J\subset K\in\mathscr F(M)$, denote
\plink{wJ;K}$w_{J;K}=w_\circ^J w_\circ^K\in W_K(M)$. Such elements are called {\em
$K$-parabolic} or just parabolic if~$K=I$; in the latter case we abbreviate~$w_{J}=w_{J;I}$.
The following is immediate.
\begin{lemma}\label{lem:prod orth parab}
Let~$K,K'\in\mathscr F(M)$ be orthogonal.
Then $w_{J;K}\times w_{J';K'}=w_{J\cup J'; K\cup K'}$
for all~$J\subset K$, $J'\subset K'$.
\end{lemma}

\begin{lemma}\label{lem:wK absorption}
Let~$K\in\mathscr F(M)$ and let~$J\subset K$. Then
\begin{enmalph}
\item\label{lem:wK absorption.a}
$D_L(w_{J;K})=K\setminus J=D_R(w_{J;K}{}^{-1})$;
\item\label{lem:wK absorption.b}
$w_\circ^K=w_\circ^J\times w_{J;K}$ and
$\ell(w_{J;K})=\ell(w_\circ^K)-\ell(w_\circ^J)$;
\item\label{lem:wK absorption.c}
for any~$J'\subset I$ such that~$J'\cap K=J$,
$w\vdash w_{J;K}$ and~$w_{J;K}{}^{-1}\vdash w$
for all $w\in W_{J'}(M)$.
\end{enmalph}
\end{lemma}
\begin{proof}
Part~\ref{lem:wK absorption.a} follows from Lemma~\ref{lem:DL uw0}
since~$D_L(w_\circ^J)=J$.
Since~$w_{J;K}=w_\circ^J w_\circ^K$,
$\ell(w_{J;K})=\ell(w_\circ^K)-\ell(w_\circ^J)$ by~\eqref{eq:ell w w0}. Thus, $\ell(w_\circ^Jw_{J;K})=\ell(w_\circ^K)=\ell(w_\circ^J)+\ell(w_{J;K})$, that is,
$w_\circ^J\vdash w_{J;K}$, and so $w_\circ^K=w_\circ^J w_{J;K}=w_\circ^J\times w_{J;K}$ by Proposition~\partref{prop:Bruhat order *.d}. This proves~\ref{lem:wK absorption.b}.

To prove~\ref{lem:wK absorption.c}, note that $J'\cap K=J$ implies $D_R(w_{J;K}{}^{-1})\cap J'=(K\setminus J)\cap J'=\emptyset$, that is,
$w_{J;K}{}^{-1}=[w_{J;K}{}^{-1}]^{J'}$, and so~$w_{J;K}{}^{-1}\vdash w$ for all~$w\in W_{J'}(M)$ by Proposition~\ref{prop:parab factorization}.
Applying~${}^{op}$, we obtain $w\vdash w_{J;K}$ for all~$w\in W_{J'}(M)$.
\end{proof}

\begin{lemma}\label{lem:w0J*w0K}
Let~$M\in\Cox I$ and $J,K\in\mathscr F(M)$. Then
\begin{equation}\label{eq:w0J*w0K}
w_\circ^J\star w_\circ^K=w_\circ^J\times w_{J\cap K;K}=
w_\circ^J w_\circ^{J\cap K} w_\circ^K=w_{J\cap K;J}{}^{-1}\times w_\circ^{J\cap K}\times w_{J\cap K;K},
\end{equation}
and
$\ell(w_\circ^J\star w_\circ^K)=
\ell(w_\circ^J)+\ell(w_\circ^K)-\ell(w_\circ^{J\cap K})$.
\end{lemma}
\begin{proof}
By Lemma~\partref{lem:wK absorption.b}, $w_\circ^K=w_\circ^{J\cap K}\times w_{J\cap K;K}$ with~$w_\circ^{J\cap K}\in W_J(M)$,
while~$D_L(w_{J\cap K;K})\cap J=(K\setminus J)\cap J=\emptyset$ by Lemma~\partref{lem:wK absorption.a}. Therefore,
${}^J[w_\circ^K]=w_{J\cap K;K}$ by Proposition~\ref{prop:parab factorization}. Then~$w_\circ^J\star w_\circ^K=w_\circ^J\times w_{J\cap K;K}$
by Lemma~\partref{lem: u * u^(-1) w_0.a'}, whence~$\ell(w_\circ^J\star w_\circ^K)=
\ell(w_\circ^J)+\ell(w_\circ^K)-\ell(w_\circ^{J\cap K})$ by Lemma~\partref{lem:wK absorption.b}.
The second equality in~\eqref{eq:w0J*w0K} is immediate from the first and the definition of parabolic elements.
Finally, applying~${}^{op}$ to $w_\circ^J=w_\circ^{J\cap K}\times w_{J\cap K;J}$ (Lemma~\partref{lem:wK absorption.b}), we obtain
$w_\circ^J=w_{J\cap K;J}{}^{-1}\times w_\circ^{J\cap K}$. Thus,
$w_\circ^J\star w_\circ^K=w_{J\cap K;J}{}^{-1}w_\circ^{J\cap K}w_{J\cap K;K}$ and, by Lemma~\partref{lem:wK absorption.b},
$$
\ell(w_{J\cap K;J}{}^{-1})+ \ell(w_\circ^{J\cap K})+\ell(w_{J\cap K;K})=\ell(w_\circ^J)+\ell(w_\circ^K)-\ell(w_\circ^{J\cap K})=\ell(w_\circ^J\star w_\circ^K),
$$
which proves the last equality
in~\eqref{eq:w0J*w0K}.
\end{proof}
Since~$w_\circ^J\star w_\circ^K\in W_{J\cup K}(M)$ and~$\ell(w)\le \ell(w_\circ^L)$ for all~$w\in W_L(M)$, $L\in\mathscr F(M)$, Lemma~\ref{lem:w0J*w0K} immediately yields
\begin{corollary}\label{cor:supermodular}
Let~$M\in\Cox I$. The function~$\mathscr F(M)\to
\ZZ_{\ge 0}$, $J\mapsto \ell(w_\circ^J)$, is supermodular on all
pairs~$J,K\in\mathscr F(M)$ such that~$J\cup K\in \mathscr F(M)$, that is,
$\ell(w_\circ^{J\cup K})\ge
\ell(w_\circ^J)+\ell(w_\circ^K)-
\ell(w_\circ^{J\cap K})$.
\end{corollary}

\begin{lemma}\label{lem:len prop wJ K}
 Let~$M\in\Cox I$ and let
 $L\subset K\subset J\in \mathscr F(M)$.
   Then
   $$
w_{L;J}=w_{L;K}\times w_{K;J}.
$$
\end{lemma}
 \begin{proof}
 Clearly, $w_{L;J}=w_\circ^L w_\circ^J=(w_\circ^L w_\circ^K)(w_\circ^K w_\circ^J)=w_{L;K}w_{K;J}$, while
 by~\eqref{eq:ell w w0}, $\ell(w_{L;K})+\ell(w_{K;J})=\ell(w_\circ^K)-\ell(w_\circ^L)+\ell(w_\circ^J)-\ell(w_\circ^K)=\ell(w_{L;J})$.
\end{proof}
\begin{lemma}\label{lem:too short}
Let~$M\in\Cox I$, $K\subset L\in\mathscr F(M)$,
and let~$J,J'\subset K$ be such
that $\ell(w_{J'\cap J;J'})<\ell(w_{J;K})$.
Then
$\ell(w_\circ^{J'}\star (w_{J;K}w_\circ^L))<
\ell(w_\circ^L)$ and~$\ell((w_\circ^L
w_{J;K}{}^{-1})\star w_\circ^{J'})<\ell(w_\circ^L)$.
\end{lemma}
\begin{proof}
By Lemma~\partref{lem:wK absorption.a}, $D_L(w_{J;K})=
K\setminus J$ whence~$D_L( w_{J;K} w_\circ^L)=
L\setminus(K\setminus J)=(L\setminus K)\cup J$ by Lemma~\ref{lem:DL uw0}. Then~$J'\cap
D_L(w_{J;K}w_\circ^L)=J'\cap J$.
Write~$w_\circ^{J'}=w_{J'\cap J;J'}\times w_\circ^{J'\cap J}$.
Then
$
w_\circ^{J'}\star (w_{J;K}w_\circ^L)
=w_{J'\cap J;J'}\star (w_{J;K}w_\circ^L),
$
by Lemma~\ref{lem:left desc},
whence
$$
\ell(w_\circ^L)-\ell(w_\circ^{J'}\star (w_{J;K}w_\circ^L))\ge
\ell(w_{J;K})-\ell(w_{J'\cap J;J'})>0.
$$
The second assertion follows from the first by applying~${}^{op}$.
\end{proof}

\section{General properties of homomorphisms
of Hecke monoids}\label{sec:Gen homs}

Throughout this chapter, we denote
standard generators of $(W(M'),\star)$
corresponding
to a Coxeter matrix~$M'$ over~$I'$
by~$s'_i$, $i\in I'$ and so on.

Let~\plink{A C H}$\He$ (respectively, $\Co$) be the category whose objects are Coxeter matrices
and morphisms are homomorphisms of corresponding Hecke monoids (respectively, Coxeter groups).
Parabolic submonoids and subgroups are, naturally, subobjects in these categories. By
Lemma~\ref{lem:free product}
categories $\He$ and~$\Co$ admit finite products and coproducts via, respectively, $(M,M')\mapsto M\times M'$ and
$(M,M')\mapsto M\coprod M'$, $M\in\Cox I$, $M'\in\Cox{I'}$.

\subsection{Homomorphisms of Hecke monoids}\label{subs:Hecke homs}
Given $M'\in\Cox{I'}$, $M\in\Cox I$ and~$\phi\in\Hom_{\He}(M',M)$
or~$\phi\in\Hom_{\Co}(M',M)$, we denote by \plink{[phi]}$[\phi]$
the map $I'\to \mathscr P(I)$ defined by $i\mapsto \supp\phi(s'_i)$, $i\in I'$ and extend it
to a map $[\phi]:\mathscr P(I')\to \mathscr P(I)$ via
$[\phi](J')=\bigcup_{j'\in J'} [\phi](j')$, $J'\subset I'$.
Clearly, $[\phi](\emptyset)=\emptyset$ and~$[\phi]$ preserves unions, that is,
$[\phi](J'\cup K')=[\phi](J')\cup[\phi](K')$ for all~$J',K'\subset I'$.
We will usually abbreviate~$[\sigma]$ as~$\sigma$ for diagram automorphisms.
\begin{definition}\label{def:types heck hom}
We say that~$\phi\in\Hom_{\He}(M',M)$
or~$\phi\in\Hom_{\Co}(M',M)$
is:
\begin{itemize}
\item[-] {\em disjoint} if $[\phi](i)\cap[\phi](j)=
\emptyset$ for all $i\not=j\in I'$;
\item[-] {\em fully supported} if $[\phi](I')=I$;

\item[-] {\em connected} if~$[\phi](J)$
is a connected subset of~$I$ for any connected~$J\subset I'$.
\end{itemize}
\end{definition}

\begin{lemma}\label{lem:[phi]comp}
Let~$M\in\Cox I$, $M'\in\Cox{I'}$ and~$M''\in\Cox{I''}$.
\begin{enmalph}
\item\label{lem:[phi]comp.a}
$\supp\phi(x)=[\phi](\supp x)$
for all~$\phi\in\Hom_{\He}(M',M)$ and for
all~$x\in W(M')$;
\item\label{lem:[phi]comp.b}
$[\phi\circ\phi']=[\phi]\circ[\phi']$ as maps $\mathscr P(I'')\to
\mathscr P(I)$ for any~$\phi\in\Hom_{\He}(M',M)$
and~$\phi'\in\Hom_{\He}(M'',M')$;
\item\label{lem:[phi]comp.e} $\phi\in\Hom_{\He}(M',M)$ is
connected if and only if the~$[\phi](i)$, $i\in I'$ are connected and~$[\phi](i)\cup [\phi](j)$
is connected whenever $m'_{ij}>2$, $i,j\in I'$;
\item\label{lem:[phi]comp.conn} if~$\phi\in\Hom_{\He}(M',M)$ and~$\phi'\in\Hom_{\He}(M'',M')$ are connected then so is~$\phi\circ\phi'$.
\end{enmalph}
\end{lemma}
\begin{proof}
Since~$\supp x\star y=\supp x\,\cup\,\supp y$ for all $x,y\in (W(M),\star)$, we have for all~$x'\in (W(M'),\star)$ $$\supp\phi(x')=\bigcup_{j\in\supp x'}
\supp \phi(s'_j)=\bigcup_{j\in\supp x'} [\phi](j)=
[\phi](\supp x'),$$
which proves~\ref{lem:[phi]comp.a}.
To prove part~\ref{lem:[phi]comp.b},
note that by part~\ref{lem:[phi]comp.a} we have
for all~$x''\in W(M'')$
$$
[\phi\circ\phi'](\supp x'')=
\supp(\phi\circ\phi')(x'')
=[\phi](\supp\phi'(x''))
=[\phi]([\phi'](\supp x'')).
$$
Since $\supp: W(M'')\to \mathscr P(I'')$
is surjective, the assertion follows.

One direction in part~\ref{lem:[phi]comp.e} is evident while the other follows by an obvious induction on the cardinality of~$I'$.
Finally, if~$J\subset I''$ is connected then so is~$[\phi'](J)$ and hence so is~$[\phi\circ\phi'](J)=[\phi]([\phi'](J))$ by part~\ref{lem:[phi]comp.b},
which proves~\ref{lem:[phi]comp.conn}.
\end{proof}

It follows from Lemma~\ref{lem:[phi]comp}
that the assignments $M\mapsto I$, $M\in\Cox I$
and $\phi\mapsto[\phi]$ for~$M\in\Cox I$,
$M'\in\Cox{I'}$ and~$\phi\in\Hom_{\He}(M',M)$, define a functor~$[\cdot]$
from~$\He$ to the Kleisli category of the power-set monad, whose objects are sets
and morphisms are maps~$X\to\mathscr P(Y)$,
the composition being the ``flattening'' operation (see above). Note that the map~$[\phi]:\mathscr P(I')\to\mathscr P(I)$ defined above is precisely the Kleisli extension of~$[\phi]:I'\to\mathscr P(I)$,
which is why it preserves unions. This category is easily seen to be isomorphic
to the category \plink{Rel}$\mathscr{R\!e\!l}$ where
objects are sets and morphisms are binary relations between two sets
(see e.g.~\cite{ML}). In the language of~$\mathscr{R\!e\!l}$, $\phi$ is disjoint (respectively, fully supported) if and only if~$[\phi]$ is a left-unique
(respectively, right-total) relation.
\begin{definition}\label{defn:total}
We say that~$\phi\in\Hom_{\He}(M',M)$, $M\in\Cox I$, $M'\in\Cox{I'}$ is {\em left-total} if so is~$[\phi]$, that is, $[\phi](i)\not=\emptyset$ for all~$i\in I'$,
or equivalently, $\phi(s'_i)\not=1$ for all~$i\in I'$.
\end{definition}
The following is immediate from the definition.
\begin{lemma}\label{lem:inj total}
Let~$M\in\Cox I$, $M'\in\Cox{I'}$ and let
$\phi\in \Hom_{\He}(M',M)$ be injective. Then~$\phi$ is left-total.
\end{lemma}
We will need the following properties
of~$[\cdot]$.
\begin{lemma}\label{lem:[phi]sets}
Let~$M\in\Cox I$, $M'\in\Cox{I'}$ and let $\phi\in\Hom_{\He}(M',M)$.
\begin{enmalph}
\item\label{lem:[phi]sets.a}
   The map~$[\phi]:\mathscr P(I')\to
    \mathscr P(I)$ is
    injective if and only if~$[\phi](i)\not\subset[\phi](I'\setminus\{i\})$ for all~$i\in I'$. In particular, if~$[\phi]$ is injective then~$\phi$ is left-total, and if
    $\phi$ is left-total and disjoint then~$[\phi]$ is injective.
\item\label{lem:[phi]sets.b} If~$\phi$ is disjoint then~$\bigcap\limits_{1\le t\le r}[\phi](J_t)=
    [\phi](\bigcap\limits_{1\le t\le r}J_t)$ for any $\{J_t\}_{1\le t\le r}\subset \mathscr P(I')$.

\noindent\hspace*{-\leftmargin}Additionally, let $M''\in\Cox{I''}$ and $\phi'\in\Hom_{\He}(M'',M')$.

\item\label{lem:[phi]sets.c} If~$\phi$ and~$\phi'$ are fully supported then so is~$\phi\circ\phi'$.
\item\label{lem:[phi]sets.d} If~$[\phi]:\mathscr P(I')\to \mathscr P(I)$ is injective and~$\phi\circ\phi'$ is fully supported then~$\phi'$ is fully supported.
\end{enmalph}
\end{lemma}
\begin{proof}
If~$[\phi]$ is injective then~$[\phi](I'\setminus\{i\})\not=[\phi](I')$, that is, $[\phi](i)\not\subset[\phi](I'\setminus\{i\})$, for all~$i\in I'$.
Conversely, suppose that~$J\not=K\subset I'$. We may assume, without loss of generality, that $J\not\subset K$. Then, for~$i\in J\setminus K$, any element of~$[\phi](i)\setminus[\phi](I'\setminus\{i\})$ lies
in~$[\phi](J)\setminus[\phi](K)$
and so~$[\phi](J)\not=[\phi](K)$.
Since~$[\phi](i)=\emptyset$ implies that~$[\phi](i)\subset [\phi](I'\setminus\{i\})$, the injectivity of~$[\phi]$ implies that~$\phi$ is left-total.
The remaining assertion in part~\ref{lem:[phi]sets.a} readily follows from
the injectivity criterion. Part~\ref{lem:[phi]sets.b} is immediate
from standard properties of~$\mathscr{R\!e\!l}$, while part~\ref{lem:[phi]sets.c} is a straightforward consequence of the functoriality of~$[\cdot]$.
Finally, if~$\phi\circ\phi'$ is fully supported then~$I=[\phi]([\phi'](I''))\subset[\phi](I')\subset I$, and so~$[\phi]([\phi'](I''))=[\phi](I')$.
Since~$[\phi]$ is injective, $[\phi'](I'')=I'$, which proves part~\ref{lem:[phi]sets.d}.
\end{proof}

Our present aim is to describe~$\Hom_{\He}(M',M)$ for any Coxeter matrices~$M'$ and~$M$.
We begin with the following observation.
\begin{lemma}\label{lem:Hecke hom w0J}
Let $M\in\Cox I$, $M'\in\Cox{I'}$ and~$\phi\in\Hom_{\He}(M',M)$.
Then~$[\phi](J')\in\mathscr F(M)$ and $\phi(w_\circ^{J'})=w_\circ^{[\phi](J')}$
for all~$J'\in\mathscr F(M')$. In particular,
$[\phi](i)\in \mathscr F(M)$ for all~$i\in I'$.
\end{lemma}
\begin{proof}
Since~$w_\circ^{J'}$ is an idempotent in~$(W(M'),\star)$,
$\phi(w_\circ^{J'})=w_\circ^J$ for some~$J\in\mathscr F(M)$
by Corollary~\partref{cor:max elts.b}. Clearly~$J\subset [\phi](J')$.
Since
$s'_j\star w_\circ^{J'}=w_\circ^{J'}$ for all~$j\in J'$
by Lemma~\ref{lem:char w_0 monoid}, it follows that
$\phi(s'_j)\star w_\circ^J=w_\circ^J$ for all~$j\in J'$ and
so $[\phi](j)=\supp\phi(s'_j)\subset J$ for all~$j\in J'$ by Proposition~\ref{prop:absorb prop}. Thus, $[\phi](J')\subset J$. The last assertion follows
by taking~$J'=\{i\}$, $i\in I'$.
\end{proof}

Let~$J,K\in\mathscr F(M)$
and define
$$
G_{J,K}:=\la w_\circ^J,w_\circ^K\ra = \{
\brd{w_\circ^J\star w_\circ^K\star}{t},\brd{w_\circ^K\star w_\circ^J\star}{t}\,:\,
t\in\ZZ_{\ge 0}\}\subset (W_{J\cup K}(M),\star).
$$
\begin{lemma}\label{lem:G JK finite}
The following are equivalent for~$J,K\in\mathscr F(M)$.
\begin{enmroman}
\item\label{lem:G JK finite.i}
$J\cup K\in\mathscr F(M)$;
\item\label{lem:G JK finite.ii} $G_{J,K}$ is finite;
\item\label{lem:G JK finite.iii} $\brd{w_\circ^J\star w_\circ^K\star}{m}
=\brd{w_\circ^K\star w_\circ^J\star}{n}$
for some $m,n\in\ZZ_{>0}$;
\item\label{lem:G JK finite.iv} $\brd{w_\circ^{J_0}\star w_\circ^{J_1}\star}{k}
=\brd{w_\circ^{J_0}\star w_\circ^{J_1}\star}{l}$
for some~$k<l\in\ZZ_{>0}$ where~$\{J_0,J_1\}=\{J,K\}$.
\item\label{lem:G JK finite.v}
$\brd{w_\circ^{J_0}\star w_\circ^{J_1}\star}{k}
=\brd{w_\circ^{J_0}\star w_\circ^{J_1}\star}{k+1}$
for some~$k\in\ZZ_{>0}$ where~$\{J_0,J_1\}=\{J,K\}$;
\item\label{lem:G JK finite.vi} $\brd{w_\circ^{J_0}\star w_\circ^{J_1}\star}{k}=w_\circ^{J\cup K}$ for some~$k\in\ZZ_{>0}$, where~$\{J_0,J_1\}=\{J,K\}$.
\end{enmroman}
In particular, if~$J\cup K\in\mathscr F(M)$
then there is a minimal~\plink{mu M}$\mu_M(J,K)\in \ZZ_{>0}$
such that~$w_\circ^{J\cup K}=\brd{w_\circ^J\star w_\circ^K\star}{\mu_{M}(J,K)}$ and
\begin{equation}\label{eq:mu_M JK KJ}
|\mu_M(J,K)-\mu_M(K,J)|\le 1.
\end{equation}
\end{lemma}
\begin{proof}
All these statements obviously hold if one of the~$J$, $K$ is empty, so we may assume, without
loss of generality that~$J,K\not=\emptyset$.

The implication \ref{lem:G JK finite.i}$\implies$\ref{lem:G JK finite.ii} is immediate since $G_{J,K}$ is a submonoid of~$(W_{J\cup K}(M),\star)$.

\ref{lem:G JK finite.ii}$\implies$
\ref{lem:G JK finite.iii}
If~$G_{J,K}$
is finite then, since~$\supp G_{J,K}=J\cup K$, by Corollary~\partref{cor:max elts.a}
$G_{J,K}$ contains a unique element of maximal length, namely
$w_\circ^{J\cup K}$, and so either
$w_\circ^{J\cup K}=\brd{w_\circ^J\star w_\circ^K\star}{m}$ or
$w_\circ^{J\cup K}=\brd{w_\circ^K\star w_\circ^J\star}{m}
$
for some~$m> 0$.
Note that if~$m$ is even then
both equalities hold at the same time since~$w_\circ^{J\cup K}$,
$w_\circ^K$ and~$w_\circ^J$ are ${}^{op}$-invariant.
If say the first equality holds for~$m$ odd
then Lemma~\ref{lem:char w_0 monoid} implies that
$$
w_\circ^{J\cup K}=w_\circ^K\star w_\circ^{J\cup K}=
w_\circ^K\star
\brd{w_\circ^J\star w_\circ^K\star}{m}
=\brd{w_\circ^{K}\star w_\circ^{J}\star}{m+1}.
$$

\ref{lem:G JK finite.iii}$\implies$
\ref{lem:G JK finite.iv}.
If~$m\le n$ then~\ref{lem:G JK finite.iv}
with~$J_0=J$, $J_1=K$, $k=m$ and~$l=n+1$
follows by Lemma~\ref{lem:char w_0 monoid}.
Otherwise, \ref{lem:G JK finite.iv}
follows with~$J_0=K$, $J_1=J$, $k=n$ and~$l=m+1$. Thus, $k=\min(m,n)$ and~$l=\max(m,n)+1$.

\ref{lem:G JK finite.iv}$\implies$\ref{lem:G JK finite.v}. Abbreviate~$x_t= \brd{w_\circ^{J_0}\star w_\circ^{J_1}\star}{t}$, $t\in\ZZ_{\ge 0}$.
By Proposition~\partref{prop:Bruhat order *.a},
$x_t\le x_{t+1}$ for all~$t\in\ZZ_{\ge 0}$.
Then~$x_k\le x_{k+1}\le\cdots \le x_l=x_k$
and so~$x_{k+1}=x_k$.

\ref{lem:G JK finite.v}
$\implies$ \ref{lem:G JK finite.vi}
We have~$x_k\star w_\circ^{J_{\overline{k+1}}}=x_k=x_{k+1}=x_{k+1}\star w_\circ^{J_{\overline{k}}}$, whence
$x_k\star x_k=x_k$. Since
$\supp x_k=\supp x_{k+1}=J\cup K$,
it follows from Corollary~\partref{cor:max elts.b} that~$x_k=w_\circ^{J\cup K}$.

The implication~\ref{lem:G JK finite.vi}$\implies$\ref{lem:G JK finite.i} is immediate from Corollary~\partref{cor:max elts.b}.

The remaining assertions are immediate, for example, from the proof of the implication
\ref{lem:G JK finite.ii}$\implies$\ref{lem:G JK finite.iii}.
\end{proof}
Thus we obtain a well-defined map
$\mu_M:\{ (J,K)\in\mathscr F(M)\times\mathscr F(M)\,:\, J\cup K\in\mathscr F(M)\}\to \ZZ_{>0}$, which we extend to a map
$\mu_M:\mathscr F(M)\times
\mathscr F(M)\to \mathbb Z_{>0}\cup\{\infty\}$
by setting $\mu_M(J,K)=\mu_M(K,J)=\infty$
if~$J\cup K\notin\mathscr F(M)$.
In particular,
$\mu_M(J,J)=1$ for all~$J\in\mathscr F(M)$ while~$\mu_M(J,K)=2$ and~$\mu_M(K,J)=1$
if~$J\subsetneq K\in\mathscr F(M)$.

\begin{example}
In~$(W(A_4),\star)$ we have
$
s_3s_4s_3\star s_2s_4\star s_3s_4s_3=
w_\circ^{\{2,3,4\}}$,
while
$
s_2s_4\star s_3s_4s_3\star s_2s_4=
s_2s_3s_4s_3s_2$ and $
(s_2s_4\star s_3s_4s_3)^{\star 2}=
w_\circ^{\{2,3,4\}}$.
Thus, $\mu_{A_4}(\{3,4\},\{2,4\})=3$,
$\mu_{A_4}(\{2,4\},\{3,4\})=4$.
\end{example}

For any~$M\in\Cox I$ and~$M'\in\Cox{I'}$
define\plink{LM'M}
$$
\Lambda(M',M):=\{\xi:I'\to \mathscr F(M)\,:\,
\max(\mu_M(\xi(i),\xi(j)),
\mu_M(\xi(j),\xi(i)))\le
m'_{ij},\,\forall\, i\not=j\in I'\}.
$$
\begin{theorem}\label{thm:Hom Heck Mon}
Let~$M'\in\Cox{I'}$, $M\in\Cox I$.
The assignments~$\phi\mapsto[\phi]$ define a bijection from $\Hom_{\He}(M',M)$ to~$\Lambda(M',M)$.
\end{theorem}
\begin{proof}
We need the following
\begin{lemma}\label{lem:[phi]Lam M'M}
If~$\phi\in\Hom_{\He}(M',M)$ then
$[\phi]\in\Lambda(M',M)$.
\end{lemma}
\begin{proof}
Note first that~$[\phi]$ is a map~$I'\to \mathscr F(M)$ by Lemma~\ref{lem:Hecke hom w0J}.

Let~$i\not=j\in I'$.
If~$m'_{ij}=\infty$ or at least one of~$[\phi](i)$, $[\phi](j)$ is empty, then the inequality
\begin{equation}\label{eq:ineq Lam M'M}
\max(\mu_M([\phi](i),[\phi](j)),
\mu_M([\phi](j),[\phi](i)))\le m'_{ij}
\end{equation}
is trivial.
Otherwise, the submonoid $G'_{i,j}=\la s'_i,s'_j\ra$ of~$(W(M'),\star)$ is finite, its
longest element being $\brd{s'_i\times s'_j\times}{m'_{ij}}=
\brd{s'_j\times s'_i\times}{m'_{ij}}$. Then
$G_{i,j}=\phi(G'_{i,j})$ is a finite submonoid of~$(W(M),\star)$ contained
in~$G_{[\phi](i),[\phi](j)}$
and~$\supp G_{i,j}=[\phi](i)\cup [\phi](j)$. By Corollary~\partref{cor:max elts.a}, $[\phi](i)\cup [\phi](j)\in\mathscr F(M)$
and the longest element of~$G_{i,j}$ is
$
w_\circ^{[\phi](i)\cup [\phi](j)}=\brd{w_\circ^{[\phi](i)}\star
w_\circ^{[\phi](j)}\star}{\mu_M([\phi](i),[\phi](j))}=\brd{w_\circ^{[\phi](j)}\star
w_\circ^{[\phi](i)}\star}{\mu_M([\phi](j),[\phi](i))}$.
Also, since $\brd{s'_i\times s'_j\times}{m'_{ij}}=
w_\circ^{\{i,j\}}$, by Lemma~\ref{lem:Hecke hom w0J} we have
$$
\phi(\brd{s'_i\times s'_j\times}{m'_{ij}})=
\phi(\brd{s'_j\times s'_i\times}{m'_{ij}})=
\phi(w_\circ^{\{i,j\}})=w_\circ^{[\phi](i)\cup
[\phi](j)}
$$
whence
$
w_\circ^{[\phi](i)\cup [\phi](j)}=
\brd{w_\circ^{[\phi](i)}\star
w_\circ^{[\phi](j)}\star}{m'_{ij}}
=\brd{w_\circ^{[\phi](j)}\star
w_\circ^{[\phi](i)}\star}{m'_{ij}}$.
The inequality~\eqref{eq:ineq Lam M'M} is now immediate.
\end{proof}
The next step is to construct
a map~$\Lambda(M',M)\to \Hom_{\He}(M',M)$.
\begin{lemma}\label{lem:Lam M'M HomM'M}
For any~$\xi\in\Lambda(M',M)$,
the assignments $s'_i\mapsto w_\circ^{\xi(i)}$, $i\in I'$
define
\plink{Theta xi}$\Theta_\xi\in\Hom_{\He}(M',M)$.
\end{lemma}
\begin{proof}
Let~$\xi\in \Lambda(M',M)$.
It suffices to verify that the~$w_\circ^{\xi(i)}$, $i\in I'$,
satisfy the defining relations of $(W(M'),\star)$.
By definition
of~$\mu_M$ we have for all $i\not=j\in I'$ with $m'_{ij}<\infty$
$$
\brd{w_\circ^{\xi(i)}\star w_\circ^{\xi(j)}\star}{m'_{ij}}=w_\circ^{\xi(i)\cup \xi(j)}\star x_{i,j}
$$
where
$$
x_{i,j}=\begin{cases}\brd{w_\circ^{\xi(i)}\star w_\circ^{\xi(j)}}{m'_{ij}-
\mu_M(\xi(i),\xi(j))},& \text{$\mu_M(\xi(i),\xi(j))$ is even,}\\
\brd{w_\circ^{\xi(j)}\star w_\circ^{\xi(i)}}{m'_{ij}-
\mu_M(\xi(i),\xi(j))},& \text{$\mu_M(\xi(i),\xi(j))$ is odd}.
\end{cases}
$$
Since $\supp x_{i,j}\subset \xi(i)\cup \xi(j)$, it follows from Lemma~\ref{lem:char w_0 monoid} that
$
\brd{w_\circ^{\xi(i)}\star w_\circ^{\xi(j)}\star}{m'_{ij}}=w_\circ^{\xi(i)\cup \xi(j)}$.
Thus, $\brd{w_\circ^{\xi(i)}\star w_\circ^{\xi(j)}\star}{m'_{ij}}=
\brd{w_\circ^{\xi(j)}\star w_\circ^{\xi(i)}\star}{m'_{ij}}$.
\end{proof}
To finish the proof of Theorem~\ref{thm:Hom Heck Mon}, it remains to observe that
$[\Theta_\xi]=\xi$ for all~$\xi\in\Lambda(M',M)$ while
$\Theta_{[\phi]}=\phi$ for all~$\phi\in\Hom_{\He}(M',M)$.
\end{proof}

As a first application of Theorem~\ref{thm:Hom Heck Mon}, we
will now describe all isomorphisms between Hecke monoids.
\begin{proposition}\label{prop:Hecke iso}
Let~$M\in\Cox I$ and~$M'\in\Cox{I'}$. Then~$(W(M'),\star)\cong(W(M),\star)$ if and only if there exists a bijection~$f:I'\to I$
such that~$m_{f(i)f(j)}=m'_{ij}$ for all~$i,j\in I'$.
\end{proposition}
\begin{proof}
Let~$\phi\in\Hom_{\He}(M',M)$ be an isomorphism. Since~$[\cdot]$ is a functor,
$[\phi]:\mathscr P(I')\to\mathscr P(I)$ is a bijection with the inverse~$[\phi^{-1}]$. As both~$[\phi]$ and~$[\phi^{-1}]$ preserve unions,
$[\phi]$ maps singletons to singletons (singletons being precisely the non-empty sets which are not unions of two proper non-empty subsets),
and so there exists a bijection~$f:I'\to I$ such that~$[\phi](i)=\{f(i)\}$ and~$[\phi^{-1}](k)=\{f^{-1}(k)\}$,
$i\in I'$, $k\in I$. Then~$\phi(s'_i)=s_{f(i)}$, $i\in I'$, by Lemma~\ref{lem:Hecke hom w0J}. Applying Theorem~\ref{thm:Hom Heck Mon} to~$\phi$ and to~$\phi^{-1}$, we obtain
$$
m'_{ij}\ge \mu_M(\{f(i)\},\{f(j)\})=m_{f(i)f(j)}\ge \mu_{M'}(\{i\},\{j\})=m'_{ij},\qquad i\not=j\in I'.
$$
Thus, $m_{f(i)f(j)}=m'_{ij}$ for all~$i,j\in I'$.

The converse is immediate.
\end{proof}
\begin{remark}\label{rem:Hecke iso Cox}
Proposition~\ref{prop:Hecke iso} fails for Coxeter groups: for instance, if~$m\ge 3$ is odd then
$W(I_2(2m))\cong W(A_1)\times W(I_2(m))$ and $W(B_m)\cong W(A_1)\times W(D_m)$ (see e.g.~\cite{Bahls}).
The reason is that the proof of Proposition~\ref{prop:Hecke iso} relies on~$[\cdot]$ being a functor, which in turn follows
from the identity~$\supp(x\star y)=\supp x\cup\supp y$, $x,y\in W(M)$. The analogous identity holds in the Artin monoid~$\Br^+(M)$,
but fails in~$W(M)$.
\end{remark}

\begin{corollary}\label{cor:tautological}
Let $M, M'\in\Cox I$ and suppose that
$m'_{ij}\ge m_{ij}$ for all~$i,j\in I$. Then
the assignments $s'_i\mapsto s_i$, $i\in I$,
define a homomorphism $(W(M'),\star)\to (W(M),\star)$.
\end{corollary}
We call such homomorphisms {\em tautological}.
For example, for any~$m'>m$ there is a
tautological homomorphism~$W(I_2(m'),\star)\to
W(I_2(m),\star)$. In other words, tautological homomorphisms are vertical with respect to the functor~$[\cdot]:\He\to\mathscr{R\!e\!l}$, that is, they are mapped to identity morphisms.

\begin{definition}\label{defn:optial}
Let~$M\in\Cox I$, $M'\in \Cox{I'}$.
We say that~$\phi\in\Hom_{\He}(M',M)$ is {\em optimal} if for all~$i\not=j\in I'$
\begin{equation}\label{eq:optimal cond}
m'_{ij}=
\max\{2,\mu_M([\phi](i),[\phi](j)),\mu_M([\phi](j),[\phi](i))\}.
\end{equation}
\end{definition}
Note that the first entry~$2$ in the right hand side of~\eqref{eq:optimal cond} only matters if~$[\phi](i)=[\phi](j)$ or one of
them is empty; otherwise the maximum of the two remaining entries is at least~$2$.

\begin{lemma}\label{lem:lift}
Let~$M\in\Cox I$ and let~$\xi:I'\to\mathscr F(M)$ be any map.
\begin{enmalph}
\item\label{lem:lift.a} There exist a unique~$M_\xi\in\Cox{I'}$ and a unique optimal~$f_\xi\in\Hom_{\He}(M_\xi,M)$ such that~$[f_\xi]=\xi$.
\item\label{lem:lift.b} If~$\phi\in\Hom_{\He}(M',M)$, $M'\in\Cox{I'}$, satisfies~$[\phi]=\xi$ then~$\phi=f_\xi\circ\phi'$,
where~$\phi'\in\Hom_{\He}(M',M_\xi)$ is tautological.
\end{enmalph}
\end{lemma}
\begin{proof}
Define~\plink{M xi}$M_\xi\in\Cox{I'}$ by~$(M_\xi)_{ii}=1$, $i\in I'$, and
\begin{equation}\label{eq:M xi}
(M_\xi)_{ij}=\max\{2,\mu_M(\xi(i),\xi(j)),
\mu_M(\xi(j),\xi(i))\},\qquad i\not=j\in I'.
\end{equation}
Then~$\xi\in\Lambda(M_\xi,M)$ and so~$f_\xi:=\Theta_\xi\in\Hom_{\He}(M_\xi,M)$ by Lemma~\ref{lem:Lam M'M HomM'M}.
Since~$[f_\xi]=\xi$, $f_\xi$ is optimal by~\eqref{eq:optimal cond} and~\eqref{eq:M xi}. If~$N\in\Cox{I'}$ and~$f\in\Hom_{\He}(N,M)$ is optimal with~$[f]=\xi$
then~$N=M_\xi$ by~\eqref{eq:optimal cond} and~\eqref{eq:M xi}, and~$f=f_\xi$ by Theorem~\ref{thm:Hom Heck Mon}. This proves part~\ref{lem:lift.a}.

To prove part~\ref{lem:lift.b}, note that~$\xi=[\phi]\in\Lambda(M',M)$ by Theorem~\ref{thm:Hom Heck Mon}, and so
$m'_{ij}\ge (M_\xi)_{ij}$ for all~$i,j\in I'$.
Thus, there is a tautological~$\phi'\in\Hom_{\He}(M',M_\xi)$ by Corollary~\ref{cor:tautological}. Since~$[\phi']$ is the identity,
$[f_\xi\circ\phi']=[f_\xi]=\xi=[\phi]$ and so~$f_\xi\circ\phi'=\phi$ by Theorem~\ref{thm:Hom Heck Mon}.
\end{proof}

\begin{proposition}\label{prop:optimal taut}
Every homomorphism of Hecke monoids can be written uniquely as a composition of a tautological homomorphism with an optimal one.
\end{proposition}
\begin{proof}
Let~$\phi\in\Hom_{\He}(M',M)$, $M'\in\Cox{I'}$, and let~$\xi=[\phi]$. By Lemma~\ref{lem:lift}, $\phi=f_\xi\circ\phi'$,
where~$\phi'\in\Hom_{\He}(M',M_\xi)$ is tautological and~$f_\xi$ is optimal.
Conversely, suppose that~$\phi=\psi''\circ\psi'$, where~$\psi'\in\Hom_{\He}(M',N)$ is tautological and~$\psi''\in\Hom_{\He}(N,M)$ is
optimal for some~$N\in\Cox{I'}$. Since~$[\psi']$ is the identity, $[\psi'']=[\phi]=\xi$, and so~$N=M_\xi$ by~\eqref{eq:optimal cond} and~\eqref{eq:M xi}, and~$\psi''=f_\xi$ by
Lemma~\partref{lem:lift.a}. Finally, $\psi'=\phi'$ since both are tautological homomorphisms from~$(W(M'),\star)$ to~$(W(M_\xi),\star)$.
\end{proof}
\begin{remark}
It should be noted that~$M_\xi$
defined by~\eqref{eq:M xi}
does not have to be of finite type even if~$M$ is of finite type.
For instance, let~$I'=\{1,2,3\}$, $I=\{1,2,3,4\}$,
$M=A_4\in\Cox I$ and
define~$\xi:I'\to \mathscr F(M)$ by~$\xi(1)=\{2\}$, $\xi(2)=\{3\}$ and~$\xi(3)=\{1,4\}$.
Clearly, $\mu_{M}(\xi(1),\xi(2))=\mu_M(\xi(2),\xi(1))=3$. Since $s_2\star w_\circ^{\{1,4\}}\star s_2=
w_\circ^{\{1,2,4\}}=w_\circ^{\{1,4\}}\star s_2\star w_\circ^{\{1,4\}}$ while~$w_\circ^{\{1,4\}}$ and~$s_2$ do not commute,
$\mu_M(\xi(1),\xi(3))=\mu_M(\xi(3),\xi(1))=3$ whence, applying the diagram automorphism of~$W(M)$, we conclude that~$\mu_M(\xi(2),\xi(3))=\mu_M(\xi(3),\xi(2))=3$.
Thus, $(M_\xi)_{ij}=3$ for all~$i\not=j$ and so~$M_\xi$ is of affine type~$\widetilde A_2$.
Since for any possible lifting of~$\xi$ to
$\phi\in\Hom_{\He}(M',M)$, $W(M')$ surjects onto~$W(M_\xi)$ by Lemma~\ref{lem:lift}, it follows that no lifting of~$\xi$ is of finite type.

Also, optimal homomorphisms are not Cartesian
with respect to the functor~$[\cdot]:\He\to \mathscr{R\!e\!l}$.
Indeed,
define~$\rho:\{1,2\}\to \mathscr F(M_\xi)$ by
$\rho(1)=\{1,2\}$, $\rho(2)=\{2,3\}$. Since~$\{1,2\}\cup\{2,3\}=\{1,2,3\}\notin \mathscr F(\widetilde A_2)$, $\mu_{\widetilde A_2}(\rho(1),\rho(2))=\mu_{\widetilde A_2}(\rho(2),\rho(1))=\infty$
and so~$M_\rho=I_2(\infty)$ whence
$f_\rho$ is the unique lifting of~$\rho$ to a morphism in~$\He$ with the codomain~$\widetilde A_2$ by Lemma~\ref{lem:lift}.
On the other hand,
$(\xi\circ\rho)(1)=\{2,3\}$, $(\xi\circ\rho)(2)=\{1,3,4\}$ and one can check that
$\mu_{A_4}(\{2,3\},\{1,3,4\})=
\mu_{A_4}(\{1,3,4\},\{2,3\}))=4$,
whence~$M_{\xi\circ\rho}=B_2$.
Thus, $f_{\xi\circ\rho}\in\Hom_{\He}(B_2,A_4)$,
$[f_{\xi\circ\rho}]=\xi\circ\rho=[f_\xi]\circ\rho$ and
$\rho$ even admits a unique lifting~$f_\rho$ to a morphism in~$\He$, but the domain of~$f_\rho$ is not~$B_2$.
\end{remark}

\begin{definition}\label{defn:locally inj}
Let~$M\in\Cox I$, $M'\in \Cox{I'}$.
We say that~$\phi\in\Hom_{\He}(M',M)$
is {\em locally injective} if $\phi|_{W_{\{i,j\}}(M')}$ is injective for
all~$i\not=j\in I'$.
\end{definition}
Clearly, an injective homomorphism is locally injective.
\begin{lemma}\label{lem:loc inj comp}
Let~$M\in\Cox I$, $M'\in\Cox{I'}$, $M''\in\Cox{I''}$. Suppose that $\phi\in\Hom_{\He}(M',M)$ is injective and~$\phi'\in\Hom_{\He}(M'',M')$ is locally injective.
Then~$\phi\circ\phi'$ is locally injective.
\end{lemma}
\begin{proof}
For all~$i\not=j\in I''$, $(\phi\circ\phi')|_{W_{\{i,j\}}(M'')}=\phi\circ\phi'|_{W_{\{i,j\}}(M'')}$ is a composition of injective maps.
\end{proof}
The reason for introducing this notion is that it has a very simple combinatorial characterization and hence can be easily verified, thus eliminating all homomorphisms which cannot possibly be injective.
\begin{proposition}\label{prop:locally inj}
Let~$M\in\Cox I$, $M'\in \Cox{I'}$. Then
$\phi\in\Hom_{\He}(M',M)$ is locally injective if and only if
$\mu_{M}([\phi](i),[\phi](j))=m'_{ij}$ for all~$i\not=j\in I'$. In particular,
a locally injective homomorphism is optimal.
\end{proposition}
\begin{proof}
Abbreviate~$\mu_{ij}=\mu_M([\phi](i),[\phi](j))$
and~$\phi_{i,j}=\phi|_{W_{\{i,j\}}(M')}$,
$i\not=j\in I$. Since~$\phi\in\Hom_{\He}(M',M)$,
$\max(\mu_{ij},\mu_{ji})\le m'_{ij}$ for all~$i\not=j\in I'$ by Theorem~\ref{thm:Hom Heck Mon}, and so proving the assertion is equivalent to proving that, given~$i\not=j\in I'$,
$\phi_{i,j}$ is not injective if and only
$\min(\mu_{i,j},\mu_{j,i})<m'_{ij}$.
Suppose that~$\min(\mu_{ij},\mu_{ji})<m'_{ij}$ for some~$i\not=j\in I'$,
say $\mu_{ij}<m'_{ij}$. Then $\mu_{ij}<\infty$ and~$\mu_{ji}\le \mu_{ij}+1\le m'_{ij}$ by~\eqref{eq:mu_M JK KJ}.
It follows that~$\brd{s'_i\times s'_j\times}{\mu_{ij}}\not=\brd{s'_j\times s'_i\times}{\mu_{ji}}$ yet
$\phi(\brd{s'_i\times s'_j\times}{\mu_{ij}})=w_\circ^{[\phi](i)\cup [\phi](j)}=
\phi(\brd{s'_j\times s'_i\times}{\mu_{ji}})$. Thus, $\phi_{i,j}$ is
not injective.

Conversely, suppose that~$\phi_{i,j}$ is not injective for some~$i\not=j\in I'$. If one of the $[\phi](i)$, $[\phi](j)$ is the empty set, or if~$[\phi](i)=[\phi](j)$ then~$\min(\mu_{ij},\mu_{ji})=1<m'_{ij}$.
Otherwise, either~$\phi(\brd{s'_i\times s'_j\times}{m})
=\phi(\brd{s'_j\times s'_i\times}{n})$
for some~$m,n\in\ZZ_{>0}$ with $\min(m,n)<
m'_{ij}$,
or $\phi(\brd{s'_\alpha\times s'_\beta\times}{k})
=\phi(\brd{s'_\alpha\times s'_\beta\times}{l})$
for some~$k<l\in\ZZ_{>0}$, where $\{\alpha,\beta\}=\{i,j\}$ and
$\max(k,l)\le
m'_{ij}$ if~$m'_{ij}<\infty$.
By
Lemma~\ref{lem:G JK finite}
applied with~$J=[\phi](i)$ and~$K=[\phi](j)$, we conclude that
either~$\brd{w_\circ^{J}\star w_\circ^{K}\star}{k}=w_\circ^{J\cup K}$
or~$\brd{w_\circ^{K}\star w_\circ^{J}\star}{k}=w_\circ^{J\cup K}$
for some~$k<m'_{ij}$ and so~$\min(\mu_{ij},\mu_{ji})\le k<m'_{ij}$.
\end{proof}
The following is immediate.
\begin{corollary}\label{cor:non-loc inj->loc inj}
Let~$M\in\Cox I$, $M_k\in\Cox{I_k}$ and let~$\pi_k\in\Hom_{\He}(M_1\times M_2,M_k)$ be
canonical surjections, $k\in\{1,2\}$.
Let~$f\in\Hom_{\He}(M,M_1\times M_2)$ and let~$f_k:=\pi_k\circ f$, $k\in\{1,2\}$. Then~$f$ is locally injective if and only for every~$i\not=j\in I$ such that
$\min\{ \mu_{M_k}([f_k](i),[f_k](j)),
\mu_{M_k}([f_k](j),[f_k](i))\}<m_{ij}$, $k\in\{1,2\}$, we have
$\mu_{M_1}([f_1](i),[f_1](j))=m_{ij}-\delta_{ij}=
\mu_{M_2}([f_2](j),[f_2](i))$,
where~$\{\delta_{ij},\delta_{ji}\}=\{0,1\}$.
\end{corollary}

The following lemma allows us to produce
additional homomorphisms of Hecke monoids from already constructed ones. We call this procedure a decoration.
\begin{lemma}\label{lem:decoration}
Let~$M\in\Cox I$, $M'\in\Cox{I'}$.
Let~$\phi\in\Hom_{\He}(M',M)$ and let
$\mathbf K=\{K_i\}_{i\in I'}\subset \mathscr F(M)$ be pairwise orthogonal. Suppose that, for each~$i\in I'$, $K_i\subset \bigcap_{j\in I'\setminus\{i\}}[\phi](j)$
and is orthogonal to~$[\phi](i)$.
Then
the assignments $s'_i\mapsto w_\circ^{[\phi](i)\cup K_i}$, $i\in I'$, define a homomorphism $\phi_{\mathbf K}\in \Hom_{\He}(M',M)$. Moreover,
if~$\phi$ is (locally) injective and~$[\phi](i)\not=K_j$, $i\not=j$,
then~$\phi_{\mathbf K}$ is also (locally) injective.
\end{lemma}
\begin{proof}
Let~$i\not=j$.
By assumptions and Lemma~\ref{lem:char w_0 monoid},
\begin{align*}
\brd{w_\circ^{[\phi](i)\cup K_i}\star w_\circ^{[\phi](j)\cup K_j}\star }{m}
=\brd{w_\circ^{[\phi](i)}\star w_\circ^{K_j}\star  w_\circ^{K_i}\star w_\circ^{[\phi](j)}\star }{m}
=\brd{w_\circ^{[\phi](i)}\star w_\circ^{[\phi](j)}\star}m.
\end{align*}
Thus, $\mu_M([\phi](i)\cup K_i,
[\phi](j)\cup K_j)=\mu_M([\phi](i),[\phi](j))$ and the first assertion follows. To prove the second, note
that our assumption guarantees the injectivity on elements of length~$1$, while for any element~$w\in W(M')$
of length~$\ge 2$ we have $\phi(w)=\phi_{\mathbf K}(w)$.
\end{proof}

\subsection{Homomorphisms and Bruhat order}
An important property of homomorphisms of Hecke monoids is that, unlike those of Coxeter groups, they are order preserving.
\begin{proposition}\label{prop:hom Bruhat}
Let~$M\in\Cox I$, $M'\in\Cox{I'}$ and let~$\phi\in\Hom_{\He}(M',M)$.
Then
$\phi(w)\le \phi(w')$ for any~
 $w\le w'\in W(M')$.
\end{proposition}
\begin{proof}
The assertion is obvious if~$w=w'$.
By Proposition~\partref{prop:Bruhat order.b}, $w<w'$ implies that there exists a chain $w=u_0<\cdots<u_k=w'$ with
$k=\ell(w')-\ell(w)$. Thus, it suffices to prove the assertion for $w<w'$ with $\ell(w')=\ell(w)+1$.

Write $w'=s'_{i_1}\times \cdots\times s'_{i_r}$ with $r=\ell(w')$.
Since $\ell(w)=\ell(w')-1$ and by Proposition~\partref{prop:Bruhat order.a} every reduced expression for~$w'$ contains a reduced expression for~$w$, there exists $1\le t\le r$ such that $w'=u\times s'_{i_t}\times v$
and $w=u\times v$ with $u=s'_{i_1}\times\cdots\times s'_{i_{t-1}}$ and $v=s'_{i_{t+1}}\times\cdots \times s'_{i_r}$. Then, since~$1$ is the minimal element of~$W(M)$ in the Bruhat order, we obtain, by Proposition~\partref{prop:Bruhat order *.a}
\begin{align*}
\phi(w)&=\phi(u\times v)=\phi(u)\star \phi(v)=\phi(u)\star 1\star \phi(v)\le
\phi(u)\star \phi(s'_{i_t})\star \phi(v)\\
&=\phi(u\star s'_{i_t}\star v)=\phi(w').\qedhere
\end{align*}
\end{proof}

Let~$M\in\Cox I$, $M'\in\Cox{I'}$ and~$\phi,\psi\in\Hom_{\He}(M',M)$. We say that~$\psi$ {\em dominates} $\phi$ if~$\phi(w)\le \psi(w)$ for all~$w\in W(M')$.
\begin{lemma}\label{lem:dominates}
Let~$M\in\Cox I$, $M'\in\Cox{I'}$ and~$\phi,\psi\in\Hom_{\He}(M',M)$.
Then~$\psi$ dominates~$\phi$ if and only if~$\phi(s'_i)\le \psi(s'_i)$ for all~$i\in I'$.
\end{lemma}
\begin{proof}
The forward direction is trivial. To prove the converse,
we use induction on~$\ell(w)$. For~$\ell(w)=0$ there is nothing to prove while for~$\ell(w)=1$ the assertion is just the hypothesis.
Suppose that~$\ell(w)>1$ and write~$w=u\times s'_i$ for some~$i\in I'$, $u\in W(M')$. Then~$\phi(w)=\phi(u)\star \phi(s'_i)\le \psi(u)\star\psi(s'_i)=
\psi(u\star s'_i)=\psi(w)$ by Proposition~\partref{prop:Bruhat order *.a}
and by the induction hypothesis.
\end{proof}

\subsection{Parabolic projections}\label{subs:prb projs}
A very important class of homomorphisms of Hecke monoids are {\em parabolic projections}.
\begin{lemma}\label{lem:p J defn}\label{lem:proj w0}
Let~$J\subset I$. The assignments
$$
s_i\mapsto \begin{cases}1,&i\in I\setminus J,\\
s_i,&i\in J,
\end{cases}
$$
for all~$i\in I$,
define a surjective homomorphism \plink{pJ}$p_J:(W(M),\star)\to (W_J(M),\star)$. Moreover, $p_J(w_\circ^K)=w_\circ^{J\cap K}$
for all~$K\in\mathscr F(M)$.
\end{lemma}
\begin{proof}
Define $\xi_J:I\to \mathscr F(M_J)$
by $\xi_J(i)=\{i\}$ if~$i\in J$ and~$\xi_J(i)=\emptyset$
if~$i\in I\setminus J$. Then~$\xi\in\Lambda(M,M_J)$
and~$p_J=\Theta_{\xi_J}$ in the notation of Lemma~\ref{lem:Lam M'M HomM'M}.
The second assertion follows from Lemma~\ref{lem:Hecke hom w0J} since~$[p_J](K)=\bigcup_{k\in K}\xi_J(k)=K\cap J$.
\end{proof}
\begin{remark}\label{rem:parab proj Cox}
Note that the assignments from Lemma~\ref{lem:p J defn}
define a homomorphism of Coxeter groups if and only
if $m_{ij}$ is even for all $i\in I\setminus J$, $j\in J$.
In particular, for irreducible Coxeter matrices of finite type, the only non-trivial examples of parabolic
projections of Coxeter groups are $p_{[1,n-1]}\in\Hom_{\Co}(B_n,A_{n-1})$,
$p_{\{1,2\}}\in\Hom_{\Co}(F_4, A_2)$, $p_{\{n\}} \in\Hom_{\Co}(B_n, A_1)$ and $p_{\{i\}}\in \Hom_{\Co}(I_2(2m), A_1)$, $m \ge 2$, $i \in \{1, 2\}$. It should also be noted that, even when
a parabolic projection of Coxeter groups is defined,
it is not the same map of sets.
For example, $p_{[1,n-1]}\in\Hom_{\Co}(B_n,A_{n-1})$ maps $s_{n-1}s_n s_{n-1}$ to~$1$, while
its Hecke counterpart maps the same element to~$s_{n-1}$.
\end{remark}
Sometimes it is convenient to treat~$p_J$ as an endomorphism of~$(W(M),\star)$.
The following is immediate from the definition of~$p_J$.
\begin{lemma}\label{lem:p_J comp}
For any~$J,K\subset I$, $p_J\circ p_K=
p_{J\cap K}$.
\end{lemma}
\begin{lemma}\label{lem:parab prod}
Let~$M\in\Cox I$ and suppose that
$J\subset I$ and~$I\setminus J$ are orthogonal. Then
$w=p_J(w)\times p_{I\setminus J}(w)$ for all~$w\in W(M)$.
\end{lemma}
\begin{proof}
Since~$J$ and~$I\setminus J$ are orthogonal, $W(M)$ is naturally isomorphic to~$W_J(M)\times W_{I\setminus J}(M)$; in particular,
for every~$w\in W$, we have $w=w'\times w''$ where~$w'\in W_J(M)$ and~$w''\in W_{I\setminus J}(M)$. It remains to observe that~$w'=p_J(w')$ and~$w''=p_{I\setminus J}(w'')$.
\end{proof}
\begin{lemma}\label{lem:prj hom}\label{lem:diag aut proj}\label{lem:taut pJ commute}
Let~$M\in\Cox I$, $M'\in\Cox{I'}$ and let~$\phi\in\Hom_{\He}(M',M)$. Then~$p_{[\phi](J')}\circ\phi$ dominates~$\phi\circ p_{J'}$
for all~$J'\subset I'$. Moreover, $\phi\circ p_{J'}=p_{[\phi](J')}\circ\phi$ for all~$J'\subset I'$ if and only if~$\phi$ is disjoint.
In particular,
if~$\sigma$ is a diagram automorphism of~$(W(M),\star)$ then
$\sigma\circ p_J=p_{\sigma(J)}\circ\sigma$ for all~$J\subset I$.
\end{lemma}
\begin{proof}
Let~$J'\subset I'$, $f_{J'}=\phi\circ p_{J'}$ and~$g_{J'}=p_{[\phi](J')}\circ\phi$. By Lemma~\ref{lem:dominates},
$g_{J'}$ dominates~$f_{J'}$
if and only if~$f_{J'}(s'_i)\le g_{J'}(s'_i)$ for all~$i\in I'$. Furthermore, since~$f_{J'}=g_{J'}$ if and only if each of them
dominates the other, $f_{J'}=g_{J'}$ if and only if~$f_{J'}(s'_i)=g_{J'}(s'_i)$ for all~$i\in I'$.
By Lemmata~\ref{lem:Hecke hom w0J} and~\ref{lem:proj w0}, for all~$i\in I'$ we have
$$
g_{J'}(s'_i)=w_\circ^{[\phi](i)\cap [\phi](J')},\qquad
f_{J'}(s'_i)=\begin{cases}
1,&i\in I'\setminus J',\\
w_\circ^{[\phi](i)},&i\in J'.
\end{cases}
$$
If~$i\in J'$ then~$[\phi](i)\subset[\phi](J')$ and so~$f_{J'}(s'_i)=g_{J'}(s'_i)$, while if~$i\in I'\setminus J'$ then~$f_{J'}(s'_i)=1\le g_{J'}(s'_i)$,
with equality if and only if~$[\phi](i)\cap[\phi](J')=\emptyset$. This proves the first assertion. Furthermore, $[\phi](i)\cap[\phi](J')=\emptyset$ for all~$J'\subset I'$
and~$i\in I'\setminus J'$ if and only if~$[\phi](i)\cap[\phi](j)=\emptyset$ for all~$i\not=j\in I'$, that is, if and only if~$\phi$
is disjoint. This proves the second assertion. The last assertion follows
since diagram automorphisms are disjoint.
\end{proof}

\begin{lemma}\label{lem: p_J max}
Let~$M\in\Cox I$ and let~$J\subset I$. Then $\downarrow p_J(w)=W_J(M)\cap \downarrow w$ for all~$w\in W(M)$.
\end{lemma}
\begin{proof}
By Lemma~\ref{lem:dominates}, $\id\in\Hom_{\He}(M,M)$
dominates~$p_J$, regarded as
an element of~$\Hom_{\He}(M,M)$,
and so~$p_J(w)\le w$ for all~$w\in W$, that is, $p_J(w)\in W_J(M)\cap \downarrow w$.
Let~$v\in W_J(M)\cap \downarrow w$. Then~$v\le w$ and
so~$v=p_J(v)\le p_J(w)$ by Proposition~\ref{prop:hom Bruhat}.
\end{proof}
Quite surprisingly, parabolic projections can be used to prove the following natural commutation criterion  for elements with disjoint supports.
\begin{proposition}\label{prop:w * w'=w'*w}
Let~$M\in\Cox I$. The following
are equivalent for~$w,w'\in W(M)$ such that~$\supp w\cap \supp w'=\emptyset$
\begin{enmalph}
\item \label{prop:w * w'=w'*w.a}
$w$ and~$w'$ commute in $(W(M),\star)$;
\item\label{prop:w * w'=w'*w.b} $w$ and~$w'$ commute in~$W(M)$;
\item\label{prop:w * w'=w'*w.c} $\supp w$ and~$\supp w'$ are orthogonal.
\end{enmalph}
\end{proposition}
\begin{proof}
Let~$J=\supp w$, $J'=\supp w'$. Since
$J\cap J'=\emptyset$, $w\star w'=ww'$ and so \ref{prop:w * w'=w'*w.a}
is equivalent to \ref{prop:w * w'=w'*w.b}. The implication~\ref{prop:w * w'=w'*w.c}$\implies$\ref{prop:w * w'=w'*w.b} is trivial.

It remains to prove that~\ref{prop:w * w'=w'*w.a} implies~\ref{prop:w * w'=w'*w.c}. Let~$i\in J$, $j\in J'$; in particular, $i\not=j$ since~$J\cap J'=\emptyset$. Then~$p_{\{i,j\}}(w)=s_i$, $p_{\{i,j\}}(w')=s_j$ and so
$s_i\star s_j=p_{\{i,j\}}(w\star w')=
p_{\{i,j\}}(w'\star w)=s_j\star s_i$. Thus, $J$ and~$J'$ are orthogonal.
\end{proof}
Clearly, if~$w$ and~$w'$ commute then so do their powers. More generally,
for any pairwise commuting set~$\{w_i\}_{1\le i\le k}$ any two elements
in the submonoid of~$(W(M),\star)$ (or the subgroup of~$W(M)$) generated by these elements commute. One can expect that the commutation problem in Hecke monoids is significantly easier than in Coxeter and other groups.

We will now use parabolic projections to describe all pairs of commuting idempotents in Hecke monoids.
\begin{proposition}\label{prop:comm idemp}
Let~$M\in\Cox I$ and $J,K\in\mathscr F(M)$. The following are equivalent:
\begin{enmalph}
\item\label{prop:comm idemp.a} $w_\circ^J\star w_\circ^K=w_\circ^K\star w_\circ^J$;
\item\label{prop:comm idemp.b} $J\cup K\in\mathscr F(M)$ and $w_\circ^J\star w_\circ^K=w_\circ^{J\cup K}$;
\item\label{prop:comm idemp.b'} $J\cup K\in\mathscr F(M)$ and $\ell(w_\circ^{J\cup K})+\ell(w_\circ^{J\cap K})= \ell(w_\circ^J)+\ell(w_\circ^K)$;
\item\label{prop:comm idemp.c}
if~$J'$ and~$K'$ are connected components of~$J$ and~$K$, respectively,
then either they are orthogonal
or one of them is contained in the other;
\item\label{prop:comm idemp.d}  if~$J'$ and~$K'$ are connected
components of~$J$ and~$K$ then
$w_\circ^{J'}$ and~$w_\circ^{K'}$ commute.
\end{enmalph}
In particular, $\mu_M(J,K)\le 2$ if and only if~$\mu_M(K,J)\le 2$.
\end{proposition}
\begin{proof}
If~\ref{prop:comm idemp.a} holds then $x:=w_\circ^J\star w_\circ^K$ satisfies $x\star w_\circ^J=w_\circ^K\star w_\circ^J\star w_\circ^J=x$ and~$x\star w_\circ^K=x$,
hence $\brd{w_\circ^J\star w_\circ^K\star}{m}=x$ for all~$m\ge 2$ and so~\ref{prop:comm idemp.b} follows from Lemma~\ref{lem:G JK finite}.

We now prove that~\ref{prop:comm idemp.b} is equivalent to~\ref{prop:comm idemp.b'}.
The implication \ref{prop:comm idemp.b}$\implies$\ref{prop:comm idemp.b'} is immediate from Lemma~\ref{lem:w0J*w0K}.
Conversely, if~\ref{prop:comm idemp.b'} holds then~$\ell(w_\circ^J\star w_\circ^K)=\ell(w_\circ^{J\cup K})$
by Lemma~\ref{lem:w0J*w0K}. Since~$w_\circ^J\star w_\circ^K\in W_{J\cup K}(M)$ and~$w_\circ^{J\cup K}$ is the unique element of the latter of maximal length, it follows that~$w_\circ^J\star w_\circ^K=w_\circ^{J\cup K}$.

We now prove that~\ref{prop:comm idemp.b} implies~\ref{prop:comm idemp.c}.
Let~$J'$ and~$K'$ be connected components of~$J$ and~$K$, respectively. Suppose that they are not orthogonal, that is, $J'\cup K'$ is connected. If both~$J'\setminus K'$,
$K'\setminus J'$ are non-empty, choose a path of minimal
length in~$\Gamma_{J'\cup K'}(M)$
connecting some vertex~$j\in J'\setminus K'$
to a vertex~$k\in K'\setminus J'$, and let~$S$ be the set of vertices of that path. By the minimality of its length,
$S\setminus\{j,k\}\subset J'\cap K'$.
Since~$j$ is adjacent to a vertex of~$S\setminus\{j\}\subset K'$, $j\notin K$, for otherwise, since~$K'$ is a connected component of~$K$, we would have $j\in K'$ which is a contradiction.
Similarly, $k\notin J$. Thus, $S\cap J=S\setminus\{k\}$ and~$S\cap K=S\setminus\{j\}$.
Applying the parabolic projection~$p_S$
to~$w_\circ^J\star w_\circ^K=w_\circ^{J\cup K}$, we
obtain, by
Lemma~\ref{lem:proj w0},
$$
w_\circ^{S\cap J}\star
w_\circ^{S\cap K}
=w_\circ^{S\setminus\{k\}}\star
w_\circ^{S\setminus\{j\}}=
w_\circ^{(J\cup K)\cap S}
=w_\circ^{S}.
$$
Then by part~\ref{prop:comm idemp.b'} applied with~$J:=S\setminus\{k\}$ and~$K:=S\setminus\{j\}$,
\begin{equation}\label{eq:length table}
\Delta(S):=\ell(w_\circ^S)+\ell(w_\circ^{S\setminus\{j,k\}})-\ell(w_\circ^{S\setminus\{j\}})-\ell(w_\circ^{S\setminus\{k\}})=0.
\end{equation}
Let~$n=|S|$.
Since $\Gamma_S(M)$ is a path and $S\in\mathscr F(M)$, $M_S$ is of type $I_2(m)$, $m\ge 3$ (with~$n=2$), $A_{n}$, $B_{n}$, $F_4$ (with~$n=4$), $H_3$ (with~$n=3$) or~$H_4$ (with~$n=4$),
with the vertices $j$ and~$k$ at the ends of the path. The values
of~$\Delta(S)$
for all possible types of~$S$ are shown in the following table
$$
\begin{array}{c|c|c|c|c|c}
M_S&\ell(w_\circ^S)&\ell(w_\circ^{S\setminus\{j\}})&\ell(w_\circ^{S\setminus\{k\}})
&\ell(w_\circ^{S\setminus\{j,k\}})&\Delta(S)\\
\hline
I_2(m),\,m\ge 3&m&1&1&0&m-2\\
A_n&\binom{n+1}2&\binom{n}2&\binom{n}2&\binom{n-1}2&1\\
B_n&n^2&(n-1)^2&\binom{n}2&\binom{n-1}2&n\\
F_4&24&9&9&4&10\\
H_3&15&5&3&1&8\\
H_4&60&15&6&3&42
\end{array}
$$
where we use the bijection~$S\to [1,n]$
with~$j\mapsto 1$ and~$k\mapsto n$ as in~\eqref{eq:Coxeter graphs}.
Since~$\Delta(S)>0$ in all cases, this contradicts~\eqref{eq:length table}. Thus, if~$J'$ and~$K'$ are not orthogonal then
one of them is contained in the other.

Suppose that~\ref{prop:comm idemp.c}
holds and let~$J'$, $K'$ be as above.
If~$J'$ and~$K'$ are orthogonal then
$w_\circ^{J'}$ and~$w_\circ^{K'}$
commute by Lemma~\ref{lem:orth factors}.
Otherwise, if say~$J'\subset K'$ then
$w_\circ^{J'}\star w_\circ^{K'}=w_\circ^{K'}=w_\circ^{K'}\star w_\circ^{J'}$ by Lemma~\ref{lem:char w_0 monoid}. Thus, \ref{prop:comm idemp.d}
follows from~\ref{prop:comm idemp.c}.

The remaining implication~\ref{prop:comm idemp.d}$\implies$\ref{prop:comm idemp.a}
is immediate since both~$w_\circ^J$ and~$w_\circ^{K}$ are products of
commuting idempotents corresponding to
their connected components.

The last assertion follows since~\ref{prop:comm idemp.b} is equivalent to~$\mu_M(J,K)\le 2$ while~\ref{prop:comm idemp.a} is symmetric in~$J$ and~$K$.
\end{proof}
\begin{remark}
Note that commutation in~$W(M)$ and~$(W(M),\star)$ are rather different.
For example, $w_\circ^{\{1,2\}}$
and~$w_\circ^{\{2,3\}}$ are easily
seen to commute in~$W(A_3)$, being disjoint transpositions, yet
they do not commute in~$(W(A_3),\star)$ by Proposition~\ref{prop:comm idemp}.

On the other hand, if~$J\subset K\in\mathscr F(M)$ then $w_\circ^J$ and~$w_\circ^K$ always commute in~$(W(M),\star)$. In the Coxeter group, under the same assumption, they commute if and only if~$\sigma_K(J)=J$, where~$\sigma_K$ is
induced by the automorphism $w\mapsto w_\circ^K w w_\circ^K$, $w\in W_K(M)$, of~$W_K(M)$.
\end{remark}
\begin{corollary}\label{cor:hom A1 x A1}
Let~$M\in \Cox I$.
The assignments~$s'_i\mapsto w_\circ^{J_i}$, $i\in [1,n]$, $J_i\in
\mathscr F(M)$, define a
homomorphism~$\phi:(W(A_1^{\times n}),\star)\to (W(M),\star)$ if and only
if every pair~$J_i$, $J_k$, $i\not=k\in [1,n]$ satisfies
the condition~\ref{prop:comm idemp.c}
of Proposition~\ref{prop:comm idemp}.
Moreover,
\begin{enmalph}
\item $\phi$ is locally injective if and only
if $J_i\not\subset J_k$
for all $i\not=k\in[1,n]$.
\item $\phi$ is injective if and only if
$J_i\not\subset\bigcup_{k\in[1,n]\setminus\{i\}}J_k$ for all~$i\in[1,n]$ (that is,
$[\phi]:\mathscr P([1,n])\to\mathscr P(I)$ is injective).
\end{enmalph}
\end{corollary}
\begin{example}\label{ex:loc inj, not inj}
Let~$n=3$, $M=A_5$ and let~$J_1=\{1,3\}$, $J_2=\{3,5\}$, $J_3=\{1,5\}$.
By Corollary~\ref{cor:hom A1 x A1},
the assignments~$s'_i\mapsto w_\circ^{J_i}$, $i\in[1,3]$, define a
locally injective~$\phi\in\Hom_{\He}(A_1^{\times 3},M)$. However,
$\phi$ is not injective since
$J_1\subset J_2\cup J_3$.
\end{example}
We will now discuss the combinatorics of commuting pairs of idempotents in irreducible
Hecke monoids of finite type.
Let~$\Gamma$ be a finite simple graph with the set of vertices~$V$. We say that~$A,B\subset V$ are {\em orthogonal} if~$A\cap B=\emptyset$ and no vertex
of~$A$ is adjacent to a vertex of~$B$, and that they are {\em nested} if~$A\subset B$ or~$B\subset A$. Given~$A\subset V$, we denote by~$\Gamma_A$ the subgraph of~$\Gamma$ induced by~$A$
and set~$\Gamma\setminus A:=\Gamma_{V\setminus A}$; thus, $\Gamma\setminus A$ is obtained from~$\Gamma$ by removing all vertices in~$A$
and all edges incident to them. The connected components of~$A\subset V$ are
those of~$\Gamma_A$, and~$\partial A$ denotes the set of vertices in~$V\setminus A$ adjacent to a vertex of~$A$.

Let~$\mathcal C(\Gamma)$ be the set of pairs~$(J,K)\in\mathscr P(V)\times\mathscr P(V)$ such that every connected component of~$J$ and every connected
component of~$K$ are either nested or orthogonal, and let~$c(\Gamma)=|\mathcal C(\Gamma)|$; we set~$c(\emptyset)=1$.
It is immediate from the definition
that if~$\Gamma_1,\dots,\Gamma_k$
are connected components of~$\Gamma$ then
\begin{equation}\label{eq:c Gamma product}
c(\Gamma)=c(\Gamma_1)\cdots c(\Gamma_k).
\end{equation}
Clearly, if~$\Gamma$ is the
Coxeter graph of~$M\in\Cox I$ then~$J,K\subset I$ are orthogonal in the above sense if and only if they are orthogonal in the sense of~\S\ref{subs:Br(M)W(M)}.

\begin{proposition}\label{prop:enumerate comm pairs}
\begin{enmalph}
\item\label{prop:enumerate comm pairs.a} Let~$v\in V$ be a leaf of~$\Gamma$. Then
$$
c(\Gamma)=2c(\Gamma\setminus\{v\})+2\sum_{J'}2^{|J'|-1}c\big(\Gamma\setminus(J'\cup\partial J')\big),
$$
where the sum is over all connected~$J'\subset V$ containing~$v$.
\item\label{prop:enumerate comm pairs.c} If~$\Gamma$ is a tree with~$n$ vertices all of which have degree at most~$3$, then~$c(\Gamma)=c_n$, where~$c_0=1$, $c_1=4$ and
$c_n=4c_{n-1}-2c_{n-2}$, $n\ge 2$. Also, $c_n=2^{\frac12n-1}\big((\sqrt2+1)^{n+1}-(\sqrt2-1)^{n+1}\big)=
\sum\limits_{k\ge 0}\binom{n+1}{2k+1}2^{n-k}$ (cf.~\cite{OEIS}*{\OEIS{A007070}}).
\item\label{prop:enumerate comm pairs.b} Let~$M\in\Cox I$ and let~$c(M)$ be the number of pairs~$(J,K)\in\mathscr F(M)\times\mathscr F(M)$ such
that~$w_\circ^J\star w_\circ^K=w_\circ^K\star w_\circ^J$. Then~$c(M)=|\mathcal C(\Gamma(M))\cap(\mathscr F(M)\times\mathscr F(M))|$. In particular,
$c(M)=c(\Gamma(M))$ if~$M$ is of finite type, while~$c(M)=c(\Gamma(M))-2^{|I|+1}+1$ if~$\mathscr F(M)=\mathscr P(I)\setminus\{I\}$
(for example, if~$M$ is of affine type).
\item\label{prop:enumerate comm pairs.d} If~$M$ is irreducible of finite type (respectively, of affine type other than~$\widetilde A_n$,
$n\ge 2$, and~$\widetilde D_4$), then~$c(M)=c_{|I|}$ (respectively, $c(M)=c_{|I|}-2^{|I|+1}+1$).
\end{enmalph}
\end{proposition}
\begin{proof}
To prove part~\ref{prop:enumerate comm pairs.a}, let~$\mathcal C=\mathcal C(\Gamma)$ and write~$\mathcal C=\mathcal C^0\sqcup
\mathcal C^1\sqcup\mathcal C^2\sqcup\mathcal C^3$, where
\begin{align*}
&\mathcal C^0=\{(J,K)\in\mathcal C\,:\,v\notin J\cup K\},&&\mathcal C^1=\{(J,K)\in\mathcal C\,:\,v\in J,\, v\notin K\},\\
&\mathcal C^2=\{(J,K)\,:\,(K,J)\in\mathcal C^1\},&&\mathcal C^3=\{(J,K)\in\mathcal C\,:\,v\in J\cap K\}.
\end{align*}
Clearly, $\mathcal C^0=\mathcal C(\Gamma\setminus\{v\})$ and~$|\mathcal C^2|=|\mathcal C^1|$.

Let~$\Gamma'=\Gamma\setminus\{v\}$. Define~$f:\mathcal C^3\to \mathscr P(V\setminus\{v\})\times\mathscr P(V\setminus\{v\})$
and~$g:\mathcal C(\Gamma')\to \mathscr P(V)\times \mathscr P(V)$ by, respectively,
$(J,K)\mapsto (J\setminus \{v\},K\setminus \{v\})$ and~$(S,T)\mapsto (S\cup\{v\},T\cup\{v\})$. Clearly, $f$ and~$g$
are inverses of each other once we establish that~$f(\mathcal C^3)\subset\mathcal C(\Gamma')$ and~$g(\mathcal C(\Gamma'))\subset\mathcal C^3$,
which yields~$|\mathcal C^3|=c(\Gamma')$. Let~$u$ be the unique vertex of~$\Gamma$ adjacent to~$v$.
Let~$(J,K)\in\mathcal C^3$ and let~$J'$, $K'$ be the connected components of~$J$, $K$, respectively, containing~$v$. Since~$J'$ and~$K'$ are not
orthogonal, they are nested, and so are~$J'\setminus\{v\}$ and~$K'\setminus \{v\}$, which are connected (or empty) since~$v$ is a leaf.
Every other component of~$J$ does not contain~$u$ and so is orthogonal to (respectively, contained in)~$K'$ if and only if it is
orthogonal to (respectively, contained in)~$K'\setminus\{v\}$, and similarly for the components of~$K$ other than~$K'$ vis-\`a-vis~$J'$.
Thus, $f(\mathcal C^3)\subset\mathcal C(\Gamma')$. Likewise, if~$(S,T)\in\mathcal C(\Gamma')$ and~$S'$, $T'$ are their respective
components containing~$u$ (or empty), then~$S'$ and~$T'$ are nested, hence so are~$S'\cup\{v\}$ and~$T'\cup\{v\}$, which are the components
of~$S\cup\{v\}$ and~$T\cup\{v\}$ containing~$v$. All other components of~$S$ (respectively, $T$) do not contain~$u$, hence are orthogonal to~$\{v\}$,
and so behave the same way vis-\`a-vis~$T'\cup\{v\}$ (respectively, $S'\cup\{v\}$) as they do with respect to~$T'$ (respectively, $S'$).
In other words, $g(\mathcal C(\Gamma'))\subset\mathcal C^3$.

Let~$(J,K)\in\mathcal C^1$ and let~$J'$ be the connected component of~$J$ containing~$v$. Then~$J\cap\partial J'=\emptyset$. A connected component
of~$K$ is either contained in or orthogonal to~$J'$ if and only if it is contained in~$J'\setminus\{v\}$ or
in~$V\setminus(J'\cup\partial J')$, which is equivalent to~$K\cap\partial J'=\emptyset$. Furthermore, all other components of~$J$, as well as the
components of~$K$ not contained in~$J'$, are contained in~$V\setminus(J'\cup\partial J')$ and hence are orthogonal to~$J'$. Thus, $(J,K)$ is
uniquely determined by~$J'$, an arbitrary subset~$K\cap J'$ of~$J'\setminus\{v\}$ and an arbitrary pair~$(J\setminus J',K\setminus J')$
in~$\mathcal C(\Gamma\setminus(J'\cup\partial J'))$, whence
$$
|\mathcal C^1|=\sum_{J'}2^{|J'|-1}c\big(\Gamma\setminus(J'\cup\partial J')\big).
$$
Since~$c(\Gamma)=|\mathcal C^0|+|\mathcal C^1|+|\mathcal C^2|+|\mathcal C^3|$, this proves part~\ref{prop:enumerate comm pairs.a}.

To prove part~\ref{prop:enumerate comm pairs.c}, suppose first that the tree~$\Gamma$ is a path,
that is, all its vertices have degree at most~$2$.
Identify~$V$ with~$[1,n]$ so that~$i$ and~$j$ are adjacent if and only if~$|i-j|=1$, and
let~$c_n=c(\Gamma)$; clearly~$c_0=1$ and~$c_1=4$. Let~$v=1$. Then all connected subsets~$J'$ of~$[1,n]$
containing~$v$ are of the form
$J'=[1,b]$, $b\in[1,n]$,
$\partial J'=\{b+1\}\cap[1,n]$ and~$\Gamma\setminus(J'\cup\partial J')$ is a path with vertex set~$[b+2,n]$.
Thus, by part~\ref{prop:enumerate comm pairs.a},
\begin{equation}\label{eq: cn c'n}
c_n=2c_{n-1}+2c'_n,\qquad c'_n=\sum_{1\le b\le n}2^{b-1}c_{[n-b-1]_+},
\end{equation}
where~$[x]_+=\max(0,x)$, $x\in\ZZ$.
It is immediate from the definition of~$c'_n$ that $$
c'_n=
c_{n-2}+\sum_{2\le b\le n} 2^{b-1} c_{[n-b-1]_+}=c_{n-2}+
2c'_{n-1},\qquad n\ge 2.
$$
It follows that
$$
c_n-2c_{n-1}=2c'_n=4c'_{n-1}+2c_{n-2}=2(c_{n-1}-2c_{n-2})+2c_{n-2},
$$
whence $c_n=4c_{n-1}-2c_{n-2}$. The first closed formula for~$c_n$ is obtained by elementary linear algebra, while the second is immediate from the first.

Suppose now that~$\Gamma$ is an arbitrary tree with~$n$ vertices all of which have degree at most~$3$. We prove that~$c(\Gamma)=c_n$
by induction on~$n$. If~$\Gamma$ is a path, the equality has already been established. Otherwise, choose a leaf~$v$ of~$\Gamma$ and let~$o$ be
the vertex of degree~$3$ closest to~$v$. Let~$P$ be the set of vertices of the path from~$v$ to~$o$ other than~$o$, so that all elements
of~$P$ have degree at most~$2$, and let~$p=|P|\ge 1$. Let~$Q$ and~$R$ be the vertex sets of the two connected components
of~$\Gamma\setminus(P\cup\{o\})$, and let~$q=|Q|$, $r=|R|$, so that~$n=p+q+r+1$. Denote by~$u_Q$ (respectively,~$u_R$)
the unique vertex in~$Q$ (respectively,~$R$) adjacent to~$o$. Note that~$\Gamma\setminus\{o\}$ is the disjoint union of~$\Gamma_P$, $\Gamma_Q$ and~$\Gamma_R$, and that~$Q$ and~$R$ may contain vertices of degree~$3$.
We now apply part~\ref{prop:enumerate comm pairs.a} to~$\Gamma$ and~$v$. Since~$\Gamma\setminus\{v\}$ is a tree with~$n-1$ vertices all of which
have degree at most~$3$, $c(\Gamma\setminus\{v\})=c_{n-1}$ by the induction hypothesis. The subsets~$J'$ in the sum in part~\ref{prop:enumerate comm pairs.a} are of the following types, where the term corresponding
to~$J'$ is~$2^{|J'|-1}c\big(\Gamma\setminus(J'\cup\partial J')\big)$:
\begin{enmroman}
\item $J'\subsetneq P$ consists of the first~$i$ vertices of the path from~$v$ to~$o$, $1\le i<p$. Then~$|J'|=i$ and
$\Gamma\setminus(J'\cup \partial J')$ is a tree with~$n-i-1$ vertices all of which have degree at most~$3$. By the induction hypothesis,
the term corresponding to~$J'$ equals~$2^{i-1}c_{n-i-1}$.
\item $J'=P$. Then~$\partial J'=\{o\}$ and~$\Gamma\setminus (J'\cup \partial J')$
is the disjoint union of~$\Gamma_Q$ and~$\Gamma_R$. By~\eqref{eq:c Gamma product} and the induction hypothesis,
the term corresponding to~$J'$ equals~$2^{p-1}c(\Gamma_Q)c(\Gamma_R)=2^{p-1}c_qc_r$.
\item $o\in J'$. Then $J'=P\cup\{o\}\cup X\cup Y$, where~$X\subset Q$ (respectively, $Y\subset R$) is either empty or connected
and contains~$u_Q$ (respectively,~$u_R$), and~$|J'|=p+1+|X|+|Y|$. Furthermore, $\Gamma\setminus(J'\cup\partial J')$ is the disjoint union
of~$\Gamma_{Q_X}$ and~$\Gamma_{R_Y}$, where~$Q_X=Q\setminus(X\cup\partial X\cup\{u_Q\})$
and~$R_Y=R\setminus(Y\cup\partial Y\cup\{u_R\})$. By~\eqref{eq:c Gamma product}, the term corresponding
to~$J'$ equals~$2^{p+|X|+|Y|}c(\Gamma_{Q_X})c(\Gamma_{R_Y})$, and so the sum of the terms corresponding to all~$J'$ of this type equals
$$
\sum_{X,Y}2^{p+|X|+|Y|}c(\Gamma_{Q_X})c(\Gamma_{R_Y})=2^pS_QS_R,\qquad S_Q=\sum_{X} 2^{|X|}c(\Gamma_{Q_X}),\quad S_R=\sum_{Y} 2^{|Y|}c(\Gamma_{R_Y}),
$$
the sums being over all~$X$ and~$Y$ as above.
\end{enmroman}
Thus,
\begin{equation}\label{eq:c Gamma tree}
c(\Gamma)=2c_{n-1}+2\sum_{1\le i\le p-1} 2^{i-1} c_{n-i-1}+2^p c_q c_r+2^{p+1}S_Q S_R.
\end{equation}
To compute~$S_Q$, let~$\Gamma^+_Q$ be the tree obtained from~$\Gamma_Q$
by attaching a new leaf~$v'$ to~$u_Q$. Since~$u_Q$ has degree at most~$2$ in~$Q$, all vertices of~$\Gamma^+_Q$ have degree at most~$3$,
and~$\Gamma^+_Q$ has~$q+1<n$ vertices. The connected subsets of the vertex set of~$\Gamma^+_Q$ containing~$v'$ are
precisely~$\{v'\}\cup X$, where~$X$ is as above, and the corresponding graph~$\Gamma^+_Q\setminus((\{v'\}\cup X)\cup\partial(\{v'\}\cup X))$
is~$\Gamma_{Q_X}$. Applying part~\ref{prop:enumerate comm pairs.a} to~$\Gamma^+_Q$ and its leaf~$v'$ and using the induction hypothesis,
we obtain~$c_{q+1}=2c_q+2S_Q$, whence~$S_Q=c'_{q+1}$ by~\eqref{eq: cn c'n}. Note that this also holds
if~$Q=\emptyset$, in which case~$S_Q=1=c'_1$. Similarly, $S_R=c'_{r+1}$. Thus, by~\eqref{eq:c Gamma tree},
$$
c(\Gamma)=2 c_{n-1}+2 \sum_{1\le i\le p-1}2^{i-1}c_{n-i-1}+2^p( c_q c_r
+2 c'_{q+1}c'_{r+1}).
$$
On the other hand, by~\eqref{eq: cn c'n}
\begin{align*}
c_n&=2c_{n-1}+2 c'_{n}
=2c_{n-1}+2\sum_{1\le i\le p-1}2^{i-1}c_{n-i-1}+2^p\sum_{1\le i\le q+r+2} 2^{i-1} c_{[q+r+1-i]_+}\\
&=c(\Gamma)+2^p( c'_{q+r+2}-c_q c_r-2c'_{q+1}c'_{r+1}).
\end{align*}
Thus, to complete the inductive step, it remains to prove that
$$
c'_{q+r+2}=c_qc_r+2c'_{q+1}c'_{r+1}
,\qquad q,r\ge 0.
$$
Since~$c'_n=\frac12(c_n-2c_{n-1})$, this is equivalent to
$$
c_{q+r+2}=2c_{q+r+1}+c_{q+1}c_{r+1}-2 c_q c_{r+1}-2 c_{q+1}c_r+6 c_q c_r,\qquad q,r\ge 0,
$$
or, using the recurrence for~$c_n$,
\begin{equation}\label{eq:c' identity}
c_{q+1}c_{r+1}-2 c_q c_{r+1}-2 c_{q+1}c_r+6 c_qc_r+2c_{q+r}-2c_{q+r+1}=0,\qquad q,r\ge0.
\end{equation}
For~$q\ge 0$ fixed, the left hand side of~\eqref{eq:c' identity}, regarded as a sequence indexed by~$r$, is a linear combination with constant
coefficients of the sequences~$\{c_{q+r+k}\}_{r\ge 0}$ and~$\{c_{r+k}\}_{r\ge 0}$, $k\in\{0,1\}$, and hence satisfies the
recurrence~$x_r=4x_{r-1}-2x_{r-2}$, $r\ge 2$.
Thus, to prove~\eqref{eq:c' identity} it remains to check that it holds for~$r=0$ and~$r=1$ and for all~$q\ge 0$. The left hand side of~\eqref{eq:c' identity} with~$r=0$ is equal to
\begin{align*}
4c_{q+1}-8 c_q-2 c_{q+1}+6 c_q+2 c_{q}-2c_{q+1}=0,
\end{align*}
while for~$r=1$ we obtain
\begin{align*}
14c_{q+1}-28 c_q-8 c_{q+1}+24 c_q+2 c_{q+1}-2c_{q+2}=2(4 c_{q+1}-2c_q-c_{q+2})=0.
\end{align*}

The first assertion in
part~\ref{prop:enumerate comm pairs.b} is immediate from Proposition~\partref{prop:comm idemp.c}. If~$M$ is of finite type
then~$\mathscr F(M)=\mathscr P(I)$. Suppose that~$\mathscr F(M)=\mathscr P(I)\setminus\{I\}$. Then~$I$ is connected,
since otherwise each of its connected components would be in~$\mathscr F(M)$, and hence so would be~$I$. Therefore, $(I,K)$ and~$(K,I)$ belong
to~$\mathcal C(\Gamma(M))$ for all~$K\subset I$, and these are~$2^{|I|+1}-1$ distinct pairs.

Finally, if~$M$ is irreducible of finite type or of affine type other than~$\widetilde A_n$, $n\ge 2$, and~$\widetilde D_4$, then~$\Gamma(M)$
is a tree all of whose vertices have degree at most~$3$, while~$\mathscr F(M)=\mathscr P(I)\setminus\{I\}$ if~$M$ is of affine type.
Thus, part~\ref{prop:enumerate comm pairs.d} follows from parts~\ref{prop:enumerate comm pairs.c} and~\ref{prop:enumerate comm pairs.b}.
\end{proof}
Our method does not apply to~$\widetilde A_n$, $n\ge 2$, since the corresponding graph is a cycle and thus has no leaves, nor
to~$\widetilde D_4$, whose graph has a vertex of degree~$4$. In fact, the second formula in Proposition~\partref{prop:enumerate comm pairs.d} fails
in these cases; for instance, $c(\widetilde A_2)=31$ and~$c(\widetilde D_4)=501$, while~$c_3-2^4+1=33$ and~$c_5-2^6+1=497$.
We conjecture that~$c(\widetilde A_n)=d_{n+1}-2^{n+1}-1$, $n\ge 2$, where~$d_0=2$, $d_1=4$ and~$d_n=4d_{n-1}-2d_{n-2}$, $n\ge 2$;
this has been verified by computer for~$n\le 11$.
\begin{remark}
Computer calculations show that~$c(\Gamma)\not=c_{|V|}$ for all trees~$\Gamma$ with at most~$10$ vertices having a vertex of degree at least~$4$. It would be interesting
to understand how~$c(\Gamma)$ depends on~$\Gamma$ in general; we intend to address this question elsewhere.
\end{remark}
\begin{remark}\label{rem:graph property}
The fact that the number
of commuting pairs of idempotents
in Hecke monoids corresponding to irreducible Coxeter matrices of finite type
depends only on the rank and not on the Coxeter matrix or even the Coxeter graph is in sharp contrast with the situation in Coxeter groups. For instance, in
$W(A_2)$ (respectively, $W(B_2)$) there are 10 (respectively, 14) pairs $(J,K)$ such that~$w_\circ^J$ and~$w_\circ^K$ commute.
\end{remark}

\subsection{Parabolic and homogeneous homomorphisms}\label{subs:parab hom}\label{subs:homogeneous}
Let~$M\in\Cox I$, $M'\in\Cox{I'}$.
\begin{definition}\label{defn:parabolic hom}
We say that~$\phi\in\Hom_{\He}(M',M)$ is {\em weakly parabolic} if
it maps every parabolic element of~$W(M')$ to a parabolic element of~$W(M)$, that is, for each~$K'\in\mathscr F(M')$ and~$J'\subset K'$ there exist~$K\in\mathscr F(M)$ and~$J\subset K$ such that~$\phi(w_{J';K'})=w_{J;K}$.
We say that~$\phi$
is
{\em parabolic} if for
every~$K'\in \mathscr F(M')$ and~$J'\subset K'$, $\phi(w_{J';K'})=
w_{J;[\phi](K')}$ for some~$J\subset [\phi](K')$.
\end{definition}
Clearly, every parabolic homomorphism is weakly parabolic, and, since~$[\cdot]$ is a functor, the composition of two (weakly) parabolic homomorphisms is also (weakly) parabolic. Thus, weakly parabolic
homomorphisms form a subcategory
of~$\He$ and parabolic homomorphisms form a subcategory of weakly parabolic homomorphisms.
\begin{remark}\label{rem:weakly not parab}
A weakly parabolic homomorphism does not have to be parabolic. For example, the assignments $s'_1\mapsto w_\circ^{[1,2]}$, $s'_2\mapsto w_\circ^{\{1,4\}}$ define an injective homomorphism
$\phi:(W(A_1\times A_1),\star)\to(W(A_4),\star)$ by Proposition~\ref{prop:comm idemp}, which is weakly parabolic since every element of~$W(A_1\times A_1)$ is parabolic and
is mapped to some~$w_\circ^K=w_{\emptyset;K}$. However, $\phi(w_{\{1\};\{1,2\}})=\phi(s'_2)=s_1s_4$ while~$[\phi](\{1,2\})=\{1,2,4\}$ and
$s_1s_4\not=w_{J;\{1,2,4\}}$ for any~$J\subset\{1,2,4\}$.
More generally, all homomorphisms
described in Corollary~\ref{cor:hom A1 x A1} are weakly parabolic, and one
can easily verify that they are parabolic if and only if, for every~$i\not=k$, any connected component of~$J_i$ and any connected component of~$J_k$ are either orthogonal or equal.
\end{remark}

Given $M'\in\Cox{I'}$, $M\in\Cox I$ and
a map~$f:W(M')\to W(M)$,
define \plink{ell f}$\ell_f:W(M')\to \ZZ_{\ge 0}$
by~$\ell_f=\ell\circ f$. We say that~$f$
is {\em homogeneous} if
$\ell_f(ww')=\ell_f(w)+\ell_f(w')$,
for all $w,w'\in W(M')$ with~$w\vdash w'$. This
is equivalent to~$\ell_f(w)=\sum_{1\le k\le r} \ell_f(s'_{i_k})$
for any reduced expression $w=s'_{i_1}\times\cdots\times s'_{i_r}$, $i_k\in I'$.
If~$f\in\Hom_{\He}(M',M)$ then, by Proposition~\partref{prop:Bruhat order *.d}, $f$ is homogeneous if and only if
$f(w)\vdash f(w')$ whenever~$w,w'\in W(M')$ satisfy~$w\vdash w'$.
Clearly, if~$f$ is homogeneous
then so is its restriction to~$W_J(M')$ for any~$J\subset I'$.

\begin{proposition}\label{prop:homogeneous<->Coxeter}
Let~$M'\in\Cox{I'}$, $M\in\Cox I$. The following
are equivalent for a map~$\phi:W(M')\to W(M)$
\begin{enmalph}
\item\label{prop:homogeneous<->Coxeter.a}
$\phi\in\Hom_{\mathscr{CH}}(M',M):=\Hom_{\Co}(M',M)\cap \Hom_{\He}(M',M)$;
\item\label{prop:homogeneous<->Coxeter.b}
$\phi\in\Hom_{\He}(M',M)$ and is homogeneous;
\item\label{prop:homogeneous<->Coxeter.c}
$\phi\in\Hom_{\Co}(M',M)$,
satisfies $\phi(s'_i)=w_\circ^{[\phi](i)}$ for all~$i\in I'$
and is homogeneous.
\end{enmalph}
\end{proposition}
\begin{proof}
Suppose that~$\phi\in\Hom_{\mathscr{CH}}(M',M)$ and
let~$w,w'\in W(M')$ be such that that~$\ell(ww')=\ell(w)+\ell(w')$ and so
$w\star w'=ww'$.
Since~$\phi\in\Hom_{\mathscr{CH}}(M',M)$, it follows
$\phi(w\times w')=\phi(w)\star\phi(w')=\phi(w)\phi(w')$
which implies that~$\ell_\phi(w)+\ell_\phi(w')=
\ell(\phi(w))+\ell(\phi(w'))=\ell(\phi(w)\phi(w'))=
\ell(\phi(ww'))=
\ell_\phi(ww')$. Thus, $\phi$ is homogeneous
and so~\ref{prop:homogeneous<->Coxeter.a} implies
both~\ref{prop:homogeneous<->Coxeter.b} and~\ref{prop:homogeneous<->Coxeter.c}

We now prove that~\ref{prop:homogeneous<->Coxeter.b}
implies~\ref{prop:homogeneous<->Coxeter.a}.
Suppose that~$\phi\in\Hom_{\He}(M',M)$ is homogeneous.
Then for all $k\not=l\in I'$ with~$m'_{kl}<\infty$
\begin{equation}
\ell_\phi(\brd{s'_k\times s'_l\times}{m'_{kl}})=
\Bfloor{\tfrac12 m'_{kl}}(\ell_\phi(s'_k)+\ell_\phi(s'_l))+\overline{m'_{kl}}\,\ell_\phi(s'_k).\label{eq:local homogen}
\end{equation}
It follows that
$
\phi(\brd{s'_k\times s'_l\times }{m'_{kl}})=
\brd{w_\circ^{[\phi](k)}\times w_\circ^{[\phi](l)}\times}{m'_{kl}}
$
and so
$
\brd{\phi(s'_k)\phi(s'_l)}{m'_{kl}}=\brd{\phi(s'_l)\phi(s'_k)}{m'_{kl}}
$
for all~$k\not=l\in I'$ with~$m'_{kl}<\infty$.
Since the $\phi(s'_i)=w_\circ^{[\phi](i)}$, $i\in I'$ are involutions in~$W(M)$,
the assignments~$s'_i\mapsto \phi(s'_i)$,
$i\in I'$ define~$f\in\Hom_{\Co}(M',M)$.
We now use induction on~$\ell(w)$ to prove that~$ f(w)=\phi(w)$
for all~$w\in W(M')$, the case~$\ell(w)\le 1$ being trivial. For the inductive step, write
$w=s'_i\times w'$, $\ell(w')<\ell(w)$.
By the induction hypothesis, $\phi(w')=f(w')$.
Since~$\phi$ is homogeneous, $\ell_\phi(w)=\ell_\phi(s'_i)+\ell_\phi(w')$ and so
$\phi(w)=\phi(s'_i)\times \phi(w')=
f(s'_i)f(w')=f(s'_iw')=f(w)$.
The proof of the implication \ref{prop:homogeneous<->Coxeter.c}$\implies$\ref{prop:homogeneous<->Coxeter.a} is similar and is omitted.
\end{proof}

\begin{example}\label{ex:homg not homg}
The assignments $s'_1\mapsto s_1$, $s'_2\mapsto s_2s_3s_2$ define a homomorphism~$f$
of Coxeter groups $W(A_2)\to W(A_3)$ which is not homogeneous since
$\ell_f(s'_1s'_2s'_1)=\ell_f(s'_2s'_1s'_2)=5\not=7=\ell_f(s'_2)+\ell_f(s'_1s'_2)$ even though~$\ell(s'_2s'_1s'_2)=1+\ell(s'_1s'_2)$. Note that it is not a homomorphism of Hecke monoids since
$s_1\star s_2s_3s_2\star s_1=s_1s_2s_3s_2s_1$ while
$s_2s_3s_2\star s_1\star s_2s_3s_2=w_\circ^{[1,3]}$.
The assignments $s'_1\mapsto w_\circ^{\{1,2\}}$,
$s'_2\mapsto w_\circ^{\{2,3,4\}}$,
$s'_3\mapsto w_\circ^{\{1,2,3\}}$ define
a homomorphism $\phi:(W(A_3),\star)\to
(W(A_4),\star)$. It
is not homogeneous since~$\ell_\phi(s'_1s'_3)=
\ell(w_\circ^{\{1,2,3\}})=6<\ell_\phi(s'_1)+
\ell_\phi(s'_3)=9$ and is not a homomorphism of
Coxeter groups since $\phi(s'_1)\phi(s'_3)=w_\circ^{\{1,2\}}w_\circ^{\{1,2,3\}}=w_\circ^{\{1,2,3\}}w_\circ^{\{2,3\}}\not=\phi(s'_3)\phi(s'_1)$.
\end{example}

\begin{proposition}\label{prop:CH-prop}
Let~$M\in\Cox I$, $M'\in\Cox{I'}$ and let~$\phi\in\Hom_{\mathscr{CH}}(M',M)$.
Then
\begin{enmalph}
\item \label{prop:CH-prop.a}
$\phi$ is disjoint;
\item \label{prop:CH-prop.a'}
if~$m'_{ij}$, $i\not=j\in I'$ is odd then~$\ell_\phi(s'_i)=
\ell_\phi(s'_j)$;
\item\label{prop:CH-prop.a''} $\phi(w_{J;K})=w_{[\phi](J);[\phi](K)}$ for
any~$J\subset K\in \mathscr F(M')$. In particular, $\phi$ is parabolic.
\item\label{prop:CH-prop.b}
for all~$i\in I'$, $w\in W(M')$
\begin{equation}
[\phi](i)\subset \begin{cases}
D_L(\phi(w)),& i\in D_L(w),\\
I\setminus D_L(\phi(w)),& i\in I'\setminus D_L(w).
\end{cases}
\label{eq:adm part property}
\end{equation}
In particular, if~$[\phi](i)\cap D_L(\phi(w))\not=\emptyset$
then~$[\phi](i)\subset D_L(\phi(w))$;
\end{enmalph}
\end{proposition}
\begin{proof}
Suppose that~$[\phi](i)\cap[\phi](j)\not=\emptyset$
for some~$i\not=j\in I'$.
Then~$\ell_\phi(s'_is'_j)=
\ell(w_\circ^{[\phi](i)}\star
w_\circ^{[\phi](j)})=
\ell(w_\circ^{[\phi](i)})+\ell(
w_\circ^{[\phi](j)})-
\ell(w_\circ^{[\phi](i)\cap[\phi](j)})<\ell_\phi(s'_i)+\ell_\phi(s'_j)$
by Lemma~\ref{lem:w0J*w0K},
which
is a contradiction since~$\phi$ is homogeneous
by Proposition~\ref{prop:homogeneous<->Coxeter}. This
proves part~\ref{prop:CH-prop.a}.
Part~\ref{prop:CH-prop.a'} is immediate since, for~$i\not=j$ with~$m'_{ij}$ odd,
\begin{align*}
\tfrac12(m'_{ij}-1)(\ell_\phi(s'_i)+\ell_\phi(s'_j))
+\ell_\phi(s'_i)
&=\ell_\phi(\brd{s'_is'_j}{m'_{ij}})\\
&=
\ell_\phi(\brd{s'_js'_i}{m'_{ij}})=
\tfrac12(m'_{ij}-1)(\ell_\phi(s'_i)+\ell_\phi(s'_j))+
\ell_\phi(s'_j).
\end{align*}

By Lemma~\ref{lem:Hecke hom w0J}, $\phi(w_\circ^J)=w_\circ^{[\phi](J)}$ for any~$J\in\mathscr F(M')$.
Then for any~$J\subset K\in\mathscr F(M')$,
$\phi(w_{J;K})
=\phi(w_\circ^J w_\circ^K)=
\phi(w_\circ^J)\phi(w_\circ^K)=
w_\circ^{[\phi](J)} w_\circ^{[\phi](K)}
=w_{[\phi](J);[\phi](K)}$. This proves part~\ref{prop:CH-prop.a''}.

To prove part~\ref{prop:CH-prop.b}, let~$w\in W(M')$
and~$i\in I'$. Suppose first that~$i\in D_L(w)$, that is, $w=s'_i\times w'$, $w'\in W(M')$. Then $\phi(w)=w_\circ^{[\phi](i)}\times \phi(w')$
since~$\phi$ is homogeneous, whence~$s_j\star \phi(w)=\phi(w)$ for all~$j\in[\phi](i)$. Therefore, $[\phi](i)\subset D_L(\phi(w))$ by Lemma~\ref{lem:left desc}. Now, suppose
that~$i\in I'\setminus D_L(w)$, that is, $\ell(s'_i w)=\ell(w)+1$, yet~$[\phi](i)\cap D_L(\phi(w))=J\not=\emptyset$.
Since~$\phi$ is homogeneous, $\ell_\phi(s'_i w)=\ell_\phi(s'_i)+\ell_\phi(w)$, that is,
$\phi(s'_i\times w)=w_\circ^{[\phi](i)}\times \phi(w)$.
Since~$J\subset D_L(\phi(w))$,
$\phi(w)=w_\circ^J\times \tilde u$, $\tilde u\in W(M)$ by
Lemma~\ref{lem:left desc}. Write~$w_\circ^{[\phi](i)}=\tilde v\times w_\circ^J$ where~$\tilde v=w_{J;[\phi](i)}{}^{-1}$. Then~$\phi(s'_i \times w)
=\tilde v\tilde u$ and so~$\ell_\phi(s'_i\times w)\le
\ell(\tilde v)+\ell(\tilde u)=
\ell_\phi(s'_i)+\ell_\phi(w)-2\ell(w_\circ^J)<
\ell_\phi(s'_i)+\ell_\phi(w)$, which is a contradiction.
\end{proof}
\begin{proposition}\label{prop:CH inj}
Let~$M\in \Cox I$, $M'\in\Cox{I'}$ and~$\phi\in\Hom_{\mathscr{CH}}(M',M)$. The
following are equivalent:
\begin{enmalph}
\item \label{prop:CH inj.a}
$\phi$ is injective;
\item\label{prop:CH inj.b}
$\phi$ is left-total;
\item\label{prop:CH inj.c}
$[\phi]:\mathscr P(I')\to \mathscr P(I)$ is injective.
\end{enmalph}
\end{proposition}
\begin{proof}
The implication~\ref{prop:CH inj.a}$\implies$\ref{prop:CH inj.b}
is Lemma~\ref{lem:inj total}.
Since~$\phi$ is disjoint by Proposition~\partref{prop:CH-prop.a}, the equivalence
of \ref{prop:CH inj.b} and~\ref{prop:CH inj.c} is immediate from
Lemma~\partref{lem:[phi]sets.a}.

It remains to prove that~\ref{prop:CH inj.b} implies~\ref{prop:CH inj.a}.
Suppose that~$\phi$ is left-total.
Since~$\phi$ is a homomorphism of Hecke monoids {\em and} Coxeter groups, it suffices to prove that~$\ker\phi$ is trivial. Suppose that~$w\in\ker\phi$ with~$\ell(w)>0$ and
write~$w=s'_i\times w'$, $i\in I'$ for some~$w'\in W(M')$, $\ell(w')=\ell(w)-1$. By Proposition~\ref{prop:homogeneous<->Coxeter},
$\phi$ is homogeneous and so
$\phi(w)=\phi(s'_i)\times \phi(w')=w_\circ^{[\phi](i)}\times\phi(w')=1$, which is a contradiction
since~$\ell(\phi(w))=\ell_\phi(w)=
\ell_\phi(s'_i)+\ell_\phi(w')\ge \ell_\phi(s'_i)=\ell(w_\circ^{[\phi](i)})>0$.
\end{proof}

\begin{theorem}\label{thm:artin parab}\label{thm:adm finite class}
Let~$M'\in\Cox{I'}$, $M\in\Cox I$ be irreducible and
of finite type. Up to diagram automorphisms and compositions, the injective fully supported homogeneous
$\phi\in\Hom_{\He}(M',M)$ which are not isomorphisms are exhausted by the following list:
\begin{enmalph}
\item\label{thm:adm finite class.unfold}
For~$M'=B_n$, $n\ge 2$,
\begin{alignat}{4}
&M=A_{2n-1}:\label{eq:unfold Bn A2n-1}
&\qquad&\phi(s'_i)=s_i s_{2n-i},&\quad &i\in [1,n-1],&\quad&
\phi(s'_n)=s_n,
\\
&M=A_{2n}:\label{eq:unfold Bn A2n}
&& \phi(s'_i)=s_i s_{2n+1-i},&& i\in [1,n-1],&& \phi( s'_n)=s_n s_{n+1}s_n,
\\
&M=D_{n+1}:\label{eq:unfold Bn Dn+1}
&& \phi(s'_i)=s_i,&& i\in [1,n-1],&& \phi(s'_n)=s_n s_{n+1};
\end{alignat}
\item\label{thm:adm finite class.unfold F4}
For~$M'=F_4$, $M=E_6$
\begin{align}
\label{eq:unfold F4 E6}
\phi(s'_1)=s_1 s_5,\quad \phi(s'_2)
=s_2 s_4,\quad \phi(s'_3)=s_3,\quad
\phi(s'_4)=s_6
\end{align}
\item \label{thm:adm finite class.odd}
For~$M'=I_2(2m+1)$, $m>0$, $M=A_{2m}$,
$$\phi(s'_i)=
w_\circ^{[1,2m+1-i]_2}=\prod_{j\in [1,2m+1-i]_2} s_j, \qquad i\in\{1,2\}
$$
(see~\S\ref{subs:gen notation} for the notation~$[a,b]_2$);
\item\label{thm:adm finite class.even}
For $M'=I_2(2m)$, $m>1$,
any~$M$ with~$h(M)=2m$ and any
partition $I=I_1\sqcup I_2$ of~$I$ into
non-empty totally disconnected subsets
$$
\phi(s'_j)=w_\circ^{I_j}=
\prod_{i\in I_j} s_i, \qquad j\in\{1,2\};
$$
\item\label{thm:adm finite class.I8 F4} For~$M'=I_2(8)$, $M=F_4$,
$\phi(s'_1)=s_1 s_4,\quad \phi(s'_2)=s_2 s_3 s_2$;

\item\label{thm:adm finite class.H3}
For~$M'=H_3$, $M=D_6$,
\begin{equation}\label{eq:unfold H3D6}
\phi(s'_1)=s_1s_5,\quad \phi(s'_2)=s_2s_4,\quad \phi(s'_3)=s_3s_6;
\end{equation}

\item\label{thm:adm finite class.H4}
For~$M'=H_4$, $M=E_8$,
\begin{equation}\label{eq:unfold H4E8}
\phi(s'_1)=s_1s_7,\quad \phi(s'_2)=s_2s_6,\quad
\phi(s'_3)=s_3s_5,\quad \phi(s'_4)=
s_4s_8.
\end{equation}
\end{enmalph}
In particular, all these homomorphisms are
parabolic.
The homomorphisms from parts~\ref{thm:adm finite class.unfold}, \ref{thm:adm finite class.unfold F4}
and~\ref{thm:adm finite class.I8 F4} are isomorphisms onto submonoids of $(W(M),\star)$ fixed by respective diagram automorphisms.
 \end{theorem}
\begin{proof}
By Proposition~\partref{prop:CH-prop.b},
$\phi\in\Hom_{\He}(M',M)$ gives rise to
a partition~$\{[\phi](i)\,:\,i\in I'\}$ of~$I$
which satisfies~$[\phi](i)\cap D_L(\phi(w))\not=\emptyset\implies [\phi](i)\subset
D_L(\phi(w))$ for all~$i\in I'$, $w\in W(M')$
and hence
is admissible in the sense of~\cite{Mue}*{Definition~3.1}. Conversely, by~\cite{Mue}*{Proposition~3.5}, such a partition gives rise
to a homomorphism of Coxeter
groups $W(M')\to W(M)$
satisfying the assumptions of Proposition~\partref{prop:homogeneous<->Coxeter.c} which,
therefore, is a homogeneous homomorphism
of corresponding Hecke monoids. Thus, the classification
of homogeneous homomorphisms between finite Hecke monoids coincides with that obtained in~\cites{Cri,God,Cas,Mue} where they are called LCM homomorphisms; in particular,
the list of such homomorphisms, up to
compositions with diagram automorphisms,
coincides with the one provided in the Theorem.
\end{proof}
\begin{remark}
Curiously, homomorphisms~\eqref{eq:unfold H3D6} and~\eqref{eq:unfold H4E8} were studied in~\cite{BBO'C} in the framework of continuous crystals.
\end{remark}
To classify homogeneous homomorphisms for all finite types, we need the following results.
\begin{lemma}\label{lem:diagonal}
Let $M_1\in\Cox{I_1}$,
$M_2\in\Cox{I_2}$ and
$M'\in\Cox{I'}$
and
let~$f_i:W(M')\to W(M_i)$,
$i\in\{1,2\}$. Then the natural map
$f:W(M')\to W(M_1)\times W(M_2)$
$w\mapsto (f_1(w),f_2(w))$, $w\in W(M')$
is a homogeneous
homomorphism
if and only if~$f_1$ and~$f_2$ are homogeneous homomorphisms.
\end{lemma}
\begin{proof}

It is easy to see that if~$\mathscr C$ is a category with products and~$\mathscr C_1$, $\mathscr C_2$ are its subcategories closed with
respect to products (that is, containing all projections and such that a morphism into a product belongs to~$\mathscr C_i$ if and only if its components do), then
$\mathscr C_1\cap\mathscr C_2$ is
also closed with respect to products.
Applying this to the category
$\mathscr C$ whose objects are Coxeter matrices and
whose morphisms
are all maps $W(M')\to W(M)$,
$M'\in\Cox{I'}$, $M\in \Cox I$,
$\mathscr C_1=\He$ and~$\mathscr C_2=\Co$, we conclude that $f\in\Hom_{\mathscr{CH}}(M',M_1\times M_2)$ if and only if
$f_i\in\Hom_{\mathscr{CH}}(M',M_i)$, $i\in\{1,2\}$. It remains to apply Proposition~\ref{prop:homogeneous<->Coxeter}.
\end{proof}
Thus, we may assume, without loss of generality, that the codomain is irreducible.

The domain requires more effort.
Given sets~$X$, $Y$, a multiplicative monoid~$\mathsf M$ and maps~$f:X\to\mathsf M$, $g:Y\to \mathsf M$, we define $f\bullet g:X\times Y\to\mathsf M$ by $(f\bullet g)(x,y)=f(x)g(y)$, $x\in X$, $y\in Y$.
The following is standard and immediate.
\begin{lemma}\label{lem:prod of monoids}
Let~$\mathsf M_1$, $\mathsf M_2$ and~$\mathsf M$ be multiplicative monoids and let~$f_i:\mathsf M_i\to \mathsf M$ be homomorphisms of monoids. Then $f_1\bullet f_2$ is a homomorphism of monoids $\mathsf M_1\times \mathsf M_2\to \mathsf M$ if and only if~$f_1(\mathsf M_1)$ and~$f_2(\mathsf M_2)$ commute elementwise. Conversely,
if~$f:\mathsf M_1\times \mathsf M_2\to\mathsf M$ is a homomorphism of monoids then~$f=f_1\bullet f_2$ where
$f_1=f|_{\mathsf M_1\times\{1_{M_2}\}}$ and
$f_2=f|_{\{1_{M_1}\}\times \mathsf M_2}$.
\end{lemma}
Let $M_i\in\Cox{I_i}$, $i\in\{1,2\}$ and~$M\in\Cox I$ and let~$f_i\in\Hom_{\He}(M_i,M)$.
Clearly, if $f_1\bullet f_2\in\Hom_{\mathscr{CH}}(M_1\times M_2,M)$ then $f_i\in\Hom_{\mathscr{CH}}(M_i,M)$, $i\in\{1,2\}$, since each
$f_1$ and~$f_2$ can be recovered as restrictions of~$f_1\bullet f_2$ to respective parabolic submonoids.

The rest of the classification is based on the following
\begin{proposition}\label{prop:coprod homogeneous}
Let~$M_i\in\Cox{I_i}$, $i\in\{1,2\}$ and~$M\in\Cox I$ and let~$f_i\in\Hom_{\mathscr{CH}}(M_i,M)$, $i\in\{1,2\}$. Then~$f_1\bullet f_2\in
\Hom_{\mathscr{CH}}(M_1\times M_2,M)$
if and only if~$[f_1](I_1)$ and~$[f_2](I_2)$ are orthogonal.
\end{proposition}
\begin{proof}
Suppose that~$f_1\bullet f_2$ is homogeneous. It suffices to prove
that~$[f_1](j)$ and~$[f_2](k)$ are orthogonal for all~$j\in I_1$ and~$k\in I_2$.
Let~$j\in I_1$, $k\in I_2$
and let~$J=[f_1](j)$, $K=[f_2](k)$.
Then~$w_\circ^J\star w_\circ^K=w_\circ^K\star w_\circ^J=w_\circ^{J\cup K}$
by Proposition~\ref{prop:comm idemp}.
Furthermore,
$w_\circ^{J\cup K}=(f_1\bullet f_2)(w_\circ^{\{j,k\}})=w_\circ^Jw_\circ^K$ and so~$w_\circ^{J\cup K}=w_\circ^Jw_\circ^K=w_\circ^Kw_\circ^J$ in~$W(M)$ as well.
In particular, $\ell(w_\circ^Jw_\circ^K)=\ell(w_\circ^J)+\ell(w_\circ^K)$.

Suppose that
$i\in J\cap K$. Write
$w_\circ^J=u\times s_i$ and~$w_\circ^K=s_i\times v$, $u\in W_J(M)$, $v\in W_K(M)$. Then
$\ell(w_\circ^Jw_\circ^K)=\ell(uv)\le
\ell(w_\circ^J)+\ell(w_\circ^K)-2
<\ell(w_\circ^J)+\ell(w_\circ^K)$
and so~$w_\circ^Jw_\circ^K<w_\circ^J\star w_\circ^K$ by Proposition~\partref{prop:Bruhat order *.d}, which is a contradiction. Thus,
$J\cap K=\emptyset$ and so
$J$ and~$K$ are orthogonal by Proposition~\ref{prop:w * w'=w'*w}.
It follows that~$[f_1](i)$ is orthogonal
to~$[f_2](I_2)$ for all~$i\in I_1$ whence $[f_1](I_1)$ and~$[f_2](I_2)$ are orthogonal.

To prove the converse, note that if~$[f_1](I_1)$ and~$[f_2](I_2)$ are orthogonal then~$f_1(W(M_1))$ and~$f_2(W(M_2))$ commute elementwise
both in the Hecke monoid $(W(M),\star)$ and
in the Coxeter group~$W(M)$
and so~$f_1\bullet f_2$ is
a homomorphism of Hecke monoids and of Coxeter groups by Lemma~\ref{lem:prod of monoids}.
\end{proof}
In view of Lemma~\ref{lem:prod of monoids} and Proposition~\ref{prop:coprod homogeneous}, the classification of
homogeneous homomorphisms from products of irreducible Hecke monoids reduces to
finding tuples of homomorphisms from individual irreducible factors to the same codomain with pairwise orthogonal supports.
Note that, by Proposition~\ref{prop:comm idemp}, the images of the factors need not have orthogonal supports for
such a tuple to define a homomorphism of Hecke monoids; orthogonality is forced by homogeneity.
\begin{example}
Let~$M\in\Cox I$ be of affine type and suppose that~$\Gamma(M)$ is bipartite, that is, $I=I_1\sqcup I_2$
where~$I_1$, $I_2$ are non-empty and totally disconnected.
By~\cite{Speyer}, the assignments $s'_j\mapsto w_\circ^{I_j}$, $j\in\{1,2\}$, define
a homogeneous homomorphism~$\phi\in\Hom_{\He}(I_2(\infty),M)$.
Other examples of homogeneous
homomorphisms are provided by affine unfoldings, such as
$(\tilde G_2,\tilde B_3)$,
$(\tilde B_n,\tilde D_{n+1})$ etc.
\end{example}

\subsection{Geometric perspective}\label{subs:geom}
Given any group~$G$ and its subgroup~$H$, we denote by $C_G(H)$ (respectively, $N_G(H)$) the
centralizer (respectively, the normalizer) of~$H$ in~$G$.
As customary, $Z(H)$ stands for the center of~$H$.
Clearly, $C_G(H)\le N_G(H)$, $H$ is normal in~$N_G(H)$ and for
any group homomorphism~$f:G\to G'$,
\begin{equation}\label{eq:norm hom}
f(N_G(H))<N_{G'}(f(H)).
\end{equation}
We need the following
\begin{lemma}\label{lem:norm centr}
Let~$G$ be a group and let~$H< G$. Then~$C_G(H)$
is normal in~$N_G(H)$ and hence
$N_G(H)\le N_G(C_G(H))$. Moreover, the equality
holds if~$H=C_G(C_G(H))$.
\end{lemma}
\begin{proof}
Let~$C=C_G(H)$, $g\in N_G(H)$ and~$c\in C$. Then for all~$h\in H$, $g^{-1}hg\in H$ and so
$$
(gcg^{-1})h(gcg^{-1})^{-1}=gc(g^{-1}hg)c^{-1}g^{-1}=g(g^{-1}hg)g^{-1}=h,
$$
that is, $gcg^{-1}\in C$, whence~$C$ is normal in~$N_G(H)$
and~$N_G(H)\le N_G(C)$.
Conversely, suppose that~$H=C_G(C)$ and let~$g\in N_G(C)$. Applying the first part to~$C$ instead of~$H$, we
conclude that~$g\in N_G(C_G(C))=N_G(H)$. Thus, $N_G(C)\le N_G(H)$.
\end{proof}
Thus, one defines the {\em Weyl group}
$W(G,H)$ as $W(G,H)=N_G(H)/C_G(H)$. In particular, if~$H$ is the center of~$C=C_G(H)$ then~$C_G(C)\le C_G(H)=C$,
whence~$C_G(C)=Z(C)=H$, and so $W(G,H)=N_G(C)/C$ by Lemma~\ref{lem:norm centr}.

Let~$G$ be a connected split reductive algebraic group and let~$T$ be its maximal torus. Then~$C_G(T)=T$ and thus the Weyl
group~$W_G:=W(G,T)$ of~$G$ is $N_G(T)/T$.
We call the centralizer of a subtorus of~$T$ a {\em Levi subgroup} of type~$T$. Clearly, each Levi subgroup of type~$T$
contains~$T$, and it is connected by~\cite{Hum}*{Theorem~22.3}.
\begin{theorem}\label{thm:groups}
Let~$G$ and~$G'$ be connected split reductive groups. Then
any morphism of algebraic groups $\rho:G\to G'$ with finite kernel yields
a family of injective homomorphisms $W_G\to W_{G'}$.
\end{theorem}
\begin{proof}
We need the following Proposition which was inspired by~\cite{BS}*{Section~2}.
\begin{proposition}
\label{prop:Levi normalizer}
Let~$T<G$ be a maximal torus and let~$L$
be a Levi subgroup of~$G$ of type~$T$.
Then
\begin{enmalph}
\item \label{prop:Levi normalizer.a}
$N_G(L)=L\mathcal N=\mathcal NL$ where~$\mathcal N=N_G(T)\cap N_G(L)$;
\item \label{prop:Levi normalizer.a'}
$L\cap \mathcal N=N_L(T)$;
\item \label{prop:Levi normalizer.b}
$\mathcal N$ is the normalizer of~$N_L(T)$ in~$N_G(L)$;
\item \label{prop:Levi normalizer.b'}
$\mathcal N$ is the normalizer of~$N_L(T)$ in~$N_G(T)$;
\item \label{prop:Levi normalizer.c}
$\mathcal N/T=N_{W_G}(W_L)$. In particular,
$W_G^L:=N_G(L)/L\cong \mathcal N/N_L(T)\cong N_{W_G}(W_L)/W_L$.
\end{enmalph}
\end{proposition}
\begin{proof}
Since~$L,\mathcal N<N_G(L)$, $L\mathcal N,\mathcal NL\subset N_G(L)$. Suppose that~$g\in N_G(L)$. Then $gTg^{-1}=T'$ is a maximal
torus of~$L$. Since~$L$ is a connected split reductive algebraic group, there exists~$l\in L$ such that~$lTl^{-1}=T'$.
Then~$l^{-1}g T g^{-1}l=T$, that is, $n:=l^{-1}g\in N_G(T)\cap N_G(L)$. Thus, $g\in L\mathcal N$. Similarly, $g\in \mathcal NL$.
Part~\ref{prop:Levi normalizer.a} is proven.

Clearly, $N_L(T)=N_G(T)\cap L$. Therefore,
$L\cap \mathcal N=L\cap N_G(L)\cap N_G(T)=L\cap N_G(T)= N_L(T)$, which proves part~\ref{prop:Levi normalizer.a'}.

To prove part~\ref{prop:Levi normalizer.b}, note
that, since the conjugation by an element of~$\mathcal N$
is an automorphism of both~$T$ and~$L$, $\mathcal N$ normalizes~$N_L(T)$. To prove that~$N_{N_G(L)}(N_L(T))=\mathcal N$, by
parts~\ref{prop:Levi normalizer.a} and~\ref{prop:Levi normalizer.a'} it suffices to show that
$N_{N_G(L)}(N_L(T))\cap L=N_L(N_L(T))$ is equal to~$N_L(T)$. We need the following
\begin{lemma}\label{lem:rigidity}
Let~$G$, $T$ be as in~Proposition~\ref{prop:Levi normalizer} and
let~$K$ be any subgroup of~$N_G(T)$ containing~$T$. Then~$N_G(K)=N_{N_G(T)}(K)$.
\end{lemma}
\begin{proof}
Since~$N_G(T)^\circ=C_G(T)=T$ (see~\cite{Hum}*{Corollary in~\S16.3}), $T$ is the unique maximal torus contained in~$N_G(T)$.
Hence, for any $g\in N_G(K)$, $gTg^{-1}<gKg^{-1}=K<N_G(T)$ is a maximal torus, and so~$gTg^{-1}=T$, that is, $g\in N_{G}(T)$.
Thus, $N_G(K)<N_{N_G(T)}(K)$. The opposite inclusion is obvious.
\end{proof}
Applying this lemma with~$K=N_L(T)$ and~$G=L$
yields part~\ref{prop:Levi normalizer.b}.

By Lemma~\ref{lem:rigidity} with~$K=N_L(T)$, $N_G(N_L(T))=N_{N_G(T)}(N_L(T))$. Every~$g\in N_G(T)$ which normalizes~$N_L(T)$
normalizes~$W_L$ and hence the set of roots of~$L$ with respect to~$T$, since the reflections in~$W_L$ are precisely the
reflections in~$W$ corresponding to roots of~$L$. Since~$L$ is generated by~$T$ and the corresponding root subgroups, $g$ normalizes~$L$. Thus,
$N_{N_G(T)}(N_L(T))=N_{N_G(L)}(N_L(T))=\mathcal N$ by part~\ref{prop:Levi normalizer.b}, which proves part~\ref{prop:Levi normalizer.b'}.

To prove part~\ref{prop:Levi normalizer.c}, note that $T\triangleleft\mathcal N$ since~$\mathcal N<N_G(T)$, and so~$\mathcal N/T$ is well-defined.
We have~$\mathcal N/T=N_{N_G(T)}(N_L(T))/T$ by part~\ref{prop:Levi normalizer.b'}.
The following is easily checked.
\begin{lemma}\label{lem:induced normalizer}
Let~$Y$ be a subgroup of~$Z$ and let~$\pi:Z\to Z'$ be a surjective homomorphism of groups such that~$\ker \pi<Y$.
Then~$\pi^{-1}(N_{Z'}(Y'))=N_Z(Y)$ where~$Y'=\pi(Y)$. In particular,
$N_{Z'}(Y')=N_Z(Y)/\ker\pi$.
\end{lemma}
Applying this lemma with~$Z=N_G(T)$, $Y=N_L(T)$,
$Z'=W_G$ and~$\pi:N_G(T)\to W_G$ being the
canonical surjection with~$\ker\pi=T$, we obtain
the first assertion in part~\ref{prop:Levi normalizer.c} since~$Y'=\pi(Y)=W_L$. Since~$N_L(T)=L\cap \mathcal N$,
$\mathcal N/N_L(T)\cong \mathcal NL/L=N_G(L)/L$. Finally,
\begin{equation*}\mathcal N/N_L(T)\cong (\mathcal N/T)/(N_L(T)/T)\cong N_{W_G}(W_L)/W_L.
\qedhere
\end{equation*}
\end{proof}
Now we have all necessary ingredients to finish the proof of Theorem~\ref{thm:groups}.
Since~$\ker\rho$ is a finite normal subgroup of the connected group~$G$, it is central (as,
for each~$n\in\ker\rho$, the image of the morphism~$G\to\ker\rho$,
$g\mapsto gng^{-1}$ is connected), and so~$\ker\rho<T$ by~\cite{Hum}*{Corollary~26.2.A}.
Replacing $G$ with $G/\ker\rho$, which does not change~$W_G$ since~$\ker\rho<T$, we may assume, without loss of generality, that $\rho$ is injective.

Let~$T'$ be a maximal torus of~$G'$ such that~$\rho(T)<T'$ (such a torus always exists).
Then $L':=C_{G'}(\rho(T))$ is a Levi subgroup of type~$T'$, since~$\rho(T)$ is a subtorus of~$T'$.
By~\eqref{eq:norm hom} and the first assertion of Lemma~\ref{lem:norm centr},
$\rho(N_G(T))<N_{G'}(\rho(T))\le N_{G'}(L')$. Thus, the restriction of~$\rho$
to~$N_G(T)$ is a homomorphism $N_G(T)\to N_{G'}(L')$ which, since~$\rho(T)<L'$, in turn induces a homomorphism
$$
\underline\rho: W_G\to W_{G'}^{L'}
$$
in the notation of Proposition~\ref{prop:Levi normalizer}. We claim that~$\underline\rho$ is injective.
Indeed, if $gT\in \ker\underline \rho$ then $\rho(g)\in L'$, that is, $\rho(g)$ centralizes~$\rho(T)$. Then the injectivity of $\rho$
implies that $g\in C_G(T)=T$.

By~\cite{How}*{Lemma~2 and Corollary~3}, the canonical surjection $N_{W_{G'}}(W_{L'})\to W_{G'}^{L'}$
from Proposition~\partref{prop:Levi normalizer.c}
splits (note that~$W_{L'}$ is conjugate to a standard parabolic subgroup of~$W_{G'}$).
Composing $\underline \rho$ with these splittings yields the desired family of  homomorphisms.
\end{proof}

Thus, we obtain a family of injective homomorphisms $W_G\to W_{G'}$, on which~$W_{L'}$ acts by conjugation of the splittings.
In particular, we obtain unfoldings
\eqref{eq:unfold Bn A2n-1}, \eqref{eq:unfold Bn A2n},
\eqref{eq:unfold Bn Dn+1} and~\eqref{eq:unfold F4 E6}
from natural embeddings of algebraic groups
$Sp_{2n}\to SL_{2n}$,
$SO_{2n+1}\to SL_{2n+1}$,
$SO_{2n+1}\to SO_{2n+2}$ and~$F_4\to E_6$.
\begin{remark}
Let~$L=C_G(S)$ be a Levi subgroup of type~$T$, where~$S$ is a subtorus of~$T$. Then~$C_G(L)\subset C_G(T)=T\subset L$, whence~$C_G(L)=Z(L)$.
On the other hand, $S\subset Z(L)$ and so~$L\subset C_G(Z(L))\subset C_G(S)=L$. Thus, $Z(L)=C_G(C_G(Z(L)))$ and so, by
Lemma~\ref{lem:norm centr}, $N_G(L)=N_G(Z(L))$ and~$W_G^L=W(G,Z(L))$.
\end{remark}

\section{Light homomorphisms of Hecke monoids}
\label{sec:Parab proj}

In this section we describe a subcategory of~$\He$ which unifies parabolic projections,
natural inclusions of parabolic submonoids and
tautological homomorphisms. We also prove that
all such homomorphisms in finite types are parabolic.

\subsection{Light homomorphisms of Hecke monoids}\label{subs:parab proj}
Tautological homomorphisms, parabolic projections and natural inclusions of parabolic submonoids belong to a larger class of homomorphisms.
\begin{definition}\label{defn:light}
Let~$M'$ and~$M$ be Coxeter matrices
over respective index sets~$I'$, $I$.
We say that~$\phi\in\Hom_{\He}(M',M)$
is {\em light} if
$|[\phi](i)|\le 1$ for all~$i\in I'$.
\end{definition}
In other words, light homomorphisms are exactly the morphisms of~$\He$ which are mapped to partial functions by the functor~$[\cdot]:\He\to\mathscr{R\!e\!l}$. This immediately implies the following
\begin{lemma}\label{lem:light cat}
A composition of light homomorphisms of Hecke monoids is again light. In other words, Coxeter matrices and light homomorphisms of respective Hecke monoids form a subcategory of~$\He$.
\end{lemma}

Clearly, parabolic projections, tautological homomorphisms and natural inclusions are light.
We now describe another class of surjective light homomorphisms.
\begin{definition}\label{defn:foldable}
Let~$\varpi:I\to J$ be a surjective map.
We say that a Coxeter matrix~$M$ over~$I$ is {\em foldable along~$\varpi$} if
 $m_{ii'}=m_{ii''}$ for all
 $i,i',i''\in I$ with $\varpi(i')=\varpi(i'')\not=
\varpi(i)$.
\end{definition}
Note that any group~$G$ of automorphisms of~$\Gamma(M)$
gives rise to a map~$\varpi_G:I\to I/G$
such that~$M$ is foldable along~$\varpi_G$.

If~$M$ is foldable along~$\varpi$,
define~\plink{MII}$M^{\varpi}$ to be the matrix over~$J$ with~$(M^\varpi)_{jj}=1$, $j\in J$ and
$(M^{\varpi})_{jj'}=m_{ii'}$
for any~$i\in \varpi^{-1}(j)$, $i'\in \varpi^{-1}(j')$, $j\not=j'\in J$.
Clearly, $M^{\varpi}$
is a Coxeter matrix.
\begin{lemma}\label{lem:Heck orb fold}
Let~$\varpi:I\to J$ be surjective and let~$M\in\Cox I$
be foldable along~$\varpi$.
The
assignments~$s_i\mapsto s^{\varpi}_{\varpi(i)}$, $i\in I$, where the~$s^\varpi_j$, $j\in J$ are the
generators of~$(W(M^{\varpi}),\star)$,
define a surjective light \plink{phi fold}$\mathbf f_{\varpi}\in\Hom_{\He}(M,M^{\varpi})$ which
is optimal if and only if all the
$\varpi^{-1}(\{j\})$, $j\in J$, are totally disconnected.
\end{lemma}
\begin{proof}
We may regard~$\varpi$
as a map~$I\to\mathscr F(M^{\varpi})$ in an obvious way. Then
$\varpi\in \Lambda(M,M^{\varpi})$ by definition of~$M^\varpi$ and so
$\mathbf f_{\varpi}:=\Theta_{\varpi}
\in\Hom_{\He}(M,M^\varpi)$ in the notation of Lemma~\ref{lem:Lam M'M HomM'M}.
The condition~\eqref{eq:optimal cond} is evidently satisfied for~$i,i'\in I$ with~$\varpi(i)\not=\varpi(i')$, while
for~$i\not=i'$ with~$\varpi(i)=\varpi(i')$ it reads~$m_{ii'}=2$.
\end{proof}
We refer to~$\mathbf f_{\varpi}$ as
the {\em folding along~$\varpi$}.
\begin{example}
Let~$M\in\Cox I$ and let~$\varpi$ be
the unique map~$I\to\{1\}$. Then~$M$
is foldable along~$\varpi$, $M^\varpi=A_1$ and $\mathbf f_\varpi(s_i)=s^\varpi_1$ for all~$i\in I$. It is optimal
if and only if~$M=A_1^{\times |I|}$.
\end{example}
\begin{example}\label{ex:light projection}
Let~$M=D_{n+1}$ and define~\plink{varpi (n,n+1)}$\varpi_{(n,n+1)}:[1,n+1]\to [1,n]$ by $\varpi_{(n,n+1)}(i)=i-\delta_{i,n+1}$,
$i\in [1,n+1]$.
Then~$M$ is foldable along~$\varpi=\varpi_{(n,n+1)}$,
$M^{\varpi}=A_n$
and, identifying~$W(M^\varpi)$ with~$W_{[1,n]}(M)$, we have $\mathbf f_{\varpi}(s_i)=
s_{i-\delta_{i,n+1}}$,
$i\in [1,n+1]$.

Similarly, if~$M=D_4$, define \plink{varpi (1,3,4)}$\varpi_{(1,3,4)}:[1,4]\to\{1,2\}$ by
$\varpi_{(1,3,4)}(i)=1$, $i\in\{1,3,4\}$ and~$\varpi_{(1,3,4)}(2)=2$. Then~$M$ is foldable along~$\varpi=\varpi_{(1,3,4)}$, $M^{\varpi}=A_2$ and, identifying~$W(M^\varpi)$ with~$W_{\{1,2\}}(M)$ we have $\mathbf f_\varpi(s_i)=s_1$, $i\in\{1,3,4\}$, $\mathbf f_\varpi(s_2)=s_2$. Both
homomorphisms in this example
are optimal.
\end{example}
\begin{example}
Let~$M\in\Cox I$
and suppose that~$I=I_1\sqcup I_2$
where
$m_{ij}=m\ge 3$ for all~$i\in I_1$,
$j\in I_2$. Define $\varpi:I\to\{1,2\}$
by $\varpi(i)=j$ provided that~$i\in I_j$.
Then~$M$ is foldable
along $\varpi$ and
$M^{\varpi}=I_2(m)$.
The homomorphism~$\mathbf f_\varpi$ is optimal if and only if both~$I_1$ and~$I_2$ are totally disconnected.
\end{example}

Since tautological homomorphisms are light, it follows from Proposition~\ref{prop:optimal taut} that, to describe all light
homomorphisms, it suffices to describe all optimal ones.
\begin{proposition}\label{prop:classify light}
Every optimal light homomorphism of Hecke monoids can be canonically presented as a composition of a parabolic projection, the
folding along a surjective map
with totally disconnected fibers,
and a natural inclusion.
\end{proposition}
\begin{proof}
Let~$M\in\Cox I$, $M'\in\Cox{I'}$ and let $\phi\in \Hom_{\He}(M',M)$
be light. Let~$I'_s=\{
i\in I'\,:\, |[\phi](i)|=s\}$, $s\in \{0,1\}$. Then
$\phi=\phi|_{W_{I'_1}(M')}\circ p_{I'_1}$. Furthermore,
if~$\phi$ is not surjective then, since~$\phi$ is light, its image
is~$(W_J(M),\star)$ for some~$J\subset I$ and so we can write
$\phi$ as a composition of a surjective
light homomorphism with the natural
inclusion~$\iota_J$. Therefore, it remains to describe light
homomorphisms with~$I'_0=\emptyset$
which are optimal and surjective.

For such a homomorphism,
$|[\phi](i)|=1$ for all~$i\in I'$ and so
we can regard~$[\phi]$ as a surjective map~$I'\to I$.
By Theorem~\partref{thm:Hom Heck Mon},
$$
m'_{ij}\ge \max(\mu_M([\phi](i),[\phi](j)),\mu_M([\phi](j),[\phi](i)))
=m_{[\phi](i)[\phi](j)},\qquad i\not=j\in I'.
$$
Since~$\phi$ is optimal, it follows
that~$m'_{ij}=m_{[\phi](i)[\phi](j)}$
for all~$i,j\in I'$ such that~$[\phi](i)\not=[\phi](j)$, while~$m'_{ij}=2$ if~$i\not=j$ and~$[\phi](i)=[\phi](j)$.
Thus, $M'$ is foldable along~$[\phi]$, the fibers of~$[\phi]$ are totally disconnected, $M'{}^{[\phi]}=M$
and~$\phi=\mathbf f_{[\phi]}$.
\end{proof}
It turns out that every surjective homomorphism of Hecke monoids has a ``light core'', that is,
restricts to a surjective light homomorphism from a maximal parabolic submonoid.
\begin{lemma}\label{lem:surj hom Heck bas}
Let~$M\in\Cox I$, $M'\in\Cox{I'}$ and
let~$\phi\in\Hom_{\He}(M',M)$. Then
there exists a unique maximal subset $J'=J'(\phi)$ of~$I'$ such that
$\phi|_{W_{J'}(M')}$ is light.
Moreover, if~$\phi$ is surjective
then so is~$\phi|_{W_{J'}(M')}$.
\end{lemma}
\begin{proof}
Let~$J'=J'(\phi)=\{ i\in I'\,:\, |[\phi](i)|\le 1\}$. Then~$\phi|_{W_{J'}(M')}$ is light, and, clearly, $J'$ is the maximal subset of~$I'$ with this property. Suppose that~$\phi$ is surjective. Then~$\phi^{-1}(s_i)\not=\emptyset$ and
for any~$x\in\phi^{-1}(s_i)
$ we have $\{i\}=\supp\phi(x)=[\phi](\supp x)$
by Lemma~\partref{lem:[phi]comp.a}
and so~$[\phi](j)\in \{\{i\},\emptyset\}$ for any~$j\in\supp x$.
It follows that
$\{j\in \supp \phi^{-1}(s_i)\,:\, [\phi](j)=\{i\}\}$ is non-empty,
whence~$s_i=\phi(s'_j)$ for some~$j\in J'$.
\end{proof}
\begin{example}
A surjective homomorphism does not have to be light, and its light core can be a proper parabolic submonoid. For instance, let~$M'$ be a Coxeter matrix over~$I'$ with~$|I'|>2$
and with $m'_{ij}=m>2$ for some~$i\not=j\in I'$.
Then the assignments $s'_i\mapsto s_1$, $s'_j\mapsto s_2$ and $s'_k\mapsto w_\circ^{\{1,2\}}=\brd{s_1s_2}m$, $k\in I'\setminus\{i,j\}$ define a surjective~$\phi\in\Hom_{\He}(M',I_2(m))$
with~$J'=J'(\phi)=\{i,j\}$. Indeed,
define~$\xi:I'\to \mathscr P(\{1,2\})$ by
$\xi(i)=\{1\}$, $\xi(j)=\{2\}$ and~$\xi(k)=\{1,2\}$,
$k\in I'\setminus\{i,j\}$. Then
$\mu_{I_2(m)}(\{t\},\{1,2\})=
\mu_{I_2(m)}(\{1,2\},\{t\})=2$ for $t\in\{1,2\}$
by Lemma~\ref{lem:char w_0 monoid},
and so~$\xi\in\Lambda(M',I_2(m))$.
Clearly, $\phi=\Theta_\xi$. Thus, $\phi$ is not light, while its restriction to~$W_{\{i,j\}}(M')$ is a surjective light
homomorphism, in accordance with Lemma~\ref{lem:surj hom Heck bas}.
\end{example}

\subsection{Parabolicity of parabolic projections}\label{subs:pJ parabolic}
In this section we prove that
parabolic projections are parabolic homomorphisms. In fact, we prove a much stronger result, namely we obtain
an explicit formula for parabolic projections of parabolic elements.
\begin{proposition}\label{prop:key prop parab}
Let~$M\in \Cox I$ be of finite type and $J,K\subset I$. Suppose that~$w_K\star w_J=w_L$ for some~$L\subset J$. Then~$p_J(w_K)=w_{L;J}$ and~$L\subset J\cap K$.
\end{proposition}
\begin{proof}
By Lemma~\ref{lem: p_J max}, $W_J(M)\cap \downarrow w_K=\downarrow p_J(w_K)$ and so it remains to show that~$W_J(M)\cap \downarrow w_K=
\downarrow w_{L;J}$.

By Lemma~\ref{lem:len prop wJ K}
$$
w_K\star w_J=w_{L}=w_{L;J}\times w_J
$$
and so~$w_{L;J}=(w_K\star w_J)w_J^{-1}\le w_K$
by Proposition~\partref{prop:Bruhat order *.b}. Since~$w_{L;J}\in W_J(M)$,
it follows that $w_{L;J}\in W_J(M)\cap \downarrow w_K$.

Let~$v\in W_J(M)\cap\downarrow w_K$. By Proposition~\ref{prop:Bruhat order *}, $v w_J\le v\star w_J\le w_K\star w_J=w_L$. Thus, $vw_\circ^Jw_\circ\le w_\circ^{L}w_\circ$.
By Proposition~\partref{prop:Bruhat order.c}, this implies
that~$w_\circ^{L}\le vw_\circ^J$. Since~$L\subset J$,
both sides of this inequality are contained in~$W_J(M)$. Applying
Proposition~\partref{prop:Bruhat order.c} again we conclude that~$v\le w_\circ^{L}w_\circ^J=w_{L;J}$. Therefore, $W_J(M)\cap \downarrow w_K=\downarrow w_{L;J}$.

Since~$w_\circ=w_\circ^K\times w_K$,
applying~$p_J$ yields~$w_\circ^J=w_\circ^{K\cap J}\star p_J(w_K)=w_\circ^{K\cap J}\star w_{L;J}$ by Lemma~\ref{lem:proj w0}.
Since~$w_\circ^J=w_\circ^L\times w_{L;J}$, it follows from Proposition~\partref{prop:Bruhat order *.b} that~$w_\circ^L=(w_\circ^{K\cap J}\star w_{L;J})w_{L;J}{}^{-1}\le w_\circ^{K\cap J}$. In particular, $L\subset K\cap J$.
\end{proof}
We need the following technical result.
\begin{lemma}\label{lem:Cox-Heck star prod}
Let
$M\in\Cox I$, $M'\in\Cox{I'}$ be of finite type and
let~$\phi\in\Hom_{\mathscr{CH}}(M',M)$ be injective. Let~$J',K'\subset I'$
and let~$J=[\phi](J')$, $K=[\phi](K')$.
If~$w_K\star w_J=w_{L}$ for some~$L\subset J$ then~$L=[\phi](L')$
for some~$L'\subset J'$ and
$w_{K'}\star w_{J'}=w_{L'}$.
In particular, if also~$w_J\star w_K=w_L$
then~$w_{J'}\star w_{K'}=w_{L'}$
as well.
\end{lemma}
\begin{proof}
By Proposition~\partref{prop:CH-prop.a''},
$\phi$ is parabolic, whence
$$
\phi(w_{K'}\star w_{J'})=
w_{K}\star w_J=w_L.
$$
Then
$w_{L;J}=p_J(w_K)$ by Proposition~\ref{prop:key prop parab},  and
so
$$
w_\circ^L=p_J(w_K)w_\circ^J
=p_{[\phi](J')}(w_{[\phi](K')})w_\circ^{[\phi](J')}=
p_{[\phi](J')}(\phi(w_{K'}))\phi(w_\circ^{J'})
=\phi(p_{J'}(w_{K'})w_\circ^{J'}),
$$
where we used Lemmata~\ref{lem:Hecke hom w0J} and~\ref{lem:prj hom}.
It follows from Lemma~\partref{lem:[phi]comp.a}
that $$L=\supp w_\circ^L=
\supp\phi(p_{J'}(w_{K'})w_\circ^{J'})=
[\phi](L')$$
where~$L'=\supp p_{J'}(w_{K'})w_\circ^{J'}\subset J'$.
Thus, $w_{L}=w_{[\phi](L')}=\phi(w_{L'})$ and so
$\phi(w_{K'}\star w_{J'})=\phi(w_{L'})$.
Since~$\phi$ is injective, $w_{K'}\star w_{J'}=w_{L'}$.
To prove the last assertion, note that~$L\subset J\cap K$ by Proposition~\ref{prop:key prop parab}. Applying the first assertion
to the pair~$(K',J')$, we obtain~$L=[\phi](L'')$ for some~$L''\subset K'$ and~$w_{J'}\star w_{K'}=w_{L''}$. Since~$[\phi]$ is
injective by Proposition~\ref{prop:CH inj}, $L''=L'$.
\end{proof}

Given~$J\in\mathscr F(M)$, we define~\plink{Psbm}$\mP_J(M)$
as the submonoid of $(W_J(M),\star)$ generated by the~$w_{J';J}$ for all~$J'\subset J$; we abbreviate~$\mP(M)=\mP_I(M)$ when~$I\in\mathscr F(M)$.
The following result was established
in~\cite{BK07} for~$M$ of crystallographic types (that is, of types $ABDE$ and for~$I_2(m)$ with~$m\in\{3,4,6\}$). A rather
sketchy argument based on direct computations was provided in~\cite{He09}. We now provide a proof
for remaining non-crystallographic types based on
Proposition~\ref{prop:key prop parab}.
\begin{proposition}\label{prop:submonoid *}
Let~$J\in\mathscr F(M)$. Then $\mP_J(M)=\{ w_{J';J}\,:\, J'\subset J\}$ and is
commutative. More precisely, for any
$J',J''\subset J$
there exists a unique \plink{*J}$J'\star_J J''=J''\star_J J'\subset J'\cap J''$ such that $w_{J';J}
\star w_{J'';J}=w_{J'\star_J J'';J}$.
\end{proposition}
\begin{proof}
We may assume, without loss of generality, that~$J=I$ and is connected. For~$M$ of types
$A$ through $G$ this result was proven in~\cite{BK07}*{Proposition~2.30}. For~$M$ of type~$I_2(m)$,
the $\star$ product of any two parabolic elements different from~$w_I=1$
is easily seen to be equal to~$w_\circ^I=w_{\emptyset}$, while~$w_I$ is the identity element.

It remains to consider the cases when~$M=H_3$ or~$M=H_4$.
But then we have homomorphisms
$\phi\in\Hom_{\mathscr{CH}}(H_3,D_6)$ (respectively,
$\phi\in\Hom_{\mathscr{CH}}(H_4,E_8)$)
described in Theorem~\ref{thm:artin parab}\ref{thm:adm finite class.H3}\ref{thm:adm finite class.H4}, which satisfy the assumptions of Lemma~\ref{lem:Cox-Heck star prod}.
The assertion for~$H_3$ and~$H_4$ is now immediate.
\end{proof}
\begin{corollary}\label{cor:CH parabolic}
Let~$M\in\Cox I$, $M'\in\Cox{I'}$ and
suppose that both~$M$ and~$M'$ are of finite type. Let~$\phi\in\Hom_{\mathscr{CH}}(M',M)$
be injective.
Then~$\phi$ restricts to an injective homomorphism
$\mP(M')\to \mP_{[\phi](I')}(M)$. Moreover,
$[\phi](J\star_{I'} K)=[\phi](J)\star_{[\phi](I')}[\phi](K)$ for all $J,K\subset I'$. In particular,
$\star_I:\mathscr{P}(I)\times
\mathscr P(I)\to \mathscr P(I)$ commutes
with diagram automorphisms of~$\Gamma(M)$.
\end{corollary}
\begin{proof}
The first assertion is immediate
from Propositions~\ref{prop:submonoid *} and~\partref{prop:CH-prop.a''}, which also imply
that
\begin{align*}
w_{[\phi](J)\star_{[\phi](I')}[\phi](K);[\phi](I')}=
w_{[\phi](J);[\phi](I')}\star
w_{[\phi](K);[\phi](I')}
&=\phi(w_{J;I'}\star w_{K;I'})\\
&=
\phi(w_{J\star_{I'} K;I'})=
w_{[\phi](J\star_{I'} K);[\phi](I')}.
\end{align*}
Since~$w_{L';L}=w_\circ^{L'}w_\circ^L$,
the second assertion follows, while the last is an immediate consequence of the second since
the automorphism of~$(W(M),\star)$
induced by
a diagram automorphism is a fully
supported injective homogeneous homomorphism.
\end{proof}

When~$J=I\in\mathscr F(M)$, we abbreviate~$J\star K=J\star_I K$, $J,K\subset I$.
We can now prove the main result of this section.
\begin{theorem}\label{thm:pJwK is w0KstarJw0J}
Let~$M\in\Cox I$. Then
$p_J(w_{K;L}) = w_{ (J\cap L)\star_L K; J\cap L}$ for any~$J\subset I$, $K\subset L\in\mathscr F(M)$.
In particular, $p_J$ is a parabolic homomorphism.
\end{theorem}
\begin{proof}
It suffices to prove the theorem for~$L=I$. By Proposition~\ref{prop:submonoid *},
for any~$J,K\subset I$ there exists
a unique~$J\star_I K=K\star_I J\subset J\cap K$ such that
$w_{J}\star w_{K}=w_{J\star_I K}$. Then~$p_J(w_K)=w_{J\star K;J}$
by Proposition~\ref{prop:key prop parab}.
\end{proof}
In the rest of this section, we use this theorem to establish some properties
of~$\star$ regarded as a map $\mathscr P(I)\times\mathscr P(I)\to \mathscr P(I)$ and to calculate it. Most of these are known (see~\cite{He09}). We chose to reprove them here for the reader's convenience and also because
of some imprecisions we found in~\cite{He09}.
Our arguments are different and do not depend on computations in~\cite{He09}.

\begin{lemma}[cf.~\cite{He09}*{Lemmata~4 and~7}]\label{lem: J' J K}
Let~$M\in \Cox I$.
Suppose that~$I\in\mathscr F(M)$ and let~$J'\subset J\subset I$. Then~$J'\star_J (J\star K)=J'\star K$ for all~$K\subset I$. In particular, $J\star (K\star L)=(J\star K)\star_J (J\star L)$ for all~$J,K,L\subset I$.
\end{lemma}
\begin{proof}
By Theorem~\ref{thm:pJwK is w0KstarJw0J}, $p_{J'}(w_K)=w_{J'\star K;J'}$. On the other hand, since~$J'\subset J$, we have~$p_{J'}(w)=p_{J'}(p_J(w))$ for all~$w\in W$. Therefore,
$$p_{J'}(p_J(w_K))=p_{J'}(w_{J\star K;J})=
w_{J'\star_J(J\star K);J'}$$ by Theorem~\ref{thm:pJwK is w0KstarJw0J}.
It follows that~$w_\circ^{J'\star K}=w_\circ^{J'\star_J(J\star K)}$
and so~$J'\star K=J'\star_J(J\star K)$.
In particular, since~$J\star L\subset J$, we have~$(J\star L)\star_J
(J\star K)=(J\star L)\star K=J\star (L\star K)=J\star (K\star L)$
since the operation~$\star:\mathscr P(I)\times \mathscr P(I)\to\mathscr P(I)$ is commutative and associative.
\end{proof}

\begin{lemma}[cf. \cite{He09}*{Lemma~6}]\label{lem: J' v J'' * K}
Let $J',J''\subset I$ be orthogonal and let $K\subset I$. Then
$(J'\cup J'')\star K=(J'\star K)\cup (J''\star K)$.
\end{lemma}
\begin{proof}
Let~$J=J'\cup J''$. Since~$J'$ and~$J''$ are orthogonal, by Lemmata~\ref{lem:p_J comp} and~\ref{lem:parab prod}
and Theorem~\ref{thm:pJwK is w0KstarJw0J}
\begin{align*}
w_{J\star K;J}&=p_{J}(w_K)\\
&=p_{J'}(p_J(w_K))\times p_{J''}(p_J(w_K))=p_{J'}(w_K)\times p_{J''}(w_K)
\\
&=w_{J'\star K;J'}\times w_{J''\star K;J''}=
w_\circ^{J'\star K}w_\circ^{J'}w_\circ^{J''\star K}w_\circ^{J''}
\end{align*}
Since $S\star K\subset S$ for any $S\subset I$ and $J',J''$ are orthogonal, it follows that
$w_\circ^{J'}w_\circ^{J''\star K}=w_\circ^{J''\star K}w_\circ^{J'}$ and so
$w_{J'\star K;J'}\times w_{J''\star K;J''}=w_\circ^{J'\star K}
w_\circ^{J''\star K}w_\circ^{J'}w_\circ^{J''}=
w_\circ^{J'\star K}
w_\circ^{J''\star K}w_\circ^{J}$. Therefore,
$w_\circ^{J\star K}=w_\circ^{J'\star K}
w_\circ^{J''\star K}=w_\circ^{(J'\star K)\cup (J''\star K)}$ and
the assertion follows.
\end{proof}
It follows from Lemmata~\ref{lem: J' J K} and~\ref{lem: J' v J'' * K}
that to calculate~$J\star K$ for all~$J,K\subset I$, it suffices
to find~$J\star K$ for~$J$, $K$ connected and satisfying $|J|=|K|=|I|-1$ and use induction on~$|I|$ for infinite series.

We will now calculate~$J\star K$ for irreducible finite types using Theorem~\ref{thm:pJwK is w0KstarJw0J}. First, for type~$A$ we obtain
a quite neat expression of parabolic elements
in terms of powers of Coxeter elements in the Hecke monoid.

\begin{proposition}\label{prop:wK explicit}
Let~$M=A_n$. Then
for any $1\le i\le j\le k\le l\le n$,
$$
w_{[j,k];[i,l]}
=\cx il{}^{\star(j-i)}\star \cxr il{}^{\star (l-k)}
= \cxr il{}^{\star (l-k)}\star\cx il{}^{\star(j-i)}.
$$
\end{proposition}
\begin{proof}
First we prove the Proposition for the case when either $i=j$ or $k=l$.
\begin{lemma}\label{lem:*powers of coxeter}
For all $1\le a\le b\le n$, $k\ge 0$ we have
\begin{align*}
&\cxr ab{}^{\star k}=w_{[a,b-k];[a,b]}=w_{[a+k,b];[a,b]}{}^{-1},\quad
\cx ab{}^{\star k}=w_{[a,b-k];[a,b]}{}^{-1}=w_{[a+k,b];[a,b]}.
\end{align*}
\end{lemma}
\begin{proof}
We only prove the first equality since the remaining ones can be obtained by applying~${}^{op}$ and or diagram automorphisms.
The argument is by induction on~$k$, the case~$k=0$ being trivial.
For the inductive step, note that if $0\le k\le b-a$ then $w_\circ^{[a,b-k]}=\cx a{(b-k)}\times
w_\circ^{[a,b-(k+1)]}$ by~\eqref{eq:w0 type A}
and so $w_{[a,b-(k+1)];[a,b]}=\cxr a{(b-k)}\times w_{[a,b-k];[a,b]}$. Now, by the induction hypothesis
\begin{align*}
\cxr ab{}^{\star (k+1)}&=\cxr ab\star w_{[a,b-k];[a,b]}=\cxr{(b-k+1)}b
\star \cxr a{(b-k)}\star  w_{[a,b-k];[a,b]}\\&=
\cxr{(b-k+1)}b\star w_{[a,b-(k+1)];[a,b]}=w_{[a,b-(k+1)];[a,b]}.
\end{align*}
The last equality follows from Lemma~\ref{lem:wK absorption} since
$[b-k+1,b]\subset [a,b]\setminus[a,b-k-1]$.

In particular, we proved that $\cxr ab{}^{\star (b-a+1)}=w_\circ^{[a,b]}$.
Since $\cxr ab\star w_\circ^{[a,b]}=w_\circ^{[a,b]}$ by Lemma \ref{lem:char w_0 monoid},
it follows that $\cxr ab{}^{\star k}=w_\circ^{[a,b]}=w_{\emptyset;[a,b]}$ for all~$k\ge b-a+1$.
\end{proof}
To treat the general case, we use induction on~$j-i$
to show that
\begin{equation}\label{eq:wKexplicit inter}
w_{[j,k];[i,l]}
=\cx il{}^{\star(j-i)}\star w_{[i,k];[i,l]}.
\end{equation}
Once~\eqref{eq:wKexplicit inter} is established, the Proposition follows by Lemma~\ref{lem:*powers of coxeter}.

The case $j=i$ is trivial.
For the inductive step, note that, for $j\le k$,
$
w_\circ^{[j,k]}=
w_\circ^{[j+1,k]}\times \cx jk=\cxr jk\times w_\circ^{[j+1,k]}
$ by~\eqref{eq:w0 type A},
and so
\begin{equation}\label{eq:w0step}
w_{[j+1,k];[i,l]}=\cx jk\times w_{[j,k];[i,l]}.
\end{equation}
Then
\begin{alignat*}{3}
w_{[j+1,k];[i,l]}&=\cx i{(j-1)}\star w_{[j+1,k];[i,l]}&&
\text{by Lemma~\ref{lem:wK absorption}}
\\
&=\cx i{(j-1)}\star \cx jk\times w_{[j,k];[i,l]}&\qquad&\text{by \eqref{eq:w0step}}\\
&=\cx ik\star w_{[j,k];[i,l]}\\
&=\cx il\star w_{[j,k];[i,l]}&&\text{by
Lemma~\ref{lem:wK absorption}}\\
&=\cx il{}^{\star (j-i+1)}\star w_{[i,k];[i,l]}&&\text{by the induction
hypothesis.}
\end{alignat*}
The inductive step is proven. The second equality is obtained from the first one using the diagram automorphism of~$W_{[i,l]}(A_n)$.
\end{proof}
\begin{corollary}[cf.~\cite{He09}]\label{cor:A J*K}
Let~$M=A_n$,
$J=[a',b']$, $K=[a,b]$, $1\le a\le b\le n$, $1\le a'\le b'\le n$.
Then $J\star K=[a+a'-1,b+b'-n]$.
\end{corollary}
\begin{proof}
We have, by Proposition~\ref{prop:wK explicit},
$$
p_J(w_K)=p_{[a',b']}(\cx1n{}^{\star(a-1)}
\star \cxr1n{}^{\star(n-b)})
=\cx{a'}{b'}{}^{\star(a-1)}\star
\cxr{a'}{b'}{}^{\star(n-b)}
=w_{[a+a'-1,b+b'-n];[a',b']}.
$$
It remains to apply Theorem~\ref{thm:pJwK is w0KstarJw0J}.
\end{proof}

To treat type~$B$, we use the homogeneous
homomorphism $(W(B_n),\star)\to (W(A_{2n-1}),\star)$
from Theorem~\partref{thm:adm finite class.unfold}.
\begin{lemma}[cf.~\cite{He09}]\label{lem:B J*K}
Let~$M=B_n$ and let $J=[a',b']$, $K=[a,b]$, $1\le b\le b'\le n$.
Then
$$
J\star K=\begin{cases}
\emptyset,&b'<n,\\
[a+a'-1,b-a'+1],&b<b'=n,\\
[a+a'-1,n],&b=b'=n.
\end{cases}
$$
\end{lemma}
\begin{proof}
Let~$\wh I=[1,n]$,
$I=[1,2n-1]$ and
let~$\phi\in\Hom_{\mathscr{CH}}(B_n,A_{2n-1})$ be the injective
homomorphism from~\eqref{eq:unfold Bn A2n-1}. Then~$[\phi]$
is injective by Proposition~\ref{prop:CH inj}.
Moreover, since~$\phi$ is fully supported, $[\phi](J\star_{\wh I}K)=[\phi](J)\star_I[\phi](K)$ for all~$J,K\subset\wh I$ by Corollary~\ref{cor:CH parabolic}.
Note that $[\phi]([a,b])=[a,b]\sqcup [2n-b,2n-a]$ if~$1\le a\le b<n$ while
$[\phi]([a,n])=[a,2n-a]$, $1\le a\le n$, and that the intervals $[a,b]$, $[2n-b,2n-a]$
are orthogonal subsets of~$I$.
If~$b,b'<n$,
we have by Lemma~\ref{lem: J' v J'' * K} and
Corollary~\ref{cor:A J*K}
\begin{align*}
[\phi](J\star_{\wh I} K)&=([a',b']\star_I [a,b])\cup
([a',b']\star_I [2n-b,2n-a])\cup\\
&\qquad ([a,b]\star_I [2n-b',2n-a'])
\cup ([2n-b',2n-a']\star_I [2n-b,2n-a])\\
&=[a+a'-1,b+b'-2n+1]\cup [2n-b+a'-1,b'-a+1]\\
&\qquad
\cup [a+2n-b'-1,b-a'+1]\cup [4n-b-b'-1,2n-a-a'+1].
\end{align*}
All these intervals are empty
since $b-a,b'-a'\le n-2$. Since~$[\phi]$ is injective, it follows that
$J\star_{\wh I} K=\emptyset$.

If~$b<b'=n$ then again by Lemma~\ref{lem: J' v J'' * K} and Corollary~\ref{cor:A J*K},
\begin{align*}
[\phi](J\star_{\wh I} K)&=([a',2n-a']\star_I [a,b])\cup
([a',2n-a']\star_I [2n-b,2n-a])\\
&=[a+a'-1,b-a'+1]\cup [2n+a'-b-1,2n-a-a'+1]
\\
&=[\phi]([a+a'-1,b-a'+1]).
\end{align*}
Since~$[\phi]$ is injective,
$J\star_{\wh I}K=
[a+a'-1,b-a'+1]$.
Finally, if $b=b'=n$ then, by Corollary~\ref{cor:A J*K},
\begin{align*}
[\phi](J\star_{\wh I} K)&=[a',2n-a']\star_I [a,2n-a]
=[a+a'-1,2n-a-a'+1]=[\phi]([a+a'-1,n]),
\end{align*}
and it remains to use the injectivity of~$[\phi]$.
\end{proof}

We now consider~$M=D_{n+1}$, $n\ge 3$.
Let~$\sigma$ be the diagram automorphism
of~$W(D_{n+1})$ corresponding to the permutation~$(n,n+1)$ (cf.~\eqref{eq:diag aut}).
The following identities are easily checked
\begin{gather}
w_\circ = {\ascprodst_{1\le i\le n}}
\cx i{(n+1)}\times\cxr i{(n-1)}\label{eq:w0 D}\\
\label{eq:w1,n D}
w_{[1,n]}={\ascprodst_{1\le i\le n}}\sigma^i(\cxr{i}{n})
=\sigma(\cxr1n \times w_{[2,n];[2,n+1]})
\\
\label{eq:w i,n+1 D}
w_{[i,n+1]}={\dscprodst_{1\le j\le i-1}}
(\cx j{(n+1)}\times\cxr j{(n-1)}),\qquad 1\le i\le n-1.
\end{gather}
\begin{proposition}\label{prop:1,n*K D}
Let $J=[1,n]$ and let $K\subset I$ be connected. We have
$$
J\star K=\begin{cases}
[1,n]_2,& K=J,\\
[1,n-1]_2,&K=\sigma(J),\\
[i,n-i+1],&K=[i,n+1],\,1\le i\le n-1,\\
\emptyset,&\text{otherwise}.
\end{cases}
$$
\end{proposition}
\begin{proof}
By~\eqref{eq:w1,n D}
and~\eqref{eq:comm cox}
$$
p_{J}(w_{J})
=\ascprod_{1\le i\le n} s_n^i\cxr in
=w_\circ^{[1,n]_2}w_\circ^{J}.
$$
Similarly,
$$
p_{J}(w_{\sigma(J)})=\ascprod_{1\le i\le n}
s_n^{i-1}\cxr in=w_\circ^{[1,n-1]_2}w_\circ^{J}.
$$
Therefore, $J\star J=[1,n]_2$ and~$J\star\sigma(J)=[1,n-1]_2$ by Theorem~\ref{thm:pJwK is w0KstarJw0J}.

Next we consider~$[i,n+1]\star J$, $2\le i\le n-1$, the case~$i=1$ being trivial. We have
$$
p_{[2,n+1]}(w_{J})=
\sigma(\cxr2n)\star\sigma(w_{[2,n];[2,n+1]})
=\sigma(\cxr2n w_{[2,n];[2,n+1]})=
\sigma(w_\circ^{[2,n-1]}w_\circ^{[2,n+1]})
=w_{[2,n-1];[2,n+1]},
$$
where we used Lemma~\ref{lem:wK absorption}.
Thus, $[2,n+1]\star J=[2,n-1]$ by Theorem~\ref{thm:pJwK is w0KstarJw0J}.
For the inductive step, note that
by Lemma~\ref{lem: J' J K} and by
the induction hypothesis
$$
[i+1,n+1]\star J=
[i+1,n+1]\star_{[i,n+1]}([i,n+1]\star J)
=[i+1,n+1]\star_{[i,n+1]} [i,n+1-i].
$$
Since~$W_{[i,n+1]}(D_{n+1})$
is of type~$D_{n+2-i}$ and both
$[i,n+1-i]$, $[i+1,n+1]$ are
invariant with respect to the diagram automorphism of~$W_{[i,n+1]}(D_{n+1})$,
using Lemma~\ref{lem:B J*K} and
the homogeneous homomorphism from~\eqref{eq:unfold Bn Dn+1}
(under which, after the identification~$s_j\mapsto s_{j-i+1}$, these subsets correspond to~$[1,n+2-2i]$ and~$[2,n+1-i]$, respectively, in~$W(B_{n+1-i})$)
we obtain $[i+1,n+1]\star_{[i,n+1]}[i,n+1-i]=
[i+1,n-i]$, which proves the inductive step.

Finally, we claim that if~$K\subsetneq J$ or $K\subsetneq \sigma(J)$ then $J\star K=\emptyset$.
Suppose first that $K\subsetneq J$. Then
either $K\subset [1,n-1]$ or $K\subset [2,n]$ and so, by Lemma~\ref{lem: J' J K},
either $K\star J=K\star_{[1,n-1]}([1,n-1]\star J)$ or~$K\star J=K\star_{[2,n]}([2,n]\star J)$.
Thus, it suffices to
prove that $[1,n-1]\star J=\emptyset=[2,n]\star J$.
We have
$$
p_{[1,n-1]}(w_{[1,n]})=\ascprod_{1\le i\le n}
\cxr i{(n-1)}=w_\circ^{[1,n-1]},
$$
and so~$[1,n-1]\star J=\emptyset$ by Theorem~\ref{thm:pJwK is w0KstarJw0J}.
On the other hand, by Lemma~\ref{lem: J' J K}
$$
[2,n]\star J=
[2,n]\star_{J}(J\star J)
=[2,n]\star_{J}[1,n]_2.
$$
Since~$W_{J}(D_{n+1})$ is of type~$A_n$,
\begin{equation}\label{eq:singleton}\{i\}\star_{J}[2,n]=\emptyset,\qquad i\in J
\end{equation}
by Corollary~\ref{cor:A J*K},
and so
$[1,n]_2\star_{J}[2,n]=\emptyset$ by Lemma~\ref{lem: J' v J'' * K}.

It remains to consider the case when $K=\sigma([i,n])$ for some~$2\le i\le n$. As before, by Lemma~\ref{lem: J' J K} it
suffices to show that~$\sigma([2,n])\star J=\emptyset$ or, equivalently, that~$[2,n]\star \sigma(J)=\emptyset$. Indeed, by Lemma~\ref{lem: J' J K},
$$
[2,n]\star \sigma(J)=[2,n]\star_J(J\star \sigma(J))=[2,n]\star_J [1,n-1]_2.
$$
It remains to apply~\eqref{eq:singleton} and Lemma~\ref{lem: J' v J'' * K}.
\end{proof}

\begin{lemma}\label{lem:2,n+1*K D}
Let~$K\subset I$ be connected. Then for $J=[2,n+1]$
$$
J\star K=
\begin{cases}
[i+1,n+1],&K=[i,n+1],\,1\le i\le n-1\\
[i+1,j-1],&K=[i,j],\,1\le i\le j\le n-1\\
[i+1,n-1],&\text{$K=[i,n]$ or $K=\sigma([i,n])$, $1\le i\le n$}\\
\end{cases}
$$
\end{lemma}
\begin{proof}
Note that~$J$ is $\sigma$-invariant.
In the first two cases, $K=\sigma(K)$ and, since~$w_\circ$
is $\sigma$-invariant,
$w_K$, $w_J$ are $\sigma$-invariant. Yet the set of~$\sigma$-invariant elements in~$(W(D_{n+1}),\star)$ is isomorphic to~$(W(B_n),
\star)$ by Theorem~\ref{thm:adm finite class} and so we can apply Lemma~\ref{lem:B J*K}.

To prove the assertion for~$K=[i,n]$ we use induction on~$i$. The case $i=1$ has already been established in Proposition~\ref{prop:1,n*K D}. For the inductive step, we have
$$
[i+1,n]\star J=[i+1,n]\star_{[i,n]}([i,n]\star J)
=[i+1,n]\star_{[i,n]}[i+1,n-1].
$$
Since~$W_{[i,n]}(D_{n+1})$ is of type~$A_{n+1-i}$, it remains to
apply Corollary~\ref{cor:A J*K}.
The result for~$K=\sigma([i,n])$ is now immediate.
\end{proof}

We will now calculate~$J\star K$ for~$J$ and~$K$ connected and of corank 1 and~$M=E_n$, $n\in\{6,7,8\}$.

First, let~$M=E_6$ and let~$\sigma$ be
its diagram automorphism (cf.~\eqref{eq:diag aut}).
Since we are only considering~$J$, $K$ which
are connected and satisfy~$|J|=|K|=5$,
we only need to consider~$J,K\in \{[1,5],[2,6],\sigma([2,6])\}$.
Using a Python program we developed for computations in Hecke monoids, we obtain
\begin{align*}
w_{[1,5]}&=s_6s_3s_2s_1s_4s_3s_2s_5s_4s_3s_6s_3s_2s_1s_4s_3s_2s_5s_4s_3s_6,\\
w_{[2,6]}&=
s_1s_2s_3s_4s_5s_6s_3s_2s_1s_4s_3s_2s_6s_3s_4s_5,\end{align*}
and so
\begin{align*}
p_{[1,5]}(w_{[1,5]})&=(s_3s_2s_1s_4s_3s_2s_5s_4s_3)^{\star 2}=w_\circ^{[1,5]},\\
p_{[1,5]}(w_{[2,6]})
&=s_1s_2s_1s_3s_2s_4s_3s_2s_1s_5s_4s_3s_2
=w_\circ^{[1,3]}s_1 \cxr14\cxr15 s_1
=s_3 w_\circ^{[1,5]}s_1=w_{\{3,5\};[1,5]},\\
p_{[2,6]}(w_{[2,6]})&=
s_2s_3s_4s_5s_4s_3s_2s_6s_3s_2s_4s_3s_5s_6=
w_\circ^{[3,4]}w_\circ^{[2,6]}\cx24\\
&=w_\circ^{[3,4]}s_6 s_3 s_4 w_\circ^{[2,6]}
=w_\circ^{\{3,4,6\}}w_\circ^{[2,6]}=w_{\{3,4,6\};[2,6]},\\
p_{[2,6]}(w_{\sigma([2,6])})&=
s_5s_4s_3s_2s_6s_3s_2s_4s_3s_5s_4s_6s_3s_2=
w_\circ^{[2,4]}w_\circ^{[2,6]}=w_{[2,4];[2,6]}.
\end{align*}
Thus, by Theorem~\ref{thm:pJwK is w0KstarJw0J},
$[1,5]\star [1,5]=\emptyset$, while
$$
[1,5]\star[2,6]=\{3,5\},\quad
[1,5]\star\sigma([2,6])=\{1,3\},\quad
[2,6]\star [2,6]=\{3,4,6\},\quad
[2,6]\star \sigma([2,6])=[2,4].
$$

For~$M=E_7$, the set~$\mathscr J_{E_7}$
of connected subsets of~$I=[1,7]$
of cardinality~$6$ is $\{[1,6],[2,7], [1,5]\cup\{7\}\}$, the corresponding parabolic submonoids being of types $A_6$, $D_6$ and~$E_6$.
We have
\begin{align*}
w_{[1,6]}&=
s_7s_3s_2s_1s_4s_3s_2s_5s_4s_3s_6s_5s_4s_7s_3s_2s_1s_4s_3s_2s_5s_4s_3s_6s_5s_4s_7s_3s_2s_1s_4s_3s_5s_4s_6s_5s_7s_3s_2s_4s_3s_7,\\
w_{[1,5]\cup\{7\}}&=s_6s_5s_4s_3s_2s_1s_7s_3s_2s_4s_3s_5s_4s_6s_5s_7s_3s_2s_1s_4s_3s_2s_7s_3s_4s_5s_6,\\
w_{[2,7]}&=s_1s_2s_3s_4s_5s_6s_7s_3s_2s_1s_4s_3s_2s_5s_4s_3s_7s_3s_2s_1s_4s_3s_2s_5s_4s_3s_6s_5s_4s_7s_3s_2s_1.
\end{align*}
Then~$p_{[1,6]}(w_{[1,6]})=p_{[1,6]}(w_{[1,5]\cup\{7\}})=
p_{[1,6]}(w_{[2,7]})=w_\circ^{[1,6]}$,
that is, $[1,6]\star K=\emptyset$ for
all~$K\in\mathscr J_{E_7}$ by Theorem~\ref{thm:pJwK is w0KstarJw0J},
while
\begin{align*}
p_{J}(w_J)&=s_1s_2s_3s_4s_5s_4s_3s_2s_1s_7s_3s_2s_1s_4s_3s_2s_5s_4s_7s_3s_2s_1s_4s_5\\
&=w_\circ^{[2,4]}w_\circ^J s_7 s_3 s_4 s_2 s_3 s_7
=w_\circ^{[2,4]}s_7 s_3 s_2 s_4 s_3 s_7 w_\circ^J
=w_\circ^{[2,4]\cup\{7\}}w_\circ^J=w_{[2,4]\cup\{7\};J},\\
p_{[2,7]}(w_J)&=
s_5s_4s_3s_2s_6s_5s_4s_3s_2s_7s_3s_2s_4s_3s_5s_4s_6s_5s_7s_3s_2s_4s_3s_7=w_{[2,4];[2,7]},\\
p_{[2,7]}(w_{[2,7]})&=s_2s_3s_2s_4s_3s_5s_4s_3s_2s_6s_5s_4s_3s_7s_3s_2s_4s_3s_5s_4s_6s_5s_7s_3s_2s_4s_3\\
&=s_4 w_\circ^{[2,5]}w_\circ^{[2,6]} \cx25 s_7 w_\circ^{[2,7]}=s_4 \cxr26\cx25 s_7 w_\circ^{[2,7]}=s_4s_6s_7 w_\circ^{[2,7]}=w_{\{4,6,7\};[2,7]},
\end{align*}
where~$J=[1,5]\cup\{7\}$. Thus, the only non-empty
products of elements of~$\mathscr J_{E_7}$ are,
by Theorem~\ref{thm:pJwK is w0KstarJw0J}
$$
J\star J=[2,4]\cup \{7\},\quad
J\star [2,7]=[2,4],\quad [2,7]\star [2,7]=\{4,6,7\}.
$$
Finally, let~$M=E_8$. The set of connected subsets of~$I=[1,8]$ of cardinality~$7$
is
$\mathscr J_{E_8}=\{[1,7],\,[2,8],\,
I\setminus\{7\}\}$, and respective parabolic submonoids are of type~$A_7$, $D_7$ and~$E_7$.
We have
\begin{align*}
w_{[1,7]}&=s_8s_3s_2s_1s_4s_3s_2s_5s_4s_3s_6s_5s_4s_7s_6s_5s_8s_3s_2s_1s_4s_3s_2s_5s_4s_3s_6s_5s_4s_7s_6s_5s_8s_3s_2s_1s_4s_3s_2s_5s_4s_3s_6s_5s_4\times\\
&\phantom{=} s_7s_6s_5s_8s_3s_2s_1s_4s_3s_2s_5s_4s_3s_6s_5s_4s_7s_8s_3s_2s_1s_4s_3s_2s_5s_4s_3s_6s_5s_4s_7s_8s_3s_2s_1s_4s_3s_5s_4s_6s_5s_8s_3s_2s_4s_3s_8,\\
w_{[2,8]}&=s_1s_2s_3s_4s_5s_6s_7s_8s_3s_2s_1s_4s_3s_2s_5s_4s_3s_6s_5s_4s_8s_3s_2s_1s_4s_3s_2s_5s_4s_3s_6s_5s_4s_7s_6s_5s_8s_3s_2s_1s_4s_3s_2s_5s_4\times\\
&\phantom{=}s_6s_5s_7s_6s_8s_3s_2s_1s_4s_3s_2s_5s_4s_3s_6s_7s_8s_3s_2s_1s_4s_3s_2s_5s_4s_3s_6s_5s_4s_8s_3s_2s_1,\\
w_{I\setminus\{7\}}&=
s_7s_6s_5s_4s_3s_2s_1s_8s_3s_2s_4s_3s_5s_4s_6s_5s_7s_6s_8s_3s_2s_1s_4s_3s_2s_5s_4s_3s_8s_3s_2s_1s_4s_3s_2s_5s_4s_3s_6s_5s_4s_7s_6s_5s_8\times\\
&\phantom{=}s_3s_2s_1s_4s_3s_2s_8s_3s_4s_5s_6s_7.
\end{align*}
Then
\begin{align*}
p_{I\setminus\{7\}}(w_{I\setminus\{7\}})&=s_1s_2s_3s_4s_5s_4s_3s_2s_1s_6s_5s_4s_3s_2s_1s_8s_3s_2s_1s_4s_3s_2s_5s_4s_3s_6s_5s_4s_8s_3s_2s_1s_4s_3\times\\
&\qquad s_2s_5s_4s_3s_6s_5s_4s_8s_3s_2s_1s_4s_3s_5s_4s_6s_5\\
&=\cx15\cxr14 w_\circ^{[1,5]}w_\circ^{I\setminus\{7\}} s_8s_3s_4s_2s_3s_8\\
&=w_\circ^{[2,4]}s_8s_3s_4s_2s_3s_8 w_\circ^{I\setminus\{7\}}=
w_\circ^{[2,4]\cup\{8\}} w_\circ^{I\setminus\{7\}}=w_{[2,4]\cup\{8\};I\setminus\{7\}}
\end{align*}
and so~$(I\setminus\{7\})\star(I\setminus\{7\})=[2,4]\cup\{8\}$ by Theorem~\ref{thm:pJwK is w0KstarJw0J}. For all other pairs~$J,K\in\mathscr J_{E_8}$, $p_J(w_K)=w_\circ^J$ that is~$J\star K=\emptyset$.

The map~$\star:\mathscr P(I)\times\mathscr P(I)\to \mathscr P(I)$ for~$F_4$, $H_3$ and~$H_4$ can be calculated using the corresponding results for~$E_6$, $D_6$ and~$E_8$ and homomorphisms from Theorem~\ref{thm:adm finite class} like in our proof of Lemma~\ref{lem:B J*K}, or directly by Theorem~\ref{thm:pJwK is w0KstarJw0J}. It turns
out that~$J\star K=\emptyset$ for all
$J,K\subset I$ in these types with~$J$, $K$ connected and of corank~1, which implies
that~$J\star K=\emptyset$ for all~$J,K\subsetneq I$.

\begin{remark}\label{rem:surj counterex}
The restriction of~$p_J$ to~$\mP(M)$ need not be surjective. For example, if $M=A_4$
and~$J=[1,3]$ then $\{1,3\}\subset J$ is not equal to
$J\star K$ for any~$K\subset [1,4]$ and
so $w_{\{1,3\};J}\not=p_J(w_K)$ for
any~$K\subset [1,4]$. Indeed,
by Corollary~\ref{cor:A J*K}, $J\star K$ is an interval for any connected~$K$, $J\star K=\{1\}$
if and only if~$K=[1,2]$ and $J\star K=\{3\}$
if and only if~$K=[2,3]$. Yet $[1,2]$ and~$[2,3]$
are not orthogonal.
\end{remark}

\subsection{Light homomorphisms of Hecke monoids are parabolic}\label{subs:light->parab}
We can now prove the following
\begin{theorem}\label{thm:light->parab}
Any light homomorphism of Hecke monoids is parabolic.
\end{theorem}
\begin{proof}
In view of Proposition~\ref{prop:classify light},
Theorem~\ref{thm:pJwK is w0KstarJw0J} and
Lemma~\ref{lem:[phi]comp},
it remains to prove that~$\phi$ is parabolic when~$\phi$ is either a folding along some surjective map
or is tautological. Note that
since a restriction of a light $\phi\in\Hom_{\He}(M',M)$
to any parabolic submonoid of~$W(M')$ is again light,
it suffices to consider the case when~$M'$ is of finite type and irreducible and to consider only $w_{J;K}$
with~$K=I'$, that is, there is no need to consider $K$-parabolic elements for~$K\subsetneq I'$.

We begin with tautological homomorphisms. Note that if~$M=I_2(m)$ then
for any~$m'>m$ we have a tautological
$\phi\in\Hom_{\He}(I_2(m'),I_2(m))$.
Since
$\phi(\brd{s'_is'_j}{k})=\brd{s_i\star s_j\star}{k}$, $k\le m'$,
and non-identity parabolic elements $w_{J;\{1,2\}}$
in~$I_2(m')$
correspond to~$k\in \{m'-1,m'\}\ge m$,
it follows that
$\phi(w_{J})=w_\circ^{\{1,2\}}$ and so~$\phi$ is parabolic.

Suppose now that~$|I|>2$ and that~$M'$ is irreducible. Note that if~$\phi\in\Hom_{\He}(M',M)$ is tautological
then~$\Gamma(M)$ is obtained from~$\Gamma(M')$
by either decreasing some labels or removing edges.

Suppose that the underlying graphs of~$\Gamma(M')$ and~$\Gamma(M)$ are isomorphic.
Then the only possibilities are:
\begin{enumerate}[label={$\arabic*^\circ$.},
ref={$\arabic*^\circ$}]
\item\label{taut-case-BnAn}
$M'=B_n$, $M=A_n$ and so~$\phi$ is
the composition of the homomorphism $(W(B_n),\star)\to
(W(A_{2n-1}),\star)$ defined by~\eqref{eq:unfold Bn A2n-1} with~$p_{[1,n]}:(W(A_{2n-1}),\star)\to (W_{[1,n]}(A_{2n-1}),\star)\cong (W(A_n),\star)$
and hence is parabolic
by Theorems~\ref{thm:adm finite class} and~\ref{thm:pJwK is w0KstarJw0J};

\item\label{taut-case-F4A4}
$M'=F_4$, $M=A_4$ and so~$\phi$ is the composition of the homomorphism $(W(F_4),\star)\to (W(E_6),\star)$ defined by~\eqref{eq:unfold F4 E6}
with the parabolic projection $p_{\{1,2,3,6\}}:(W(E_6),\star)\to (W_{\{1,2,3,6\}}(E_6),\star)\cong (W(A_4),\star)$
hence the assertion holds in this case
by Theorems~\ref{thm:adm finite class} and~\ref{thm:pJwK is w0KstarJw0J};

\item\label{taut-case-HnBn}
$M'=H_n$, $M=B_n$, $n\in\{3,4\}$.
We claim
that~$\phi(w_{J})=w_\circ^{[1,n]}$ for all~$J\subsetneq I$. By Lemma~\ref{lem:len prop wJ K},
it suffices to prove
the claim for~$J$ with~$|J|=n-1$.

If~$n=3$, let~$c=s_1s_3s_2$ and~$c'=s_1s_2s_3$. Then~$w_\circ^{[1,3]}=c^{\times 5}=c'^{\times 5}$ in~$W(H_3)$
by Proposition~\partref{prop:Coxeter splitting.b}
and so
\begin{align*}
&w_{\{1,2\}}=s_3s_2s_3\times c'^{\times 3},
\quad w_{\{1,3\}}=s_2 \times c^{\times 4},\\
&w_{\{2,3\}}=s_1s_2s_3s_2s_1s_3s_2s_3s_2s_1.
\end{align*}
Since~$c^{\star 3}={c'}^{\star 3}=w_\circ^{[1,3]}$ in~$(W(B_3),\star)$, the claim is obvious
for~$J\in\{\{1,2\},\{1,3\}\}$
while
\begin{align*}
\phi(w_{\{2,3\}})&=s'_1\star s'_2\star s'_3\star s'_2\star s'_1\star s'_3\star s'_2\star s'_3\star s'_2\star s'_1
\\
&=s'_1\star s'_2\star s'_3\star s'_2\star s'_1\star s'_2\star s'_3\star s'_2\star s'_3\star s'_1\\
&=s'_1\star s'_2\star s'_1\star s'_3\star s'_2\star s'_1\star s'_3\star s'_2\star s'_3\star s'_1
=w_\circ^{[1,3]}\star s'_1=w_\circ^{[1,3]}.
\end{align*}
Similarly, for~$n=4$ we have
\begin{align*}
&w_{[1,3]}=(w_\circ^{[1,3]}\cx14{}^3)\times \cx14{}^{\times 12}=s_4s_3s_4s_2s_3s_4\times \cx14{}^{\times 12},\\
&w_{\{1,2,4\}}=
s_3s_2s_1\times \cxr14{}^{\times 12}\times s_4s_3s_2s_4s_3,\\
&w_{\{1,3,4\}}=
s_2s_1s_3s_2\times \cxr14{}^{\times10}
\times \cxr24\times\cxr14\times\cxr24\\
&w_{[2,4]}=\cx13\times (\cxr14\times
\cxr24\times \cxr14)^{\times 3}\times \cxr24
\times\cxr34\times\cxr14.
\end{align*}
Since~$\cxr14{}^{\star 4}=w_\circ^{[1,4]}$ in~$(W(B_4),\star)$, the claim is immediate
for~$J\not=\{2,3,4\}$ with~$|J|=3$. Since
$\cxr14\star\cxr24\star\cxr14{}^{\star 2}\star \cxr24
=w_\circ^{[1,4]}$ in~$(W(B_4),\star)$,
the claim follows for~$J=\{2,3,4\}$ as well.

\item\label{taut-case-HnAn}
$M'=H_n$, $M=A_n$, $n\in\{3,4\}$.
Then~$\phi$ is the composition
of a homomorphism from~\ref{taut-case-HnBn} with
the respective homomorphism from~\ref{taut-case-BnAn}
and hence is parabolic.
\end{enumerate}

It remains to consider the case when~$\Gamma(M)$
is obtained from~$\Gamma(M')$ by removing some
edges. Let~$I_1,\dots,I_r$ be vertex sets
of connected components of~$\Gamma(M)$.

Let~$J\subset I$. Then by Lemmata~\ref{lem:taut pJ commute}, \ref{lem:parab prod} and Theorem~\ref{thm:pJwK is w0KstarJw0J}
\begin{align*}\phi(w_J)=
\prodst_{1\le i\le r} p_{I_i}(\phi(w_J))=
\prodst_{1\le i\le r}\phi(p'_{I_i}(w_J))
=\prodst_{1\le i\le r}\phi(w_{J\star I_i;I_i})
\end{align*}
where~$J\star I_i$, $1\le i\le r$ is taken in~$W(M')$. Since the restriction of~$\phi$ to~$W_{I_i}(M')$, $1\le i\le r$ is light and hence parabolic, as $\Gamma_{I_i}(M')$ and~$\Gamma_{I_i}(M)$ are connected, it follows that $\phi(w_{J\star I_i;I_i})=
w_{J'_i;I_i}$, $J'_i\subset I_i$ and so~$\phi(w_J)=
w_{J'_1\cup\cdots\cup J'_r}$ by Lemma~\ref{lem:prod orth parab}.

We now turn our attention to foldings.
The only non-trivial optimal foldings in finite types with irreducible~$M'$
are $\mathbf f_{\varpi_{(n,n+1)}}:(W(D_{n+1}),\star)\to (W(A_n),\star)$, $n\ge 2$ (where~$D_3=A_3$)
and $\mathbf f_{\varpi_{(1,3,4)}}:(W(D_4),\star)\to (W(A_2),\star)$ where $\varpi_{(n,n+1)}:[1,n+1]\to[1,n]$ and~$\varpi_{(1,3,4)}:
[1,4]\to [1,2]$ are defined as in Example~\ref{ex:light projection}.

Consider first~$\mathbf f_{\varpi(n,n+1)}$.
Let~$\sigma$ be the diagram automorphism of~$W(D_{n+1})$
which corresponds to the permutation $(n,n+1)$ of~$I=[1,n+1]$.
Since, obviously,
$\mathbf f_{\varpi(n,n+1)}(\sigma(w))=\mathbf f_{\varpi(n,n+1)}(w)$, it suffices to consider
the case when either $\sigma(J)=J$ or
$\{n,n+1\}\cap J=\{n\}$.
In the first
case~$w_{J}=\sigma(w_{J})$ and hence
is contained in the image of~$(W(B_n),\star)$ in~$(W(D_{n+1}),\star)$ under the injective
parabolic homomorphism provided by~\eqref{eq:unfold Bn Dn+1}. Then the assertion
follows since the restriction of $\mathbf f_{\varpi(n,n+1)}$ to the image of~$(W(B_n),\star)$
coincides with the tautological homomorphism
$(W(B_n),\star)\to (W(A_n),\star)$ which
is parabolic by~\ref{taut-case-BnAn}.
Suppose now that~$J\cap \{n,n+1\}=\{n\}$.
Then $w_{J}=w_{J;[1,n]}\times
w_{[1,n]}$ by Lemma~\ref{lem:len prop wJ K} and, since
$\mathbf f_{\varpi(n,n+1)}(w_{J;[1,n]})
=w_{J;[1,n]}$ as the restriction of~$\mathbf f_{\varpi(n,n+1)}$ to~$W_{[1,n]}(D_{n+1})$ is just the
identity map, it suffices to prove that
$\mathbf f_{\varpi(n,n+1)}(w_{[1,n]})=w_{K;[1,n]}$
for some~$K\subset [1,n]$. Then~$\mathbf f_{\varpi(n,n+1)}(w_J)=w_{J;[1,n]}\star w_{K;[1,n]}=w_{J\star_{[1,n]}K;[1,n]}$.
By~\eqref{eq:w1,n D}, we have
$$
\mathbf f_{\varpi(n,n+1)}(w_{[1,n]})=
\ascprodst_{1\le i\le n} \cxr in=w_\circ^{[1,n]}=w_{\emptyset;[1,n]}.
$$

It remains to consider~$\mathbf f_{\varpi(1,3,4)}$.
If~$|J|<3$ then~$J\subset J_i:=[1,4]\setminus\{i\}$
for some~$i\in [1,4]$ and so~$w_J=w_{J;J_i}\times
w_{J_i}$ by Lemma~\ref{lem:len prop wJ K}. Since
the restriction of~$\mathbf f_{\varpi(1,3,4)}$
to~$W_{J_i}(D_4)$ with~$i\not=2$ is parabolic
by the previous case, while its restriction
to~$W_{\{1,3,4\}}(D_4)$ is obviously parabolic,
it suffices to consider~$J=J_i$. Using
the diagram automorphism corresponding
to the permutation $(1,3,4)$ of~$[1,4]$
we may assume, without loss of generality, that either~$J=[1,3]$ or~$J=\{1,3,4\}$.
By~\eqref{eq:w1,n D}, we have
$$
\mathbf f_{\varpi(1,3,4)}(w_{[1,3]})=
\mathbf f_{\varpi(1,3,4)}(s_4\times s_2\times s_1
\times s_3\times s_2\times s_4)=
(s_1\star s_2 \star s_1)^{\star 2}=w_\circ^{[1,2]}.
$$
Since~$w_\circ^{[1,4]}=(s_1\times s_3\times s_4\times s_2)^{\times 3}$, it follows that
$w_{\{1,3,4\}}=s_2\times (s_1\times s_3\times s_4\times s_2)^{\times 2}$ and so
$$
\mathbf f_{\varpi(1,3,4)}(w_{\{1,3,4\}})
=s_2\star s_1\star s_2\star s_1\star s_2=
w_\circ^{[1,2]}.
$$
This completes the proof of Theorem~\ref{thm:light->parab}.
\end{proof}

\subsection{Connection with Lie theory}\label{subs:Lie}
Let~$I$ be a finite set and let~$A=(a_{i,j})_{i,j\in I}$ be a (generalized) Cartan matrix over~$I$, that is $a_{i,i}=2$, $i\in I$,
$a_{i,j}\in\mathbb Z_{\le 0}$ if $i\not=j\in I$ and
$a_{i,j}=0$ implies that~$a_{j,i}=0$.
The Lie algebra~$\lie n(A)$ associated with~$A$
is generated by the $e_i$, $i\in I$ subject to
the Serre relations
$$
(\ad e_i)^{1-a_{i,j}}(e_j)=0,\qquad i\not=j\in I.
$$
We say that $M\in \Cox I$
is of {\em Weyl type} if $m_{ij}\in \{2,3,4,6,\infty\}$
for all~$i\not=j\in I$.
Given a Cartan matrix~$A$ over~$I$, define
\plink{C(A)}$\mathsf C(A)\in\Cox I$ of Weyl type via $\mathsf C(A)_{i,i}=1$
for all~$i\in I$ and
$$
\mathsf C(A)_{i,j}=\begin{cases}
2+a_{i,j}a_{j,i},& a_{i,j}a_{j,i}\le 1,\\
2a_{i,j}a_{j,i},& a_{i,j}a_{j,i}\in\{2,3\},\\
\infty,&a_{i,j}a_{j,i}>3
\end{cases}
$$
for all~$i\not=j\in I$.
\begin{lemma}\label{lem:Cartan->Cox}
Let~$M\in\Cox I$ be of Weyl type. Then~$M=\mathsf C(A)$
for some Cartan matrix~$A$ over~$I$.
\end{lemma}
\begin{proof}
Fix a total order on~$I$.
Set $a_{i,i}=2$, $i\in I$, $a_{i,j}=a_{j,i}=0$ if
$m_{i,j}=2$, $i\not=j\in I$. Finally, for all~$i<j$
with~$m_{i,j}>2$ set $a_{i,j}=-1$ and
$$
a_{j,i}=\begin{cases}
-\floor{\frac12m_{i,j}},& m_{i,j}\in\{3,4,6\},\\
-4,& m_{i,j}=\infty.
\end{cases}
$$
It follows from the definition that~$\mathsf C(A)=M$.
\end{proof}
\begin{remark}
If~$m_{ij}\le 3$ for all~$i,j\in I$ then the Cartan
matrix~$A$ such that~$M=\mathsf C(A)$ is unique. If
all entries of~$M$ are finite and~$\Gamma(M)$ is a forest, such a matrix is unique up to the choice of the total order on~$I$; in particular,
for~$M=B_n$, $M=F_4$ and~$M=G_2$, $A$ is unique up to the transpose, and for $M=F_4$ (respectively, $M=G_2$) the corresponding Lie algebras are isomorphic.
If~$\Gamma(M)$ contains a cycle,
then a choice of~$a_{i,j}<-1$ along the edges of the cycle need not be induced by a total order.
\end{remark}
\begin{theorem}\label{thm:Lie light}
Let~$M\in\Cox I$, $M'\in\Cox{I'}$
be of Weyl type and let
$\phi\in\Hom_{\He}(M',M)$ be light.
Then for any Cartan matrix~$A$ such that~$M=\mathsf C(A)$
there exists a Cartan matrix~$A'$
such that $M'=\mathsf C(A')$
and the assignments
$$
e_{i'}\mapsto \begin{cases}
e_i,& [\phi](i')=\{i\}\not=\emptyset,\\
0,& [\phi](i')=\emptyset,
\end{cases}
$$
$i'\in I'$,
define a homomorphism of Lie algebras
$\wh{[\phi]}:\lie n(A')\to \lie n(A)$.
\end{theorem}
\begin{proof}
We need the following
\begin{lemma}\label{lem:subord Lie hom}
Let~$I$, $I'$ be finite sets, let $f:I'\to \mathscr P(I)$
with~$|f(i)|\le 1$, $i\in I'$ and let~$A$ (respectively, $A'$) be Cartan matrices over~$I$ (respectively, $I'$).
Suppose that~$A'$ is $f$-subordinate to~$A$, that
is $a'_{i',j'}\le a_{i,j}$ whenever $f(i')=\{i\}$ and~$f(j')=\{j\}$ are non-empty, $i'\not=j'\in I'$.
Then the assignments
$$
e_{i'}\mapsto \begin{cases}
e_i,&f(i')=\{i\}\not=\emptyset,\\
0,&f(i')=\emptyset,
\end{cases}
$$
$i'\in I'$,
define a homomorphism of Lie algebras~$\wh f:\lie n(A')\to \lie n(A)$.
\end{lemma}
\begin{proof}
Let~$i'\not=j'\in I'$. Clearly, if at least one of $f(i')$, $f(j')$ is empty then the images of $e_{i'}$
and~$e_{j'}$ in~$\lie n(A)$ trivially satisfy the Serre relations. Suppose that~$f(i')=\{i\}$, $f(j')=\{j\}$ for some~$i,j\in I$.
If~$i=j$ then $[e_i,e_j]=0$ and so
$(\ad e_i)^{1-a'_{i',j'}}(e_j)=0=(\ad e_j)^{1-a'_{j',i'}}(e_i)$. Otherwise, we have
$$
(\ad e_i)^{1-a'_{i',j'}}(e_j)=(\ad e_i)^{a_{i,j}-a'_{i',j'}}((\ad e_i)^{1-a_{i,j}}(e_j))=0,
$$
and similarly with the role of~$i$ and~$j$ interchanged.
\end{proof}

\begin{lemma}\label{lem:Lie light}
Let~$M\in\Cox I$, $M'\in\Cox{I'}$ be of Weyl type and let
$\phi\in\Hom_{\He}(M',M)$ be a light homomorphism.
Let~$A$ be a Cartan matrix such that~$M=\mathsf C(A)$.
Then there exists a Cartan matrix~$A'$ such that
$M'=\mathsf C(A')$ and $A'$ is $[\phi]$-subordinate to $A$.
\end{lemma}
\begin{proof}
By Lemma~\ref{lem:Cartan->Cox}, there exists a Cartan matrix~$A''$ with~$\mathsf C(A'')=M'$. We obtain~$A'$ by modifying
entries of~$A''$ as follows.
Fix a total order on~$I'$ and let
$i'<j'\in I'$ be such that~$[\phi](i')=\{i\}$,
$[\phi](j')=\{j\}$ are non-empty and~$i\not=j$. Since~$\phi$ is a homomorphism of Hecke monoids, it follows that~$m'_{i',j'}\ge
m_{i,j}$.
If~$m'_{i',j'}=m_{i,j}$, we set $a'_{i',j'}=a_{i,j}$ and~$a'_{j',i'}=a_{j,i}$.
Suppose that~$m'_{i',j'}>m_{i,j}$.
If~$m_{i,j}\in\{2,3\}$,
let~$a'_{i',j'}=-1$ and let~$a'_{j',i'}=-1$
if $m'_{i',j'}=3$, $a'_{j',i'}=-\frac12 m'_{i',j'}$
if~$m'_{i',j'}\in\{4,6\}$ and~$a'_{j',i'}=-4$ if~$m'_{i',j'}=\infty$. If~$m_{i,j}\in\{4,6\}$
let~$a'_{i',j'}=a_{i,j}-\chi(i,j)$,
$a'_{j',i'}=a_{j,i}-\chi(j,i)$
where $\chi(s,t)=0$ if~$a_{s,t}=-1$, $\chi(s,t)=1$
if~$a_{s,t}<-1$ and~$m'_{i',j'}<\infty$ and~$\chi(s,t)=2$ otherwise.
All other entries of~$A'$ are those of~$A''$; note that $[\phi]$-subordination imposes no condition on them.
By construction, $A'$ is a Cartan matrix which is $[\phi]$-subordinate to~$A$, and~$M'=\mathsf C(A')$.
\end{proof}

\begin{remark}
It is not always possible to reverse this procedure. For example,
let~$M'=\left(\begin{smallmatrix}1&\infty\\\infty&1
\end{smallmatrix}\right)$, $M=I_2(6)$ and let $\phi\in\Hom_{\He}(M',M)$ be the tautological homomorphism. Then $A'=\left(\begin{smallmatrix}2&-2\\-2&2
\end{smallmatrix}\right)$ satisfies~$M'=\mathsf C(A')$.
Yet~$A$ such that~$M=\mathsf C(A)$ is either $\left(\begin{smallmatrix}2&-3\\-1&2
\end{smallmatrix}\right)$ or its transpose and
hence~$A'$ is not $[\phi]$-subordinate to~$A$.
\end{remark}
The assertion is an immediate consequence
of Lemmata~\ref{lem:Lie light} and~\ref{lem:subord Lie hom}.
\end{proof}

Exponentiating, we obtain a homomorphism of the corresponding unipotent (ind)groups $U(A')\to U(A)$.
If~$\lie n(A)$ and~$\lie n(A')$ are finite-dimensional,
this is a homomorphism of ordinary Lie groups. The study of such homomorphisms brought light homomorphisms of Hecke monoids to our attention and to the discovery of their remarkable property vis-\`a-vis parabolic elements (Theorem~\ref{thm:mthm I}).

\section{Locally injective
homomorphisms for classical series}\label{sec:loc inj Homs}

In this section we classify all connected (Definition~\ref{def:types heck hom})
locally injective (Definition~\ref{defn:locally inj})  homomorphisms for classical series. Since every~$\phi\in\Hom_{\He}(M',M)$ is the composition of a fully supported homomorphism
to~$(W_{[\phi](I')}(M),\star)$ with the natural (injective) inclusion of the latter into~$(W(M),\star)$,
in view of Lemma~\ref{lem:loc inj comp} we will only consider fully supported homomorphisms.

\subsection{From type~\texorpdfstring{$A$}{A}
to type~\texorpdfstring{$A$}{A}}\label{subs:A->A}
Given~$m,r\in\ZZ_{>0}$, let
\plink{Ar(m)}$\mathcal A_r(m)$ be the set of all integer
partitions of~$m$ with exactly~$r$ parts, such that
the largest $r-2$ parts are equal to each other, that is
$$
\mathcal A_r(m)=\{ (\lambda_1,\dots,\lambda_r)\in\ZZ_{>0}^r\,:\,
\lambda_1=\dots=\lambda_{r-2}\ge \lambda_{r-1}\ge \lambda_r,\,\sum_{1\le j\le r}\lambda_j=m\}.
$$
Clearly, $\mathcal A_r(m)=\emptyset$ if~$r>m$.
Also, for~$1\le r\le 3$,
$\mathcal A_r(m)$ is the set of partitions of~$m$ with exactly~$r$ parts. Thus, $|\mathcal A_2(m)|=\lfloor\frac12m\rfloor$,
while
$|\mathcal A_3(m)|$ is the nearest integer
to~$\frac1{12}m^2$ (the sequence~\OEIS{A001399} from~\cite{OEIS} up to the shift) by~\cite{Hardy}.
If~$r\ge 3$ then, using transposition
of partitions, it is easy to see that~$\mathcal A_r(m)$ is
in bijection with the set of all partitions of~$m$ with
parts equal to $r$, $r-1$ and~$r-2$ and the largest part being~$r$, whence
$$
\sum_{m\ge 0} |\mathcal A_r(m)|x^m=\frac{x^r}{(1-x^r)(1-x^{r-1})(1-x^{r-2})},\qquad r\ge 3.
$$

Given~$\boldsymbol\lambda=(\lambda_1,\dots,\lambda_r)\in\ZZ^r$, define
\plink{a_i(l)}$a_i(\boldsymbol\lambda)$, $i\in[0,r]$ by
$$
a_0(\boldsymbol\lambda)=0,\qquad
a_i(\boldsymbol\lambda)=\lambda_r+\sum_{1\le j\le i-1}\lambda_j,\qquad i\in[1,r].
$$
In particular, $a_r(\boldsymbol{\lambda})=\sum_{1\le i\le r}\lambda_i$.
Furthermore, if~$\boldsymbol{\lambda}=(\lambda_1,\dots,\lambda_r)\in \ZZ_{\ge 0}^r$ we set\plink{J_i(l)}
$$
J_i(\boldsymbol{\lambda}):=[a_{i-1}(\boldsymbol{\lambda})+1,a_{i+1}(\boldsymbol{\lambda})-1]\subset [1,a_r(\boldsymbol{\lambda})-1],\qquad i\in[1,r-1].
$$
Thus, $|J_i(\boldsymbol\lambda)|=\lambda_i+\lambda_{i-1}-1$, $i\in[1,r-1]$, where we set~$\lambda_0=\lambda_r$.
We will need the following combinatorial fact.
\begin{lemma}\label{lem:diag aut J_i}
Let~$\boldsymbol\lambda=(\lambda_1,\dots,\lambda_r)\in \mathbb Z_{\ge0}^r$. Then~$a_r(\boldsymbol\lambda)-J_{r-i}(\boldsymbol\lambda)=J_i(\tilde{\boldsymbol\lambda})$ for all~$i\in[1,r-1]$, where
\plink{tilde l}$\tilde{\boldsymbol\lambda}=(\lambda_{r-2},\dots,\lambda_1,\lambda_r,\lambda_{r-1})$.
\end{lemma}
\begin{proof}
Write~$\tilde{\boldsymbol\lambda}=
(\tilde\lambda_1,\dots,\tilde\lambda_r)$. Thus, $\tilde\lambda_i=\lambda_{r-i-1}$,
$i\in [1,r-2]$, $\tilde\lambda_{r-1}=\lambda_r$ and~$\tilde\lambda_r=\lambda_{r-1}$.
Clearly, $a_0(\tilde{\boldsymbol\lambda})=0
=a_r(\boldsymbol\lambda)-a_r(\boldsymbol\lambda)$ and
$a_r(\tilde{\boldsymbol\lambda})=a_r(\boldsymbol{\lambda})$, while,
for all~$i\in[1,r-1]$,
\begin{align*}
a_i(\tilde{\boldsymbol\lambda})&=
\sum_{1\le j\le i-1}\tilde\lambda_j+\tilde\lambda_r=
\sum_{1\le j\le i-1}\lambda_{r-j-1}+\lambda_{r-1}=\sum_{r-i\le j\le r-1}
\lambda_j
=a_r(\boldsymbol\lambda)-\lambda_r-
\sum_{1\le j\le r-i-1}\lambda_j\\
&=a_r(\boldsymbol\lambda)-a_{r-i}(\boldsymbol\lambda).
\end{align*}
Then~$J_i(\tilde{\boldsymbol\lambda})=
[a_{i-1}(\tilde{\boldsymbol\lambda})+1,
a_{i+1}(\tilde{\boldsymbol\lambda})-1]
=[a_r(\boldsymbol\lambda)-a_{r-i+1}(\boldsymbol\lambda)+1,
a_r(\boldsymbol\lambda)-a_{r-i-1}(\boldsymbol\lambda)-1]
=a_r(\boldsymbol\lambda)-J_{r-i}(\boldsymbol\lambda)$.
\end{proof}

The main result of this section is the following
\begin{theorem}\label{thm:Ak to An}
For any~$\boldsymbol\lambda\in \mathcal A_{k+1}(n+1)$,
the assignments $s'_i\mapsto w_\circ^{J_i(\boldsymbol\lambda)}$, $i\in [1,k]$, define
a locally injective  connected \plink{psi l}$\psi_{\boldsymbol\lambda}
\in\Hom_{\He}(A_k,A_n)$.
Moreover, if~$\phi\in\Hom_{\He}(A_k,A_n)$
is locally injective, connected and fully supported then, up to compositions with diagram automorphisms, $\phi=\psi_{\boldsymbol{\lambda}}$
for some~$\boldsymbol{\lambda}\in\mathcal A_{k+1}(n+1)$.
\end{theorem}
\begin{proof}
The case~$k=2$, which is the most important step in this argument,
is established in the following
\begin{proposition}\label{prop:A2->An}
For any~$n\ge 2$ and $\boldsymbol\lambda=(\lambda_1,\lambda_2,\lambda_3)\in\mathcal A_3(n+1)$ the assignments~$s'_i\mapsto w_\circ^{J_i(\boldsymbol\lambda)}$,
$i\in\{1,2\}$,
define injective~$\psi_{\boldsymbol\lambda}\in\Hom_{\He}(A_2,A_n)$,
which is parabolic if and only if~$\boldsymbol\lambda=(n-1,1,1)$.
Conversely, if~$\phi\in\Hom_{\He}(A_2,A_n)$ is injective, connected and fully supported then, up to compositions with diagram automorphisms, $\phi=\psi_{\boldsymbol\lambda}$ for
some~$\boldsymbol{\lambda}\in\mathcal A_3(n+1)$.
\end{proposition}
\begin{proof}
First, we need to establish two technical results.
\begin{lemma}\label{lem: w [1,k]*w[l,n]}
For any~$\boldsymbol\nu=
(\nu_1,\nu_2,\nu_3)\in
\ZZ_{\ge0}\times\ZZ_{>0}\times\ZZ_{>0}$
\begin{align}
w_\circ^{J_1(\boldsymbol\nu)}\star w_\circ^{J_2(\boldsymbol\nu)}
&=w_{[\nu_1+1,\nu_1+\nu_2+\nu_3-1]\setminus\{\nu_1+\nu_3\};[\nu_1+1,\nu_1+\nu_2+\nu_3-1]}w_\circ^{[1,\nu_1+\nu_2+\nu_3-1]}
\label{eq: w [1,k]*w[l,n]}
\\
&=w_\circ^{[1,\nu_1+\nu_2+\nu_3-1]}w_{[1,\nu_2+\nu_3-1]\setminus\{\nu_3\};[1,\nu_2+\nu_3-1]}{}^{-1},
\label{eq: w [1,k]*w[l,n] II}
\end{align}
\end{lemma}

\begin{proof}
Abbreviate~$m=\nu_1+\nu_2+\nu_3-1$.
It suffices to prove~\eqref{eq: w [1,k]*w[l,n]}.
Then~\eqref{eq: w [1,k]*w[l,n] II}
follows from~\eqref{eq: w [1,k]*w[l,n]} since
$w w_\circ^{[1,r]}=w_\circ^{[1,r]}\sigma(w)$
for any~$w\in W(A_r)$ where~$\sigma$ is the diagram
automorphism of~$W(A_r)$.

Since~$J_1(\boldsymbol\nu)\cap
J_2(\boldsymbol\nu)=[1,\nu_1+\nu_3-1]
\cap [\nu_3+1,m]=[\nu_3+1,\nu_1+\nu_3-1]$,
by Lemma~\ref{lem:w0J*w0K}
\begin{equation}\label{eq:star to perm step1}
w_\circ^{J_1(\boldsymbol{\nu})}\star
w_\circ^{J_2(\boldsymbol{\nu})}
=
w_\circ^{[1,\nu_1+\nu_3-1]}
w_\circ^{[\nu_3+1,\nu_1+\nu_3-1]}
w_\circ^{[\nu_3+1,m]}.
\end{equation}
Identify~$W(A_m)$ with the symmetric
group~$S_{m+1}$ acting on~$[1,m+1]$ where $s_i$ is mapped to the transposition interchanging~$i$ and~$i+1$, $i\in[1,m]$, and permutations are multiplied as maps,
$(\pi\pi')(x)=\pi(\pi'(x))$. Write~$\pi\in S_{m+1}$ as $(\pi(1),\dots,\pi(m+1))$. Then~$w_\circ^{[a,b]}$, $1\le a\le b\le m$, identifies
with~$(1,\dots,a-1,b+1,b,\dots,a,b+2,\dots,m+1)$,
while~$w_\circ^{[a,b]}=1$ if~$a>b$. With these conventions,
the right-hand side of~\eqref{eq:star to perm step1} is mapped to the permutation
$$
(\nu_1+\nu_3,\dots,\nu_1+1,m+1,m,\dots,\nu_1+\nu_3+1,\nu_1,\dots,1).
$$
On the other hand, the right-hand side of~\eqref{eq: w [1,k]*w[l,n]}
is equal to
$$
w_\circ^{[\nu_1+1,\nu_1+\nu_3-1]}
w_\circ^{[\nu_1+\nu_3+1,m]}w_\circ^{[\nu_1+1,m]}w_\circ^{[1,m]}.
$$
As permutations,
$$
w_\circ^{[1,m]}=(m+1,\dots,1),\qquad
w_\circ^{[\nu_1+1,m]}w_\circ^{[1,m]}=(\nu_1+1,\dots,m+1,\nu_1,\dots,1),
$$
and the subsequent left multiplication
by~$w_\circ^{[\nu_1+1,\nu_1+\nu_3-1]}
w_\circ^{[\nu_1+\nu_3+1,m]}$, which reverses the values in~$[\nu_1+1,\nu_1+\nu_3]$ and in~$[\nu_1+\nu_3+1,m+1]$, yields
the same permutation as above. This proves~\eqref{eq: w [1,k]*w[l,n]}.
\end{proof}

\begin{lemma}\label{lem:mu J1 J2}
Let~$\boldsymbol\lambda=(\lambda_1,\lambda_2,\lambda_3)\in\ZZ_{>0}^3$
and let~$m=\lambda_1+\lambda_2+\lambda_3-1$. Then
\begin{enmalph}
\item \label{lem:mu J1 J2.a}
$\min(\mu_{A_m}(J_1(\boldsymbol\lambda),J_2(\boldsymbol\lambda)),\mu_{A_m}(J_2(\boldsymbol\lambda),J_1(\boldsymbol\lambda)))\ge 3$;
\item \label{lem:mu J1 J2.b}
$\mu_{A_m}(J_1(\boldsymbol\lambda),J_2(\boldsymbol\lambda))
=3$
(respectively,
$\mu_{A_m}(J_2(\boldsymbol\lambda),J_1(\boldsymbol\lambda))
=3$) if and only if~$\lambda_1\ge \lambda_2$
(respectively, $\lambda_1\ge \lambda_3$).
\item\label{lem:mu J1 J2.c}
If~$\lambda_1<\min(\lambda_2,\lambda_3)$
then
$\min(\mu_{A_m}(J_1(\boldsymbol\lambda),J_2(\boldsymbol\lambda)),\mu_{A_m}(J_2(\boldsymbol\lambda),J_1(\boldsymbol\lambda)))>4$.
\end{enmalph}
\end{lemma}
\begin{proof}
By~\eqref{eq: w [1,k]*w[l,n]},
$\ell(w_\circ^{J_1(\boldsymbol{\lambda})}\star
w_\circ^{J_2(\boldsymbol{\lambda})})
=\ell(w_\circ^{[1,m]})-\lambda_2\lambda_3<
\ell(w_\circ^{[1,m]})$,
whence
$$\min(\mu_{A_m}(J_1(\boldsymbol{\lambda}),
J_2(\boldsymbol{\lambda})),
\mu_{A_m}(J_2(\boldsymbol{\lambda}),
J_1(\boldsymbol{\lambda})))> 2.
$$
If~$\lambda_1\ge \lambda_2$ then
$
w_\circ^{J_1(\boldsymbol\lambda)}\star w_\circ^{J_2(\boldsymbol{\lambda})}\star w_\circ^{J_1(\boldsymbol\lambda)}=w_\circ^{[1,m]}$  by~\eqref{eq: w [1,k]*w[l,n]} and Lemma~\partref{lem: u * u^(-1) w_0.c},
whence $\mu_{A_m}(J_1(\boldsymbol{\lambda}),
J_2(\boldsymbol{\lambda}))=3$.

Suppose
that~$\lambda_1<\lambda_2$.
Then by Lemma~\partref{lem: u * u^(-1) w_0.b} with~$u=w_\circ^{[1,m-\lambda_1]
\setminus\{\lambda_3\}}$ and~$v=w_\circ^{[1,m-\lambda_1]}$ as well as Lemma~\ref{lem:char w_0 monoid}
\begin{align*}
w_\circ^{J_1(\boldsymbol\lambda)}\star w_\circ^{J_2(\boldsymbol\lambda)}\star w_\circ^{J_1(\boldsymbol\lambda)}
&=(w_\circ^{[1,m]}w_{[1,m-\lambda_1]\setminus\{\lambda_3\};[1,m-\lambda_1]}{}^{-1})
\star w_\circ^{[1,\lambda_1+\lambda_3-1]}\\
&=
(w_\circ^{[1,m]}w_\circ^{[1,m-\lambda_1]})
\star w_\circ^{[1,\lambda_3-1]}\star
w_\circ^{[\lambda_3+1,\lambda_2+\lambda_3-1]}
\star w_\circ^{[1,\lambda_1+\lambda_3-1]}
\\
&=(w_\circ^{[1,m]}w_\circ^{[1,m-\lambda_1]})
\star
w_\circ^{[\lambda_3+1,\lambda_2+\lambda_3-1]}
\star w_\circ^{[1,\lambda_1+\lambda_3-1]}.
\end{align*}
Using~\eqref{eq: w [1,k]*w[l,n]} with
$\boldsymbol{\nu}=(\lambda_1,\lambda_2-\lambda_1,\lambda_3)$ and applying~${}^{op}$
and Lemmata~\partref{lem: u * u^(-1) w_0.b} and~\ref{lem:len prop wJ K}, we
obtain
\begin{align}
w_\circ^{J_1(\boldsymbol\lambda)}&\star w_\circ^{J_2(\boldsymbol\lambda)}\star
w_\circ^{J_1(\boldsymbol\lambda)}\nonumber
\\
&=
(w_\circ^{[1,m]}w_\circ^{[1,\lambda_2+\lambda_3-1]})
\star (
w_\circ^{[1,\lambda_2+\lambda_3-1]}w_{[\lambda_1+1,\lambda_2+\lambda_3-1]\setminus\{\lambda_1+\lambda_3\};[\lambda_1+1,\lambda_2+\lambda_3-1]}{}^{-1})\nonumber\\
&=w_{[1,\lambda_2+\lambda_3-1];[1,m]}{}^{-1}
\star w_{
[\lambda_1+1,\lambda_2+\lambda_3-1];[1,\lambda_2+\lambda_3-1]}{}^{-1}
\star w_\circ^{[\lambda_1+1,\lambda_2+\lambda_3-1]\setminus\{\lambda_1+\lambda_3\}}\nonumber\\
&=
w_\circ^{[1,m]}w_{[\lambda_1+1,\lambda_2+\lambda_3-1]\setminus\{\lambda_1+\lambda_3\};[\lambda_1+1,\lambda_2+\lambda_3-1]}{}^{-1}.
\label{eq:1st brd Am}
\end{align}
Since~$\lambda_1<\lambda_2$ and~$\lambda_3>0$
it follows that~$\lambda_1+1\le \lambda_1+\lambda_3\le\lambda_2+\lambda_3-1$ and so
$\ell(w_\circ^{J_1(\boldsymbol\lambda)}
\star w_\circ^{J_2(\boldsymbol\lambda)}
\star w_\circ^{J_1(\boldsymbol\lambda)})<
\ell(w_\circ^{[1,m]})$, that is
$\mu_{A_m}(J_1(\boldsymbol{\lambda}),
J_2(\boldsymbol\lambda))>3$.
This proves
the first assertion in part~\ref{lem:mu J1 J2.b}.
The second follows from the first applied to~$\tilde{\boldsymbol\lambda}=(\lambda_1,\lambda_3,\lambda_2)$,
since~$m+1-J_i(\boldsymbol\lambda)=J_{3-i}(\tilde{\boldsymbol\lambda})$, $i\in\{1,2\}$, by Lemma~\ref{lem:diag aut J_i}, and so
$\mu_{A_m}(J_2(\boldsymbol\lambda),J_1(\boldsymbol\lambda))=\mu_{A_m}(J_1(\tilde{\boldsymbol\lambda}),J_2(\tilde{\boldsymbol\lambda}))$
by applying the diagram automorphism of~$A_m$.

Finally, if~$\lambda_1<\min(\lambda_2,\lambda_3)$,
we obtain, using~\eqref{eq: w [1,k]*w[l,n]},
\eqref{eq: w [1,k]*w[l,n] II}, Lemma~\partref{lem: u * u^(-1) w_0.b}
and Corollary~\ref{cor:A J*K}
\begin{align*}
(w_\circ^{J_1(\boldsymbol\lambda)}&\star w_\circ^{J_2(\boldsymbol\lambda)})^{\star 2}
=w_\circ^{[\lambda_1+1,m]\setminus\{\lambda_1+\lambda_3\}}\star (w_\circ^{[\lambda_1+1,m]}w_\circ^{[1,m]})
\star (w_\circ^{[1,m]}w_\circ^{[1,m-\lambda_1]})
\star w_\circ^{[1,m-\lambda_1]\setminus\{\lambda_3\}}\\
&=w_\circ^{[\lambda_1+1,m]\setminus\{\lambda_1+\lambda_3\}}\star w_{[\lambda_1+1,m];[1,m]}\star w_{[\lambda_1+1,m];[1,m]}
\star w_\circ^{[1,m-\lambda_1]\setminus\{\lambda_3\}}\\
&=w_\circ^{[\lambda_1+1,\lambda_1+\lambda_3-1]}
\star w_\circ^{[\lambda_1+\lambda_3+1,m]}\star w_{[2\lambda_1+1,m];[1,m]}
\star w_\circ^{[1,\lambda_3-1]}
\star w_\circ^{[\lambda_3+1,m-\lambda_1]}.
\end{align*}
Since~$w_{[2\lambda_1+1,m];[1,m]}=w_{[1,\lambda_2-\lambda_1+\lambda_3-1];[1,m]}{}^{-1}$, by Lemma~\partref{lem:wK absorption.a} we have
$D_R(w_{[2\lambda_1+1,m];[1,m]})\cap[1,\lambda_3-1]=([1,m]\setminus[1,\lambda_2-\lambda_1+\lambda_3-1])\cap[1,\lambda_3-1]=\emptyset$
and~$D_L(w_{[2\lambda_1+1,m];[1,m]})\cap[2\lambda_1+1,2\lambda_1+\lambda_3-1]=[1,2\lambda_1]\cap[2\lambda_1+1,2\lambda_1+\lambda_3-1]=\emptyset$.
Therefore, by Corollary~\ref{cor:descent w0J},
\begin{align*}
w_{[2\lambda_1+1,m];[1,m]}\star w_\circ^{[1,\lambda_3-1]}
&=w_\circ^{[1,m]}w_\circ^{[1,\lambda_2-\lambda_1+\lambda_3-1]}w_\circ^{[1,\lambda_3-1]}
=w_\circ^{[2\lambda_1+1,2\lambda_1+\lambda_3-1]}w_\circ^{[2\lambda_1+1,m]}w_\circ^{[1,m]}\\
&=w_\circ^{[2\lambda_1+1,2\lambda_1+\lambda_3-1]}\star w_{[2\lambda_1+1,m];[1,m]}.
\end{align*}
Using~\eqref{eq: w [1,k]*w[l,n] II}
with~$\boldsymbol\nu=(\lambda_1,\lambda_2-\lambda_1,\lambda_3-\lambda_1)$, we obtain
\begin{align*}
u:&=w_\circ^{[\lambda_1+\lambda_3+1,m]}\star w_{[2\lambda_1+1,m];[1,m]}
\star w_\circ^{[1,\lambda_3-1]}\\
&=w_\circ^{[\lambda_1+\lambda_3+1,m]}\star
w_\circ^{[2\lambda_1+1,2\lambda_1+\lambda_3-1]}
\star w_{[2\lambda_1+1,m];[1,m]}\\
&=(w_{[2\lambda_1+1,\lambda_2+\lambda_3-1]\setminus\{\lambda_1+\lambda_3\};
[2\lambda_1+1,\lambda_2+\lambda_3-1]}
w_\circ^{[2\lambda_1+1,m]})\star w_{[2\lambda_1+1,m];[1,m]}
\\
&=w_\circ^{[2\lambda_1+1,\lambda_2+\lambda_3-1]\setminus\{\lambda_1+\lambda_3\}}
\star w_{[2\lambda_1+1,\lambda_2+\lambda_3-1];
[2\lambda_1+1,m]}
\star w_{[2\lambda_1+1,m];[1,m]}\\
&=w_\circ^{[2\lambda_1+1,\lambda_2+\lambda_3-1]\setminus\{\lambda_1+\lambda_3\}}
\star w_{[2\lambda_1+1,\lambda_2+\lambda_3-1];[1,m]}\\
&=w_{[2\lambda_1+1,\lambda_2+\lambda_3-1]\setminus\{\lambda_1+\lambda_3\};
[2\lambda_1+1,\lambda_2+\lambda_3-1]}w_\circ^{[1,m]}
\\
&=w_\circ^{[1,m]}
w_{[\lambda_1+1,\lambda_2-\lambda_1+\lambda_3-1]\setminus\{\lambda_3\};[\lambda_1+1,\lambda_2-\lambda_1+\lambda_3-1]}{}^{-1},
\end{align*}
where we used Lemmata~\partref{lem: u * u^(-1) w_0.b} and~\ref{lem:len prop wJ K}.
Since~$D_L(u)=[1,m]\setminus\{\lambda_1+\lambda_3\}$ and~$D_R(u)=[1,m]\setminus \{\lambda_3\}$ by
Lemmata~\ref{lem:DL uw0} and~\ref{lem:wK absorption},
it follows that
\begin{align*}
(w_\circ^{J_1(\boldsymbol\lambda)}&\star w_\circ^{J_2(\boldsymbol\lambda)})^{\star 2}
=w_\circ^{[\lambda_1+1,\lambda_1+\lambda_3-1]}
\star u
\star w_\circ^{[\lambda_3+1,m-\lambda_1]}=u.
\end{align*}
In particular, since~$\lambda_1+1\le \lambda_3\le \lambda_2-\lambda_1+\lambda_3-1$,
it follows that~$\ell(u)=\ell(u^{-1})<\ell(w_\circ^{[1,m]})$ and so
$\min(\mu_{A_m}(J_1(\boldsymbol\lambda),J_2(\boldsymbol{\lambda})),\mu_{A_m}(J_2(\boldsymbol\lambda),J_1(\boldsymbol{\lambda})))>4$.
\end{proof}

By Lemma~\ref{lem:mu J1 J2} and Theorem~\ref{thm:Hom Heck Mon},
if~$\boldsymbol\lambda\in \mathcal A_3(n+1)$
then the assignments~$s'_i\mapsto w_\circ^{J_i(\boldsymbol{\lambda})}$, $i\in\{1,2\}$, define a homomorphism~$\psi_{\boldsymbol{\lambda}}\in
\Hom_{\He}(A_2,A_n)$, which is injective
by Proposition~\ref{prop:locally inj}.
By Lemma~\ref{lem: w [1,k]*w[l,n]} it is parabolic if and only if~$\lambda_1+1=\lambda_1+\lambda_3=\lambda_1+\lambda_2+\lambda_3-1=n$, which yields~$\boldsymbol\lambda=(n-1,1,1)$. Conversely, let~$\phi\in\Hom_{\He}(A_2,A_n)$ be injective and fully supported. Applying the diagram automorphism of~$A_2$ if necessary, we may assume that~$1\in[\phi](1)$. Then~$n\notin[\phi](1)$ since otherwise~$[\phi](2)\subset[\phi](1)$ which contradicts the injectivity by Proposition~\ref{prop:locally inj}. Thus, we can write~$[\phi](i)=J_i(\boldsymbol\lambda)$, $i\in\{1,2\}$, for some
$(\lambda_1,\lambda_2,\lambda_3)\in\ZZ_{\ge 0}^3$ with~$\lambda_1+\lambda_2+\lambda_3=n+1$. Since~$[\phi](1)\cup [\phi](2)=[1,n]$,
$\min [\phi](2)-\max[\phi](1)\le 1$ and so~$\lambda_1\ge 0$. Also, $\lambda_2>0$ and~$\lambda_3>0$ for
otherwise one of~$[\phi](1)$, $[\phi](2)$ is
a subset of the other and we get a contradiction by Proposition~\ref{prop:locally inj}. Thus,~$\boldsymbol\lambda\in \ZZ_{>0}^3$.
Applying the diagram automorphism of~$A_n$ if necessary we may assume, without loss of generality, that $n-\lambda_2=|J_1(\boldsymbol\lambda)|\le
|J_2(\boldsymbol\lambda)|=n-\lambda_3$, that is
$\lambda_2\ge \lambda_3$. It follows from
Lemma~\ref{lem:mu J1 J2},
Theorem~\ref{thm:Hom Heck Mon}
that~$\lambda_1\ge\max(\lambda_2,\lambda_3)$.
Thus, $\boldsymbol\lambda\in \mathcal A_3(n+1)$.
\end{proof}
We now proceed by induction on~$k\ge 2$, the induction
base being Proposition~\ref{prop:A2->An}.
For the inductive step, let~$\boldsymbol\lambda\in\mathcal A_{k+1}(n+1)$.
Then~$\boldsymbol\lambda'=(\lambda_1,\dots,
\lambda_{k-2},\lambda_{k-1},\lambda_{k+1})\in
\mathcal A_{k}(n+1-\lambda_k)$
and so the assignments~$s'_i\mapsto
w_\circ^{J_i(\boldsymbol{\lambda'})}$, $i\in[1,k-1]$ define a locally injective connected~$\psi_{\boldsymbol{\lambda'}}\in
\Hom_{\He}(A_{k-1},A_{n-\lambda_k})$.
Clearly,
$J_i(\boldsymbol{\lambda'})=
J_i(\boldsymbol{\lambda})$, $i\in[1,k-1]$.
In particular,
$$
\mu_{A_n}(J_i(\boldsymbol{\lambda}),
J_l(\boldsymbol{\lambda}))=
\begin{cases}
2,& |i-l|\ge2,\\
3,& |i-l|=1.
\end{cases}
$$
Furthermore, since~$\min J_k(\boldsymbol{\lambda})=a_{k-1}(\boldsymbol{\lambda})+1=
a_{i+1}(\boldsymbol\lambda)+\sum\limits_{i+1\le j\le k-2}\lambda_j+1\ge \max J_i(\boldsymbol{\lambda})+2$, $i\in[1,k-2]$,
it follows that all the~$J_i(\boldsymbol{\lambda})$, $i\in[1,k-2]$ are orthogonal to~$J_k(\boldsymbol\lambda)$ whence
by Corollary~\ref{cor:hom A1 x A1}
$$
\mu_{A_n}(J_i(\boldsymbol{\lambda}),
J_k(\boldsymbol{\lambda}))=\mu_{A_n}(J_k(\boldsymbol{\lambda}),
J_i(\boldsymbol{\lambda}))=2,\qquad i\in[1,k-2].
$$
Finally, write~$J_{k-1}(\boldsymbol{\lambda})
=a_{k-2}(\boldsymbol{\lambda})+[1,\lambda_{k-2}+
\lambda_{k-1}-1]$
and~$J_k(\boldsymbol{\lambda})=
a_{k-2}(\boldsymbol{\lambda})+[\lambda_{k-2}+1,
\lambda_{k-2}+\lambda_{k-1}+\lambda_k-1]$.
Thus, $J_{k-2+i}(\boldsymbol{\lambda})
=a_{k-2}(\boldsymbol{\lambda})+J_i(\boldsymbol{\nu})$, $i\in \{1,2\}$ where~$\boldsymbol{\nu}=(\lambda_{k-1},\lambda_k,\lambda_{k-2})
=(\lambda_{k-2},\lambda_k,\lambda_{k-2})$.
By Lemma~\ref{lem: w [1,k]*w[l,n]},
$$
\mu_{A_{\lambda_{k-2}+\lambda_{k-1}+\lambda_k-1}}(
J_1(\boldsymbol{\nu}),J_2(\boldsymbol{\nu}))=
\mu_{A_{\lambda_{k-2}+\lambda_{k-1}+\lambda_k-1}}(
J_2(\boldsymbol{\nu}),J_1(\boldsymbol{\nu}))=3.
$$
Since~$(W_{[a_{k-2}(\boldsymbol{\lambda})+1,n]}(A_n),\star)\cong (W(A_{\lambda_{k-2}+\lambda_{k-1}+\lambda_k-1}),\star)$ via
$s_i\mapsto s_{i-a_{k-2}(\boldsymbol{\lambda})}$,
$i\in [a_{k-2}(\boldsymbol{\lambda})+1,n]$, it follows that $
\mu_{A_n}(J_{k-1}(\boldsymbol{\lambda}),
J_k(\boldsymbol{\lambda}))=\mu_{A_n}(J_{k}(\boldsymbol{\lambda}),
J_{k-1}(\boldsymbol{\lambda}))=3$.
It remains to apply Theorem~\ref{thm:Hom Heck Mon} and Proposition~\ref{prop:locally inj}.

To prove the converse, we also use induction on~$k$,
the induction base being Proposition~\ref{prop:A2->An}. For the inductive step, using the diagram automorphism
of~$(W(A_k),\star)$
if necessary we may assume, without loss of generality, that~$1\in[\phi](1)$. Since the restriction
of~$\phi$ to~$(W_{[1,k-1]}(A_k),\star)
\cong (W(A_{k-1}),\star)$
is also a locally injective and connected
homomorphism in~$\Hom_{\He}(A_{k-1},A_m)$
for some~$m\le n$,
by induction hypothesis
$[\phi](i)=J_i(\boldsymbol\lambda)$, $i\in[1,k-1]$ for
some~$\boldsymbol\lambda=(\lambda_1,\dots,\lambda_{k})\in\ZZ_{>0}^{k}$
which is in~$\mathcal A_{k}(m+1)$ up to reordering of~$\lambda_{k}$ and~$\lambda_{k-1}$; that permutation accounts for the diagram automorphism of~$W(A_m)$.
Write~$[\phi](k)=[x,y]$, $1\le x\le y\le n$.
By Proposition~\ref{prop:locally inj} and
Lemma~\ref{lem: w [1,k]*w[l,n]}, $[\phi](k)$
must be orthogonal to~$[\phi](i)$ for all~$i\in[1,k-2]$ and hence to~$[\phi]([1,k-2])
=[1,a_{k-1}(\boldsymbol\lambda)-1]$,
whence~$x\ge a_{k-1}(\boldsymbol\lambda)+1$. If~$m=n$
then~$[\phi](k)\subset [a_{k-1}(\boldsymbol{\lambda})+1,n]\subset
J_{k-1}(\boldsymbol\lambda)=[\phi](k-1)$,
which is a contradiction by Proposition~\ref{prop:locally inj}. Thus,
$m<n$ and~$y=n$. We now consider the restriction of~$\phi$ to~$(W_{[k-1,k]}(A_k),\star)
\cong (W(A_2),\star)$. Since~$J_{k-1}(\boldsymbol\lambda)=a_{k-2}(\boldsymbol\lambda)+
[1,\lambda_{k-2}+\lambda_{k-1}-1]$,
write~$[x,n]=a_{k-2}(\boldsymbol\lambda)+[x'+1,y'-1]$ where, since~$n>m=a_k(\boldsymbol\lambda)-1$, $y'>\lambda_{k-2}+\lambda_{k-1}$, while,
as~$x\ge a_{k-1}(\boldsymbol\lambda)+1$,
$x'\ge \lambda_{k-2}$. In particular, $(W_{[a_{k-2}(\boldsymbol\lambda)+1,n]}(A_n),\star)\cong (W(A_{y'-1}),\star)$. Since~$|J_1(\boldsymbol{\lambda})|=\lambda_1+\lambda_k-1=\lambda_{k-2}+\lambda_k-1$
and~$|[\phi](k)|=y'-x'-1$, by applying the diagram automorphism of~$A_n$ if necessary we may assume that~$\lambda_{k-2}+\lambda_k\le y'-x'$.
Write, as in Proposition~\ref{prop:A2->An},
$[1,\lambda_{k-2}+\lambda_{k-1}-1]=J_1(\boldsymbol{\nu})$,
$[x'+1,y'-1]=J_2(\boldsymbol{\nu})$, where
$\boldsymbol{\nu}=(\nu_1,\nu_2,\nu_3)=(\lambda_{k-2}+\lambda_{k-1}-x',y'-\lambda_{k-1}-\lambda_{k-2},x')$.
By Proposition~\ref{prop:locally inj}, $\mu_{A_{y'-1}}(J_i(\boldsymbol{\nu}),J_{3-i}(\boldsymbol{\nu}))
=\mu_{A_n}([\phi](k+i-2),[\phi](k-i+1))=3$,
$i\in\{1,2\}$. Therefore,
by Lemma~\ref{lem:mu J1 J2},
$(\nu_1,\nu_2,\nu_3)\in\ZZ_{>0}^3$
and~$\nu_1\ge \max(\nu_2,\nu_3)$. The inequality~$\nu_1\ge \nu_3$ then yields~$\lambda_{k-2}\le x'\le \lambda_{k-1}$ which, since~$\lambda_{k-1}\le \lambda_{k-2}$,
forces~$\lambda_{k-1}=\lambda_{k-2}=x'$.
Then~$\lambda_{k-1}=\nu_1\ge \nu_2=y'-2\lambda_{k-1}$ which, together with~$y'>\lambda_{k-1}+\lambda_{k-2}=2\lambda_{k-1}$,
implies that~$2\lambda_{k-1}<y'\le 3\lambda_{k-1}$ and so~$y'=\lambda_{k-2}+\lambda_{k-1}+\mu$
with~$0<\mu\le \lambda_{k-1}$.
Finally, as $\lambda_{k-2}+\lambda_k\le y'-x'=\lambda_{k-2}+\mu$, it follows that~$\mu\ge \lambda_k$. Therefore,
$\boldsymbol\mu=(\lambda_1,\dots,\lambda_{k-1},\mu,\lambda_k)
\in\mathcal A_{k+1}(n+1)$,
$[\phi](i)=J_i(\boldsymbol{\mu})=J_i(\boldsymbol{\lambda})$, $i\in[1,k-1]$ while~$[\phi](k)=[a_{k-2}(\boldsymbol\lambda)+\lambda_{k-2}+1,
a_{k-2}(\boldsymbol{\lambda})+\lambda_{k-2}+\lambda_{k-1}+\mu-1]=
J_k(\boldsymbol\mu)$. Therefore, $\phi=\psi_{\boldsymbol{\mu}}$.
\end{proof}
\begin{remark}\label{rem:u JK 1}\label{rem:B2 A2n-1 parab noninj}
Let~$\boldsymbol{\lambda}=(\lambda_1,\lambda_2,\lambda_3)\in\ZZ_{>0}^3$ and set~$n=\lambda_1+\lambda_2+\lambda_3-1$.
By the proof of Lemma~\ref{lem:mu J1 J2},
if~$\lambda_3\le \lambda_1<\lambda_2$ or
$\lambda_2\le \lambda_1<\lambda_3$ then
the assignments~$s'_i\mapsto w_\circ^{J_i(\boldsymbol{\lambda})}$, $i\in\{1,2\}$ define a non-injective homomorphism
in~$\Hom_{\He}(B_2,A_n)$ which is
parabolic if and only if~$n=2l-1$, $l>1$, and~$\boldsymbol\lambda
\in\{(l-1,l,1),(l-1,1,l)\}$. More generally,
one can show, along the lines of the
proof of Lemma~\ref{lem: w [1,k]*w[l,n]}, that
$\mu_{A_n}(J_1(\boldsymbol\lambda),
J_2(\boldsymbol\lambda))=2(r+2)$
(respectively, $\mu_{A_n}(J_1(\boldsymbol\lambda),
J_2(\boldsymbol\lambda))=2r+3$), $r\ge 0$
if and only if~$r\lambda_1<\lambda_3\le (r+1)\lambda_1<\lambda_2$
(respectively, $r\lambda_1<\lambda_2\le (r+1)\lambda_1$ and~$\lambda_3>r\lambda_1$),
and then we have for~$\lambda_3>r\lambda_1$
\begin{align*}
u_{J_1(\boldsymbol{\lambda}),
J_2(\boldsymbol{\lambda})}=
\begin{cases}
w_{[(r+1)\lambda_1+1,\lambda_2+\lambda_3-r\lambda_1-1]\setminus\{\lambda_2\};[(r+1)\lambda_1+1,\lambda_2+\lambda_3-r\lambda_1-1]},&\lambda_3\le (r+1)\lambda_1<\lambda_2,
\\
w_{[r\lambda_1+1,\lambda_2+\lambda_3-r\lambda_1-1]\setminus\{\lambda_2\};[r\lambda_1+1,\lambda_2+\lambda_3-r\lambda_1-1]},&r\lambda_1<\lambda_2\le(r+1)\lambda_1.
\end{cases}
\end{align*}
The inequalities for~$\mu_{A_n}(J_2(\boldsymbol\lambda),J_1(\boldsymbol\lambda))$ are obtained by interchanging~$\lambda_2$ and~$\lambda_3$, while
the expressions for~$u_{J_2(\boldsymbol\lambda),J_1(\boldsymbol\lambda)}$ are obtained by applying the diagram
automorphism of~$W(A_n)$ to~$u_{J_1(\boldsymbol\lambda),J_2(\boldsymbol\lambda)}$ and then interchanging the role of~$\lambda_2$ and~$\lambda_3$. Then
it is easy to see that if one of
$\mu_{A_n}(J_1(\boldsymbol\lambda),J_2(\boldsymbol\lambda))$, $\mu_{A_n}(J_2(\boldsymbol\lambda),J_1(\boldsymbol\lambda))$ is even then the other one is odd and
so there are no connected {\em injective} fully supported homomorphisms in~$\Hom_{\He}(I_2(2r),A_n)$, $r\ge 2$. On the other hand,
the assignments~$s'_i\mapsto w_\circ^{J_i(\boldsymbol\lambda)}$,
$i\in\{1,2\}$, define an injective
connected fully supported homomorphism in~$\Hom_{\He}(I_2(2r+3),A_n)$, $r\ge 0$ provided that
$r\lambda_1<\lambda_2\le\lambda_3\le (r+1)\lambda_1$
and, up to diagram automorphisms, all such homomorphisms are obtained this way. They are
parabolic if and only if~$\lambda_2=\lambda_3=r\lambda_1+1$ and so~$n=(1+2r)\lambda_1+1$.
\end{remark}

\subsection{From type \texorpdfstring{$B$}{B}
to type~\texorpdfstring{$B$}{B}}\label{subs:B->B}
First, we collect some standard facts about parabolic
elements of~$W(B_n)$ which will be used
in the sequel. Their proof is an easy exercise
which is left to the reader.
\begin{lemma}\label{lem:Bn facts}
Let~$m\ge 2$, $a\in [1,m-1]$. Then
\begin{enmalph}
\item $w_\circ^{[a,m]}$ is central in~$W_{[a,m]}(B_m)$
and~$(W_{[a,m]}(B_m),\star)$;
\item $\cx am\times\cxr a{(m-1)}$
centralizes~$W_{[a+1,m]}(B_m)$ both
in the Coxeter group and in the Hecke monoid. Moreover,
$w_\circ^{[a,m]}=
\cx am\cxr a{(m-1)}\times w_\circ^{[a+1,m]}=
\cxr a{(m-1)}\times
w_\circ^{[a+1,m]}\times \cx a{(m-1)}$
and so~$w_{[a+1,m];[a,m]}=\cx am\cxr a{(m-1)}$.
\end{enmalph}
\end{lemma}
First we describe all connected injective
homomorphisms $(W(B_2),\star)\to
(W(B_n),\star)$. Unlike in the case of type~$A$,
they do not fit in the series.
\begin{proposition}\label{prop:B2 Bn all}
Let~$n\ge 2$, $2\le l\le \ceil{\frac{n+1}2}$ and
$k\in[2l-3,n-1]$. The
assignments $s'_1\mapsto w_\circ^{[1,k]}$,
$s'_2\mapsto w_\circ^{[l,n]}$ define an
injective homomorphism
$(W(B_2),\star)\to (W(B_n),\star)$,
which is parabolic if and only if~$k=n-1$, $l=2$.
Moreover, up to the diagram automorphism
of~$(W(B_2),\star)$, this exhausts
all injective connected homomorphisms
$(W(B_2),\star)\to (W(B_n),\star)$.
\end{proposition}
\begin{remark}
The number of pairs~$(k,l)\in[1,n]^2$ satisfying the assumptions of the Proposition is
\begin{align*}\sum_{2\le l\le \ceil{\frac{n+1}2}} (n-2l+3)=\floor{\tfrac12(n+1)}(\ceil{\tfrac12(n+1)}-1)
=\ceil{\tfrac12 n}\floor{\tfrac12 n}=\floor{\tfrac14 n^2}.
\end{align*}
\end{remark}
\begin{proof}
We need the following
\begin{lemma}\label{lem:B2 Bn all}
For all~$n\ge 2$, $l\in[2,n]$ and~$k\in[l-1,n-1]$
\begin{equation}\label{eq:first brd B2 Bn}
w_\circ^{[1,k]}\star w_\circ^{[l,n]}\star
w_\circ^{[1,k]}=w_\circ^{[1,n]}w_{[k-l+3,n]\setminus\{k+1\};[k-l+3,n]}.
\end{equation}
In particular,
$\mu_{B_n}([1,k],[l,n])=
\mu_{B_n}([l,n],[1,k])=4$
if and only
if~$2\le l\le  \ceil{\frac12(n+1)}$, $k\in [2l-3,n-1]$ and
$\min(\mu_{B_n}([1,k],[l,n]),
\mu_{B_n}([l,n],[1,k]))>4$ otherwise.
\end{lemma}
\begin{proof}
To prove~\eqref{eq:first brd B2 Bn}, we use induction on~$n$, the case~$n=2$ being obvious.

For the inductive step, we first prove that
\begin{align}\label{eq: B2 l=2}
&w_\circ^{[1,k]}\star w_\circ^{[2,n]}\star w_\circ^{[1,k]}=w_\circ^{[1,n]}w_{[k+2,n];[k+1,n]}=w_\circ^{[1,n]}\cx{(k+1)}n\cxr{(k+1)}{(n-1)},\\
&w_\circ^{[2,n]}\star w_\circ^{[1,k]}\star w_\circ^{[2,n]}=s_1 w_\circ^{[1,n]}.
\label{eq: B2 l=2'}
\end{align}
Indeed, using~\eqref{eq:w0 type A} and Lemma~\ref{lem:char w_0 monoid} we obtain
$$
w_\circ^{[1,k]}\star w_\circ^{[2,n]}\star w_\circ^{[1,k]}=\cxr1k
\star w_\circ^{[2,k]}\star w_\circ^{[2,n]}
\star w_\circ^{[2,k]}\star \cx1k
=\cxr1k
\star w_\circ^{[2,n]}\star \cx1k
$$
and similarly
$$
w_\circ^{[2,n]}\star w_\circ^{[1,k]}\star w_\circ^{[2,n]}=w_\circ^{[2,n]}\star s_1\star w_\circ^{[2,n]}
=w_\circ^{[2,n]}\star (s_1w_\circ^{[2,n]}).
$$
Since~$w_\circ^{[2,n]}=\cx1n\cxr1{(n-1)}w_\circ^{[1,n]}$ and~$w_\circ^{[1,n]}$ is central, by Lemma~\partref{lem: u * u^(-1) w_0.a}
\begin{align*}
w_\circ^{[1,k]}&\star w_\circ^{[2,n]}\star w_\circ^{[1,k]}=\cxr1k\star
((\cx1n\times \cxr1{(n-1)})w_\circ^{[1,n]})\star \cx1k
\\
&=((\cx{(k+1)}n\times\cxr1{(n-1)})w_\circ^{[1,n]})
\star \cx1k=(w_\circ^{[1,n]}(\cx{(k+1)}n\times\cxr1{(n-1)}))
\star \cxr1k\\
&=w_\circ^{[1,n]}\cx{(k+1)}n\cxr{(k+1)}{(n-1)}.
\end{align*}
On the other hand, since
$s_1w_\circ^{[2,n]}=(\cx2n\times\cxr1{(n-1)})w_\circ^{[1,n]}$, we obtain
by Lemma~\partref{lem: u * u^(-1) w_0.a'}
\begin{align*}
w_\circ^{[2,n]}\star w_\circ^{[1,k]}\star w_\circ^{[2,n]}&=w_\circ^{[2,n]}\star (\cx2n\times\cxr1{(n-1)})w_\circ^{[1,n]}\\
&=w_\circ^{[2,n]}\star (\cxr1{(n-1)}w_\circ^{[1,n]})
=w_\circ^{[2,n]}\star (s_1w_\circ^{[1,n]}).
\end{align*}
Since~$D_L(s_1w_\circ^{[1,n]})=[2,n]$,
it remains to apply Lemma~\ref{lem:left desc}.

Thus, we may assume that~$l>2$. Then by the induction hypothesis, in~$W(B_{n-1})$ we have
$$
w_\circ^{[1,k-1]}\star w_\circ^{[l-1,n-1]}\star
w_\circ^{[1,k-1]}=w_\circ^{[1,n-1]}w_{[k-l+3,n-1]\setminus\{k\};[k-l+3,n-1]},
$$
whence, using the natural
isomorphism $(W(B_{n-1}),\star)\cong (W_{[2,n]}(B_n),\star)$,
$s_i\mapsto s_{i+1}$, $i\in[1,n-1]$,
$$
w_\circ^{[2,k]}\star w_\circ^{[l,n]}\star
w_\circ^{[2,k]}=w_\circ^{[2,n]}w_{[k-l+4,n]\setminus\{k+1\};[k-l+4,n]}
$$
Then by~\eqref{eq:w0 type A}
\begin{align*}
w_\circ^{[1,k]}\star w_\circ^{[l,n]}\star
w_\circ^{[1,k]}
=\cxr1k \star (w_\circ^{[2,n]}w_{[k-l+4,n]\setminus\{k+1\};[k-l+4,n]})\star\cx 1k
\end{align*}
Write~$w_\circ^{[1,n]}=u\times w_\circ^{[2,n]}w_{[k-l+4,n]\setminus\{k+1\};[k-l+4,n]}$
where by Lemma~\partref{lem:Bn facts}
\begin{align*}
u&=w_\circ^{[1,n]}(w_\circ^{[2,n]}w_{[k-l+4,n]\setminus\{k+1\};[k-l+4,n]})^{-1}
=
\cx1n\times\cxr1{(n-1)}\times w_{[k-l+4,n]\setminus\{k+1\};[k-l+4,n]}\\
&=\cx1{(k-l+2)}\times\cx{(k-l+3)}n\cxr{(k-l+3)}{(n-1)}
w_{[k-l+4,n]\setminus\{k+1\};[k-l+4,n]}\times
\cxr1{(k-l+2)}\\
&=\cx1{(k-l+2)}\times
w_{[k-l+4,n]\setminus\{k+1\};[k-l+3,n]}\times
\cxr1{(k-l+2)}.
\end{align*}
Since~$w_\circ^{[1,n]}$ is central, by~\eqref{eq:w0 type A} and Lemmata~\partref{lem: u * u^(-1) w_0.a} and~\ref{lem:wK absorption},
\begin{align*}
w_\circ^{[1,k]}&\star w_\circ^{[l,n]}\star
w_\circ^{[1,k]}
=\cxr1k \star (w_\circ^{[1,n]}u)\star\cx 1k\\
&=\cxr{(k-l+3)}k\star (w_\circ^{[1,n]}w_{[k-l+4,n]\setminus\{k+1\};[k-l+3,n]})\star \cx{(k-l+3)}k
\\
&=\cxr{(k-l+3)}k\star (w_\circ^{[1,n]}(\cx{(k-l+3)}k\times
w_{[k-l+3,n]\setminus\{k+1\};[k-l+3,n]}))\star \cx{(k-l+3)}k\\
&=(w_\circ^{[1,n]}
w_{[k-l+3,n]\setminus\{k+1\};[k-l+3,n]})\star \cx{(k-l+3)}k.
\end{align*}
Then, since~$D_R(w_\circ^{[1,n]}w_{[k-l+3,n]\setminus\{k+1\};[k-l+3,n]})=[1,n]\setminus\{k+1\}$ by Lemmata~\ref{lem:wK absorption} and~\ref{lem:DL uw0}, this product
is equal to its first factor by Lemma~\ref{lem:left desc}.

Next, we prove that
\begin{equation}
\label{eq: mu[l,n][1,k]}
\mu_{B_n}([l,n],[1,k])\ge4.
\end{equation}
For~$l=2$ this follows from~\eqref{eq: B2 l=2'}. Suppose
that~$\mu_{B_n}([l,n],[1,k])\le 3$ for some~$l>2$, that is,
$w_\circ^{[l,n]}\star w_\circ^{[1,k]}\star w_\circ^{[l,n]}=w_\circ^{[1,n]}$.
Since~$w_\circ^{[l,n]}\le w_\circ^{[2,n]}$, by Proposition~\partref{prop:Bruhat order *.a}
$$
w_\circ^{[1,n]}=w_\circ^{[l,n]}\star w_\circ^{[1,k]}\star w_\circ^{[l,n]}\le w_\circ^{[2,n]}\star w_\circ^{[1,k]}\star w_\circ^{[2,n]},
$$
and so $w_\circ^{[2,n]}\star w_\circ^{[1,k]}\star w_\circ^{[2,n]}=w_\circ^{[1,n]}$, which contradicts~\eqref{eq: B2 l=2'}.

Since~$\ell(w_{[k-l+3,n]\setminus\{k+1\};[k-l+3,n]})=\frac12(l-1)(4(n-k-1)+l)>0$,
it follows from~\eqref{eq:first brd B2 Bn}
that
$$\mu_{B_n}([1,k],[l,n])\ge4.$$ If~$l\le k-l+3$,
that is, if $k\ge 2l-3$, which, since~$k\le n-1$,
forces~$2l-3\le n-1$ or~$l\le \ceil{\frac12(n+1)}$, then~\eqref{eq:first brd B2 Bn} and Lemma~\partref{lem: u * u^(-1) w_0.c}
imply that~$\mu_{B_n}([1,k],[l,n])=4$.
Applying~${}^{op}$, we conclude that
$\mu_{B_n}([l,n],[1,k])\le 4$
and so~$\mu_{B_n}([l,n],[1,k])=4$ by~\eqref{eq: mu[l,n][1,k]}.

Suppose that~$k<2l-3$. Let
$K=[k-l+3,n]$, $J=[k-l+3,n]\setminus\{k+1\}$
and~$J'=[l,n]\subsetneq K$. Then
$J\cap J'=[l,n]\setminus \{k+1\}$ and
\begin{align*}
\ell(w_{J;K})&-\ell(w_{J'\cap J;J'})=
\ell(w_\circ^{[k-l+3,n]})-\ell(w_\circ^{[k-l+3,k]})-\ell(w_\circ^{[l,n]})+\ell(w_\circ^{[l,k]})\\
&=\tfrac12(2l-3-k)(4n-3k-2)>0
\end{align*}
since~$n\ge 2$ and~$k<n$. Applying Lemma~\ref{lem:too short} we conclude that~$\mu_{B_n}([1,k],[l,n])>4$. In particular,
this also forces~$\mu_{B_n}([l,n],[1,k])>4$
since $\mu_{B_n}([l,n],[1,k])\ge4$ by~\eqref{eq: mu[l,n][1,k]} and the equality would imply,
by applying~${}^{op}$, that~$\mu_{B_n}([1,k],[l,n])\le 4$, which is a contradiction.
\end{proof}
It follows from the Lemma and Theorem~\ref{thm:Hom Heck Mon}
that the assignments
$s'_1\mapsto w_\circ^{[1,k]}$, $s'_2\mapsto
w_\circ^{[l,n]}$ define
a homomorphism~$(W(B_2),\star)\to (W(B_n),\star)$
if and only if~$k\in[2l-3,n-1]$
and hence~$2\le l\le \ceil{\frac12(n+1)}$.
By Proposition~\ref{prop:locally inj},
all of these homomorphisms are injective.
For the converse, given injective fully supported connected~$\phi\in\Hom_{\He}(B_2,B_n)$,
we may assume, without loss of generality, that~
$1\in[\phi](1)$ and~$n\in[\phi](2)$, for
otherwise we immediately get a contradiction by Proposition~\ref{prop:locally inj}. Thus,
$[\phi](1)=[1,k]$ and~$[\phi](2)=[l,n]$
with~$l\in[2,n]$ and~$k\in[l-1,n-1]$.
Then by Lemma~\ref{lem:B2 Bn all},
it follows that~$2\le l\le \ceil{\frac12(n+1)}$
and~$k\in[2l-3,n-1]$.
Finally, note that the image of
the parabolic element~$s'_1s'_2s'_1$ is parabolic
if and only if~$[k-l+3,n]=\{k+1\}$ which is equivalent to $(k,l)=(n-1,2)$, and in that
case the image of~$s'_2s'_1s'_2$ is parabolic by~\eqref{eq: B2 l=2'}.
\end{proof}
\begin{remark}\label{rem: uJK 2}
If~$2\le l\le \ceil{\frac12(n+1)}$ and~$k\in[2l-3,n-1]$ then it follows from~\eqref{eq:first brd B2 Bn} that $u_{[1,k],[l,n]}=w_{[k-l+3,n]\setminus\{k+1\};[k-l+3,n]}$.
One can check that~$u_{[l,n],[1,k]}=w_{[1,2l-3]\setminus\{l-1\};[1,2l-3]}$ in that case.
\end{remark}
\begin{proposition}\label{prop:simply laced to B, B to A}
Let~$M'\in\Cox{I'}$, $|I'|\ge 2$, be irreducible and simply laced. Then no
connected fully supported homomorphism
from~$(W(M'),\star)$ to~$(W(B_r),\star)$, $r\ge 2$, is locally injective. Similarly, no connected fully supported homomorphism
from~$(W(B_s),\star)$, $s\ge 2$, or~$(W(F_4),\star)$ to~$(W(A_n),\star)$, $n\ge 2$, is locally injective.
\end{proposition}
\begin{proof}
We need the following
\begin{lemma}\label{lem:B2 An do not exist}
A connected fully supported homomorphism
$(W(B_2),\star)\to (W(A_n),\star)$,
or~$(W(A_2),\star)\to (W(B_r),\star)$,
$r,n\ge 2$, is never injective.
\end{lemma}
\begin{proof}
Let~$\phi\in\Hom_{\He}(B_2,A_n)$
be injective, connected and fully supported. Then, as in the proof of Proposition~\ref{prop:A2->An},
$[\phi](i)=J_i(\boldsymbol\lambda)$
for some~$\boldsymbol\lambda=(\lambda_1,\lambda_2,\lambda_3)\in\ZZ_{>0}^3$
and $\mu_{A_n}(J_1(\boldsymbol{\lambda}),
J_2(\boldsymbol{\lambda}))=
\mu_{A_n}(J_2(\boldsymbol{\lambda}),
J_1(\boldsymbol{\lambda}))=4$
by Theorem~\ref{thm:Hom Heck Mon}
and Proposition~\ref{prop:locally inj}.
Yet
$\min(\mu_{A_n}(J_1(\boldsymbol{\lambda}),
J_2(\boldsymbol{\lambda})),
\mu_{A_n}(J_2(\boldsymbol{\lambda}),J_1(\boldsymbol{\lambda})))\not=4$
by Lemma~\ref{lem:mu J1 J2}, which is a contradiction.
The argument for~$\Hom_{\He}(A_2,B_r)$ is similar with
Lemma~\ref{lem:mu J1 J2} replaced by Lemma~\ref{lem:B2 Bn all}.
\end{proof}
Let~$\phi$ be a connected fully supported homomorphism from~$(W(M'),\star)$ to~$(W(B_r),\star)$.
Since~$\phi$ is fully supported, $r\in[\phi](i)$ for some~$i\in I'$.
Since~$M'$ is irreducible and simply laced, there exists~$j\not=i$ with~$m'_{ij}=3$.
Then~$L=[\phi](\{i,j\})$ is connected and contains~$r$, hence~$W_L(B_r)$ is of type~$B_{|L|}$ (or~$|L|=1$ and then~$W_L(B_r)$ has two elements),
and~$\phi|_{W_{\{i,j\}}(M')}$ is a connected fully supported
homomorphism from~$(W(A_2),\star)$ to~$(W_L(B_r),\star)$, which is not injective by the Lemma. Thus, $\phi$ is not locally injective.

Similarly, if~$M'$ is of type~$B_s$ or~$F_4$, let~$i,j\in I'$ be such that~$m'_{ij}=4$. Then~$L=[\phi](\{i,j\})$ is connected,
so~$W_L(A_n)$ is of type~$A_{|L|}$, and the restriction of~$\phi$ to~$W_{\{i,j\}}(M')$ is a connected fully supported
homomorphism from~$(W(B_2),\star)$ to~$(W_L(A_n),\star)$, which is not injective by the Lemma.
\end{proof}

Given~$m,r\in\ZZ_{>0}$, let~\plink{Br(m)}$\mathcal B_r(m)$ be the set of
all partitions of~$m$ with $r$ parts such that
all parts except possibly the maximal and the minimal are equal, that is
$$
\mathcal B_r(m)=\{ (\lambda_1,\dots,\lambda_r)\in\ZZ_{>0}^r\,:\,
\lambda_1\ge \lambda_2=\cdots=\lambda_{r-1}\ge \lambda_r,\, \sum_{1\le j\le r}\lambda_j=m\}.
$$
Clearly, $\mathcal B_r(m)=\mathcal A_r(m)$ for~$1\le r\le 3$. If~$r\ge 3$ then, using the transposition of partitions,
it is easy to see that~$\mathcal B_r(m)$ is in bijection
with the set of partitions of~$m$ whose parts
are in~$\{r,r-1,1\}$ and the maximal part is~$r$. In particular,
$$
\sum_{m\ge 0} |\mathcal B_r(m)|x^m=\frac{x^r}{(1-x)(1-x^{r-1})(1-x^r)},\qquad r\ge 3.
$$

\begin{theorem}\label{thm:Bk Bn}
For all~$m\ge r\ge 3$, $\boldsymbol\mu\in \mathcal B_r(m)$, the assignments
$s'_i\mapsto w_\circ^{J_i(\boldsymbol\mu')}$, $i\in [1,r-1]$ and~$s'_r\mapsto
w_\circ^{[m-\mu_1+1,m]}$,
where~$\boldsymbol\mu'=(\mu_2,\mu_2,\dots,\mu_r)\in \mathcal A_r(m-\mu_1+\mu_2)$,
define a locally injective connected fully supported~\plink{theta mu}$\theta_{\boldsymbol\mu}\in
\Hom_{\He}(B_r,B_m)$. Conversely,
if~$\phi\in\Hom_{\He}(B_r,B_m)$
is locally injective, connected and fully supported
then~$\phi=\theta_{\boldsymbol{\mu}}$
for some~$\boldsymbol{\mu}\in\mathcal B_r(m)$.
\end{theorem}
\begin{proof}
Abbreviate~$J_i=J_i(\boldsymbol{\mu}')$, $i\in[1,r-1]$ and~$J_r=[m-\mu_1+1,m]$.
By Theorem~\ref{thm:Ak to An},
the assignments $s'_i\mapsto w_\circ^{J_i}$,
$i\in [1,r-1]$ define a locally injective connected
homomorphism from~$(W_{[1,r-1]}(B_r),\star)
\cong (W(A_{r-1}),\star)$ to
$(W_{[1,m-\mu_1+\mu_2-1]}(B_m),\star)$,
which in turn is isomorphic to~$
(W(A_{m-\mu_1+\mu_2-1}),\star)$.
Since~$\min J_r=m-\mu_1+1=
\sum_{2\le j\le r}\mu_j+1
=a_{r-1}(\boldsymbol{\mu}')+1
\ge
a_i(\boldsymbol{\mu'})+1$ for all~$i\in[1,r-1]$,
it follows that~$J_r$ is orthogonal to
all the~$J_i$ with~$i\in[1,r-2]$ and
so~$\mu_{B_m}(J_i,J_r)=
\mu_{B_m}(J_r,J_i)=2$ for all~$i\in[1,r-2]$.
We now write~$J_{r-1}=a_{r-2}(\boldsymbol\mu')+[1,k]$,
$J_r=a_{r-2}(\boldsymbol{\mu}')+[l,n]$ where
$k=\mu'_{r-2}+\mu'_{r-1}-1=2\mu_2-1$,
$l=\mu_2+1$ and~$n=\mu_1+\mu_2$. Then~$l\ge 2$,
$n\ge 2(l-1)$, whence~$l\le\ceil{\frac{n+1}2}$, and~$k=2l-3\le n-1$. By Lemma~\ref{lem:B2 Bn all}
$\mu_{B_n}([1,k],[l,n])=
\mu_{B_n}([l,n],[1,k])=4$. Since~$(W(B_n),\star)\cong
(W_{[a_{r-2}(\boldsymbol\mu')+1,m]}(B_m),\star)$
via $s_i\mapsto s_{i+a_{r-2}(\boldsymbol\mu')}$, $i\in [1,n]$,
it follows that~$\mu_{B_m}(J_{r-1},J_r)=
\mu_{B_m}(J_r,J_{r-1})=4$.
It remains to apply Theorem~\ref{thm:Hom Heck Mon} and to observe that~$\bigcup_{1\le i\le r}J_i=[1,m]$ and that~$J_{r-1}\cup J_r=a_{r-2}(\boldsymbol\mu')+[1,n]$ is connected,
so that~$\theta_{\boldsymbol\mu}$ is fully supported and connected.

To prove the converse, let~$\phi\in \Hom_{\He}(B_r,B_m)$ be locally injective, connected and fully supported.
Since~$[1,r-1]$ is connected, $L:=[\phi]([1,r-1])$ is a connected subset of~$[1,m]$.
If~$m\in L$ then~$W_L(B_m)$ is of type~$B$ and the restriction of~$\phi$ to
$(W_{[1,r-1]}(B_r),\star)\cong(W(A_{r-1}),\star)$ is a locally injective connected fully supported homomorphism
to~$(W_L(B_m),\star)$. This contradicts
Proposition~\ref{prop:simply laced to B, B to A} since~$r-1\ge 2$. Thus, $L=[a,b]\subset[1,m-1]$ and, since~$\phi$ is
fully supported and connected, $[\phi](r)=[l,m]$ for some~$l\in[1,m]$. If~$a>1$ then~$1\in[\phi](r)$, that is, $[\phi](r)=[1,m]\supset[\phi](r-1)$,
whence~$\mu_{B_m}([\phi](r),[\phi](r-1))\le 2$ by Proposition~\ref{prop:comm idemp},
which contradicts the local injectivity of~$\phi$ by Proposition~\ref{prop:locally inj}.
Thus, $L=[1,b]$.

The restriction of~$\phi$ to~$(W_{[1,r-1]}(B_r),\star)$ is a locally injective connected fully supported homomorphism
from~$(W(A_{r-1}),\star)$ to~$(W_L(B_m),\star)\cong(W(A_b),\star)$.
By Theorem~\ref{thm:Ak to An}, taking into account the diagram automorphisms of~$A_{r-1}$ and of~$A_b$, there
exists~$\boldsymbol\lambda=(\lambda_1,\dots,\lambda_r)\in\mathcal A_r(b+1)$ such that the restriction of~$[\phi]:[1,r]\to
\mathscr P([1,m])$ to~$[1,r-1]$
is one of the following four maps
\begin{subequations}
\begin{align}
&i\mapsto J_i(\boldsymbol\lambda)\label{eq:cs 1},\\
&i\mapsto b+1-J_i(\boldsymbol{\lambda}),
\label{eq:cs 2}\\
&i\mapsto J_{r-i}(\boldsymbol{\lambda}),
\label{eq:cs 3}\\
&i\mapsto b+1-J_{r-i}(\boldsymbol\lambda).\label{eq:cs 4}
\end{align}
\end{subequations}
Note that $1\in [\phi](1)$ and
$b\in [\phi](r-1)$ in cases~\eqref{eq:cs 1} and~\eqref{eq:cs 4},
while~$1\in [\phi](r-1)$
and~$b\in [\phi](1)$ in cases~\eqref{eq:cs 2} and~\eqref{eq:cs 3}.

Let~$i\in[1,r-2]$. Since~$m'_{ir}=2$
and~$\phi$ is locally injective,
$[\phi](i)$ and~$[\phi](r)$ must be orthogonal by Corollary~\ref{cor:hom A1 x A1}.
Thus,
\begin{equation}\label{eq:Bk Bn orth}
l\ge \max [\phi](i)+2,\qquad i\in[1,r-2].
\end{equation}
In cases~\eqref{eq:cs 2} and~\eqref{eq:cs 3}, $b\in[\phi](1)$, and then~\eqref{eq:Bk Bn orth} with~$i=1$ implies that~$l\ge b+2$, whence
$b+1\notin[\phi]([1,r])=[1,b]\cup[l,m]$, which contradicts the assumption that~$\phi$ is fully supported.
In the case~\eqref{eq:cs 4},
it follows from Lemma~\ref{lem:diag aut J_i} that
$[\phi](i)=J_i(\tilde{\boldsymbol\lambda})$, $i\in[1,r-1]$, where,
since~$\lambda_1=\cdots=\lambda_{r-2}$,
$\tilde{\boldsymbol\lambda}$
is obtained from~$\boldsymbol\lambda$ by interchanging~$\lambda_r$ and~$\lambda_{r-1}$. Thus,
$$
[\phi](i)=J_i(\boldsymbol\nu),\quad i\in[1,r-1],\qquad \boldsymbol\nu=(\nu_1,\dots,\nu_1,\nu_{r-1},\nu_r)\in\ZZ_{>0}^r,\quad \nu_1\ge \max(\nu_{r-1},\nu_r).
$$
Then~$a_i(\boldsymbol\nu)=\nu_r+(i-1)\nu_1$, $i\in[1,r-1]$, $a_r(\boldsymbol\nu)=b+1$, and so
$$
\max[\phi](r-2)=a_{r-1}(\boldsymbol\nu)-1=\nu_r+(r-2)\nu_1-1,\qquad [\phi](r-1)=[a_{r-2}(\boldsymbol\nu)+1,b].
$$
In particular, $l\ge \nu_r+(r-2)\nu_1+1$ by~\eqref{eq:Bk Bn orth}.

Finally, consider the restriction of~$\phi$ to~$(W_{\{r-1,r\}}(B_r),\star)\cong(W(B_2),\star)$.
Since~$l\le b+1$ (as~$\phi$ is fully supported),
$[\phi](\{r-1,r\})=[a_{r-2}(\boldsymbol\nu)+1,m]$ and this restriction is an injective connected homomorphism
to~$(W_{[a_{r-2}(\boldsymbol\nu)+1,m]}(B_m),\star)$.
Taking its composition with the
natural isomorphism of the
latter parabolic submonoid onto
$(W(B_{m'}),\star)$, $s_{a_{r-2}(\boldsymbol\nu)+j}\mapsto s_j$,
$j\in [1,m']$,
where~$m'=m-a_{r-2}(\boldsymbol\nu)$,
yields
an injective connected fully supported homomorphism~$(W(B_2),\star)
\to (W(B_{m'}),\star)$ satisfying
$$
s'_1\mapsto w_\circ^{[1,k]},
\quad s'_2\mapsto w_\circ^{[l',m']},
$$
where~$k=b-a_{r-2}(\boldsymbol\nu)=\nu_1+\nu_{r-1}-1$ and~$l'=l-a_{r-2}(\boldsymbol\nu)$. By Proposition~\ref{prop:B2 Bn all},
$2\le l'\le\ceil{\frac12(m'+1)}$ and~$k\in[2l'-3,m'-1]$. Since~$l'\ge \nu_r+(r-2)\nu_1+1-a_{r-2}(\boldsymbol\nu)=\nu_1+1$, we obtain
$\nu_1+\nu_{r-1}-1=k\ge 2l'-3\ge 2\nu_1-1$, whence~$\nu_{r-1}\ge \nu_1$. Since~$\nu_1\ge \max(\nu_{r-1},\nu_r)$, this forces
$\nu_{r-1}=\nu_1$, whence $l'=\nu_1+1$, that is, $l=\nu_r+(r-2)\nu_1+1$, and~$k=2\nu_1-1\le m'-1$, that is,
$m\ge \nu_r+(r-1)\nu_1$.

Set~$\mu_1=m-\nu_r-(r-2)\nu_1$, $\mu_2=\cdots=\mu_{r-1}=\nu_1$ and~$\mu_r=\nu_r$. Then~$\mu_1\ge \mu_2\ge\mu_r$ and~$\sum_{1\le j\le r}\mu_j=m$,
that is, $\boldsymbol\mu=(\mu_1,\dots,\mu_r)\in\mathcal B_r(m)$. Furthermore, $\boldsymbol\nu=\boldsymbol\mu'$, so
$[\phi](i)=J_i(\boldsymbol\mu')$, $i\in[1,r-1]$, and $[\phi](r)=[l,m]=[m-\mu_1+1,m]$. Therefore, $\phi=\theta_{\boldsymbol\mu}$.
\end{proof}
\begin{remark}
It is easy to see, using the functor~$[\cdot]$ and Lemma~\ref{lem:Hecke hom w0J}, that if~$\boldsymbol\mu=(\mu_1,\dots,\mu_r)\in \mathcal B_r(m)$ with~$\mu_1>\mu_2$ then
$$
\theta_{\boldsymbol\mu}=\theta_{(2,1,\dots,1)}\circ \theta_{(\mu_1-1,\mu_2,\dots,\mu_r)},
$$
where~$(2,1,\dots,1)\in\mathcal B_{m-1}(m)$, so that~$\theta_{(2,1,\dots,1)}\in\Hom_{\He}(B_{m-1},B_m)$. This decomposition is non-trivial
unless~$\boldsymbol\mu=(2,1,\dots,1)$. Thus, every~$\theta_{\boldsymbol\mu}$ is a composition of homomorphisms~$\theta_{\boldsymbol\nu}$
with~$\nu_1=\nu_2$ and of the~$\theta_{(2,1,\dots,1)}$.
\end{remark}

\subsection{From type~\texorpdfstring{$B$}{B} to type
\texorpdfstring{$D$}{D}}\label{subs:B->D}
Let~$\sigma$ be the diagram automorphism
of~$W(D_{n+1})$ corresponding to the
permutation~$(n,n+1)$. We begin with the following immediate consequence of Theorem~\ref{thm:adm finite class}.
\begin{lemma}\label{lem:sigma equiv}
Let~$M\in\Cox I$, $\phi\in\Hom_{\He}(M,D_{n+1})$ and
suppose that~$\phi$ is $\sigma$-invariant, that is $\sigma([\phi](i))=[\phi](i)$, $i\in I$. Then~$\phi$ is the composition
of~\eqref{eq:unfold Bn Dn+1} with
some $\phi'\in\Hom_{\He}(M,B_n)$.
Moreover, $\phi$ is fully supported (respectively, injective, locally injective) if and only if so is~$\phi'$,
and if~$\phi$ is connected then so is~$\phi'$.
\end{lemma}
\begin{proof}
Let~$\zeta\in\Hom_{\He}(B_n,D_{n+1})$ be the homomorphism~\eqref{eq:unfold Bn Dn+1}.
By Theorem~\ref{thm:adm finite class}, $\zeta$ is
an isomorphism from $(W(B_n),\star)$ to
the submonoid of~$(W(D_{n+1}),\star)$
consisting of~$\sigma$-invariant elements, and the image of~$\phi$ is manifestly contained in that submonoid. Then~$\phi':=\zeta^{-1}\circ\phi$
is a well-defined homomorphism
in~$\Hom_{\He}(M,B_n)$.
Since~$\zeta$ is injective (hence left-total), fully supported and disjoint, the assertions about full support and
(local) injectivity follow from Lemmata~\ref{lem:[phi]sets} and~\ref{lem:loc inj comp}.
Finally,
since~$[\phi]=[\zeta]\circ[\phi']$ by Lemma~\partref{lem:[phi]comp.b}, $[\phi'](J)=[\phi](J)\setminus\{n+1\}$ for all~$J\subset I$.
If~$J$ is connected and~$\phi$ is connected then~$[\phi](J)$ is a $\sigma$-invariant connected subset of~$[1,n+1]$, and so~$n\in[\phi](J)$
implies~$n-1\in[\phi](J)$. Therefore, $[\phi'](J)=[\phi](J)\setminus\{n+1\}$ is connected.
\end{proof}
Thus, in this section we only consider homomorphisms which are {\em not} $\sigma$-invariant.

The case of~$D_4$ must be treated separately since
its group of diagram automorphisms is~$S_3$.
The following is easily checked (for example, using
our Python program for computations in Hecke monoids or Sage, or even by hand)
\begin{lemma}\label{lem:B2 to D4}
Up to diagram automorphisms,
the only injective fully supported homomorphisms from
$(W(B_2),\star)$ to $(W(D_4),\star)$ which
are not~$\sigma$-invariant, are
given by assignments $s'_1\mapsto w_\circ^{[1,3]}$,
$s'_2\mapsto w_\circ^{J}$ where
$$J\in \{ \{4\},\{2,4\},\{3,4\},\{2,3,4\},\{1,3,4\}\}.
$$
Of these, the ones corresponding to~$J=\{\{2,4\},\{2,3,4\}\}$ are parabolic and
the ones corresponding to~$J\in \{\{4\},\{2,4\},
\{2,3,4\}\}$ are connected. Furthermore,
there are no such homomorphisms
from $(W(A_2),\star)$ to~$(W(D_4),\star)$.
\end{lemma}
Note that homomorphisms corresponding to
$J=\{3,4\}$ and~$J=\{1,3,4\}$ are obtained from
the one corresponding to~$J=\{4\}$ using Lemma~\ref{lem:decoration}. However, one of
them turns out to be a part of a series, even
though it is the only non-connected one in that series.

The following is immediate from~\eqref{eq:w0 D}.
\begin{equation}\label{eq:w0 D rec}
w_\circ^{[a,n+1]}= \cx a{(n+1)}\cxr a{(n-1)}\times
w_\circ^{[a+1,n]}=w_\circ^{[a+1,n]}\times \cx a{(n+1)}\cxr a{(n-1)}.
\end{equation}
\begin{proposition}\label{prop:B2 Dn+1}
Let~$n\ge 4$. For all $2\le k\le
\ceil{\frac14(n+2)}+\floor{\frac14(n+2)}$
the assignments $s'_1\mapsto w_\circ^{[1,n]}$,
$s'_2\mapsto w_\circ^{[k,n+1]}$ define
a homomorphism $(W(B_2),\star)\to
(W(D_{n+1}),\star)$. This homomorphism
is injective except when~$n$ is even and~$k=2$,
and is parabolic if $k=2$. Moreover, up
to diagram automorphisms
all injective and connected~$\phi\in\Hom_{\He}(B_2,D_{n+1})$
which are not~$\sigma$-invariant are obtained this way.
\end{proposition}
\begin{proof}
We need the following
\begin{lemma}\label{lem:first braid B2 Dn+1}
\begin{enmalph}
\item\label{lem:first braid B2 Dn+1.a}
Let~$n\ge 4$, $r\in [1,n]$. Then
\begin{equation}\label{eq: first braid D'}
w_\circ^{[1,n]}\star w_\circ^{\sigma([r,n])}\star w_\circ^{[1,n]}=
w_{[3,n];[3,n+1]}w_\circ^{[1,n+1]}.
\end{equation}
In particular,
$\min(\mu_{D_{n+1}}([1,n],\sigma([r,n])),
\mu_{D_{n+1}}(\sigma([r,n]),[1,n]))>4$
for all~$r\in[1,n]$.
\item\label{lem:first braid B2 Dn+1.b}
Let~$n\ge 3$, $k\in[2,n]$. Then
\begin{align}
w_\circ^{[1,n]}\star w_\circ^{[k,n+1]}\star w_\circ^{[1,n]}
=w_{[n+2-l(\bar n,k),n];[n+2-l(\bar n,k),n+1]}w_\circ^{[1,n+1]}\label{eq: first braid D}
\end{align}
where~$l(\epsilon,2)=\epsilon$, $l(\epsilon,3)=3-2 \epsilon$ and $l(\epsilon,k)=k-\overline{k+\epsilon+1}$,
$k\ge 4$, $\epsilon\in\{0,1\}$.
In particular,
$\mu_{D_{n+1}}([1,n],[k,n+1])
=\mu_{D_{n+1}}([k,n+1],[1,n])=4$ if and only
if~$3-\bar n\le k\le \floor{\frac14(n+2)}+\ceil{\frac14(n+2)}$
and~$\min(\mu_{D_{n+1}}([1,n],[k,n+1]),
\mu_{D_{n+1}}([k,n+1],[1,n]))>4$ otherwise.
\end{enmalph}
\end{lemma}
\begin{proof}
Since~$(W_{\sigma([r,n])}(D_{n+1}),\star)
\cong (W(A_{n-r+1}),\star)$, we have
by Lemma~\ref{lem:w0J*w0K} and~\eqref{eq:w0 type A}
$$
w_\circ^{[1,n]}\star w_\circ^{\sigma([r,n])}=w_\circ^{[1,n]}\times \sigma(w_{[r,n-1];[r,n]})
=w_\circ^{[1,n]}\times \sigma(\cxr rn)=w_\circ^{[1,n]}\times s_{n+1}\times \cxr r{(n-1)},
$$
and so, by Lemma~\ref{lem:char w_0 monoid},
$$
w_\circ^{[1,n]}\star w_\circ^{\sigma([r,n])}
\star w_\circ^{[1,n]}
=w_\circ^{[1,n]}\star s_{n+1}\star \cxr r{(n-1)}\star w_\circ^{[1,n]}
=w_\circ^{[1,n]}\star s_{n+1}
\star w_\circ^{[1,n]}.
$$
Thus, it suffices to prove~\eqref{eq: first braid D'} for~$r=n$.
We have by~\eqref{eq:w0 type A} and Lemma~\ref{lem:char w_0 monoid}
\begin{align*}
w_\circ^{[1,n]}&\star s_{n+1}
\star w_\circ^{[1,n]}
=\cx1n\cx1{(n-1)}\star w_\circ^{[1,n-2]}
\star s_{n+1}\star w_\circ^{[1,n]}\\
&=\cx1n\cx1{(n-1)}
\star s_{n+1}\star w_\circ^{[1,n]}.
\end{align*}
Note that the element~$\cx1n\cx1{(n-1)}$ is reduced hence is the same both in the Hecke monoid and in the Coxeter group. Using~\eqref{eq:comm cox} we obtain
$\cx1n \cx1{(n-1)}=\cx2n\cx1n$ and so
\begin{align*}
w_\circ^{[1,n]}&\star s_{n+1}
\star w_\circ^{[1,n]}
=\cx2n\cx1{(n+1)}\star w_\circ^{[1,n]}=\cx2n\sigma(\cx1n)\star s_n
\star w_\circ^{[1,n]}=
\cx2n\sigma(\cx1n)\star w_\circ^{[1,n]}.
\end{align*}
By~\eqref{eq:w1,n D},
\begin{align*}
w_\circ^{[1,n]}&=
\sigma(\cxr1n\times w_{[2,n];[2,n+1]})w_\circ^{[1,n+1]}
=(s_{n+1}\cxr1{(n-1)}\times \sigma(w_{[2,n];[2,n+1]}))w_\circ^{[1,n+1]}\\&
=
(s_{n+1}\cxr1{(n-1)}\times \cxr2n\times w_{[3,n];[3,n+1]})w_\circ^{[1,n+1]},
\end{align*}
and it remains to apply Lemma~\partref{lem: u * u^(-1) w_0.a}.

We now prove that
\begin{equation}\label{eq: long braid D}
(w_\circ^{\sigma([1,n])}\star w_\circ^{[1,n]})^{\star 2}=w_{\sigma([4,n]);[4,n+1]}w_\circ^{[1,n+1]}.
\end{equation}
Indeed,
by~\eqref{eq: first braid D'}
and~\eqref{eq:w0 type A},
\begin{align*}
(w_\circ^{\sigma([1,n])}\star w_\circ^{[1,n]})^{\star 2}
&=w_\circ^{\sigma([1,n])}\star
(w_{[3,n];[3,n+1]}w_\circ^{[1,n+1]})\\&
=w_\circ^{[1,n-1]}\star s_{n+1}\star
\cxr1{(n-1)}\star (w_{[3,n];[3,n+1]}w_\circ^{[1,n+1]})
\\&=w_\circ^{[1,n-1]}\star (s_{n+1}w_{[3,n];[3,n+1]}w_\circ^{[1,n+1]}),
\end{align*}
since~$D_L(w_{[3,n];[3,n+1]}w_\circ^{[1,n+1]})=
[1,n]$ by Lemmata~\ref{lem:DL uw0} and~\ref{lem:wK absorption}. By~\eqref{eq:w1,n D}, $s_{n+1}w_{[3,n];[3,n+1]}=
\cxr3{(n-1)}\times \sigma(w_{[4,n];[4,n+1]})
=\cxr3{(n-1)}\times w_{\sigma([4,n]);[4,n+1]}$.
Then by Lemma~\partref{lem: u * u^(-1) w_0.a'},
\begin{align*}
(w_\circ^{\sigma([1,n])}\star w_\circ^{[1,n]})^{\star 2}
&=w_\circ^{[1,n-1]}\star ((\cxr3{(n-1)}\times w_{\sigma([4,n]);[4,n+1]})w_\circ^{[1,n+1]})\\
&
=w_\circ^{[1,n-1]}\star (w_{\sigma([4,n]);[4,n+1]}w_\circ^{[1,n+1]}),
\end{align*}
and~\eqref{eq: long braid D} follows from Lemma~\ref{lem:left desc}
since~$D_L(w_{\sigma([4,n]);[4,n+1]}w_\circ^{[1,n+1]})=\sigma([1,n])\supset[1,n-1]$.

For~$r\in[1,n]$, abbreviate $m_r=\mu_{D_{n+1}}(\sigma([r,n]),[1,n])$ and~$m'_r=\mu_{D_{n+1}}([1,n],\sigma([r,n]))$.
By~\eqref{eq: long braid D}, $m_1>4$.

Suppose that~$m_r\le 4$ for some~$2 \le r\le n$.
Then, since~$w_\circ^{\sigma([r,n])}\le w_\circ^{\sigma([1,n])}$,
$$
w_\circ^{[1,n+1]}=w_\circ^{\sigma([r,n])}
\star w_\circ^{[1,n]}\star w_\circ^{\sigma([r,n])}\star w_\circ^{[1,n]}
\le w_\circ^{\sigma([1,n])}
\star w_\circ^{[1,n]}\star w_\circ^{\sigma([1,n])}\star w_\circ^{[1,n]}
\le w_\circ^{[1,n+1]}
$$
by Proposition~\partref{prop:Bruhat order *.a}.
Thus, $m_1\le 4$, which is a contradiction.
Therefore, $m_r>4$ and so~$m'_r\ge 4$ by~\eqref{eq:mu_M JK KJ} for all~$r\in[1,n]$.
If~$m'_r=4$ for some~$r\in[1,n]$ then,
by applying~${}^{op}$, we conclude
that~$m_r\le 4$, which is a contradiction.
Thus, $\min(m_r,m'_r)>4$ for all~$r\in[1,n]$, and
part~\ref{lem:first braid B2 Dn+1.a} is proven.

We now prove part~\ref{lem:first braid B2 Dn+1.b}. First,
let~$k=2$. We claim that $w_\circ^{[1,n]}\star w_\circ^{[2,n+1]}
\star w_\circ^{[1,n]}=w_\circ^{[1,n+1]}$ if~$n$ is even
and~$w_\circ^{[1,n]}\star w_\circ^{[2,n+1]}
\star w_\circ^{[1,n]}=s_{n+1}w_\circ^{[1,n+1]}$
if~$n$ is odd. By~\eqref{eq:w0 type A}, \eqref{eq:w0 D rec} and Lemmata~\ref{lem:char w_0 monoid} and~\partref{lem: u * u^(-1) w_0.a}
\begin{align*}
w_\circ^{[1,n]}&\star w_\circ^{[2,n+1]}
\star w_\circ^{[1,n]}
=\cxr1n\star w_\circ^{[2,n]}\star w_\circ^{[2,n+1]}
\star w_\circ^{[2,n]}\star \cx1n
=\cxr1n\star w_\circ^{[2,n+1]}\star \cx1n\\
&=\cxr1n\star (w_\circ^{[1,n+1]}\cx1{(n-1)}\cxr1{(n+1)})
\star \cx1n
=\cxr1n\star (w_\circ^{[1,n+1]}\sigma(\cx1n)).
\end{align*}
If~$n$ is even then $w_\circ^{[1,n+1]}\sigma(\cx1n)
=\cx1n w_\circ^{[1,n+1]}$, whence
again by Lemma~\partref{lem: u * u^(-1) w_0.a}
$$
w_\circ^{[1,n]}\star w_\circ^{[2,n+1]}
\star w_\circ^{[1,n]}=w_\circ^{[1,n+1]}.
$$
If~$n$ is odd then~$w_\circ^{[1,n+1]}$ is central. Therefore, by Lemma~\partref{lem: u * u^(-1) w_0.a}
\begin{align*}
w_\circ^{[1,n]}\star w_\circ^{[2,n+1]}
\star w_\circ^{[1,n]}
&=\cxr1n\star (\cx1{(n-1)}s_{n+1}w_\circ^{[1,n+1]})=s_n\star (s_{n+1}w_\circ^{[1,n+1]}).
\end{align*}
Since~$D_L(s_{n+1}w_\circ^{[1,n+1]})=[1,n]$ it follows by Lemma~\ref{lem:left desc}
that~$s_n\star (s_{n+1}w_\circ^{[1,n+1]})=
s_{n+1}w_\circ^{[1,n+1]}$.

For~$k>2$ we use induction on~$n$, the case~$n=3$
being easy to check. By the induction hypothesis,
$
w_\circ^{[2,n]}\star w_\circ^{[k,n+1]}\star w_\circ^{[2,n]}
=w_{[n+2-l(1-\bar n,k-1),n];[n+2-l(1-\bar n,k-1),n+1]}w_\circ^{[2,n+1]}$.
Therefore,
\begin{align*}
w_\circ^{[1,n]}\star w_\circ^{[k,n+1]}\star w_\circ^{[1,n]}=\cxr1n \star
(w_{[n+2-l(1-\bar n,k-1),n];[n+2-l(1-\bar n,k-1),n+1]}w_\circ^{[2,n+1]})\star \cx1n.
\end{align*}
If~$k=3$ then
\begin{align*}
w_\circ^{[1,n]}\star w_\circ^{[3,n+1]}\star w_\circ^{[1,n]}=
\cxr1n\star (w_\circ^{[n+1+\bar n,n+1]}w_\circ^{[2,n+1]})\star \cx1n.
\end{align*}
If~$n$ is odd then~$w_\circ^{[1,n+1]}$ is
central and, using Lemma~\partref{lem: u * u^(-1) w_0.a} and Lemma~\ref{lem:left desc} (as above)
we obtain
\begin{align*}
w_\circ^{[1,n]}\star w_\circ^{[3,n+1]}\star w_\circ^{[1,n]}&=
\cxr1n\star w_\circ^{[2,n+1]}\star \cx1n
=\cxr1n \star (w_\circ^{[1,n+1]}\cx1{(n-1)}\cxr1{(n+1)})\star \cx1n
\\
&=\cxr1n \star (w_\circ^{[1,n+1]}\cx1{(n-1)}s_{n+1})
=s_n\star (s_{n+1}w_\circ^{[1,n+1]})=s_{n+1}w_\circ^{[1,n+1]}\\
&=w_{[n+2-l(1,3),n];[n+2-l(1,3),n+1]}w_\circ^{[1,n+1]}.
\end{align*}
If~$n$ is even then by Lemma~\partref{lem: u * u^(-1) w_0.a}
\begin{align*}
w_\circ^{[1,n]}&\star w_\circ^{[3,n+1]}\star w_\circ^{[1,n]}=
\cxr1n\star (s_{n+1}w_\circ^{[2,n+1]})\star \cx1n
\\
&=\cxr1n\star (s_{n+1}w_\circ^{[1,n+1]}\cx1{(n-1)}\cxr1{(n+1)})\star \cx1n\\
&=\cxr1n\star (w_\circ^{[1,n+1]}s_n\cx1{(n-1)}\cxr1{(n+1)})\star \cx1n
=\cxr1n\star (w_\circ^{[1,n+1]}s_n\cx1{(n-1)}s_{n+1})\\
&=\cxr1n\star (\cx1{(n-2)}s_{n+1}s_{n-1}s_n w_\circ^{[1,n+1]})
=s_n s_{n-1}\star (s_{n+1}s_{n-1}s_n w_\circ^{[1,n+1]}).
\end{align*}
Since~$D_L(s_{n+1}s_{n-1}s_n)=\{n+1\}$,
$D_L(s_{n+1}s_{n-1}s_n w_\circ^{[1,n+1]})=
[1,n]$ by Lemma~\ref{lem:DL uw0} and
so~$s_ns_{n-1}\star (s_{n+1}s_{n-1}s_n w_\circ^{[1,n+1]})
=s_{n+1}s_{n-1}s_n w_\circ^{[1,n+1]}$ by
Lemma~\ref{lem:left desc}. It remains to
observe that~$s_{n+1}s_{n-1}s_n=
w_{[n-1,n];[n-1,n+1]}=
w_{[n+2-l(0,3),n];[n+2-l(0,3),n+1]}$.

Before we pass to the general case,
note that if~$k>4$ then~$l(1-\bar n,k-1)=k-1-\overline{k-n+1}=l(\bar n,k)-1$, while
$l(1-\bar n,3)-l(\bar n,4)=
1+2\bar n-(3+\bar n)=\bar n-2$
and so the case when~$k=4$ and~$n$ is even requires a separate treatment.  But then we
have
\begin{align*}
w_\circ^{[1,n]}\star w_\circ^{[4,n+1]}
\star w_\circ^{[1,n]}&=
\cxr1n \star (s_{n+1}w_\circ^{[2,n+1]})\star \cx1n=w_{[n-1,n];[n-1,n+1]}w_\circ^{[1,n+1]}
\end{align*}
as shown above. Since~$l(0,4)=3$, the assertion follows
in this case. Thus, we assume that either~$k>4$
or~$k=4$ and~$n$ is odd. Then
\begin{align*}
w_\circ^{[1,n]}\star w_\circ^{[k,n+1]}
\star w_\circ^{[1,n]}
&=\cxr1n \star  (w_{[n+3-l,n];[n+3-l,n+1]}w_\circ^{[2,n+1]})\star \cx1n,
\end{align*}
where we abbreviate~$l=l(\bar n,k)\ge 4$.
Write~$w_\circ^{[1,n+1]}=u\times (w_{[n+3-l,n];[n+3-l,n+1]}w_\circ^{[2,n+1]})$. Then
by~\eqref{eq:w0 D rec} and~\eqref{eq:w0 type A}
\begin{align*}
u&=w_\circ^{[1,n+1]}w_\circ^{[2,n+1]}
w_{[n+3-l,n];[n+3-l,n+1]}{}^{-1}\\
&=\cx1{(n+1)}\cxr1{(n-1)} w_\circ^{[n+3-l,n+1]}
w_\circ^{[n+3-l,n]}=\cx1{(n+1-l)} w_\circ^{[n+2-l,n+1]}w_\circ^{[n+3-l,n]}
\cxr1{(n+1-l)}\\
&=\cx1{(n+1-l)}\times (w_\circ^{[n+2-l,n+1]}
w_\circ^{[n+2-l,n]})\times \cxr 1n,
\end{align*}
where we replace the usual product in the
Coxeter group with~$\times$ in the last equality by the comparison of lengths.
Therefore,
by Lemma~\partref{lem: u * u^(-1) w_0.a}
\begin{align*}
w_\circ^{[1,n]}&\star w_\circ^{[k,n+1]}
\star w_\circ^{[1,n]}
=\cxr1n \star (u^{-1}w_\circ^{[1,n+1]})\star \cx1n\\
&
=(w_{[n+2-l,n];[n+2-l,n+1]}\cxr1{(n+1-l)}w_\circ^{[1,n+1]})\star\cx1n\\
&=(w_{[n+2-l,n];[n+2-l,n+1]}w_\circ^{[1,n+1]}\cxr1{(n+1-l)})\star\cx1n\\
&=(w_{[n+2-l,n];[n+2-l,n+1]}w_\circ^{[1,n+1]})
\star \cx{(n+2-l)}n.
\end{align*}
Note that~$(W_{[n+2-l,n+1]}(D_{n+1}),\star)
\cong (W(D_l),\star)$. Also,
$l=k-\overline{k+n+1}$
has the same parity as~$n+1$. Thus, if~$n$ is odd
then $w_\circ^{[n+1]}$ is central in~$(W(D_{n+1}),\star)$, $w_\circ^{[n+2-l,n+1]}$ is central in~$(W_{[n+2-l,n+1]}(D_{n+1}),\star)$ and so
$$
w_{[n+2-l,n];[n+2-l,n+1]}w_\circ^{[1,n+1]}
=w_\circ^{[1,n+1]}w_\circ^{[n+2-l,n+1]}w_\circ^{[n+2-l,n]}
=w_\circ^{[1,n+1]}w_{[n+2-l,n];[n+2-l,n+1]}{}^{-1}.
$$
If~$n$ is even then
\begin{align*}
w_{[n+2-l,n];[n+2-l,n+1]}w_\circ^{[1,n+1]}
&=w_\circ^{[1,n+1]}w_\circ^{[n+2-l,n-1]\cup\{n+1\}}
w_\circ^{[n+2-l,n+1]}\\&=
w_\circ^{[1,n+1]}w_\circ^{[n+2-l,n+1]}
w_\circ^{[n+2-l,n]}=
w_\circ^{[1,n+1]}w_{[n+2-l,n];[n+2-l,n+1]}{}^{-1}.
\end{align*}
In either case, $D_R(w_{[n+2-l,n];[n+2-l,n+1]}w_\circ^{[1,n+1]})=[1,n]$
by Lemmata~\ref{lem:DL uw0} and~\ref{lem:wK absorption} and hence contains~$[n+2-l,n]$. Then
$
(w_{[n+2-l,n];[n+2-l,n+1]}w_\circ^{[1,n+1]})
\star \cx{(n+2-l)}n=w_{[n+2-l,n];[n+2-l,n+1]}w_\circ^{[1,n+1]}$
by Lemma~\ref{lem:left desc},
which completes the proof of the inductive step.

We now prove that~$\mu_{D_{n+1}}([k,n+1],[1,n])\ge 4$ for all~$2\le k\le n$.
The argument is rather similar to that in type~$B$.
Writing~$w_\circ^{[1,n]}=\cxr1n\star w_\circ^{[2,n]}$
we obtain by \eqref{eq:w0 D rec}, Lemmata~\ref{lem:char w_0 monoid}, \partref{lem: u * u^(-1) w_0.a'} and~\ref{lem:left desc}
\begin{align}
w_\circ^{[2,n+1]}\star
w_\circ^{[1,n]}\star w_\circ^{[2,n+1]}
&=w_\circ^{[2,n+1]}\star \cxr1n\star w_\circ^{[2,n+1]}
=w_\circ^{[2,n+1]}\star s_1\star w_\circ^{[2,n+1]}\nonumber\\
&=w_\circ^{[2,n+1]}\star (s_1 w_\circ^{[2,n+1]})
=w_\circ^{[2,n+1]}\star (\cx2{(n+1)}\cxr{1}{(n-1)}
w_\circ^{[1,n+1]})\nonumber\\
&=w_\circ^{[2,n+1]}\star (\cxr1{(n-1)}w_\circ^{[1,n+1]})
=s_1 w_\circ^{[1,n+1]},\label{eq:[2,n+1]*[1,n]*[2,n+1]}
\end{align}
where we used that~$D_L(s_1w_\circ^{[1,n+1]})=[2,n+1]$ and that~$s_1$ commutes with~$w_\circ^{[1,n+1]}$.
Thus, $\mu_{D_{n+1}}([2,n+1],[1,n])=4$
by Lemma~\partref{lem: u * u^(-1) w_0.c}.
Now, if
$\mu_{D_{n+1}}([k,n+1],[1,n])\le 3$
for some~$k\in[3,n]$
then, since $w_\circ^{[k,n+1]}\le w_\circ^{[2,n+1]}$, we obtain by Proposition~\partref{prop:Bruhat order *.a}
and~\eqref{eq:[2,n+1]*[1,n]*[2,n+1]}
$$
w_\circ^{[1,n+1]}=
w_\circ^{[k,n+1]}\star
w_\circ^{[1,n]}\star w_\circ^{[k,n+1]}
\le w_\circ^{[2,n+1]}\star
w_\circ^{[1,n]}\star w_\circ^{[2,n+1]}=s_1w_\circ^{[1,n+1]}<w_\circ^{[1,n+1]},
$$
which is a contradiction.

We now prove that~$\mu_{D_{n+1}}([1,n],[k,n+1])=4$
if and only if~$3-\bar n\le
k\le \floor{\frac14(n+2)}+
\ceil{\frac14(n+2)}$. The lower
bound is immediate. For the upper,
note first that~$k\le n+2-l(\bar n,k)$
is equivalent to~$k\le \floor{\frac14(n+2)}+
\ceil{\frac14(n+2)}$. Indeed, if we write~$n=4m+r$, $r\in\{-1,0,1,2\}$, $m\in\ZZ_{>0}$, then
$$
\floor{\tfrac14(n+2)}+
\ceil{\tfrac14(n+2)}=2m+1+\delta_{2,r}.
$$
If~$k=3$,
then~$n+2-l(\bar n,3)=n-1$ if~$n$ is even
and~$n+1$ if~$n$ is odd. Since~$n\ge 3$,
$3\le n+2-l(\bar n,3)$ and also
$3\le \floor{\frac14(n+2)}+\ceil{\frac14(n+2)}$.
For~$k>4$, we have~$l(\bar n,k)=
k-\overline{k+n+1}$ and so the inequality becomes~$2k\le n+2+\overline{n+k+1}
=n+3-\overline{n+k}$
or~$k\le 2m+1+\frac12(1+r-\overline{r+k})$.
It is now easy to see that the precise
upper bound for~$k$ is~$2m+1+\delta_{r,2}$.

By~\eqref{eq: first braid D}, $\mu_{D_{n+1}}([1,n],[k,n+1])>3$ if~$k\ge 3-\bar n$.
If~$k\le \floor{\frac14(n+2)}+
\ceil{\frac14(n+2)}$ then
$[n+2-l(\bar n,k),n+1]\subset [k,n+1]$
and so~$\mu_{D_{n+1}}([1,n],[k,n+1])=4$
by Lemma~\partref{lem: u * u^(-1) w_0.c} and~\eqref{eq: first braid D}. Applying~${}^{op}$ and taking into account
that~$\mu_{D_{n+1}}([k,n+1],[1,n])\ge 4$,
we conclude that~$\mu_{D_{n+1}}([k,n+1],[1,n])=4$.

Suppose that~$k>\floor{\frac14(n+2)}+
\ceil{\frac14(n+2)}$, that is,
$J'=[k,n+1]\subset [n+2-l,n+1]=K$, where we
abbreviate~$l=l(\bar n,k)$. Let~$J=[n+2-l,n]$.
Then~$J'\cap J=[k,n]$ and so
$$
\ell(w_{J;K})-\ell(w_{J'\cap J;J'})=
\ell(w_\circ^{K})-\ell(w_\circ^J)-
(\ell(w_\circ^{J'})-\ell(w_\circ^{J'\cap J}))=\tfrac12 (k-(n+2-l))(n+l-k+1)>0
$$
since~$k>n+2-l$ and~$l\in \{k,k-1\}$.
Then~$\mu_{D_{n+1}}([k,n+1],[1,n])>4$ by
Lemma~\ref{lem:too short}.
It follows from~\eqref{eq:mu_M JK KJ} that $\mu_{D_{n+1}}([1,n],[k,n+1])\ge 4$
and
the equality,
by applying~${}^{op}$, would yield
$\mu_{D_{n+1}}([k,n+1],[1,n])\le 4$ which is
a contradiction.
\end{proof}
The first assertion is immediate from Lemma~\partref{lem:first braid B2 Dn+1.b}
and Theorem~\ref{thm:Hom Heck Mon}.
The injectivity follows from Proposition~\ref{prop:locally inj}, and the parabolicity for~$k=2$ is obvious.

For the converse, suppose that~$\phi\in\Hom_{\He}(B_2,D_{n+1})$
is injective, connected and not~$\sigma$-invariant. Then
we have at least one~$i\in\{1,2\}$ such that~$\{n,n+1\}\not\subset [\phi](i)$.
If both~$[\phi](1)$ and~$[\phi](2)$
have that property then we must have
$[\phi](1)=[s,n]$, $[\phi](2)=\sigma([r,n])$
for some~$r,s\in[1,n]$. Since~$\phi$
is fully supported, by applying diagram automorphisms we may assume that~$[\phi](1)=[1,n]$ and~$[\phi](2)=\sigma([r,n])$
for some~$r\in [1,n]$. But
then $\max(\mu_{D_{n+1}}([\phi](1),[\phi](2)),
\mu_{D_{n+1}}([\phi](2),[\phi](1)))>4$
by Lemma~\partref{lem:first braid B2 Dn+1.a},
which is a contradiction by Theorem~\ref{thm:Hom Heck Mon}.
If one of the~$[\phi](i)$ is $\sigma$-invariant
then it cannot contain~$1$ and~$n,n+1$
at the same time, since otherwise
the image of~$s'_i$ is~$w_\circ^{[1,n+1]}$ and~$\phi$ cannot possibly be injective.
Thus, in that case, without loss of generality,
we must have~$[\phi](1)=[1,n]$ and~$[\phi](2)=
[r,n+1]$ for some~$r\in[2,n-1]$. It
remains to apply Lemma~\partref{lem:first braid B2 Dn+1.b}.
\end{proof}

\begin{proposition}\label{prop:negative}
A connected fully supported homomorphism between monoids
$(W(A_r),\star)$, $r\ge 2$ and $(W(D_{n+1}),\star)$, $n\ge 3$
cannot be locally injective.
\end{proposition}
\begin{proof}
We need the following
\begin{lemma}\label{lem:negative rank 2 reduction}
Let~$\phi\in\Hom_{\He}(A_r,D_{n+1})$, $r\ge 2$, $n\ge 3$, be locally injective,
connected and fully supported.
Then~$\{n-2,n,n+1\}\subset[\phi](\{i,i'\})$ for some~$i,i'\in[1,r]$
with~$m'_{ii'}=3$.
In particular, $[\phi](\{i,i'\})$
is a connected subset of~$[1,n+1]$
of type~$D_{k+1}$, $k\ge 3$.
\end{lemma}
\begin{proof}
Since~$\phi$ is locally injective,
$[\phi](i)\not\subset[\phi](j)$
for all~$i\not=j\in [1,r]$.
Since~$\phi$ is fully supported,
there exist $i,j\in[1,r]$ such that~$n\in [\phi](i)$ and~$n+1\in[\phi](j)$.

If~$i=j$ then, since~$[\phi](i)$ is connected, $\{n-1,n,n+1\}\subset [\phi](i)$.
If~$n-2\in[\phi](i)$ then we can take any~$i'\in [1,r]$ with~$m'_{ii'}=3$. Otherwise, $[\phi](i)=\{n-1,n,n+1\}$;
let~$i'\in[1,r]$ with~$m'_{ii'}=3$. Then~$[\phi](i')\not\subset[\phi](i)$ and~$[\phi](i')$
is not orthogonal to~$[\phi](i)$,
since otherwise~$\mu_{D_{n+1}}([\phi](i),[\phi](i'))\le 2$
by Proposition~\ref{prop:comm idemp}, which contradicts the local injectivity of~$\phi$ by Proposition~\ref{prop:locally inj}.
Since~$[\phi](i')$ is connected and~$n-2$ is the only vertex of~$\Gamma(D_{n+1})$ outside of~$\{n-1,n,n+1\}$ and adjacent to one of its vertices, it follows that~$n-2\in[\phi](i')$.

Suppose that~$i\not=j$. If~$m'_{ij}=2$ then, by Proposition~\ref{prop:comm idemp},
Corollary~\ref{cor:hom A1 x A1}
and the local injectivity of~$\phi$, $[\phi](i)$ and~$[\phi](j)$ must be orthogonal. Then neither of them contains~$n-1$ and so, since~$[\phi](i)$ and~$[\phi](j)$ are connected, $[\phi](i)=\{n\}$ and~$[\phi](j)=\{n+1\}$. Let~$i'\in[1,r]$ with~$m'_{ii'}=3$.
Since~$\mu_{D_{n+1}}([\phi](i'),[\phi](i))=3$ by Proposition~\ref{prop:locally inj}, the sets~$[\phi](i')$ and~$\{n\}$ are neither orthogonal nor nested. Hence~$n\notin[\phi](i')$
and~$n-1\in[\phi](i')$, since~$n-1$ is the only vertex of~$\Gamma(D_{n+1})$ adjacent to~$n$. If~$m'_{i'j}=2$ then~$[\phi](i')$ would be orthogonal to~$\{n+1\}$,
which is impossible since~$n-1\in[\phi](i')$; thus~$m'_{i'j}=3$, and, as above, $n+1\notin[\phi](i')$.
Since~$[\phi](i')$ is connected, $[\phi](i')=[k,n-1]$ for some~$k\le n-1$.
Identify~$(W_{\{i,i'\}}(A_r),\star)$
with~$(W(A_2),\star)$ via
$s'_{i'}\mapsto s'_1$, $s'_{i}\mapsto s'_2$
and $(W_{[k,n]}(D_{n+1}),\star)$
with~$(W(A_{n-k+1}),\star)$ via
$s_t\mapsto s_{t-k+1}$, $t\in[k,n]$. Then by Lemma~\partref{lem:mu J1 J2.b},
$[\phi](i')=k-1+[1,\nu_1+\nu_3-1]$
and~$[\phi](i)=k-1+[\nu_3+1,
\nu_1+\nu_2+\nu_3-1]$
for some~$(\nu_1,\nu_2,\nu_3)\in\ZZ_{>0}^3$ with~$\nu_1\ge\max(\nu_2,\nu_3)$.
Since~$[\phi](i)=\{n\}$, we have~$\nu_1+\nu_2=2$, whence
$\nu_1=\nu_2=\nu_3=1$, and so~$n-1=\max[\phi](i')=k-1+\nu_1+\nu_3-1=k$. Thus, $[\phi](i')=\{n-1\}$.
Since~$\phi$ is fully supported,
there is~$j'\in [1,r]\setminus\{i,j,i'\}$ with~$n-2\in [\phi](j')$ (in particular, $r>3$). Since the only
vertices adjacent to~$i'$ are~$i$ and~$j$, $m'_{i'j'}=2$ and so~$[\phi](j')$ must be orthogonal to~$[\phi](i')=\{n-1\}$ by Corollary~\ref{cor:hom A1 x A1},
which is a contradiction since~$n-2\in[\phi](j')$.

Finally, suppose that~$m'_{ij}=3$.
In particular, $[\phi](i)$ and~$[\phi](j)$ are not orthogonal and so~$n-1\in [\phi](i)\cup[\phi](j)$. If~$n-2\in[\phi](i)\cup[\phi](j)$ then we can take~$i'=j$. Otherwise,
$[\phi](i)\cup[\phi](j)=\{n-1,n,n+1\}$, and, since~$\phi$ is fully supported, there exists~$i'\in[1,r]\setminus\{i,j\}$
such that~$n-2\in [\phi](i')$. Since~$i'$ is adjacent to at most one of~$i$, $j$, $[\phi](i')$ is orthogonal to at least one
of~$[\phi](i)$, $[\phi](j)$ by Corollary~\ref{cor:hom A1 x A1}, and not to both since~$n-1\in[\phi](i)\cup[\phi](j)$ is adjacent to~$n-2$.
Applying the diagram automorphism of~$D_{n+1}$ if necessary, we may assume that~$m'_{ii'}=3$ and that~$[\phi](i')$ is orthogonal to~$[\phi](j)$.
Then~$n-1\notin[\phi](j)$, whence~$[\phi](j)=\{n+1\}$, $[\phi](i)=\{n-1,n\}$, and, since~$[\phi](i')$ is connected and orthogonal to~$\{n+1\}$,
$[\phi](i')=[k,n-2]$ for some~$k\le n-2$.
Arguing as in the previous case,
we can write~$[k,n-2]=k-1+[1,\nu_1+\nu_3-1]$,
$\{n-1,n\}=k-1+[\nu_3+1,\nu_1+\nu_2+\nu_3-1]$
where~$(\nu_1,\nu_2,\nu_3)\in\ZZ_{>0}^3$ and~$\nu_1\ge\max(\nu_2,\nu_3)$.
Note that~$[1,\nu_1+\nu_3-1]\cap
[\nu_3+1,\nu_1+\nu_2+\nu_3-1]=
[\nu_3+1,\nu_1+\nu_3-1]$ must be empty.
Then~$\nu_1+\nu_2=3$, whence~$\nu_1=2$,
and so~$[\phi](i')\cap[\phi](i)=k+[\nu_3,\nu_1+\nu_3-2]$ is non-empty, which is a contradiction.
\end{proof}
Suppose that~$\phi\in\Hom_{\He}(A_r,D_{n+1})$, $r\ge 2$, $n\ge 3$, is locally injective, connected and fully supported.
By Lemma~\ref{lem:negative rank 2 reduction},
there exist~$i,i'\in[1,r]$
such that~$(W_{\{i,i'\}}(A_r),\star)\cong (W(A_2),\star)$
and $(W_{[\phi](\{i,i'\})}(D_{n+1}),\star)\cong
(W(D_k),\star)$
where~$k=|[\phi](\{i,i'\})|\ge 4$.
Therefore, $\phi$ yields
an injective connected fully
supported homomorphism from
$(W(A_2),\star)$ to~$(W(D_k),\star)$. Therefore,
it suffices to consider the case~$r=2$.

If~$\phi$ is $\sigma$-invariant then, by Theorem~\ref{thm:adm finite class}, it is a composition of~\eqref{eq:unfold Bn Dn+1} with an injective
connected homomorphism from $(W(A_2),\star)$ to
$(W(B_n),\star)$, which is fully supported since~$\phi$ is. Such a homomorphism
does not exist by Lemma~\ref{lem:B2 An do not exist}.
If~$\phi$ is not $\sigma$-invariant, the case~$n=3$ was already discussed in Lemma~\ref{lem:B2 to D4}, while for~$n\ge 4$
the assertion follows from Lemma~\ref{lem:first braid B2 Dn+1}.

We now turn our attention to
homomorphisms from~$(W(D_{n+1}),\star)$ to~$(W(A_r),\star)$. Since
every connected subset of~$[1,r]$ is of type~$A$, it suffices to consider the case~$n=3$.
Let~$\phi\in\Hom_{\He}(D_4,A_r)$
be locally injective, connected
and fully supported. We claim that~$1\in[\phi](i)$ for some~$i\in\{1,3,4\}$. Indeed, otherwise~$[\phi](2)=[1,c]$ for some~$c$, and,
for each~$i\in\{1,3,4\}$, $[\phi](i)$ is neither orthogonal to~$[1,c]$ nor contained in it by local injectivity of~$\phi$,
Proposition~\ref{prop:comm idemp}
and Corollary~\ref{cor:hom A1 x A1}. Then~$c+1\in[\phi](i)$ for all~$i\in\{1,3,4\}$,
which contradicts the fact that the~$[\phi](i)$, $i\in\{1,3,4\}$, are pairwise orthogonal by Corollary~\ref{cor:hom A1 x A1}.
Using diagram automorphisms of~$D_4$, we may assume, without loss of generality, that~$1\in[\phi](1)$.

Then,
since~$(W_{[1,3]}(D_4),\star)
\cong (W(A_3),\star)$, by Theorem~\ref{thm:Ak to An} and Lemma~\ref{lem:diag aut J_i}, $[\phi](i)=J_i(\boldsymbol\lambda)$, $i\in[1,3]$,
for some~$\boldsymbol{\lambda}=(\lambda_1,\lambda_2,\lambda_3,\lambda_4)\in\ZZ_{>0}^4$
which is in~$\mathcal A_4(t+1)$
for some~$t\le r$ up to interchanging~$\lambda_3$ and~$\lambda_4$; in particular, $\lambda_1=\lambda_2$. Let~$[\phi](4)=[a,b]$.
Since~$[\phi](1)$ and~$[\phi](4)$ must be orthogonal by Corollary~\ref{cor:hom A1 x A1},
$a\ge \max J_1(\boldsymbol\lambda)+2
=\lambda_1+\lambda_4+1$. Consider
now the restriction of~$\phi$ to~$(W_{\{2,4\}}(D_4),\star)$, which is isomorphic to~$(W(A_2),\star)$. Write~$J_2(\boldsymbol\lambda)=\lambda_4+[1,\lambda_1+\lambda_2-1]=\lambda_4+J_1(\boldsymbol\nu)$
and~$[\phi](4)=\lambda_4+J_2(\boldsymbol\nu)$,
where~$\boldsymbol{\nu}=(\nu_1,\nu_2,\nu_3)=(\lambda_1+\lambda_2+\lambda_4-a+1,
1+b-\lambda_1-\lambda_2-\lambda_4,a-\lambda_4-1)$.
By Lemma~\partref{lem:mu J1 J2.b} we must
have~$\nu_1\ge\max(\nu_2,\nu_3)$.
In particular,
$\lambda_1+\lambda_2+\lambda_4-a+1\ge
a-\lambda_4-1$, which, since~$\lambda_1=\lambda_2$, implies that~$a=\lambda_1+\lambda_4+1$ and so~$[\phi](4)=[\lambda_1+\lambda_4+1,b]$.
Since~$[\phi](3)=J_3(\boldsymbol\lambda)=[\lambda_1+\lambda_4+1,t]$,
either~$[\phi](3)\subset [\phi](4)$
or~$[\phi](4)\subset [\phi](3)$,
which contradicts the local injectivity of~$\phi$ by Corollary~\ref{cor:hom A1 x A1}.
\end{proof}
\plink{Dr(m)}Define, for all $r\in\ZZ_{>1}$, $m\in\ZZ_{\ge 0}$,
$$
\mathcal D_r(m)=\{
\boldsymbol\lambda=(\lambda_1,\dots,\lambda_r)\in\mathcal A_r(m)\,:\, \lambda_1=\cdots=
\lambda_{r-1}\in 2\ZZ_{>0}\}.
$$
This set is in bijection, via transposition of partitions, with the
set of all partitions of~$m$
with parts in~$\{r,r-1\}$ such that
the total number of parts is even and
at least one part is equal to~$r$.
In particular,
$$
\sum_{m\ge 0} |\mathcal D_r(m)|x^m=
\frac{x^{2r-1}(1+x)}{(1-x^{2r})(1-x^{2r-2})},\qquad r\ge 2.
$$

\begin{theorem}\label{thm:Br D_n+1}
Let~$r\ge 3$, $n\ge 4$ and let~$\boldsymbol\lambda=(\lambda_1,\dots,\lambda_r)\in\mathcal D_r(n+1)$. Then the assignments
$$
s'_i\mapsto w_\circ^{J_i(\boldsymbol\lambda)},\quad i\in[1,r-1],\qquad
s'_r\mapsto \begin{cases}
s_{n+1},&\lambda_1=2,\\
w_\circ^{[n+2-\lambda_1,n+1]},&\lambda_1\ge 4,
\end{cases}
$$
define a locally injective connected fully supported homomorphism~\plink{eta l}$\eta_{\boldsymbol\lambda}:(W(B_r),\star)\to
(W(D_{n+1}),\star)$ which is not $\sigma$-invariant. Conversely, up to
diagram automorphism, every locally injective connected fully supported homomorphism
from~$(W(B_r),\star)$ to~$(W(D_{n+1}),\star)$ which is not $\sigma$-invariant
is of the form~$\eta_{\boldsymbol\lambda}$ for some~$\boldsymbol\lambda\in\mathcal D_r(n+1)$.
\end{theorem}

\begin{proof}
Let~$\boldsymbol\lambda=(\lambda_1,\dots,\lambda_1,\lambda_r)\in\mathcal D_r(n+1)$.
Since~$\boldsymbol\lambda\in\mathcal A_r(n+1)$, by Theorem~\ref{thm:Ak to An} the assignments $s'_i\mapsto w_\circ^{J_i(\boldsymbol{\lambda})}$, $i\in[1,r-1]$, define a fully supported
connected locally injective homomorphism
$\phi:(W(A_{r-1}),\star)\cong
(W_{[1,r-1]}(B_r),\star)\to
(W(A_n),\star)\cong (W_{[1,n]}(D_{n+1}),\star)$. In particular,
$[\phi](r-1)=[n+2-2\lambda_1,n]$.
Since~$\lambda_1\ge 2$,
$n+2-2\lambda_1\le n-2$ and so $(W_{[n+2-2\lambda_1,n+1]}(D_{n+1}),\star)\cong
(W(D_{2\lambda_1}),\star)$ via
$s_i\mapsto s_{i-n+2\lambda_1-1}$, $i\in[n+2-2\lambda_1,n+1]$.
Let~$S(\lambda_1)=\{n+1\}$ if~$\lambda_1=2$
and~$S(\lambda_1)=[n+2-\lambda_1,n+1]$ if~$\lambda_1\ge 4$.
If~$\lambda_1=2$, it follows from Lemma~\ref{lem:B2 to D4} that
$\mu_{D_{n+1}}([\phi](r-1),S(\lambda_1))=
\mu_{D_{n+1}}(S(\lambda_1),[\phi](r-1))
=\mu_{D_4}([1,3],\{4\})=4$.
Otherwise, $S(\lambda_1)=n-2\lambda_1+1+
[\lambda_1+1,2\lambda_1]$.
Since~$2<\lambda_1+1=
\lceil \frac14(2\lambda_1+1)\rceil+
\lfloor \frac14(2\lambda_1+1)\rfloor$ and~$2\lambda_1-1$ is odd,
it follows from Proposition~\ref{prop:B2 Dn+1} that
\begin{align*}
\mu_{D_{n+1}}([\phi](r-1),S(\lambda_1))&=\mu_{D_{n+1}}(S(\lambda_1),[\phi](r-1))=
\mu_{D_{2\lambda_1}}([1,2\lambda_1-1],
[\lambda_1+1,2\lambda_1])=4.
\end{align*}
Since~$\bigcup_{1\le i\le r-2} J_i(\boldsymbol{\lambda})=[1,n-\lambda_1]$ and~$\min S(\lambda_1)\ge n+2-\lambda_1$,
the set~$S(\lambda_1)$ is orthogonal to the~$J_i(\boldsymbol\lambda)$, $i\in[1,r-2]$, and so
$\mu_{D_{n+1}}(J_i(\boldsymbol\lambda),S(\lambda_1))=
\mu_{D_{n+1}}(S(\lambda_1),J_i(\boldsymbol\lambda))=2$, for all~$i\in[1,r-2]$. Thus, by Theorem~\ref{thm:Hom Heck Mon} and Proposition~\ref{prop:locally inj},
the assignments $s'_i\mapsto w_\circ^{J_i(\boldsymbol\lambda)}$, $i\in[1,r-1]$, $s'_r\mapsto w_\circ^{S(\lambda_1)}$, define a
locally injective homomorphism in~$\Hom_{\He}(B_r,D_{n+1})$. It is fully supported since~$[1,n]\cup S(\lambda_1)=[1,n+1]$,
and connected since all~$J_i(\boldsymbol\lambda)$ and~$S(\lambda_1)$ are connected and so is~$J_{r-1}(\boldsymbol\lambda)\cup S(\lambda_1)=[n+2-2\lambda_1,n+1]$.
Since~$n\in J_{r-1}(\boldsymbol\lambda)$ and~$n+1\notin J_{r-1}(\boldsymbol\lambda)$, it is not $\sigma$-invariant.

Conversely, let~$\phi\in\Hom_{\He}(B_r,D_{n+1})$ be locally injective, connected, fully supported
and not~$\sigma$-invariant. Then its restriction to~$(W_{[1,r-1]}(B_r),\star)\cong(W(A_{r-1}),\star)$ is a locally injective,
connected and fully supported homomorphism
to $(W_J(D_{n+1}),\star)$ where~$J=[\phi]([1,r-1])$
is connected. Since~$r-1\ge 2$, by Proposition~\ref{prop:negative}
$W_J(D_{n+1})$ must be of type~$A$.

Suppose
first that~$J$ is~$\sigma$-invariant. Then
$[\phi](r)$ cannot be~$\sigma$-invariant, for otherwise so would be~$\phi$, and we may assume without loss of generality
that~$n\in [\phi](r)$ and~$n+1\notin[\phi](r)$.
Since~$\phi$ is fully supported, $n+1\in J$ and, since~$J$ is $\sigma$-invariant, connected and of type~$A$,
$J=\{n-1,n,n+1\}$. Since~$[\phi](r)$ is connected, $n\in[\phi](r)$ and
$\phi$ is fully supported, it follows that
$[\phi](r)=[1,n]$. Furthermore, $\mu_{D_{n+1}}([\phi](r),[\phi](1))=\mu_{D_{n+1}}([\phi](1),[\phi](r))=2$ by Proposition~\ref{prop:locally inj}
and
$[\phi](1)\subset \{n-1,n,n+1\}$ is connected, hence not orthogonal to~$[\phi](r)\ni n-1$. Therefore, by Proposition~\ref{prop:comm idemp},
one of~$[\phi](1)$, $[\phi](r)$ is contained in the other, and so~$[\phi](1)\subset[\phi](r)$.
But then~$\mu_{D_{n+1}}([\phi](r),[\phi](1))=1$, which contradicts the
local injectivity of~$\phi$.

We conclude that~$J$ is not $\sigma$-invariant.
In particular, exactly one of~$n$, $n+1$ is in~$J$ and we may assume without loss of generality that~$n\in J$.
Since~$\phi$ is fully supported, $n+1\in[\phi](r)$.
Furthermore, $J=[1,n]$. Indeed, $J$ is connected, contains~$n$ but not~$n+1$, and~$1\notin[\phi](r)$, since otherwise~$[\phi](r)$,
being connected and containing~$1$ and~$n+1$, would contain~$[\phi](r-1)$, contradicting the local injectivity of~$\phi$ by Proposition~\ref{prop:locally inj}. Therefore, $1\in J$ since~$\phi$ is fully supported, and so~$J=[1,n]$
since~$\phi$ is connected.

By Theorem~\ref{thm:Ak to An},
there exists~$\boldsymbol\lambda\in \mathcal A_r(n+1)$ such that
the restriction of~$[\phi]$ to
$[1,r-1]$ is one of the four
maps~\eqref{eq:cs 1}--\eqref{eq:cs 4} with~$b=n$. Since~$[\phi](r-1)$
cannot be orthogonal to~$[\phi](r)$
by the local injectivity of~$\phi$
and Proposition~\ref{prop:locally inj}, we must have~$n-1\in [\phi](r-1)$. In cases~\eqref{eq:cs 2} and~\eqref{eq:cs 3}, $1\in[\phi](r-1)$, which forces
$[\phi](r-1)\supset [1,n-1]$.
Furthermore, $n\notin[\phi](r-1)$,
since otherwise~$[\phi](r-1)=[1,n]=J\supset[\phi](r-2)$, which contradicts the local injectivity of~$\phi$ by Proposition~\ref{prop:locally inj}.
Thus, $[\phi](r-1)=[1,n-1]$. Since~$n\in [\phi](1)$ in these cases, $r=3$, for otherwise~$[\phi](1)$ and~$[\phi](r-1)$
must be orthogonal by Corollary~\ref{cor:hom A1 x A1}, which is impossible since~$n\in[\phi](1)$ is adjacent to~$n-1\in[\phi](r-1)$.
Since~$[\phi](1)$ must also be orthogonal to~$[\phi](3)\ni n+1$, we have~$n-1\notin[\phi](1)$, whence~$[\phi](1)=\{n\}$. Yet the assignments
$s'_1\mapsto s_n$,
$s'_2\mapsto w_\circ^{[1,n-1]}$
define an injective fully supported homomorphism
$(W(A_2),\star)\to (W(A_n),\star)$
if and only if~$n=2$ by Proposition~\ref{prop:A2->An}, which is a contradiction since~$n\ge 4$.
Thus, $[\phi]|_{[1,r-1]}$ is
either~\eqref{eq:cs 1} or~\eqref{eq:cs 4} and so
$[\phi](i)=J_i(\boldsymbol\lambda)$,
$i\in[1,r-1]$, by Lemma~\ref{lem:diag aut J_i}, where
$\boldsymbol\lambda\in\mathcal A_r(n+1)$ up to interchanging~$\lambda_r$ and~$\lambda_{r-1}$. In particular,
$\lambda_1\ge\max(\lambda_{r-1},\lambda_r)$, $(r-2)\lambda_1+\lambda_{r-1}+\lambda_r=n+1$,
$[\phi](r-1)=[n+2-\lambda_1-\lambda_{r-1},n]$ and~$[\phi](r-2)=[\max(1,n+2-2\lambda_1-\lambda_{r-1}),\,n-\lambda_{r-1}]$.
Since~$[\phi](r-2)$ must be
orthogonal to~$[\phi](r)\ni n+1$ by
Corollary~\ref{cor:hom A1 x A1},
$n-1\notin[\phi](r-2)$, that is, $\lambda_{r-1}\ge 2$, which forces~$\lambda_1\ge 2$.

Let~$k=\lambda_1+\lambda_{r-1}-1=|[\phi](r-1)|\ge 3$. Then~$L:=[n+1-k,n+1]=[\phi](r-1)\cup\{n+1\}$ is of type~$D_{k+1}$, and we
identify~$(W_L(D_{n+1}),\star)$ with~$(W(D_{k+1}),\star)$ via~$s_i\mapsto s_{i-n+k}$, $i\in L$; then~$[\phi](r-1)$
corresponds to~$[1,k]$.

Suppose first that~$[\phi](r)$ is not $\sigma$-invariant. Since~$n+1\in[\phi](r)$ and~$[\phi](r)$ is connected, $[\phi](r)=\sigma([l,n])$
for some~$l\le n$, and~$l\ge n-\lambda_{r-1}+2\ge n+1-k$ since~$[\phi](r)$ must be orthogonal to~$[\phi](r-2)$ by the local injectivity of~$\phi$ and Corollary~\ref{cor:hom A1 x A1}. Thus, $[\phi](r)\subset L$ corresponds
to~$\sigma_k([l-n+k,k])$, where~$\sigma_k$ is the diagram automorphism of~$\Gamma(D_{k+1})$.
If~$k=3$ then~$\lambda_1=\lambda_{r-1}=2$, $l=n$, that is, $[\phi](r)=\{n+1\}$, and by Lemma~\ref{lem:B2 to D4},
$$
\mu_{D_{n+1}}([\phi](r-1),[\phi](r))=\mu_{D_{n+1}}([\phi](r),[\phi](r-1))=\mu_{D_4}([1,3],\{4\})=4.
$$
Thus,
$\phi=\eta_{\boldsymbol\lambda}$
with~$\boldsymbol{\lambda}\in\mathcal D_{r}(n+1)$ and~$\lambda_1=2$.
If~$k\ge 4$ then
Lemma~\partref{lem:first braid B2 Dn+1.a}
yields
\begin{align*}
\min(\mu_{D_{n+1}}([\phi](r-1),&[\phi](r)),\mu_{D_{n+1}}([\phi](r),[\phi](r-1)))\\
&=
\min(\mu_{D_{k+1}}([1,k],\sigma_k([l-n+k,k])),
\mu_{D_{k+1}}(\sigma_k([l-n+k,k]),[1,k]))>4,
\end{align*}
which is a contradiction.

Suppose now that~$[\phi](r)$ is $\sigma$-invariant. Then~$\{n,n+1\}\subset[\phi](r)$ and so~$[\phi](r)=[a,n+1]$ with~$a\le n-1$,
since~$[\phi](r)$ is connected. As before, $a\ge n-\lambda_{r-1}+2$ by the orthogonality of~$[\phi](r)$ and~$[\phi](r-2)$, whence~$[\phi](r)\subset L$
corresponds to~$[a',k+1]$, where~$a'=a-n+k\ge k-\lambda_{r-1}+2=\lambda_1+1$. By Proposition~\ref{prop:B2 Dn+1},
$$
\lambda_1+1\le a'\le \lceil\tfrac14(k+2)\rceil+\lfloor\tfrac14(k+2)\rfloor\le \tfrac12(k+3)=\tfrac12(\lambda_1+\lambda_{r-1}+2),
$$
and if~$k$ is even then~$a'>2$.
But the above inequality implies that~$\lambda_1\le\lambda_{r-1}$ whence~$\lambda_{r-1}=\lambda_1$
and $k=2\lambda_1-1$ is odd. Then
$$\lceil\tfrac14(k+2)\rceil+\lfloor\tfrac14(k+2)\rfloor=
\lambda_1-\bar\lambda_1+
\lceil\tfrac14+\tfrac12\bar\lambda_1\rceil+\lfloor\tfrac14+\tfrac12\bar\lambda_1\rfloor=\lambda_1+1-\bar\lambda_1.
$$
Thus, $\lambda_1+1\le a'\le
\lambda_1+1-\bar\lambda_1$, whence
$\lambda_1$ is even and~$a'=\lambda_1+1$, that is,
$[\phi](r)=[n+2-\lambda_1,n+1]$.
Finally, $n+2-\lambda_1\le n-1$ forces~$\lambda_1\ge 4$, as~$\lambda_1$ is even. Since~$\lambda_{r-1}=\lambda_1\ge\lambda_r$,
$\boldsymbol\lambda\in\mathcal D_r(n+1)$ and~$\phi=\eta_{\boldsymbol\lambda}$
with~$\lambda_1\ge 4$.
\end{proof}
\begin{lemma}\label{lem:Dr Dn}
Up to diagram automorphisms,
the only connected locally injective fully supported homomorphism from $(W(D_{r+1}),\star)$
to~$(W(D_{n+1}),\star)$ is the isomorphism when~$r=n$.
\end{lemma}
\begin{proof}
We first consider the case~$r=3$.
Let~$\phi\in\Hom_{\He}(D_4,D_{n+1})$, $n\ge 3$,
be locally injective, connected and fully supported. Then its restriction to~$(W_{[1,3]}(D_4),\star)
\cong (W(A_3),\star)$ is a locally injective
connected fully supported homomorphism to~$(W_J(D_{n+1}),\star)$
where~$J=[\phi]([1,3])$.
Thus, $J$ is connected, and
so $(W_J(D_{n+1}),\star)\cong (W(A_m),\star)$
for some~$m\ge 3$ by Proposition~\ref{prop:negative}. By using diagram automorphisms of~$W(D_4)$ and of~$W(D_{n+1})$ we may assume, without loss of generality, that~$J=[1,m]$ for some~$3\le m\le n$.
By Theorem~\ref{thm:Ak to An}
and Lemma~\ref{lem:diag aut J_i},
$[\phi](i)=J_i(\boldsymbol{\lambda})$
for some~$\boldsymbol{\lambda}=(\lambda_1,\lambda_1,\lambda_3,\lambda_4)$
with~$2\lambda_1+\lambda_3+\lambda_4=m+1$
and~$\lambda_1\ge\max(\lambda_3,\lambda_4)$.
Again, by Proposition~\ref{prop:negative},
$K=[\phi](\{2,4\})$ must be connected
and~$(W_K(D_{n+1}),\star)$ must
be isomorphic to~$(W(A_l),\star)$ for some~$l\ge 2$. If~$m<n$ then
$\{n,n+1\}\subset[\phi](4)$, which immediately leads to a contradiction
since~$\Gamma_{[\phi](2)\cup [\phi](4)}(D_{n+1})$
then contains a subdiagram of type~$D$. Thus,
$m=n$, $n+1\in[\phi](4)$ and~$n\notin [\phi](4)$. Write~$[\phi](2)=\lambda_4+
[1,2\lambda_1-1]=\lambda_4+[1,\nu_1+\nu_3-1]$,
$[\phi](4)=\lambda_4+[\nu_3,\nu_1+\nu_2+\nu_3-2]\cup\{n+1\}$, $\nu_1+\nu_2+\nu_3=l+1=
n-\lambda_4$
and~$\nu_1\ge \max(\nu_2,\nu_3)$.
Since~$[\phi](3)=[\lambda_1+\lambda_4+1,
n]$ must be orthogonal
to~$[\phi](4)$ by Corollary~\ref{cor:hom A1 x A1},
it follows that~$[\phi](3)=\{n\}$ and~$[\phi](4)=\{n+1\}$ which
forces~$\lambda_1+\lambda_3=2$. Then~$\boldsymbol\lambda=(1,1,1,1)$,
that is, $n=3$ and~$\phi$ is the identity map.

Suppose that~$r>3$. Restricting~$\phi$ to~$(W_{[r-2,r+1]}(D_{r+1}),\star)\cong (W(D_4),\star)$ and using
Proposition~\ref{prop:negative},
we conclude that~$[\phi]([r-2,r+1])$ must be of type~$D$. By the case~$r=3$, this forces~$[\phi]([r-2,r+1])=[n-2,n+1]$ and we may assume, without loss of generality,
that $[\phi](i)=\{i-r+n\}$, $i\in[r-2,r+1]$. Now, the restriction
of~$\phi$ to~$(W_{[1,r-1]}(D_{r+1}),\star)\cong (W(A_{r-1}),\star)$ is a connected locally injective homomorphism to a parabolic submonoid of~$(W(D_{n+1}),\star)$ and so,
by Proposition~\ref{prop:negative},
$L:=[\phi]([1,r-1])$ must be of type~$A$.
Since~$i\in[1,r-2]$ is not adjacent to~$r$ and~$r+1$, $[\phi](i)$ is orthogonal to~$\{n\}$ and~$\{n+1\}$ by Corollary~\ref{cor:hom A1 x A1},
and so~$L=[\phi]([1,r-1])\subset[1,n-1]$. But~$n-1\in L$ and, since~$\phi$ is fully supported, $1\in L$, whence~$L=[1,n-1]$.
Then Theorem~\ref{thm:Ak to An}
implies that, up to diagram
automorphisms, the restriction of~$\phi$ to~$(W_{[1,r-1]}(D_{r+1}),\star)$ is obtained
from~$\psi_{\boldsymbol\lambda}$
with~$\boldsymbol\lambda\in \mathcal A_r(n)$ and~$J_{r-1}(\boldsymbol\lambda)$ is
a singleton. That forces
$\boldsymbol{\lambda}=(1,\dots,1)$,
$|[\phi](i)|=1$
for all~$1\le i\le r+1$ and so~$\phi$ is an isomorphism.
\end{proof}
\begin{corollary}\label{cor:Dr Bn}
A connected fully supported
homomorphism $(W(D_{r+1}),\star)\to
(W(B_n),\star)$ cannot be locally injective.
\end{corollary}
\begin{proof}
Let~$\phi$ be such a homomorphism and let~$\zeta\in\Hom_{\He}(B_n,D_{n+1})$ be given by~\eqref{eq:unfold Bn Dn+1}. We claim that~$[\phi](i)\not=\{n\}$ for all~$i\in[1,r+1]$.
Indeed, otherwise let~$j$ be adjacent to~$i$. By the local injectivity of~$\phi$, $[\phi](j)$ is neither orthogonal to~$\{n\}$ nor contains it, whence~$n-1\in[\phi](j)$
and~$n\notin[\phi](j)$. Then~$L=[\phi](\{i,j\})$ is connected, $(W_L(B_n),\star)$ is of type~$B$, and the restriction of~$\phi$
to~$(W_{\{i,j\}}(D_{r+1}),\star)\cong(W(A_2),\star)$ is an injective connected fully supported homomorphism to~$(W_L(B_n),\star)$, which contradicts
Lemma~\ref{lem:B2 An do not exist}. Since~$\phi$ is locally injective and hence left-total, $[\phi](J)=\{n\}$ for some~$J\subset[1,r+1]$
would force~$[\phi](i)=\{n\}$ for all~$i\in J$, whence $[\phi](J)\not=\{n\}$ for every non-empty~$J\subset[1,r+1]$.
Since~$[\zeta](L)$ is connected for every connected~$L\subset[1,n]$ with~$L\not=\{n\}$ and~$[\zeta\circ\phi]=[\zeta]\circ[\phi]$
by Lemma~\partref{lem:[phi]comp.b}, it follows
that~$\zeta\circ\phi$ is connected.
By Lemmata~\partref{lem:[phi]sets.c} and~\ref{lem:loc inj comp}, $\zeta\circ\phi$ is also fully supported and locally injective, hence
is an isomorphism by Lemma~\ref{lem:Dr Dn}, and~$n=r$. Therefore, $\phi$ must be injective.
Yet~$|W(D_{n+1})|=2^n (n+1)!=
(n+1)|W(B_n)|$, which is a contradiction.
\end{proof}

\subsection{Concluding remarks}
\plink{D'r(m)}Define, for all $r\in\ZZ_{>1}$, $m\in\ZZ_{\ge 0}$,
$$
\mathcal D'_r(m)=\{ \boldsymbol\lambda=(\lambda_1,\dots,\lambda_r)\in\mathcal D_r(m)\,:\, \lambda_r>1\}.
$$
It is easy to see that
$$
\sum_{m\ge 0}|\mathcal D'_r(m)|x^m=\frac{x^{2r}(1+x^{2r-1})}{(1-x^{2r})(1-x^{2r-2})}.
$$
We make the following
\begin{conjecture}\label{conj:families of inj}
\begin{enmalph}
\item\label{conj:families of inj.a} The homomorphisms
\begin{align*}
&\psi_{\boldsymbol\lambda}\in \Hom_{\He}(A_r,A_n),&&\boldsymbol\lambda\in\mathcal A_{r+1}(n+1),\ n>r\ge 3&&\text{(cf. Theorem~\ref{thm:Ak to An})},\\
&\theta_{\boldsymbol\mu}\in\Hom_{\He}(B_r,B_m),&&\boldsymbol\mu\in\mathcal B_r(m),\ m>r\ge 3&&\text{(cf. Theorem~\ref{thm:Bk Bn})},\\
&\eta_{\boldsymbol\nu}\in\Hom_{\He}(B_r,D_{n+1}),&&\boldsymbol{\nu}\in\mathcal D'_r(n+1),\ n>r\ge 3&&\text{(cf. Theorem~\ref{thm:Br D_n+1})}
\end{align*}
are injective.
\item These homomorphisms  together
with homogeneous homomorphisms
from Theorem~\ref{thm:adm finite class},
those from
Propositions~\ref{prop:B2 Bn all} and~\ref{prop:B2 Dn+1},
exhaust, up to isomorphisms, compositions, and decorations
as in Lemma~\ref{lem:decoration},
all injective fully supported homomorphisms
in~$\Hom_{\He}(M',M)$ where
both $M'$ and~$M$ belong to
classical series.
\end{enmalph}
\end{conjecture}
Part~\ref{conj:families of inj.a} of this Conjecture was verified
on a computer for~$3\le r\le 7$ and~$r+1\le n\le 30$.
We have to exclude
$\mathcal D_r(n+1)\setminus \mathcal D'_r(n+1)$ since
for~$\eta_{\boldsymbol\nu}$ with~$\nu_r=1$,
$\eta_{\boldsymbol\nu}(w_{[1,r-2];[1,r]}{}^{-1}s'_r)=\eta_{\boldsymbol\nu}(w_{[1,r-2];[1,r]}{}^{-1})$ (this was checked for $3\le r\le 7$).
One can check that injective
homomorphisms from~$(W(A_2),\star)$ and
$(W(B_2),\star)$ to exceptional types
exist only for the following pairs $(M',M)$, up
to diagram automorphisms and decorations,
$$
\begin{array}{c|c|c|c}
M'&M& [\phi](1)&[\phi](2)\\
\hline
A_2&E_6&[1,4]\cup\{6\}&[2,6]\\
B_2&F_4&[1,3]&[2,4]\\
B_2&E_6&[1,5]&[2,6]\\
B_2&E_7&[1,4]\cup\{7\}&[2,7]\\
B_2&E_7&[1,5]\cup\{7\}&[2,7]\\
B_2&E_8&[1,6]\cup\{8\}&[2,8]
 \end{array}
$$
Note that~$[\phi](i)$, $i\in\{1,2\}$ are
connected for all~$\phi$ in the above list. None of these homomorphisms is parabolic.
We hope that
homomorphisms from Conjecture~\ref{conj:families of inj}, together with homogeneous homomorphisms from Theorem~\ref{thm:adm finite class} and the ones listed above in exceptional types yield the bulk, if not all, solutions of Problem~\ref{prob:injective} for classical series.

\section*{List of symbols}
\def\bqq{{\setbox0\hbox{$\widehat{U}_q^+$}\setbox2\null\ht2\ht0\dp2\dp0\box2}}
\def\hr#1{\hyperlink{#1}{\pageref*{page:#1}}}

\noindent
{
\scriptsize
\begin{tabular}{p{1.55in}@{\bqq}l@{\hskip.25in}p{1.55in}@{\bqq}l@{\hskip.25in}p{1.55in}@{\bqq}l}
$\bar s$&p.~\hr{bar s}&
$[a,b]_2$&p.~\hr{[a,b]2}&
$a+bJ$&p.~\hr{a+bJ}\\
$\mathscr P(S)$&p.~\hr{PS}&
$\ascprod$, $\dscprod$&p.~\hr{ascp}&
$\brd{xy}{m}$&p.~\hr{brd}\\
$\Cox I$&p.~\hr{Cox I}&
$\Gamma(M)$&p.~\hr{Gamma(M)}&
$\Br^+(M)$&p.~\hr{Br+(M)}\\
$\ell$&p.~\hr{ell}&
${}^{op}$&p.~\hr{op}&
$W(M)$&p.~\hr{W(M)}\\
$\pi_M$&p.~\hr{piM}&
$\SQF^+(M)$&p.~\hr{SQF}&
$\Br^+_J(M)$, $W_J(M)$&p.~\hr{Br+J(M)}\\
$\iota_J$&p.~\hr{iotaJ}&
$\mathscr F(M)$&p.~\hr{F(M)}&
$\supp$&p.~\hr{supp}\\
$w_\circ^J$&p.~\hr{w0J}&
$\pi^\star_M$&p.~\hr{pi*M}&
$(W(M),\star)$&p.~\hr{HeMon}\\
$\times,\vdash$&p.~\hr{times}&
$D_L(w)$,$D_R(w)$&p.~\hr{DL(w)}&
${}^K[x]$, $[x]^K$&p.~\hr{parab fact}\\
$\downarrow w$, $\uparrow w$&p.~\hr{dar}&
$\cx ab$, $\cxr ab$&p.~\hr{cab}&
$h(M)$&p.~\hr{h(M)}\\
$w_{J;K}$, $w_J$&p.~\hr{wJ;K}&
$\Co$, $\He$&p.~\hr{A C H}&
$[\phi]$&p.~\hr{[phi]}\\
$\mathscr{R\!e\!l}$&p.~\hr{Rel}&
$\mu_M$&p.~\hr{mu M}&
$\Lambda(M',M)$&p.~\hr{LM'M}\\
$\Theta_\xi$&p.~\hr{Theta xi}&
$M_\xi$, $f_\xi$&p.~\hr{M xi}&
$p_J$&p.~\hr{pJ}\\
$\ell_f$&p.~\hr{ell f}&
$M^{\varpi}$&p.~\hr{MII}&
$\mathbf f_\varpi$&p.~\hr{phi fold}\\
$\mP_K(M)$&p.~\hr{Psbm}&
$\star_J$&p.~\hr{*J}&
$\mathsf C(A)$&p.~\hr{C(A)}\\
$\mathcal A_r(m)$&p.~\hr{Ar(m)}&
$a_i(\boldsymbol\lambda)$&p.~\hr{a_i(l)}&
$J_i(\boldsymbol{\lambda})$&p.~\hr{J_i(l)}\\
$\tilde{\boldsymbol\lambda}$&p.~\hr{tilde l}&
$\psi_{\boldsymbol\lambda}$&p.~\hr{psi l}&
$\mathcal B_r(m)$&p.~\hr{Br(m)}\\
$\theta_{\boldsymbol\mu}$&p.~\hr{theta mu}&
$\mathcal D_r(m)$&p.~\hr{Dr(m)}&
$\eta_{\boldsymbol\lambda}$&p.~\hr{eta l}\\
$\mathcal D'_r(m)$&p.~\hr{D'r(m)}\\
\end{tabular}
}

\renewcommand{\PrintDOI}[1]{DOI: \href{https://dx.doi.org/#1}{#1}}
\renewcommand{\MR}[1]{\relax}

\end{document}